\documentclass[a4,11pt,usenames]{article}
\usepackage{amssymb}
\usepackage{mathrsfs} 
\usepackage{amsmath}
\usepackage{amsthm}
\usepackage{color}
\usepackage{marginnote}
\usepackage{graphicx}
\usepackage{psfrag}
\usepackage[all]{xy}
\usepackage{authblk}
\usepackage{accents}

\DeclareFontFamily{OT1}{pzc}{}
\DeclareFontShape{OT1}{pzc}{m}{it}{<-> s * [1.150] pzcmi7t}{}
\DeclareMathAlphabet{\mathpzc}{OT1}{pzc}{m}{it}

\usepackage[totalwidth=17cm,totalheight=24.2cm]{geometry}
\usepackage{parskip}

\makeatletter
\let\mcnewpage=\newpage
\newcommand{\Trick}{
\renewcommand\newpage{%
        \if@firstcolumn
            \hrule width\linewidth height0pt
                \columnbreak
        \else
                \mcnewpage
        \fi
}
}
\makeatother

\theoremstyle{plain}

\newtheorem{theorem}{Theorem}[section]
\newtheorem{proposition}[theorem]{Proposition}
\newtheorem{corollary}[theorem]{Corollary}
\newtheorem{lemma}[theorem]{Lemma}

\newtheorem{maintheorem}{Theorem}

\newenvironment{mainthm}[1]
{\renewcommand{\themaintheorem}{#1}
\begin{maintheorem}}
{\end{maintheorem}}

\theoremstyle{definition}
\newtheorem{definition}[theorem]{Definition}
\newtheorem{example}[theorem]{Example}
\newtheorem{remark}[theorem]{Remark}
\newtheorem{problem}[theorem]{Problem}

\newtheorem*{notation}{Notation}

\newcommand{\thistheoremname}{}
\newtheorem*{genericthm*}{\thistheoremname}
\newenvironment{namedthm*}[1]
  {\renewcommand{\thistheoremname}{#1}%
   \begin{genericthm*}}
  {\end{genericthm*}}

\makeatletter

\newenvironment{figurehere}
  {\def\@captype{figure}}
  {}
\makeatother

\def\dis{\displaystyle}

\def\Ccal{\mathcal{C}}

\def\PSL{\mathrm{PSL}}

\def\SO{\mathrm{SO}}

\def\Crit{\mathrm{Crit}}

\def\Acrit{\mathrm{ACrit}}

\def\id{\mathrm{id}}

\def\ZZ{\mathbb{Z}}

\def\RR{\mathbb{R}}
\def\CC{\mathbb{C}}

\def\PP{\mathbb{P}}

\def\Sph{\mathbb{S}^2}

\def\DD{\mathbb{D}}

\def\diam{\mathrm{diam}}

\def\Mcal{\text{\footnotesize{$\mathpzc{M}$}}}
\def\ol#1{\overline{#1}}

\def\bm#1{\text{\boldmath$#1$}}

\def\wh#1{\widehat{#1}}
\def\vc#1{{\bm{#1}}}   % vectors

\def\pa{\partial} % boundary/derivations
\def\1{\vc{1}}
\def\0{\vc{0}}
\def\e{\varepsilon}
\def\th{\vartheta}
\def\IM{\mathrm{Im}} % imaginary part

\def\floor#1{\lfloor{#1}\rfloor}

\def\inte#1{\accentset{\circ}{#1}}

\def\Acal{\mathcal{A}}

\def\Bcal{\mathcal{B}}
\def\sys{\mathrm{sys}}
\def\Ext{\mathrm{Ext}}
\def\Ab{\mathrm{Ab}}
\def\area{\mathrm{Area}}
\def\Area{\mathrm{Area}}

\def\Ab{\mathrm{NB}}

\def\Ncal{\mathpzc{N}}

\def\MSPH{\text{\footnotesize{$\mathpzc{MSph}$}}}
\def\Rep{\text{\footnotesize{$\mathpzc{Rep}$}}}

\def\MP{\text{\footnotesize{$\mathpzc{MP}$}}}

\def\rot{\mathrm{rot}}

\def\rar{\rightarrow}
\def\lra{\longrightarrow}

\def\Alg{\mathrm{Alg}}

\def\Vor{\mathcal{V}}
\def\d{\mathpzc{d}}

\def\Kcal{\mathcal{K}}
\def\DV{D^{\Vor}}
\def\DVint{\inte{D}^{\Vor}}
\def\DVbar{\bar{D}^{\Vor}}

\def\lev{\lambda}

\def\q{\mathpzc{q}}
\def\sing{\mathrm{sing}}

\def\W{B^{\max}}
\def\Wint{\inte{B}^{\max}}

\def\P{\mathcal{P}}
\def\Q{\mathcal{Q}}
\def\R{\mathcal{R}}

\def\Ical{\mathcal{I}}

\def\BU{\mathpzc{B}}

\newcommand{\gm}[1]{}

\begin{document}

\title{Spherical surfaces with conical points:
systole inequality
and moduli spaces with many connected components}

\author[1]{Gabriele Mondello}
\affil[1]{{\small{``Sapienza'' Universit\`a di Roma, Department of Mathematics (mondello@mat.uniroma1.it)}}}

\author[2]{Dmitri Panov}
\affil[2]{{\small{King's College London, Department of Mathematics (dmitri.panov@kcl.ac.uk)}}}

\maketitle

\abstract{
\noindent
In this article we address a number of features of the moduli space 
%$\MSPH_{g,n}(\bm{\th})$ 
of spherical metrics on connected, compact, orientable surfaces
%on genus $g$ surfaces
with %$n$ 
conical singularities of assigned angles, %$2\pi\cdot\bm{\th}=2\pi\cdot(\th_1,\dots,\th_n)$,
such as its non-emptiness and connectedness.
We also consider some features of
the forgetful map from the above moduli space of spherical surfaces
with conical points to the associated moduli space of pointed Riemann surfaces,
%$F_{g,n,\bm{\th}}:\MSPH_{g,n}(\bm{\th})\rar\Mcal_{g,n}$
such as its properness, which follows from an explicit systole inequality
that relates metric invariants (spherical systole) and
conformal invariant (extremal systole).
%%%%%%%%%
%%%%%%%%%%%
}

% OLD ABSTRACT
%
%\abstract{\noindent
%In this article we address a number of features of the moduli spaces 
%$\MSPH_{g,n}(\bm{\th})$ of spherical metrics on genus $g$ surfaces
%with $n$ conical singularities of assigned angles $2\pi\cdot\bm{\th}=2\pi\cdot(\th_1,\dots,\th_n)$.
%
%\noindent
%Concerning the global properties of such moduli spaces,
%we first show that in positive genus $\MSPH_{g,n}(\bm{\th})$ is always non-empty provided the obvious
%Gauss-Bonnet constraint is satisfied (Theorem \ref{main:existence}). On the other hand,
%we exhibit examples of $\MSPH_{0,3+m}(\bm{\th})$ that have at least $3^m$ connected components
%(Theorem \ref{main:many}).
%
%\noindent
%Furthermore, we prove the following non-existence result for spherical metrics.
%%Moreover, we analyse some properties 
%%of an apparently different flavour about conformal spherical metrics 
%%on a fixed Riemann surface $(S,J)$.
%For every $(g,n)\neq (0,2),\,(0,3)$ with $n\geq 2$ 
%and for every fixed conformal class
%we produce an open subset of $\bm{\th}\in\RR_{>0}^n$
%such that $\MSPH_{g,n}(\bm{\th})$ is non-empty but it contains no metric in the given conformal class, all $\bm{\th}$'s occurring in the subset we construct having very small $\th_1$
%(Theorem \ref{main:non-existence-small}).
%Such result relies on an explicit systole inequality (Theorem \ref{main:systole}) which relates
%metric invariants (systole) and conformal invariants (extremal systole) of
%spherical surfaces in $\MSPH_{g,n}(\bm{\th})$,
%if $\bm{\th}$ does not belong to a locally finite union of affine hyperplanes of $\RR^n$.
%}

\setcounter{tocdepth}{2}
\tableofcontents

\setlength{\parindent}{0pt}
\setlength{\parskip}{0.2\baselineskip}

\section{Introduction}

%
%The aim of this paper is to investigate certain metric
%and conformal features of surfaces with spherical metrics and conical singularities,
%such as the link between systole and extremal systole.
%We will then apply some of the developed techniques to
%discuss non-existence of spherical metrics with a conical point of small angle in a given conformal class
%and the possibe number of connected components of the moduli space of such spherical metrics.
%
\gm{Abstract is reduced. Most of its former content is moved here, to this new
pre-section. In general, references to bibliography are put in round
parentheses next to the numbering of the theorem.}
%%%%%%%%%%%%%%%%%%%%%%%%%%%%%%
The aim of the present paper is to
investigate certain topological properties
of the moduli space
$\MSPH_{g,n}(\bm{\th})$ of spherical metrics on genus $g$ surfaces
with $n$ conical singularities of angles $2\pi\cdot\bm{\th}=2\pi\cdot(\th_1,\dots,\th_n)$, and of the associated forgetful map
$F_{g,n,\bm{\th}}:\MSPH_{g,n}(\bm{\th})\rar\Mcal_{g,n}$
that sends a spherical surface to its underlying Riemann surface
of genus $g$
with $n$ marked points.

Our main results, labeled as Theorems  \ref{main:existence}-\ref{main:many}-\ref{main:systole}-\ref{main:non-existence-small}-\ref{main:properness}, 
can be summarized as follows:
\begin{itemize}
\item[(\ref{main:existence})]
$\MSPH_{g,n}(\bm{\th})$ is always non-empty 
for $g>0$, provided the obvious
Gauss-Bonnet constraint is satisfied;
% (Theorem \ref{main:existence});
\item[(\ref{main:many})]
for suitable $\bm{\th}$, the moduli space
$\MSPH_{0,3+m}(\bm{\th})$ and its image through the
forgetful map $F_{0,3+m,\bm{\th}}$
have at least $3^m$ connected components; 
%(Theorem \ref{main:many});
\item[(\ref{main:systole})]
if $\bm{\th}$ does not belong to a certain well-understood
locally finite union of affine hyperplanes of $\RR^n$
(at which some bubbling phenomenon can occur),
spherical systole and
extremal systole  of
spherical surfaces in $\MSPH_{g,n}(\bm{\th})$
bound each other
through an explicit systole inequality;
% (Theorem \ref{main:systole})
\item[(\ref{main:non-existence-small})]
for every $(g,n)\neq (0,2),\,(0,3)$ with $n\geq 2$ 
and for every fixed conformal class
there exists an open subset of $\bm{\th}\in\RR_{>0}^n$
such that $\MSPH_{g,n}(\bm{\th})$ is non-empty but it contains no metric in the given conformal class; moreover,
the $\bm{\th}$'s occurring in such open subset 
can have arbitrarily small $\th_1$;
%(Theorem \ref{main:non-existence-small});
\item[(\ref{main:properness})]
under the same (non-bubbling) hypotheses of Theorem \ref{main:systole},
the forgetful map $F_{g,n,\bm{\th}}$ is proper.
%(Theorem \ref{main:properness}).
\end{itemize}
Complete statements can be found in Subsection \ref{sec:intro-main}.\\

We remark that Theorems \ref{main:systole}-\ref{main:non-existence-small}-\ref{main:properness} are genus independent. 

On the contrary,  Theorem \ref{main:existence} only holds for positive genus and it is
false in genus zero, since the simple topology of the sphere imposes
some constraints on the possible monodromy representation of a spherical metric (see Subsection \ref{sec:genus0}). 

The content of Theorem \ref{main:many} is essentially to provide examples of moduli spaces with many connected components. Our construction works in genus zero because $\MSPH_{0,3}\left(\frac{1}{2}+m_1,\frac{1}{2}+m_2,\frac{1}{2}+m_3\right)$ consists of a single point for $m_1,m_2,m_3\in\ZZ_{\geq 0}$, and
because such a point represents a spherical surface with 
discrete (and so finite) monodromy. 
The study of such phenomena of disconnectedness in higher genus
is object of future work.\\

Theorems stated in Subsection \ref{sec:context}, which
are not labeled by letters, are extracted from previous works
of various authors.

%
%The aim of this paper is to prove 
%show that the existence
%of a Riemannian metric of curvature $1$ on a punctured Riemann surface
%with conical singularities of assigned amplitude at the punctures
%may constrain the underlying conformal structure. Such constraints
%arise as the sum of two conical angles is smaller than $2\pi$ and they
%can be explicitly expressed in terms of inequalities that involve these angles
%and the extremal lengths of simple loops.
%
\subsection{Setting}\label{sec:setting}

In this subsection we introduce the notion of spherical surface
with conical points, of (metric) systole and of extremal systole,
and of non-bubbling parameter. Moreover, we recall the definition of
forgetful map from the
moduli space of spherical surfaces with conical points
to the moduli space of Riemann surfaces with marked points.\\

Let $S$ be a compact, connected, oriented surface and let 
$\dot{S}$ be the punctured surface obtained from $S$ by
removing a subset $\bm{x}=(x_1,\dots,x_n)$ of $n$ marked points.

A metric of curvature $1$ on a surface is locally isometric to $\Sph$
and so it can be locally written as $d\rho^2+\sin^2(\rho)d\phi^2$ in polar coordinates $(\rho,\phi)$.
The model for a closed neighborhood of a conical point of angle $2\pi\theta$ inside
a spherical surface will be the disk
$\DD_\theta(r)=\{z=\rho e^{i\phi}\in\CC\,|\,|z|\leq r\}$ with $0<r<\pi$, endowed with the metric 
$d\rho^2+\theta^2 \sin^2(\rho)d\phi^2$.

\begin{definition}[Spherical metric]
A {\it{spherical metric on $S$ with conical singularities}}
at $\bm{x}$ of angles $2\pi\bm{\th}=2\pi(\th_1,\dots,\th_n)$
is a Riemannian metric $h$ of curvature $1$ on $\dot{S}$
such that every $x_i$ has a closed neighbourhood isometric
to $\DD_{\th_i}(r)$ for some small $r>0$.
We will call $(S,\bm{x},h)$ a {\it{spherical surface}} and each $x_i$ a {\it{conical point}}.
\end{definition}

\begin{remark}[Developing map and monodromy representation]
For every spherical surface $(S,\bm{x},h)$
we can construct
a locally isometric {\it{developing map}} $\iota:(\widetilde{\dot{S}},\tilde{h})\rar \Sph$
from the universal cover $\widetilde{\dot{S}}$ of $\dot{S}$
(endowed with the pull-back metric $\tilde{h}$), which
is equivariant with respect to
a {\it{monodromy representation}}
$\rho:\pi_1(\dot{S})\rar \SO_3(\RR)$.
The couple $(\iota,\rho)$ is well-defined up to
the action $\Q\cdot (\iota,\rho)=(\Q\circ\iota,\Q\rho\Q^{-1})$ by elements $\Q\in\SO_3(\RR)$.
\end{remark}

Each spherical surface $(S,h)$ admits a triangulation, where all edges are geodesic segments and all triangles are isometric to spherical triangles with respect to the induced metric $h$.
In other words, each spherical surface is a polyhedral space (see \cite[Section 8]{Petrunin} for a nice  introduction to the topic). In this description all conical points sit at some vertices of the triangulation. 

\begin{notation}
If $p\in S$ is any point in the spherical surface $S$ and $r>0$,
we denote by $B_p(r)$ the closed ball centered at $p$ of radius $r$,
namely the locus of points at distance at most $r$ from $p$.
\end{notation}

The diameter of a spherical surface with $n$ conical points
is at most $\pi(n+1)$ (see Lemma \ref{sphdiameter}). For this reason,
a sequence of spherical metrics degenerates if and only if 
two (or more) conical points clash or if a geodesic loop based at a conical point shrinks to a singular point.
Thus, the following quantity can be interpreted as a measure of how far the metric is from degenerating.

%Since the diameter of a spherical metric
%on such a surface 
%is at most $\pi(n+1)$ (see Lemma \ref{sphdiameter}),
%a sequence of 
%spherical metrics degenerates if and only if 
%two (or more) conical points clash or if a geodesic loop based at a conical point shrinks to a singular point.
%Thus, the following quantity can be interpreted as a measure of how far the metric is from degenerating.

\begin{definition}[Spherical systole]
The {\it{systole}} $\sys(S,\bm{x})$ of a spherical surface $(S,\bm{x})$
is the supremum of all $r>0$ for which the $n$ balls $B_{x_i}(r)$ are pairwise disjoint
and each $B_{x_i}(r)$ is isometric to $\DD_{\th_i}(r)$.
\end{definition}

We recall that the datum of a conformal class of metrics on an oriented surface $S$
is equivalent to giving an almost-complex structure $J:TS\rar TS$, and so to 
giving a complex structure by Gauss's existence of isothermal coordinates.
In what follows, we will denote a Riemann surface by $(S,J)$.

The conformal counterpart to the spherical systole, 
which measures how far the underlying conformal structure on a surface is from
degenerating, is introduced in Definition
\ref{def:extremal-systole} below.
The definition and some basic properties
of the extremal length and of the conformal modulus
of a cylinder are summarized in Appendix \ref{sec:appendix}.

\begin{definition}[Extremal systole]\label{def:extremal-systole}
Let $(S,J)$ be a Riemann surface with marked points $\bm{x}$ and assume that $\dot{S}$ is not a sphere with at most $3$ punctures.
The {\it{extremal systole}} $\Ext\sys(\dot{S})$ of $\dot{S}$ is the infimum of
the extremal lengths $\Ext_\gamma(\dot{S})$, where $\gamma\subset\dot{S}$
is a simple closed essential curve.
\end{definition}

We will use the notation $\Ext\sys(\dot{S},J)$ whenever we want to emphasize the dependence on $J$.
Clearly, if $(S,\bm{x})$ is a spherical surface, then $\Ext\sys(\dot{S})$ must be understood as the extremal systole
of the underlying conformal structure.\\

Next, we introduce a quantity that depends only on the topology of $(S,\bm{x})$ and on 
the {\it{angle vector}} $\bm{\th}$.
Consider the subset $\Crit_{\bm{\th}}\subset\RR$ defined as 
\[
\Crit_{\bm{\th}}:=\left\{
\|\bm{\th}_I\|_1-\|\bm{\th}_{I^c}\|_1+2b
\ \Big|\
I\subsetneq\{1,2,\dots,n\},\ b\in \ZZ_{\geq 0}
\right\}
\]
where $\|\bm{\th}_I\|_1:=\sum_{i\in I}\th_i$ and $I^c=\{1,2,\dots,n\}\setminus I$.
%(for $I=\{1,2,\dots,n\}$,
%the quantity $\|\bm{\th}_I\|_1$ will be also denoted by $\|\bm{\th}\|_1:=\th_1+\dots+\th_n$).

\begin{definition}[Non-bubbling parameter]
The {\it{non-bubbling parameter}} of 
$(S,\bm{x},\bm{\th})$ is
\[
\Ab_{\bm{\th}}(S,\bm{x}):=d_{\RR}\Big(\chi(\dot{S}),\Crit_{\bm{\th}}\Big)
\]
where $d_\RR$ is the usual distance as subsets of $\RR$.
\end{definition}

\begin{remark}
We stress that
\[
\Ab_{\bm{\th}}(S,\bm{x})=0
\iff
\chi(\dot{S})+\sum_{i=1}^n(\pm\th_i)\in 2\ZZ_{\geq 0}
\]
for a suitable choice of the signs.
\end{remark}

We will explain in Section \ref{sec:bubbling} that
the condition $\Ab_{\bm{\th}}(S,\bm{x})\geq\e$
prevents spherical metrics on $(S,\bm{x})$ with
conical angles $2\pi\bm{\th}$ from being
geometrically close to a bouquet of bubbles
(as in Figure \ref{fig:nearly-complete-bubbling-intro}).

Another feature of the non-bubbling parameter
is that it provides an obstruction 
for spherical surfaces
to having degenerate monodromy representation in the following sense.

\begin{definition}[Coaxial representation]
A representation $\rho$ in $\SO_3(\RR)$
is {\it{coaxial}} if it takes values in a $1$-parameter subgroup of
$\SO_3(\RR)$. We say that a metric has {\it{coaxial monodromy}}
(or that it is a {\it{coaxial metric}})
if its monodromy representation is coaxial.
\end{definition}

The link between non-bubbling parameter
and coaxiality is the following.

\begin{namedthm*}{Lemma \ref{lemma:coaxial-zero-NB}}[Coaxial metrics have vanishing non-bubbling parameter]
Let $(S,\bm{x})$ be a spherical surface with conical points of angles
$2\pi\cdot\bm{\th}$. If $(S,\bm{x})$ has coaxial monodromy,
then $\Ab_{\bm{\th}}(S,\bm{x})=0$.
\end{namedthm*}

The above lemma can be proven in more than one way.
In Section \ref{sec:systole} we will derive it from Theorem \ref{main:systole}.

%Thus, fixed $(S,\bm{x})$, the space $\RR^n_{>0}$ of $\bm{\th}$ is subvided into
%connected components ({\it{chambers}}) by the collection of affine hyperplanes ({\it{walls}})
%described by $\Ab_{\bm{\th}}(S,\bm{x})=0$.

\begin{notation}
Since $\Ab_{\bm{\th}}(S,\bm{x})$ 
only depends on $\bm{\th}$ and on the topology of the surface $S$ of genus $g$ with $n$ marked points, 
such quantity will be also denoted by $\Ab_{\bm{\th}}(g,n)$.
\end{notation}

Finally, we  recall
the definition of moduli spaces of Riemann surfaces and of spherical surfaces.

\begin{definition}[Moduli space of Riemann surfaces]\label{modulidef}
Let $g,n\geq 0$ with $2g-2+n>0$.
The {\it{moduli space $\Mcal_{g,n}$ of Riemann surfaces}} of genus $g$ with $n$ marked points
is the space of isomorphism classes $[S,\bm{x},J]$,
where $J$ is a complex structure on the connected oriented surface $S$ of genus $g$
and $\bm{x}=(x_1,\dots,x_n)$ is a collection of $n$ distinct points on $S$.
\end{definition}

We recall that $\Mcal_{g,n}$ is a complex-analytic connected orbifold
of complex dimension $3g-3+n$.

\begin{definition}[Moduli space of spherical surfaces]\label{def:moduli-spherical}
Let $g,n\geq 0$ with $2g-2+n>0$ and let $\bm{\th}\in\RR^n_{>0}$.
The {\it{moduli space $\MSPH_{g,n}(\bm{\th})$ of spherical surfaces}} of genus $g$ with $n$ conical
singularities of angles $2\pi\cdot\bm{\th}$ is the space of isometry classes
$[S,\bm{x},h]$, where $h$ is a spherical metric on the connected oriented surface $S$
of genus $g$ with conical singularity at $x_i$ of angle $2\pi\th_i$ for $i=1,\dots,n$.
\end{definition}

In the case $\Ab_{\bm{\th}}(S,\bm{x})>0$,
it can be shown that $\MSPH_{g,n}(\bm{\th})$ is locally the quotient of
a real-analytic variety of dimension $6g-6+2n$ by a finite group
(see Section \ref{sec:moduli-spaces}).

The procedure of forgetting the spherical metric $h$
on the surface $(S,\bm{x})$ and remembering only
the underlying conformal structure $J$ determines a real-analytic 
{\it{forgetful map}}
\[
F_{g,n,\bm{\th}}:\MSPH_{g,n}(\bm{\th})\lra \Mcal_{g,n}
\]
defined as $F_{g,n,\bm{\th}}[S,\bm{x},h]:=[S,\bm{x},J]$.

\subsection{Main results}\label{sec:intro-main}

\subsubsection{Components of the moduli space of spherical surfaces}

Our first result is the existence of a spherical metric in some conformal class,
provided $g>0$ and the obvious Gauss-Bonnet constraint is satisfied.
The constraint can be stated as follows: if $S$ supports a spherical metric with $n$ conical singularities
of angles $2\pi\bm{\th}$, its area
\[
\frac{1}{2\pi}\area(S)=\chi(S,\bm{\th}):=\chi(\dot{S})+\sum_i\th_i
\]
must be positive. 

\begin{mainthm}{A}[Existence of spherical metrics in positive genus]\label{main:existence}
Let $g,n>0$ and let $\bm{\th}\in\RR^n_{>0}$ such that $\chi(S,\bm{\th})>0$.
Then the moduli space $\MSPH_{g,n}(\bm{\th})$ is non-empty.
\end{mainthm}

In \cite{mondello-panov:constraints} we proved that a conical singularity can be split
into several conical points, provided the metric is angle-deformable.
The above statement is a simple consequence of such result.

\begin{remark}
Theorem \ref{main:existence}
can be contrasted with what happens for genus $0$
surfaces (see Theorem \ref{thm:mondello-panov}
in Section \ref{sec:genus0}).
\end{remark}

At present very little seems to be known about the number of connected components
of $\MSPH_{g,n}(\bm{\th})$. In particular, we are not aware of any special case prior to this work
in which such moduli space is shown to be disconnected.

In our next result we produce examples in which $\MSPH_{g,n}(\bm{\th})$
has an arbitrarily large number of components.
Moreover each component parametrizes surfaces with quite a different conformal type.

\begin{mainthm}{B}[The moduli space of spherical surfaces can have many components]\label{main:many}
Let $m\geq 0$. Given integers $m_1,m_2,m_3\geq m$
and $\e_1,\dots,\e_m\in \left(0,\frac{1}{2m+2}\right)$, set $\bm{\th}=\left(\frac{1}{2}+m_1,\frac{1}{2}+m_2,
\frac{1}{2}+m_3,\e_1,\dots,\e_m\right)$.
Then
\begin{itemize}
\item[(a)]
the moduli space $\MSPH_{0,m+3}(\bm{\th})$ has at least $3^m$ connected components;
\item[(b)]
if $\max\{\e_j\}\leq \frac{1}{16}\exp\left(-2\pi\cdot\max\{m_i\}\right)$, then the image of
$F_{0,m+3,\bm{\th}}:\MSPH_{0,m+3}(\bm{\th})\rar\Mcal_{0,m+3}$ has at least $3^m$ connected components.
\end{itemize}
\end{mainthm}

\begin{remark}
The same statement holds for $(\th_1,\th_2,\th_3)$ ranging in
a small neighbourhood of the point $\left(\frac{1}{2}+m_1,\frac{1}{2}+m_2,
\frac{1}{2}+m_3\right)$ inside $\RR^3$. 
\end{remark}

The space of spherical polygons 
(i.e.~spherical disks with piecewise geodesic boundary)
can be thought of as a ``real'' locus of the
moduli space of spherical surfaces of genus $0$.
A disconnectedness result for certain moduli spaces of quadrilaterals
can be found in Eremenko-Gabrielov \cite{EG:curvature1}. More examples of such behaviour
can be also found in Eremenko-Gabrielov-Tarasov \cite{EGT:four}.

\subsubsection{Systole inequality and surfaces with one conical point of small angle}

By the very definition of extremal length (see Appendix \ref{sec:appendix}),
a lower bound for the spherical systole of $(S,\bm{x})$ always implies the following lower bound for its extremal systole
\[
\Ext\sys(\dot{S})\geq
\frac{\inf_{\gamma}\ell(\gamma)^2}{\Area(S)}\geq \frac{\left(2\,\sys(S,\bm{x})\right)^2}{2\pi\left(\chi(\dot{S})+\|\bm{\th}\|_1\right)}
=\frac{2\,\sys(S,\bm{x})^2}{\pi\left(\chi(\dot{S})+\|\bm{\th}\|_1\right)}.
\]
where $\gamma$ ranges over all essential simple closed curves in $\dot{S}$
and where $\|\bm{\th}\|_1:=\th_1+\dots+\th_n$.
Here we are using the inequality
$\ell(\gamma)>2\sys(S,\bm{x})$, which is proven in Lemma \ref{birkhoff}.

Conversely, our next main result provides a lower bound for $\sys(S,\bm{x})$ in terms of $\Ext\sys(\dot{S})$,
as long as the angle vector $\bm{\th}$ does not satisfy $\Ab_{\bm{\th}}(g,n)=0$.

\begin{mainthm}{C}[Systole inequality]\label{main:systole}
Let $S$ be a surface with spherical metric and conical singularities 
at $\bm{x}$ of angles $2\pi\bm{\th}$.
Assume that $\chi(\dot{S})<0$ and $\dot{S}$ is not a $3$-punctured sphere.\\
Suppose that there exists $\varepsilon\in\left(0,\frac{1}{2}\right)$
such that
\[
\Ab_{\bm{\th}}(S,\bm{x})\geq \varepsilon.
\]
Then
\[
\Ext\sys(\dot{S})\geq \frac{2\pi\|\bm{\th}\|_1}{\log(1/\e)}
\quad
\text{implies}\quad
\sys(S,\bm{x})\ge \left(\frac{\varepsilon}{4\pi\|\bm{\th}\|_1}\right)^{-3\chi(\dot S)+1}
\hspace{-1cm}.
\]
\end{mainthm}

As a first application of the above result, we are able to prove non-existence
of spherical metrics in a given conformal class with one sufficiently small angle.

\begin{mainthm}{D}[Non-existence of spherical metrics with one small angle]\label{main:non-existence-small}
Let $(S,J)$ be a Riemann surface with $n$ marked points $\bm{x}$ and assume that $\chi(\dot{S})<-1$.
Let $\bm{\hat{\th}}=(0,\th_2,\dots,\th_n)$ with $\th_2,\dots,\th_n>0$
and suppose that
\begin{itemize}
\item[(i)]
$\chi(S,\bm{\hat{\th}})> 0$
\item[(ii)]
$\Ab_{\bm{\hat{\th}}}(S,\bm{x})>0$.
\end{itemize}
Then there is a $\th_1^\star\in (0,10^{-6})$ that depends only on
$\Ext\sys(\dot{S},J)$, $\|\bm{\hat{\th}}\|_1$, $\Ab_{\bm{\hat{\th}}}(S,\bm{x})$ and $\chi(\dot{S})$
such that there exists no spherical metric on $S$ with angles $2\pi\bm{\th}$ at $\bm{x}$
and underlying conformal structure $J$ for any $\th_1<\th_1^\star$.
\end{mainthm}

%I THINK THAT THE ABOVE REMARK SHOULD REPLACE THE TWO LINES BELOW.
%
%Clearly, the above statement is not very interesting if $\chi(S,\bm{\th})<0$ by
%Gauss-Bonnet.
%Also,
%the cases $\chi(\dot{S})=-1$ can be easily analyzed using holonomy considerations
%(for $g=0$) and again
%Gauss-Bonnet (for $g=1$).

\begin{remark}
In Theorem \ref{main:non-existence-small} it is possible to take
\[
\th_1^\star=
%\begin{cases}
%1 & \text{if $\Ab_{\bm{\hat{\th}}}(S,\bm{x})\geq 1$}\\
\frac{1}{\pi}\left(\frac{\e}{\pi(1+4\|\bm{\hat{\th}}\|_1)}\right)^{1-3\chi} 
%& \text{otherwise}
%\end{cases}
\]
where $\chi=\chi(\dot{S})$ and
\[
\e=\min\left\{
\frac{1}{2}\Ab_{\bm{\hat{\th}}}(S,\bm{x}),
\quad
\exp\left(
\frac{-\pi(1+2\|\bm{\hat{\th}}\|_1)}{\Ext\sys(\dot{S},J)}
\right)
\right\}.
\]
\end{remark}

\begin{remark} 
Note that, by Theorem \ref{main:existence}, for any 
$g>0$, any $\bm{\hat{\th}}=(0,\th_2,\dots,\th_n)$ satisfying  $\chi(S,\bm{\hat{\th}})\geq 0$ and  for any choice of $\th_1>0$, we can construct a spherical metric on $S$ with angles $2\pi\bm{\th}=2\pi(\th_1,\th_2,\dots,\th_n)$. 
Hence, Theorem \ref{main:non-existence-small} 
provides us
an open subset of $\RR^n$
such that, for every fixed $\bm{\th}$ in such subset,
the existence of a spherical metric with conical points of angles $2\pi\bm{\th}$ 
depends on the conformal structure of the surface.
Moreover, the smaller is $\th_1$ the smaller is the subset of conformal structures for which the metric exists.
\end{remark}

Another application of Theorem \ref{main:systole} emphasizes the relation between metric and conformal invariants of a spherical surface.

\begin{mainthm}{E}[Properness of the forgetful map]\label{main:properness}
Let $g,n\geq 0$ with $2g-2+n>0$ and let $\bm{\th}\in\RR^n_{>0}$ such that
$\Ab_{\bm{\th}}(g,n)>0$.
Then the forgetful map $F_{g,n,\bm{\th}}$ is proper.
\end{mainthm}

%The case of $S$ of genus $0$ with $3$ conical points holds, since
%it is known by Eremenko \cite{eremenko:three}
%that for certain values of $\bm{\th}$
%there exists a unique metric spherical metric in each conformal class,
%and for the other values of $\bm{\th}$ the moduli space $\MSPH_{g,n}(\bm{\th})$ is empty.

%%%%%%%%%%%%%%%%%%%%%%%%%%%%%%%%%%%%%%%%%%%%%%%%%%%%
%%%%%%%%%%%%%%%%%%%%%%%%%%%%%%%%%%%%%%%%%%%%%%%%%%%%%

\subsection{Context and known results}\label{sec:context}

In this section we recall some relevant results 
in the literature on spherical surfaces with conical points
that relate to our main results.\\

Consider metrics of constant curvature on a surface of genus $g$.
Up to rescaling, we can always assume that such metrics are
{\it{hyperbolic}} (curvature $-1$), {\it{spherical}} (curvature $1$) or {\it{flat}} (curvature $0$) metrics.
Moreover, the sign of the curvature
must agree with the sign of 
$\chi(S,\bm{\th})=\chi(\dot{S})+\|\bm{\th}\|_1$ by Gauss-Bonnet.
%We will always assume that this is the case.

Generalizing the definition given in the Section \ref{sec:setting}, a 
hyperbolic/flat/spherical metric $h$ on $S$ of
has a conical singularity at the point $x_i\in S$ of angle $2\pi\th_i$ if
\[
h=\begin{cases}
d\rho^2+\th_i^2 \rho^2 d\phi^2 & \text{in the flat case}\\
d\rho^2+\th_i^2 \sinh^2(\rho) d\phi^2 & \text{in the hyperbolic case}\\
d\rho^2+\th_i^2 \sin^2(\rho) d\phi^2 & \text{in the spherical case}
\end{cases}
\]
with respect to local polar coordinates $(\rho,\phi)$ centered at $x_i$.

Here is one of the fundamental problems in the theory.

\begin{problem}[Spherical metrics in a fixed conformal class]\label{problem:main}
Fix $\bm{\th}=(\th_1,\dots,\th_n)\in\RR^n_{>0}$
and a Riemann surface $(S,J)$ of genus $g$ with $n$ marked points $\bm{x}$.
Study the set of $J$-conformal spherical metrics on
$S$ with conical singularities at $\bm{x}$ of angles $2\pi\cdot\bm{\th}$, namely
the fiber $F_{g,n,\bm{\th}}^{-1}[S,\bm{x},J]$ of
the forgetful map $F_{g,n,\bm{\th}}:\MSPH_{g,n}(\bm{\th})\rar\Mcal_{g,n}$.
\end{problem}

From this point of view, the existence of a spherical metric 
in each conformal class
can be rephrased in terms of surjectivity of $F_{g,n,\bm{\th}}$.
As an example, such surjectivity is verified whenever
it is possible to define a degree of $F_{g,n,\bm{\th}}$
and this degree is nonzero.

\subsubsection{The unpunctured case}
Consider first the case $n=0$.
In genus $0$ it is well-known that there is a $\PSL_2(\CC)/\SO_3(\RR)$ family
of conformally equivalent spherical metrics in each conformal class.
In genus $1$, Riemann surfaces are isomorphic to flat tori $\CC/\Lambda$, which
thus admit a unique flat metric up to rescaling.
For genus greater than $1$ the celebrated uniformization
theorem by Koebe \cite{koebe:I} \cite{koebe:II} and Poincar\'e \cite{poincare:uniformization}
states that there exists a unique hyperbolic in each
conformal class.\\

Since the situation in the unpunctured case is clear, we will assume now on that $n>0$.

\subsubsection{Existence of conformal metrics of constant curvature in the subcritical case}

First of all, we wish to recall the following important existence result by Troyanov,
which we only state in the case of constant curvature.

\begin{theorem}[Existence of conformal metrics of constant $K$ in the subcritical case \cite{troyanov:conical}]\label{thm:troyanov}
Let $(S,J)$ be a compact connected Riemann surface, $\bm{x}=(x_1,\dots,x_n)$
be a subset of $n\geq 0$ distinct points of $S$ and let $\bm{\th}=(\th_1,\dots,\th_n)\in\RR^n_{>0}$.
Assume that
\[
\chi(S,\bm{\th})<\tau(S,\bm{\th}):=2\cdot\min\{\th_1,\,\th_2,\dots,\,\th_n,\,1\}.
\]
Then there exists a conformal metric on $(S,J)$ with constant curvature
and conical singularities at $\bm{x}$ of angle $2\pi\cdot\bm{\th}$. Moreover,
if $\chi(S,\bm{\th})\leq 0$, then such metric is unique up to rescaling.
\end{theorem}

%Assume first that $\chi(S,\bm{\th})\leq 0$.
%
%Theorem \ref{thm:troyanov} then ensures the existence
%of a metric of constant nonpositive curvature in each conformal class,
%since the above hypotheses are automatically satisfied.

%
\gm{This paragraph ``The approach taken....'' is added.}
%%%%%%%%%%%%%%%%%%%%%%%%%
The approach taken by Troyanov is analytic. He fixes a background metric $h_0$
on $S$ in the given conformal class and with the prescribed conical behaviour at $\bm{x}$
and he looks for a conformal factor
$u:\dot{S}\rar \RR$ such that $h=e^{2u}h_0$ has wished curvature function $K:\dot{S}\rar\RR$
(in our case, the constant function $K=1$).
This translates into a Liouville equation $\Delta_{h_0}u=K e^{2u}-K_{h_0}$
and in turn, following an idea already contained in \cite{berger:riemannian}, into a variational problem.
More precisely, Troyanov consider the functionals $\mathcal{F},\mathcal{G}:\ol{W}^{1,2}(S,h_0)\rar\RR$
on the space $\ol{W}^{1,2}(S,h_0)$ of $W^{1,2}$ functions on $(S,h_0)$ with 
zero mean defined 
as $\mathcal{F}(u):=\int_S \left(\|du\|_{h_0}^2+K_{h_0}u\right)dA_{h_0}$
and $\mathcal{G}(u):=\int_S K e^{2u}dA_{h_0}$.
The key step to show the existence of a solution is to
prove that $\mathcal{F}$ is coercive
on each level $\mathcal{G}^{-1}(c)$. 
Under the subcritical hypothesis $\chi(S,\bm{\th})<\tau(S,\bm{\th})$,
this is achieved through the Moser-Trudinger inequality \cite{moser:sharp}.

The uniqueness result for $\chi(S,\bm{\th})\leq 0$
relies on an application of the maximum principle
and it had been previously proven by Troyanov himself \cite{troyanov:euclidean} in a different
way for $\chi(S,\bm{\th})=0$ and by McOwen \cite{mcowen:conical} for $\chi(S,\bm{\th})<0$.
As a consequence, 
the analogue of Problem \ref{problem:main}
for metrics of constant curvature is
trivially solved
for $\chi(S,\bm{\th})\leq 0$.
%In this case, it can be also easily shown that $F_{g,n,%\bm{\th}}$ is submersive and so a diffeomorphism.
\\

In what follows we will consider only the case $\chi(S,\bm{\th})>0$.

%%%%%%%%%%%%%%%%%%%%%%%%%

%We are going to illustrate one reason for which the spherical case
%can be very different from the case of non-positive curvature.

%%%%%%%%%%%

\subsubsection{Bubbling phenomenon in positive curvature}\label{sec:bubbling} % and Theorems \ref{main:systole} and \ref{main:properness}}

\gm{This section has been expanded.}
%%%%%%%%%%%%%%%%%%%%%%%%

The purpose of this subsection is to
discuss a typical phenomenon of positive curvature:
%, which thus does not occur in the case non-positive curvature,
the existence of degenerating sequences of metrics, with bounded (or even fixed) underlying
conformal structures. \\
%Such
%sequences give rise to phenomena of concentration of area
%at finitely many points.

A local example of such behaviour can be obtained
by fixing $\theta>0$ and
by considering
for every integer $m\geq 1$
the metric 
$\left(\frac{2m\theta |z|^{\theta-1}|dz|}{1+m^2|z|^{2\theta}} \right)^2
=f_m^*\left(\frac{2|dw|}{1+|w|^2}\right)^2$
on the complex plane $\Pi=\CC$ obtained by pulling back
the standard spherical metric $\left(\frac{2|dw|}{1+|w|^2}\right)^2$
through the (possibly multi-valued) map $f_m:\Pi\rar\CC$ defined as
$f_m(z)=mz^\theta$. The above metrics 
on $\Pi$ are all conformally equivalent
to one another and their area concentrates at the (possibly conical) point $z=0$
as $m\rar +\infty$.

It is known since \cite{brezis-merle:blowup} that, if a sequence of conformal spherical metrics 
$(h_m)$ on a fixed Riemann surface $(S,J)$ is 
not bounded, then (up to subsequences)
the area of $h_m$ concentrates at finitely many points of the surface
and it goes to zero elsewhere as $m\rar\infty$.

This phenomenon can be prevented by requiring that the non-bubbling 
parameter $\Ab_{\bm{\th}}(S,\bm{x})$ remains strictly positive, as in 
the following result by Bartolucci-Tarantello,
which we only state in the constant curvature case.

%Theorem \ref{thm:BT} below, which appears in \cite{BDMM:supercritical} and which we state in the constant
%curvature case.
%Such result relies on the local analysis 
%carried over by Nagasaki-Suzuki in \cite{nagasaki-suzuki:asymptotic}, Suzuki in \cite{suzuki:introduction}
%and Br\'ezis-Merle in \cite{brezis-merle:blowup}
%for smooth planar domains,
%and by Bartolucci-Tarantello in \cite{bartolucci-tarantello:compactness}
%for planar domains with conical singularities
%(a similar local analysis for $\th_i\in(0,1)$
%is performed by Bartolucci-Montefusco in \cite{bartolucci-montefusco:compactness}).

\begin{theorem}[Compactness of the space of spherical metrics in a 
conformal class {\cite{bartolucci-tarantello:compactness}}]\label{thm:BT}
Given a Riemann surface $(S,J)$ of genus $g$ with $n$ marked points $\bm{x}$
(with $2g-2+n>0$) 
and an angle vector $\bm{\th}\in (1,+\infty)^n$
such that $\Ab_{\bm{\th}}(g,n)>0$,
the space of conformal spherical metrics on $S$
with conical singularities at $\bm{x}$ of angles $2\pi\bm{\th}$ is compact.
\end{theorem}

The key point of the above result is to show that, if a sequence of spherical metrics
$(h_m)$ on $S$ is blowing up at the point $p$, then the
area that concentrates at $p$ is
exactly $4\pi\th_i$ if $p=x_i$ and it is $4\pi$ if $p$ is not a conical point.
This is achieved essentially by comparing $h_m$ to a model metric near the blow-up points.\\

The above result is implied and extended to all possible values of $\bm{\th}$
by Theorem \ref{main:systole}, which can be in fact
seen as a geometric and quantitative counterpart to Theorem \ref{thm:BT}.
The properness of the forgetful map (Theorem \ref{main:properness}) is also a qualitative consequence of
Theorem \ref{main:systole}.\\

In order to understand the geometric meaning of the assumption $\Ab_{\bm{\th}}(g,n)>0$, let us consider a special class of spherical surfaces,
which we call {\it{almost bubbling}} (see Figure 
\ref{fig:nearly-complete-bubbling-intro}
%\ref{fig:almost-bubbling-intro}
and Definition \ref{def:bubbling}).

\begin{center}
\begin{figurehere}
\psfrag{S'}{$\textcolor{Purple}{S^c}$}
\psfrag{S}{$\textcolor{blue}{S}$}
\psfrag{S4}{$\textcolor{blue}{\!\!\BU^1_4}$}
\psfrag{S2}{$\textcolor{blue}{\BU^1_2}$}
\psfrag{B1}{$\textcolor{blue}{\BU^0_1}$}
\psfrag{x2}{$\textcolor{Sepia}{x_2}$}
\psfrag{x4}{$\textcolor{Sepia}{x_4}$}
\includegraphics[width=0.4\textwidth]{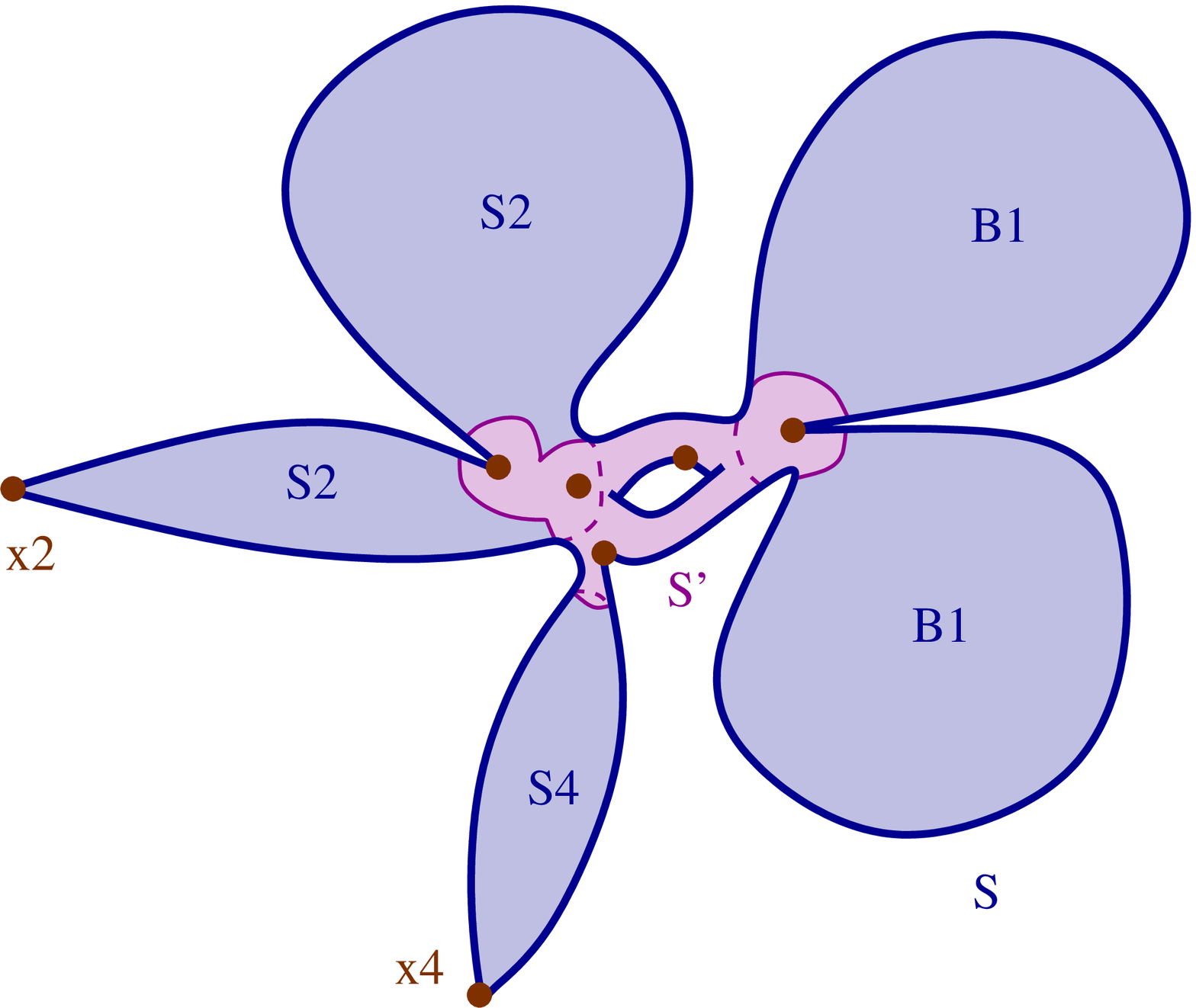}
\caption{{\small An example of almost bubbling spherical surface.}}\label{fig:nearly-complete-bubbling-intro}
\end{figurehere}
\end{center}

%
%\begin{center}
%\begin{figurehere}
%\psfrag{S'}{$\textcolor{Purple}{S^c}$}
%\psfrag{S}{$\textcolor{blue}{S}$}
%\psfrag{S4}{$\textcolor{blue}{\!\!\BU^1_4}$}
%\psfrag{S2}{$\textcolor{blue}{\BU^1_2}$}
%\psfrag{B1}{$\textcolor{blue}{\BU^0_1}$}
%\psfrag{B2}{$\textcolor{blue}{\BU^0_2}$}
%\psfrag{B3}{$\textcolor{blue}{\BU^0_3}$}
%\psfrag{x2}{$\textcolor{Sepia}{x_2}$}
%\psfrag{x4}{$\textcolor{Sepia}{x_4}$}
%\includegraphics[width=0.4\textwidth]{almost-bubbling-intro.eps}
%\caption{{\small An example of almost bubbling surface $S$.}}\label{fig:almost-bubbling-intro}
%\end{figurehere}
%\end{center}

Informally speaking, a spherical surface $S$ is almost bubbling
if there exists
a subset $I\subseteq\{1,2,\dots,n\}$ such that
$S$ can be partitioned into the disjoint union of
\begin{itemize}
\item
finitely many disks $\BU^0_j$ with no conical points,
\item
a disk $\BU^1_i$ 
with exactly
one conical point $x_i$ for every $i\in I$, and
\item
a connected subsurface $S^c$ (which will be called the {\it{core}}),
\end{itemize}
so that the core has small area
and the disks (which will be also called {\it{bubbles}})
have short boundary compared to their size.
In particular, we will say that $S$ is
{\it{$\e$-bubbling}} if the area of $S^c$ and the lengths
of the boundaries of the disks can be estimated in terms of $\e$ in a precise way (see Section \ref{sec:almost-bubbling} for a formal definition).

In Section \ref{sec:systole}
we will construct a partition as above by taking
$S^c$ equal to the locus of points that sit at distance at most $r$
from $\bm{x}\setminus\{x_i\,|\,i\in I\}$
for some $r<\pi/2$: under appropriate hypotheses,
the surface $S$ will be almost bubbling with core given by such $S^c$ ({\it{Voronoi core}})
and the boundary of $S^c$ will be piecewise smooth of constant extrinsic curvature.

A feature of $\e$-bubbling surfaces 
with conical singularities of angles $2\pi\bm{\th}$
is that 
they have non-bubbling parameter smaller than $\e$
(Theorem \ref{thm:almost-bubbling}) or,
equivalently, that
the value $\frac{1}{2\pi}\Area(S)$ is
at distance smaller than $\e$ from the subset
\[
\Acrit_{\bm{\th}}=\left\{2b+2\sum_{i\in I}\th_i\ \big|\ b\in\ZZ_{\geq 0},\ I\subseteq\{1,2,\dots,n\}\right\}
\]
inside $\RR_{\geq 0}$.

The above assertion 
depends on the fact that each bubble $\BU^0_j$ has area approximately $4\pi b^0_j$ and
each bubble $\BU^1_i$ has area approximately $2\pi(1+\th_i)+4\pi b^1_i$
for certain integers $b^0_j,b^1_i\geq 0$.
A quantitative estimate in terms of $\e$
of the area of the bubbles $\BU^0_j$ and $\BU^1_i$
is obtained in Theorem \ref{diskstructure}.
%
%achieved by comparing them to appropriate model bubbles.
%More precisely, 
%we will see in Section \ref{sec:disks-one}
%that, up to neglecting a region of small area, $\BU^0_j$ can be obtained 
%from the disjoint union of $b^0_j$ 
%round spheres by removing a small domain from each sphere,
%and by gluing them to one another along portions of their boundary.
%Quite similarly, up to neglecting a region of small area, $\BU^1_i$ can
%be obtained from the disjoint union of $b^1_i$ round spheres and
%one spherical surface of genus $0$ with $2$ conical points of angles $2\pi\th_i$
%by removing a small domain from each round sphere and a small
%neighbourhood of a conical point,
%and by gluing the surfaces to one another along portions of their boundary.
%
Though logically independent,
such calculation of ours
may remind the area estimate near a bubbling point 
for the metrics $h_m$ performed in \cite{bartolucci-tarantello:compactness}.

%
%Another feature is that, for every small $\e>0$,
%one can find $\e$-bubbling spherical surfaces whose underlying conformal structure
%is very far from being degenerate (i.e. the extremal systole stays bounded below away from zero).
The proof of Theorem \ref{main:systole} 
essentially shows that, if the 
extremal systole remains bounded from below,
the only way that a sequence of spherical metrics can degenerate is by
getting $\e$-bubbling with smaller and smaller $\e$.
%Here, the needed quantitative estimate in terms of $\e$
%of the area of the bubbles $\BU^0_j$ and $\BU^1_i$
%is achieved by comparing them to appropriate model bubbles 
%(see Section \ref{sec:disks-one}).

To sum up, the hypothesis $\Ab_{\bm{\th}}(g,n)\geq \e$ prevents $\e$-bubbling
and so it prevents degenerations at all in a fixed conformal class.

\subsubsection{Existence of conformal spherical metrics in positive genus}

Theorem \ref{thm:troyanov}
implies that the forgetful map $F_{g,n,\bm{\th}}$ is surjective
in the spherical case if $\chi(S,\bm{\th})<\tau(S,\bm{\th})$ (subcritical case).
In the variational formulation employed by Troyanov, the above numerical condition on the angles
is exploited to show that the functional he considers is proper.

%
%More explicitly, an almost degenerate metric has a core of very small diameter
%and a complement which consists of a disjoint union of bubbles
%(see, for instance, Figure \ref{fig:almost-bubbling-intro}).
%
%Such bubbles may be of two types: {\it{round bubbles}}, isometric to
%the complement of a small disk in a round sphere (see \cite{brezis-merle:blowup}),
%and an {\it{$x_j$-bubble}}, namely the complement of a small disk centered
%at $x'_j$ inside a spherical surface of genus $0$ with two conical points $x_j,x'_j$
%both of angle $2\pi\th_j$ (see \cite{bartolucci-tarantello:compactness} for the case $\th_j>1$ and \cite{bartolucci-montefusco:compactness} for the case $\th_j<1$).\\
%%
%In particular, away from the conical points the metric can blow up forming
%{\it{round bubbles}}, that is isometric copies of round spheres
%(see \cite{brezis-merle}).
%At the conical point $x_j$ the metric can blow up forming an {\it{$x_j$-bubble}},
%namely a spherical surface of genus $0$ with two conical points both of angle $2\pi\th_j$,
%one marked by $x_j$  and the other one closer to the rest of the surface

As seen in Section \ref{sec:bubbling},
a diverging sequence of spherical metrics in a fixed conformal class
gets $\e$-bubbling for smaller and smaller $\e$.
By Theorem \ref{diskstructure}, the area of 
a $\BU^0_j$ or a $\BU^1_i$ in an $\e$-bubbling surface is
at least $2\pi\cdot (\tau(S,\bm{\th})-\e)$ and so
this is also a lower bound for the total area of an $\e$-bubbling surface.
%
%Since $2\pi\cdot\tau(S,\bm{\th})$ is the smallest
%area of a possible bubble and $2\pi\cdot\chi(S,\bm{\th})$ is the total area of a spherical metric on $S$ 
%with conical points of angles $\bm{\th}$, 
Hence, the subcritical hypothesis is equivalent to asking 
that there is not enough room for $\e$-bubbling if $\e$ is too small.

%%%
\gm{``The first general existence result...'' is changed to ``The following general existence
result''.}
%%%%%%%%%%%%%%%%%%%%%
The following general existence result
for $\tau(S,\bm{\th})>\chi(S,\bm{\th})$ (supercritical case) is due to Bartolucci-De Marchis-Malchiodi.

\begin{theorem}[Existence of supercritical spherical metrics in positive genus \cite{BDMM:supercritical}]\label{thm:BDMM}
Let $(S,J)$ be a compact connected Riemann surface of genus $g$, 
let $\bm{x}=\{x_1,\dots,x_n\}$
be a subset of $n\geq 0$ distinct points of $S$ and let $\bm{\th}=(\th_1,\dots,\th_n)\in\RR^n_{>0}$.
Assume that
\begin{itemize}
\item[(i)]
$g>0$ and $\th_i\geq 1$ for all $i$
\item[(ii)]
$\chi(S,\bm{\th})>\tau(S,\bm{\th})$
\item[(iii)]
$\Ab_{\bm{\th}}(S,\bm{x})>0$.
%$\chi(\dot{S})\notin \Crit_{\bm{\th}}:=
%\left\{
%2m+\sum_{i\in I}\th_i-\sum_{j\notin I}\th_j\in \RR\ \big|\ m\in\ZZ_{\geq 0},\ I\subseteq\{1,\dots,n\}
%\right\}$.
\end{itemize}
Then there exists at least one conformal spherical 
metric on $(S,J)$ with
conical singularities at $\bm{x}$ of angles $2\pi\cdot\bm{\th}$.
\end{theorem}

%\begin{remark}\label{rmk:A-critical}
%Condition (c) is equivalent to asking that
%$\frac{1}{2\pi}\area(S)=\chi(S,\bm{\th})$ does not belong
%to the set 
%$\Acrit_{\bm{\th}}:=\{2m+2\sum_{i\in I}\th_i\,|\,m\in\ZZ_{\geq 0},\ I\subseteq\{1,\dots,n\}\}$.
%\end{remark}

%%%
\gm{Here I have modified the paragraph to cite the lower bound for the number of solutions
that Bartolucci-De Marchis-Malchiodi have found.}
%%%%%%%%%%%%%%%%%%%
Both Theorem \ref{thm:troyanov} in the spherical case and Theorem \ref{thm:BDMM}
do not provide any uniqueness result for the metric.
On the contrary, an inspection of \cite{BDMM:supercritical}
shows that, if counted with the appropriate multiplicities,
there should be at least 
$\dis\binom{g-1+\floor{\frac{1}{2}\chi(S,\bm{\th})}}{g-1}$
spherical metrics in the given conformal class.
The strategy is again
to consider a variational formulation of a Liouville-type equation, 
but now critical points of the functional are
found by examining the topology of the sublevels:
hypothesis (i) is used to ensure that the topology indeed changes.

%As seen in Section \ref{sec:bubbling},
%Condition (iii) prevents the occurrence of complete bubbling degenerations.
 
%About Condition (c), note that
%$2\pi(2m+2\|\bm{\th}_I\|_1)$ is exactly the area of a degenerate spherical metric on $S$,
%in which an $x_i$-bubble for $i\in I$ and $m$ round bubbles survive and the rest of the surface is
%shrunk to a point. In order to prevent such degeneration, it is enough to prescribe
%that $\chi(S,\bm{\th})$ does not belong to the set
%\[
%\Acrit_{\bm{\th}}=\left\{2m+2\sum_{i\in I}\th_i\ \big|\ m\in\ZZ_{\geq 0},\ I\subseteq\{1,2,\dots,n\}\right\}\subset \RR_{\geq 0}
%\]
%of critical values for $\frac{1}{2\pi}\area(S)$. We underline that $0\in\Acrit_{\bm{\th}}$.
%It is easy to see that $\Ab_{\bm{\th}}(S,\bm{x})>0$ is an equivalent reformulation of the above constraint on $\chi(S,\bm{\th})$.\\

\begin{remark}\label{rmk:tori}
By contrast with Theorem \ref{thm:BDMM},
Lin-Wang \cite{lin-wang:elliptic} prove that
the existence of a spherical metric with one conical point
of angle $6\pi$ in a prescribed conformal class
on a torus depends on the conformal class.
Note though that, in this case, $\Ab_{\bm{\th}}(S,\bm{x})=0$
and $\chi(S,\bm{\th})=\tau(S,\bm{\th})$.
It also follows from Theorem \ref{main:existence} and Theorem \ref{main:non-existence-small} that,
on a given pointed Riemann surface $(S,J,\bm{x})$,
for suitable $\bm{\th}$ (with small $\th_1$), existence of a spherical metric depends on the
conformal class (see Subsection \ref{sec:small}).
\end{remark}

A major progress in the enumeration of solutions of such Liouville-type equations
was achieved by Chen-Lin \cite{chen-lin:topological}, who exhibited
a recursive formula for the number of solutions (counted with signed multiplicity)
intended as a Schauder degree of an endomorphism of a Sobolev space.
We do not recall the full statement of such a result for brevity but
we emphasize that, in particular, it implies that 
the forgetful map $F_{g,n,\bm{\th}}$ is surjective whenever such degree is non-zero.
In view of Remark \ref{rmk:tori} though, Chen-Lin's result
does not always detect which moduli spaces $\MSPH_{g,n}(\bm{\th})$
are non-empty.

\subsubsection{Non-emptyness and connectedness for moduli spaces of spherical surfaces of genus zero}\label{sec:genus0}

If $n\leq 3$, Riemann surfaces of genus $0$ with $n$ marked points are all isomorphic to each other. 
In this case, spherical surfaces of genus $0$ with $n$ conical singularities
were classified by Troyanov \cite{troyanov:bigons} for $n=1$ (the round sphere) and $n=2$
(the rugby ball and the cyclic cover of $\Sph$ branched at two points) and by Eremenko
\cite{eremenko:three} for $n=3$.
% (obtained doubling spherical triangles).
Moreover, the results in 
\cite{eremenko:three} give a complete
answer to Problem \ref{problem:main} for $(g,n)=(0,3)$.

In the case of genus $0$ and $n\geq 4$, 
explicit inequalities in $\RR^n$ describing
the set of all $n$-uples of vectors $\bm{\th}$
for which a spherical metric with angles $2\pi\bm{\th}$ and non-coaxial monodromy
exists were determined by Mondello-Panov, and they
can be phrased as follows.

\begin{theorem}[Monodromy constraints for spherical surfaces of 
genus zero \cite{mondello-panov:constraints}]\label{thm:mondello-panov}
Let $n\geq 3$ and let $\bm{\th}\in \RR_{>0}^n$ an angle vector such that
the Gauss-Bonnet constraint $2-n+\|\bm{\th}\|_1>0$ is satisfied.
Then
\[
\MSPH_{0,n}(\bm{\th})\ 
\begin{cases}
\text{is empty} & \text{if $d_1(\bm{\th}-\bm{1},\ZZ^n_o)<1$}\\
\text{contains non-coaxial metrics} & \text{iff $d_1(\bm{\th}-\bm{1},\ZZ^n_o)>1$}
\end{cases}
\]
where $d_1(\bm{\th}-\bm{1},\ZZ^n_o)$ is the standard $L^1$-distance in $\RR^n$
between the vector $\bm{\th}-\bm{1}=(\th_1-1,\dots,\th_n-1)$ and the subset
$\ZZ^n_o=\{\bm{p}\in\ZZ^n\,|\,\text{$p_1+\dots+p_n$ is odd}\}$.
\end{theorem}

%Here we recall that a subgroup of $\SO_3(\RR)$ %(or $\SU_2$) 
%is {\it{coaxial}} if it is contained in a $1$-parameter subgroup and a spherical metric is coaxial is its holonomy group is.

\begin{remark}\label{rmk:existence-splitting}
Inspecting the proof of the case $d_1(\bm{\th}-\bm{1},\ZZ^n_o)>1$ in \cite{mondello-panov:constraints},
one can realize that all metrics for $n\geq 5$ are constructed starting from
metrics with fewer singularities by splitting some conical points.
Moreover, one can notice that the choice of which point to split is harder if the angles are small.
In any case, the metrics obtained through such splitting procedure have a
very small spherical systole (and so they can be thought of being ``very close to degenerating'').
\end{remark}

%Before \cite{mondello-panov:constraints}, 
%a proof that $d_1(\bm{\th}-\bm{1},\ZZ^n_o)>1$ is necessary for the existence
%of a spherical metric on a surface of genus $0$ with conical points of angles smaller than $2\pi$
%is already in Luo-Tian \cite{luo-tian:spherical}.
%
\gm{Citation of Kapovich's paper inserted.}
%%%%%%%%%%%%%%%%%%%%%%%%
After \cite{mondello-panov:constraints},
the case $d_1(\bm{\th}-\bm{1},\ZZ^n_o)=1$ was analyzed by Dey, Kapovich and Eremenko.
Under this hypothesis,  Dey \cite{dey:coaxial} showed that $\MSPH_{0,n}(\bm{\th})$ is empty if all $\th_i$'s are non-integral, and
Kapovich \cite{kapovich:branched} found
a simple criterion to determine
for which $\bm{\th}\in\ZZ^n_{\geq 0}$
the moduli space $\MSPH_{0,n}(\bm{\th})$ is non-empty.
Finally, Eremenko \cite{eremenko:coaxial} determined 
all the values of $\bm{\th}$
%with some $\th_i$ integral and $d_1(\bm{\th}-\bm{1},\ZZ^n_o)=1$ 
for which a spherical metric in genus $0$ with coaxial monodromy exists.
Altogether the above results completely determine when $\MSPH_{0,n}(\bm{\th})$ is non-empty.

%Such results in genus $0$ can be contrasted with our Theorem \ref{main:existence}, in which we show
%that for $g>0$ the amplitude of the angles $\bm{\th}$ provides no constraint 
%for the existence of a spherical metric, so that
%the moduli space $\MSPH_{g,n}(\bm{\th})$ is always non-empty
%whenever $\chi(S,\bm{\th})>0$.

\begin{remark}\label{rmk:genus0-convex}
The case of $g=0$ and all $\th_i<1$
is rather special.
The constraints in \cite{mondello-panov:constraints}
for the existence of some spherical metric 
with non-coaxial monodromy
are equivalent to $0<\chi(S,\bm{\th})<\tau(S,\bm{\th})$
and they had already been proven necessary by Luo-Tian \cite{luo-tian:spherical}.
Being in the subcritical case,
Theorem \ref{thm:troyanov} provides a much stronger conclusion than
\cite{mondello-panov:constraints}
by granting existence in each conformal class.
Furthermore, uniqueness also holds
and it was proven by Luo-Tian \cite{luo-tian:spherical}.
\end{remark}

The connectedness problem for $\MSPH_{g,n}(\bm{\th})$ has a simple answer
when $g=0$ and $\bm{\th}$ is integral.
%In fact, the Hurwitz space of isomorphism classes of
%holomorphic maps $f:(S,J)\rar\CC\PP^1$ with distinct ramification
%points $x_1,\dots,x_n$ on the Riemann surface $(S,J)$ of genus $0$
%of index $\th_1-1,\,\dots,\,\th_n-1$
%can be identified to a $\PSL_2(\CC)/\SO_3(\RR)$-bundle over $\MSPH_{0,n}(\bm{\th})$. 
%Through this identification, it is easy to see that the forgetful map $F_{0,n,\bm{\th}}$ is
%quasi-finite, and so its image is Zariski-dense whenever $\MSPH_{0,n}(\bm{\th})$ is non-empty.
A corollary of Liu-Osserman's proof of the connectedness
of the Hurwitz scheme in genus $0$ is the following.

\begin{theorem}[Connectedness of Hurwitz spaces \cite{liu-osserman:irreducibility}]\label{thm:liu-osserman}
For every $\bm{\th}\in\ZZ^n_{>0}$ the moduli space
$\MSPH_{0,n}(\bm{\th})$ is smooth and connected.
\end{theorem}

The above result must be contrasted with our Theorem \ref{main:many}, in which we exhibit
smooth moduli spaces $\MSPH_{0,3+m}(\bm{\th})$ that 
have at least $3^m$ connected components.

\begin{remark}
Determining the behavior of the forgetful map is an non-obvious problem
even in genus $0$. For instance,
Eremenko \cite{eremenko:coaxial} noticed that
Lin-Wang's results in Remark \ref{rmk:tori}
imply that $F_{0,4,\bm{\th}}$ is not surjective
for $\bm{\th}=\left(\frac{1}{2},\frac{1}{2},\frac{1}{2},\frac{3}{2}\right)$.
%
%
%on a 
%Riemann surface $(S,J)$ of genus $g=0$
%with marked points $\bm{x}=(x_1,x_2,x_3,x_4)$
%the existence
%of a conformal spherical metric 
%with conical singularities of angles 
%$(\pi,\pi,\pi,3\pi)$ really depends on $J$.
\end{remark}

We briefly mention that the case of integral angles had already been analyzed
by Goldberg and Scherbak. In particular,
Goldberg \cite{goldberg:catalan} showed that, for $\th_1=\dots=\th_n=2$,
the forgetful map $F_{0,n,\bm{\th}}$ is surjective and she computed its degree,
and Scherbak \cite{scherbak:critical} settled the case of a general $\bm{\th}\in\ZZ^n_{>0}$. 

More recently, Eremenko-Tarasov analyzed the case of surfaces of genus $0$ with exactly three non-integral angles
and proved the following result.

\begin{theorem}[Spheres with three non-integral angles {\cite[Theorem 2.5]{eremenko-tarasov:three-fuchsian}}]
Let $\bm{\th}=(\th_1,\dots,\th_n)$ be an angle vector such that
$\th_4=\dots=\th_n\in\ZZ_{>0}$ and $d_1(\bm{\th}-\bm{1},\ZZ^n_o)>1$.
Then the forgetful map $F_{0,n,\bm{\th}}$ is surjective with finite fibers
of cardinality at most $\th_4\cdot\th_5\cdots\th_n$.
Moreover, this upper bound is attained for generic $(\th_1,\th_2,\th_3)$
at the generic point of $\Mcal_{0,n}$.
\end{theorem}

\subsubsection{Spherical surfaces with one small angle}\label{sec:small}

We recall that in positive genus Theorem \ref{thm:BDMM} 
ensures the existence of a spherical metric in {\it{every}} conformal class,
provided $\Ab_{\bm{\th}}(g,n)>0$ and all $\th_i\geq 1$.
The non-existence result in Theorem \ref{main:non-existence-small}
shows that it is not possible to strengthen the statement of Theorem \ref{main:existence}
so to claim existence of a metric in every conformal class for every $\bm{\th}$ that satisfies
the Gauss-Bonnet constraint.

The rather delicate nature of existence of spherical metrics with conical singularities of small angles
can in turn be compared with the content of \cite{carlotto:solvability},
in which Carlotto studies the solvability of a singular Liouville differential equation.
In positive genus our non-existence result does not logically overlap with Carlotto's, 
since Carlotto's setting requires $\bm{\th}\in(0,1)^n$
and so solutions to his differential equation are not conformal factors of a spherical metric
(with respect to a given background).
%On the other hand, our hypothesis $\Ab_{\bm{\hat{\th}}}(S,\bm{x})>0$
%implies Carlotto's $x_1$-stability (indeed it is strictly related to it),
%which is responsible in \cite{carlotto:solvability} for the possible failure of the existence of solutions.
%Moreover, his non-existence result for small values of $\th_1$
%requires the analogous of $\Ab_{\thno}(S)>0$ (see Condition (ii)).
%\end{remark}

\begin{remark}
The hypotheses on $\bm{\th}$ in Theorem \ref{main:non-existence-small}
are never satisfied in the subcritical case, which is coherent with the existence result in
Theorem \ref{thm:troyanov}.
\end{remark}

Another equivalent way of rephrasing Theorem \ref{main:non-existence-small} is the following.

\begin{mainthm}{\ref{main:non-existence-small}'}[Image of the forgetful map for small $\th_1$]\label{main:non-existence-small'}
Let $g\geq 0$ and $n>1$ such that $2g-2+n>1$ and let
$\bm{\hat{\th}}=(0,\th_2,\dots,\th_n)$ with $\th_2,\dots,\th_n>0$
such that $\Ab_{\bm{\hat{\th}}}(S,\bm{x})>0$.
For every compact subset $K\subset \Mcal_{g,n}$ there exists $\th_1^\star(K)>0$
such that the image of $F_{g,n,\bm{\th}}$ avoids $K$ for all
$\bm{\th}=(\th_1,\th_2,\dots,\th_n)$ with $0<\th_1<\th_1^\star(K)$.
\end{mainthm}

Note that the behavior of the forgetful map for $\bm{\th}$ as in Theorem \ref{main:non-existence-small'}
is completely different than the behavior of $F_{0,n,\bm{\th}}$ with $\bm{\th}\in\ZZ^n_{>0}$,
since $F_{0,n,\bm{\th}}$ has dense image as discussed above.

\subsection{Ideas of the proofs of Theorem \ref{main:systole} and 
of Theorems \ref{main:non-existence-small} and \ref{main:properness}}

The proof of the systole inequality
(Theorem \ref{main:systole}) breaks into two parts.
First we show that an almost bubbling surface has small non-bubbling parameter,
then we prove that a spherical surface
with small systole
and large extremal systole must be almost bubbling.
%
% which is metrically almost degenerate
%(small systole) but conformally far from degenerating %(large extremal systole)
%must be almost bubbling.

The former statement essentially relies on estimating the area of the ``bubbles'',
namely spherical disks with at most one conical point and short boundary.
For the latter statement we explicitly construct a decomposition
of the surface into a core and a collection of bubbles by means of the Voronoi function.

\subsubsection{Area estimate for the bubbles} %\label{sec:intro-bubbling}

The first step in the proof of
Theorem \ref{main:systole}
is the following.

\begin{namedthm*}{Theorem \ref{thm:almost-bubbling}}
If the spherical surface $(S,\bm{x})$ is $\e$-bubbling, then
$\Ab_{\bm{\th}}(S,\bm{x})<\e$.
\end{namedthm*}

Such a theorem completely relies on the following two
estimates.

\begin{namedthm*}{Corollary \ref{shorthindisk}}
A spherical disk $\BU^0$ without conical points
with $\ell(\pa\BU^0)<2\pi$ satisfies
\[
\frac{1}{2\pi}|\Area(\BU^0)-4\pi b^0|<(\ell(\pa\BU^0)/2\pi)^2
\]
for some $b^0\in\ZZ_{\geq 0}$.
\end{namedthm*}

\begin{namedthm*}{Theorem \ref{diskstructure}}
A spherical disk $\BU^1$ with a conical point $x$ of angle $2\pi\th$
and $\ell(\pa\BU^1)/d(x,\pa\BU^1)<1/2$ satisfies 
\[
\frac{1}{2\pi}|\Area(\BU^1)-4\pi(b^1+\th)|<\ell(\pa\BU^1)/d(x,\pa\BU^1)
\]
for some $b^1\in\ZZ_{\geq 0}$.
\end{namedthm*}

In order to calculate the area of $\BU^0$,
we consider its developing map to $\Sph$.
By taking care of the local degrees of the developing map, we
reduce the problem to the classical isoperimetric inequality $\ell(\pa\Omega)^2>2\pi\cdot\Area(\Omega)$
for domains $\Omega\subset\Sph$.
For disks with one conical point
we proceed in a similar way, as we consider the disk obtained by
cutting
$\BU^1$ along a geodesic that joins $x$ to $\pa\BU^1$.
In this case though, a further estimate for the area
of isosceles triangles in $\Sph$ (Section \ref{sec:isosceles})
is needed: this explains the different nature of the estimate for 
$\BU^1$, which is linear in $\ell(\pa\BU^1)$.

%In order to calculate the area of $\BU^1$,
%we cut it along a curve that joins $x$ to $\pa\BU^1$
%and we consider its developing map to $\Sph$.
%By taking care of the local degrees of the developing map, we
%reduce the problem to area estimates for
%isosceles triangles in $\Sph$ (Section \ref{sec:isosceles})
%and for domains $\Omega\subset\Sph$ with short perimeter, for which
%we can invoke the classical isoperimetric inequality $\ell(\pa\Omega)^2>2\pi\cdot\Area(\Omega)$.
%For $\BU^0$

%A somehow subtle point is that a simple closed curve in $\Sph$ determines
%two complementary domains. 
%In order to distinguish between the areas of the first or the second domain, 
%we introduce the notion of {\it{algebraic area}} of a loop in $\Sph$ and we exploit its additive properties (see Section \ref{sec:algebraic-area}).

\subsubsection{The Voronoi function}

Before proceeding with the proof of Theorem \ref{main:systole},
we investigate some features of the {\it{Voronoi function}} $\Vor:S\rar\RR_{\geq 0}$ associated to a spherical surface $(S,\bm{x})$,
which assigns to a point its distance from $\bm{x}$.
The Voronoi function will be central in many of our constructions,
where it will often play the role of a ``topological Morse-Bott function''.

We completely classify the types of critical points
of $\Vor$ (Theorem \ref{locmaxint}) and we
prove a number of results that will be useful
in particular
in the proofs of Theorem \ref{main:many} and Theorem \ref{main:systole},
among which we highlight the following:
\begin{itemize}
\item
the lowest positive critical value of $\Vor$ is the systole of $S$ (Proposition \ref{critnumber});
\item
the maximum value of $\Vor$ is at least $\sqrt{2\chi(S,\bm{\th})\|\bm{\th}\|_1^{-1}}$
(Lemma \ref{lengthoflevel});
\item
saddle critical points of $\Vor$ are midpoints of special geodesic arcs and loops
(see Section \ref{saddlegeosec});
such geodesics provide
a cellular {\it{Morse-Delaunay decomposition}} of $S$ (Proposition \ref{prop:Voronoi-cellular-decomposition});
\item
$\Vor$ has at most $-3\chi(\dot{S})$ saddle critical values (Proposition \ref{critnumber});
\item
area and perimeter of sublevel sets of $\Vor$ are bounded from above 
(Lemma \ref{lengthoflevel}) by
$\ell(\pa\Vor^{-1}[0,r])\leq 2\pi r\|\bm{\th}\|_1$ and $\Area(\Vor^{-1}[0,r])\leq \pi r^2\|\bm{\th}\|_1$;
\item
if $\Vor$ has no critical values in the interval $(r',r'')$, each connected component
$C$ of $\Vor^{-1}(r',r'')$ is a {\it{(Voronoi) cylinder}} (Definition \ref{defVoronoi} and Corollary \ref{cor:Vor})
with conformal modulus $M(C)>\frac{\log(r''/r')}{2\pi\|\bm{\th}\|_1}$ (Lemma \ref{cylinderslength}).
\end{itemize}

\subsubsection{Detecting $\e$-bubbling surfaces}

The second essential step in Theorem \ref{main:systole} is to decompose
the surface $S$ into a core $S^c$ and a collections of disks $\BU^0_j$ and $\BU^1_i$
that satisfy the condition of $\e$-bubbling.
In fact, we want to obtain our $S^c$ as a {\it{Voronoi core}}, namely
of a component of a sublevel set $\Vor^{-1}([0,r'])$ %with $r'>\sys(S,\bm{x})$
(see Definition \ref{def:Voronoi-core}).

In order for such construction to work,
the value $r'$ must belong to the interval $(\sys(S,\bm{x}),\max\,\Vor)$,
because
$S^c$ must contain a geodesic that realizes the systole.
In addition,
all components of $\Vor^{-1}(r')$ that bound $S^c$ must be non-essential.
In fact, such two conditions are also sufficient (Lemma \ref{lemma:subcore}).

Now, the non-essentiality of the components in $\Vor^{-1}(r')$
is achieved by showing that each of them bounds a
Voronoi cylinder of large modulus (namely, larger than
$1/\Ext\sys(\dot{S})$).
More precisely, we find a non-critical 
$r''<\frac{\pi}{2}$ in the interval $(r',\max\,\Vor)$ in such a way that
$[r',r'']$ does not contain any saddle value 
and $|\log(r'/r'')|$ is small enough
compared to $1/\Ext\sys(\dot{S})$.
In such a situation each component of $\Vor^{-1}([r',r''])$
is either a disk with no conical points or a Voronoi
cylinder, and the non-essentiality then follows from the modulus estimate for such cylinders (Corollary \ref{cor:non-essential}).

%
%and the cylindrical components of $\Vor^{-1}([r',r''])$ are non-essential
%(Theorem \ref{Abound}).
%In fact, if a component of $S\setminus S^c$ is contained in $\Vor^{-1}([r',r''])$,
%then it is easily seen to be a disk without conical points.
%Otherwise, a component $S'$ of $S\setminus S^c$ not in $\Vor^{-1}([r',r''])$
%must contain a cylinder $C$ of $\Vor^{-1}([r',r''])$, and so $\pa S'$ is
%non-essential.
%Since we know how to bound the modulus of any Voronoi cylinder,
%the non-essentiality of $C$
%can be obtained 
%by choosing $|\log(r'/r'')|$ small enough
%compared to $1/\Ext\sys(\dot{S})$.

After establishing that $S^c$ is a Voronoi core,
we observe that
the sum of the values
$\lambda(\BU^0_j)$ is bounded above in terms of $r'$
and the sum of the values $\lambda(\BU^1_i)$ in terms of $r'/r''$. 
Being a component of a sublevel of $\Vor$, 
the area of
$S^c$ is also easily estimated in terms of $r'$.

%The estimate for the areas of the $\CAP_l$'s follows from Paul Levy's 
%soperimetric inequality, provided we require that $r''<\frac{\pi}{2}$,
%and total area of the caps can be bounded in terms of $r'$.

Finally, we must check that
it is possible to find regular values $r',r''<\frac{\pi}{2}$ with
$0<\sys(S,\bm{x})<r'<r''<\max\,\Vor$
such that 
\begin{itemize}
\item
$r'$ is small
and
$r'/r''$ is small
\item
the interval $(r',r'')$ does not contain saddle values.
\end{itemize}
Assuming that the systole is small enough (of order $\e^{-3\chi(\dot{S})}$), 
this is just a consequence of the pigeonhole principle because
the maximum number of saddle values is the topological constant $-3\chi(\dot{S})$.

More precisely, the exact values of $r',r''$ we will take 
are $r''=\delta$ and $r'=\frac{\e\delta}{4\pi\|\bm{\th}\|_1}$
for a suitable $\delta \in (\sys(S,\bm{x}),\max\, \Vor)$.

\subsubsection{Non-existence for small $\th_1$ and properness of the forgetful map}

The non-existence result for spherical metrics with small $\th_1$
(Theorem \ref{main:non-existence-small}) follows from Theorem \ref{main:systole} after
noticing 
%(Lemma \ref{anglesysbound}) 
that $\sys(S,\bm{x})\leq \pi\th_1$
for every spherical surface $(S,\bm{x})$ in $\MSPH_{g,n}(\bm{\th})$
(Lemma \ref{anglesysbound}).

On the other hand,
it is well-known that the level supsets of the continuous functions
$\sys:\MSPH_{g,n}(\bm{\th})\rar\RR_+$ and 
$\Ext\sys:\Mcal_{g,n}\rar\RR_+$ are compact
(see Lemma \ref{lemma:systole-function} and Lemma \ref{lemma:ext-systole-function}).
Thus, Theorem \ref{main:properness} is also a direct consequence of
Theorem \ref{main:systole}.

%%%%%%%%%%%%%%%%%%5

\subsection{Ideas of the proofs of Theorem \ref{main:existence}
and of Theorem \ref{main:many}}

The proof of Theorem \ref{main:existence}, namely
the non-emptiness of $\MSPH_{g,n}(\bm{\th})$ for $g>0$
and $\chi(S,\bm{\th})>0$, is a rather simple
application of some results contained in \cite{mondello-panov:constraints}.
First we produce one special metric with a single conical point  by
identifying the sides of a spherical bigon in a suitable way (Lemma \ref{lemma:existence-one}). This already settles
the case $n=1$.
The result for $n>1$ can then be inductively achieved 
by splitting conical points in a controlled way
(Proposition \ref{prop:splitting}).\\

The construction of many connected components
(Theorem \ref{main:many}) in certain moduli spaces 
$\MSPH_{0,n}(\bm{\th})$ is more elaborate.
We will see that, similarly to what happens in Theorem \ref{main:non-existence-small}, the presence of conical points with small angles imposes strong constraints on the metric
and the conformal structure.

The spherical surfaces $(S,h)$ we consider
have genus $0$ and $n=3+m$ conical points
$x_1,x_2,x_3,y_1,\dots,y_m$ with angles $2\pi\cdot(\th_1,\dots,\th_{3+m})$,
where $\th_1,\th_2,\th_3$ are close to $\frac{1}{2}+\ZZ_{\geq 0}$
and $\th_4,\dots,\th_{3+m}$ are of order $\e$
(see Figure \ref{fig:disconnected}).

The proof of (a) relies on two main observations.
First, the monodromy gives some control on the lengths
of geodesics between conical points (Section \ref{sec:holonomy0}).
Second, the monodromy representation of the above $\dot{S}$
is informally speaking an ``$\e$-deformation'' of the monodromy of
a sphere
with $3$ conical points of angles odd multiples of $\pi$,
which is completely understood.
More precisely, 
Proposition \ref{gap} shows that
\begin{itemize}
\item
the distance between $x_j,x_l$ is at least $\pi(\frac{1}{2}-m\e)$;
\item
every $y_i$ is {\it{tied}} to $x_{\kappa(i)}$, namely
$x_{\kappa(i)}$ is the conical point closest to $y_i$,
whereas all the other conical points are at distance
strictly greater than $d(y_i,x_{\kappa(i)})$.
\end{itemize}
This second property
allows us to construct a continuous map $\Kcal:\MSPH_{0,3+m}(\bm{\th})\rar \{1,2,3\}^m$
that sends a surface to the function $\kappa:\{1,2,\dots,m\}\rar\{1,2,3\}$.
Surjectivity of $\Kcal$ is easily proven: for each $\kappa$
we produce a spherical surface $(S,h)$ with $\Kcal(S,h)=\kappa$
by a standard splitting procedure
(see Proposition \ref{prop:splitting})
starting from a sphere with $3$ conical points 
of angles odd multiples of $\pi$.
This ensures that $\MSPH_{0,3+m}(\bm{\th})$ has at least $3^m$ components.

In order to prove (b),
%that the image of 
%the forgetful map $F_{0,3+m,\bm{\th}}:\MSPH_{0,3+m}(\bm{\th})\rar\Mcal_{0,3+m}$
%has at least $3^m$ connected components, 
a deeper analysis of such surfaces is needed.
In particular,
Proposition \ref{goeodesicloop} asserts that smooth geodesic loops $\gamma$ based at some $x_j$
of length $\ell(\gamma)<\pi$ are exactly of two types:
\begin{itemize}
\item
curves $\gamma$ with $\ell(\gamma)\leq 2\pi m\e$
such that
a component of $S\setminus\gamma$ does not contain any $x_l$;
\item
curves $\gamma$ with $\ell(\gamma)\geq \frac{\pi}{2}-2\pi m\e$
such that
both components of $S\setminus\gamma$ contain some $x_l$.
\end{itemize}
In Section \ref{sec:separating-cylinders}
the described gap property allows us 
to construct for each $j=1,2,3$ a Voronoi cylinder
$C_j\subset S$ with large modulus such that a component of $S\setminus C_j$ exactly contains
$x_j$ and all the points $y_i$ tied to $x_j$.
The existence of such conformally long cylinders $C_j$
permits us to conclude that components
of $\MSPH_{0,3+m}(\bm{\th})$ corresponding to distinct functions $\kappa$
are mapped to disjoint subsets of $\Mcal_{0,3+m}$ by the forgetful map.

\subsection{Content of the paper}

Section \ref{sec:existence} contains the proof of Theorem \ref{main:existence}, which is rather elementary and does not require much technology.
In Section \ref{sec:preliminaries} we introduce the notions of injectivity radius and systole
and we prove some basic properties of theirs. 
%Moreover, we show that two conical points with small angles
%cannot be closest to each other and that the smallest angle gives a bound from below for the systole.
In Section \ref{sec:voronoi} we analyze the geometry of the Voronoi function, the types and number
of its critical points and values, and the induced Voronoi and Delaunay cellular decompositions.
%the induced Voronoi foliation
In Section \ref{sec:voronoi2} we give some basic estimates of the modulus of a Voronoi cylinder and of the area 
and the perimeter of a sublevel set for the Voronoi function.
In Section \ref{sec:main} we present two applications of the systole inequality.
Assuming Theorem \ref{main:systole}, we prove the non-existence result of metrics with small $\th_1$ (Theorem \ref{main:non-existence-small}) and the properness of the
forgetful map (Theorem \ref{main:properness}).
Section \ref{sec:disconnectedness} contains the proof of Theorem \ref{main:many}.
Section \ref{sec:disks-one} is dedicated to estimating the area of a disk with one conical singularity and short boundary. 
In Section \ref{sec:systole}, we will use
this result to calculate the area of an $\e$-bubbling surface
and we will complete the proof of Theorem \ref{main:systole}.
In Appendix \ref{sec:appendix} we collect some well-known properties of the extremal length. Moreover, we show in one example that the estimate of the extremal length
that follows from Theorem \ref{main:systole} is reasonably sharp.

%optimal up to a factor $3$ (at worst).

\subsection{Acknowledgements}

We would like to thank Daniele Bartolucci and Misha Kapovich for fruitful discussions
and Francesca De Marchis for helpful explanations on \cite{BDMM:supercritical}.
We are particularly grateful to Alexandre Eremenko for a number of valuable remarks that allowed
us to improve the exposition. We also thank an anonymous referee for carefully reading the
paper and for useful expository suggestions.
\gm{Thanks to the referee added.}
%%%%%%%%%%%%%%%%%

G.M.'s research was partially supported by FIRB 2010 Italian grant ``Low-dimensional
geometry and topology'' (code {\texttt{RBFR10GHHH\_003}}) and by GNSAGA INdAM
research group. D.P. was supported by a Royal Society University Research Fellowship and partially supported by grant {\texttt{RGF{$\backslash$}EA{$\backslash$}180221}}.

%\subsection{Notation}
%
%%\subsubsection{Spherical surface}
%Let $S$ be a compact, oriented surface, which we assume to be connected unless differently specified, and let
%$\bm{x}=(x_1,\dots,x_n)$ be a numbered collection
%of distinct marked points $x_i\in S$. We denote by
%$\dot{S}:=S\setminus \bm{x}$ the punctured surface.
%
%A simple closed curve $\beta\subset\dot{S}$ is
%{\it{essential}} if each connected component of $\dot{S}\setminus \beta$ has negative Euler characteristic.

%Let $\bm{\th}=(\th_1,\dots,\th_n)\in\RR_{>0}^n$ be an angle
%vector. A spherical metric $h$ on $(S,\bm{x})$ with angles
%$2\pi\cdot\bm{\th}$ is a metric of curvature $K=1$ on $\dot{S}$ 
%with conical singularity of angle $2\pi\th_i$ at $x_i$.

%The defect vector $\bm{\delta}=(\delta_1,\dots,\delta_n)$
%is obtained as $\bm{\th}-\bm{1}$, where $\bm{1}=(1,\dots,1)$. 

%\subsubsection{Group}
%Let $\mathcal{B}$ the invariant positive-definite 
%symmetric bilinear form on
%$\su_2$ defined as $\mathcal{B}(M,N):=-\frac{1}{2}\tr(M^t N)$
%and consider $S^2\subset \su_2$ as the unit sphere
%with respect to $\mathcal{B}$.
%
%For every $R\in \SU_2$, we will write $\bar{\delta}_R$ for the unique element
%in $\RR/2\ZZ$ (well-defined up to sign) such that
%$R=\exp(\pi\bar{\delta}_R N)$ for some $N\in S^2$.
%Given $R,R'\in\SU_2$,
%we will write $d(\bar{\delta}_R,\bar{\delta}_{R'})$ for the minimum of the distances
%in $\RR/2\ZZ$ between $\pm \bar{\delta}_R$ and $\pm\bar{\delta}_{R'}$.

%%%%%%%%%%%%%%%%%

\section{Existence of spherical metrics in positive genus}\label{sec:existence}

The purpose of this section is to illustrate a simple existence result
of a spherical metric on a surface of genus $g>0$ with conical singularities of assigned angles.
The only constaint will be given by the Gauss-Bonnet formula.
%
%marked points a spherical metric with conical singularities at the marked points exists as long as the obvious Gauss-Bonnet constraint is satisfied.
%The second statement says that on a surface of any genus and with any number of punctures there exists a conformal structure
%whose extremal length is universally bounded from below.

\subsection{Existence of spherical metrics with assigned angles}

Let us start this section with a slightly more precise
version of Theorem \ref{main:existence}.

\begin{theorem}[Existence of a spherical metric with assigned angles]\label{thm:existence}
For any $g>0$ and for any $\bm{\th}=(\th_1,\dots,\th_n)\in\RR_{>0}^n$ satisfying $\sum(\th_i-1)>2g-2$ there is a genus $g$ surface with an 
angle-deformable spherical metric with conical points $\bm{x}=(x_1,\dots,x_n)$ of angles $2\pi\cdot\bm{\th}$.
Moreover, such metric has non-coaxial monodromy, except if $g=1$ and $\bm{\th}\in\ZZ_{>0}^n$.
\end{theorem}

Here we recall that a spherical metric $h$ on a surface $S$ with conical points $\bm{x}$ of angles $2\pi\cdot\bm{\th}$
is {\it{angle-deformable}} if there exists a neighbourhood $\Ncal\subset\RR_{>0}^n$ of $\bm{\th}$
and a continuous family $\Ncal\ni \bm{\phi}\mapsto h_{\bm{\phi}}$ of spherical metrics on $\dot{S}$ such that
$h_{\bm{\phi}}=h$ and $h_{\bm{\phi}}$ has conical singularity of angle $2\pi \phi_i$ at $x_i$ for every $\bm{\phi}\in \Ncal$.\\

We start with the case of one conical point.

\begin{lemma}[Existence of a spherical metric with one conical point]\label{lemma:existence-one}
Let $S$ be a surface of genus $g>0$ with one marked point $x_0$.
For every $\th_0>2g-1$ there exists a spherical metric $h_{\th_0}$ on $S$ with conical singularity at $x_0$
of angle $2\pi\th_0$. Such $h_{\th_0}$ depends continuously on $\th_0$ and so it is angle-deformable.\\
Moreover, $h_{\th_0}$ has non-coaxial monodromy, except if $g=1$ and $\th_0\in\ZZ_{>0}$.
\end{lemma}
\begin{proof}
Let $\theta=\th_0-2g+1>0$ and let
$\Bcal$ be the bigon (unique up to isometry) with sides of length $\pi$ and angles $\pi\theta=\pi(\th_0-2g+1)$. 
Subdivide each side of $\Bcal$ into $2g$ segments of equal length $\frac{\pi}{2g}$ and glue together the sides of the obtained $4g$-gon $\Bcal'$ in a standard way. Namely,
label the cyclically ordered edges of $\Bcal'$ by
\[
\alpha_1,\beta_1,\check{\alpha}_1,\check{\beta}_1,\alpha_2,\beta_2,\check{\alpha}_2,\check{\beta}_2,\dots,
\alpha_g,\beta_g,\check{\alpha}_g,\check{\beta}_g
\]
and call $v$ the vertex of $\Bcal'$ adjacent to $a_1$ and $\check{b}_g$.
For example, the case of $g=2$ looks like Figure~\ref{fig:glued-bigon}.

\begin{center}
\begin{figurehere}
\psfrag{B}{$\textcolor{blue}{\Bcal'}$}
\psfrag{fi}{$\textcolor{red}{\!\!\pi\theta}$}
\psfrag{a1}{$\textcolor{blue}{\alpha_1}$}
\psfrag{a2}{$\textcolor{blue}{\alpha_2}$}
\psfrag{v}{$\textcolor{Brown}{v}$}
\psfrag{b1}{$\textcolor{blue}{\beta_1}$}
\psfrag{b2}{$\textcolor{blue}{\beta_2}$}
\psfrag{a1c}{$\textcolor{blue}{\check{\alpha}_1}$}
\psfrag{a2c}{$\textcolor{blue}{\check{\alpha}_2}$}
\psfrag{b1c}{$\textcolor{blue}{\check{\beta}_1}$}
\psfrag{b2c}{$\textcolor{blue}{\check{\beta}_2}$}
\includegraphics[width=0.4\textwidth]{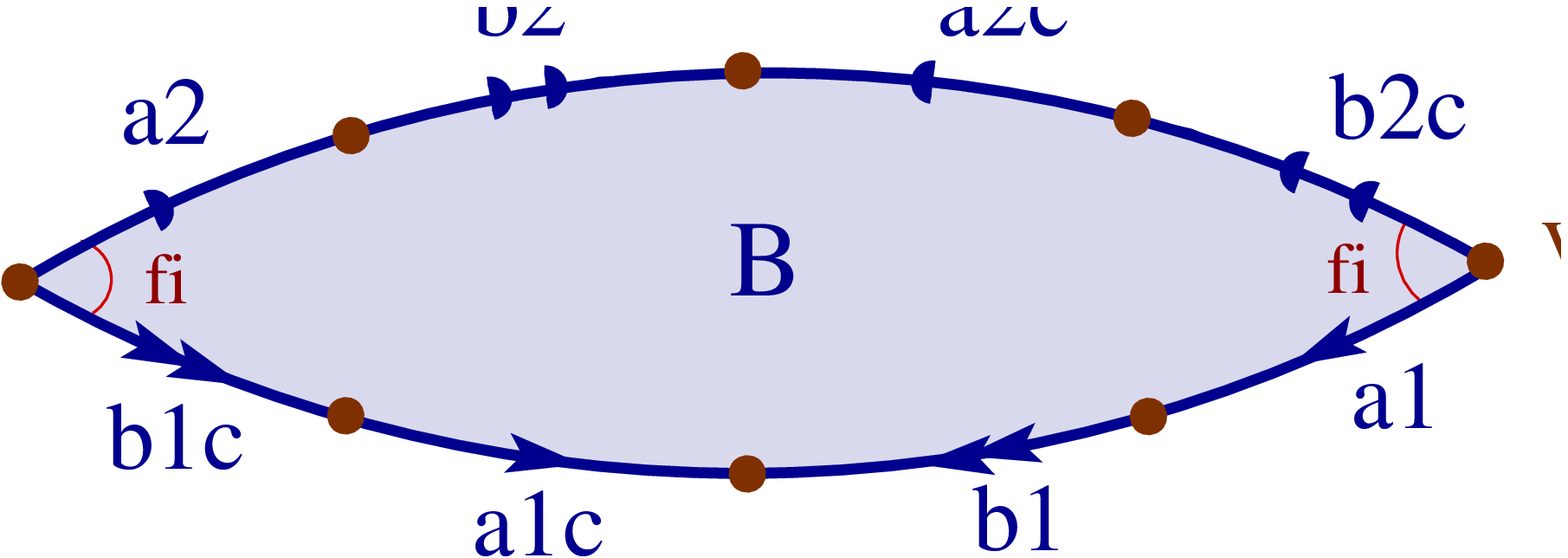}
\caption{{\small Construction of a spherical surface with $(g,n)=(2,1)$.}}\label{fig:glued-bigon}
\end{figurehere}
\end{center}

Identify $\check{\alpha}_k\sim \alpha_k^{-1}$ and $\check{\beta}_k\sim \beta_k^{-1}$
in order to obtain a surface $S$ of genus $g$ with a spherical metric as desired. 
The continuous dependence of the metric on $\th_0$ is clear.\\
As for the monodromy, if $g\geq 2$, then consider a developing map $\iota:\Bcal'\rar \Sph$. The monodromy group
contains the rotations of angle $\pi$ centered at the midpoint of $\iota(\beta_1)$ (that takes $\iota(\alpha_1)$ to $\iota(\check{\alpha}_1)$)
and at the midpoint of $\iota(\check{\alpha}_1)$ (that takes $\beta_1$ to $\check{\beta}_1$). Since the midpoints of $\iota(\beta_1)$ and $\iota(\check{\alpha}_1)$
are at distance $\frac{\pi}{2g}<\pi$, such rotations are not coaxial and so neither is the monodromy group.\\
If $g=1$ and $\th_0\notin\ZZ$, then the monodromy group
contains nontrivial rotations centered at $\iota(v)$.
On the other hand, the $1$-parameter subgroup of rotations that fix 
$\iota(v)$ cannot take $\alpha_1$ to $\check{\alpha}_1$, because $\check{\alpha}_1$
does not have $v$ as endpoint. Thus, the monodromy cannot be
coaxial.
\end{proof}

In order to split the conical point into several ones, we recall the following result.

\begin{lemma}[Splitting one conical point into $2$ conical points]
Let $\bm{\th}=(\th_1,\dots,\th_n)$ and $\bm{\th'}=(\th_1,\dots,\th_{n-2},\th_{n-1}+\th_n-1)$
and suppose that there exists a surface $S'$ of genus $g$ endowed with an angle-deformable spherical metric 
$h'$ with $n-1$ conical singularities $\bm{x'}$ of angles $2\pi\cdot\bm{\th'}$.
Then for every $\epsilon>0$ small enough there exists an angle-deformable spherical metric $h$ on a surface $S$ of genus $g$ with $n$ conical singularities
$x_1,\dots,x_n$ of angles $2\pi\cdot\bm{\th}$ such that
\begin{itemize}
\item[(a)]
$x_{n-1},x_n$ are at distance less than $\epsilon$
\item[(b)]
$|d(x_i,x_j)-d(x'_i,x'_j)|<\e$ for $1\leq i<j\leq n-2$
\item[(c)]
$|d(x_i,x_l)-d(x'_i,x'_{n-1})|<2\e$ for $1\leq i\leq n-2$ and $l=n-1,n$.
\end{itemize}
%
%
%\begin{itemize}
%\item[(a)]
%$x_{n-1},x_n$ are at distance less than $\epsilon$
%\item[(b)]
%there is a $(1+\epsilon)$-biLipschitz map
%from $S\setminus U$ to $S'\setminus U'$ that sends $x_i$ to $x'_i$ for $i=1,\dots,n-2$,
%where $U\subset S$ is a contractible domain of diameter at most $2\epsilon$ that contains
%$x_{n-1},x_n$, and $U'\subset S'$ is a disk of diameter at most $2\epsilon$ than contains $x'_{n-1}$.
%\end{itemize}
Moreover, if $h'$ has non-coaxial monodromy, so has $h$.
\end{lemma}
\begin{proof}
The existence of such an $S$ is essentially the content of Lemma 3.41 from \cite{mondello-panov:constraints} in the case of a positive merging operation.
Let us recall how the surgery in such lemma is operated.

We construct small spherical triangles $T_\epsilon$
with vertices $y_1,y_2,y_3$ and 
angles $\pi(\th_{n-1},\th_{n},\th_{n-1}+\th_n-1+\epsilon)$ and small $|\epsilon|$ as in Proposition
3.17 of \cite{mondello-panov:constraints}. 
We consider a small deformation of the metric on $S'$ that slightly moves the angle at $x'_{n-1}$
and keeps the other angles fixed. Then we remove a small neighbourhood of $x'_{n-1}$ from $S'$
and of $y_3$ from the double $DT_\epsilon$ of $T_\epsilon$ and we glue the two boundary components.

The surgery being a local procedure,
it is clear that such procedure can be performed on a surface $S'$ of any genus.

By looking at Proposition 3.17 and Figure 8 of \cite{mondello-panov:constraints}, one can check
that the triangles $T_\epsilon$ can be chosen to be as small as desired.
As a consequence, both the deformation of $S'$ and the size of 
the removed neighbourhoods can be taken as small as desired.
This ensures that $d(x_{n-1},x_n)<\epsilon$
and so properties (b) and (c) are also satisfied for $\epsilon$ small enough.

Finally, if the monodromy %$\hol(\pi_1(\dot{S}'))\subset\SO_3(\RR)$
of $h'$ is non-coaxial, so is the monodromy of its
small deformations $h'_{\bm{\phi}}$. Since the monodromy group of the metric $h$ obtained
by the above surgery from some $h'_{\bm{\phi}}$ contains the monodromy group of such $h'_{\bm{\phi}}$,
it follows that $h$ is non-coaxial too for $\epsilon$ small.
\end{proof}

An iterated application of the above lemma yield the following.

\begin{proposition}[Splitting one conical point into $k$ conical points]\label{prop:splitting}
Let $\bm{\th}=(\th_1,\dots,\th_n)$ and $\bm{\th'}=(\th_1,\dots,\th_{n-k},\th_0)$
with $\th_0=\th_{n-k+1}+\dots+\th_n-k+1$
and suppose that there exists a surface $S'$ of genus $g$ endowed with an angle-deformable spherical metric 
$h'$ with $n-k+1$ conical singularities $\bm{x'}$ of angles $2\pi\cdot\bm{\th'}$.
Then for every small $\eta>0$ there exists an angle-deformable spherical metric $h$ on a surface $S$ of genus $g$ with $n$ conical singularities
$x_1,\dots,x_n$ of angles $2\pi\cdot\bm{\th}$ such that
\begin{itemize}
\item[(a)]
$x_{n-k+1},\dots,x_n$ are at distance less than $\eta$ from each other
\item[(b)]
$|d(x_i,x_j)-d(x'_i,x'_j)|<\eta$ for $1\leq i<j\leq n-k$
\item[(c)]
$|d(x_i,x_l)-d(x'_i,x'_{n-k})|<2\eta$ for $1\leq i\leq n-k$ and $n-1\leq l\leq n$.
\end{itemize}
%$x_{n-k+1},\dots,x_n$ are at distance less than $\eta$ from each other.
%\begin{itemize}
%\item[(a)]
%$x_{n-k+1},\dots,x_n$ are at distance less than $\eta$ from each other
%\item[(b)]
%there is a $(1+\eta)$-biLipschitz map
%from $S\setminus B$ to $S'\setminus B'$ that sends $x_i$ to $x'_i$ for $i=1,\dots,n-k$,
%where $B\subset S$ is a contractible domain of diameter at most $2\eta$ that contains
%$x_{n-k+1},\dots,x_n$, and $B'\subset S'$ is a contractible domain of diameter at most $2\eta$ than contains $x'_{n-1}$.
%\end{itemize}
Moreover, if $h'$ has non-coaxial monodromy, so has $h$.
\end{proposition}

\begin{remark}
The local structure of the space of spherical surfaces near
a degeneration in which some conical points coalesce
has been analyzed by
Mazzeo-Zhu \cite{mazzeo-zhu:compactified1}.
\end{remark}

%
%
%
%\begin{corollary}[Splitting several conical points]\label{cor:splitting}
%Let $\bm{\th}=(\th_1,\dots,\th_n)$ and $\th_0:=1+\sum_{i=1}^n(\th_i-1)$
%and suppose that there exists a surface $S'$ of genus $g$ endowed 
%with an angle-deformable spherical metric $h'$
%with one conical singularity of angle $2\pi\cdot\th_0$.
%Then for every $\eta>0$ there exists an angle-deformable spherical metric $h$ on a surface $S$ of genus $g$ with $n$ conical singularities
%of angles $2\pi\cdot\bm{\th}$ such that 
%all conical points are at distance less than $\eta$ from each other.\\
%Moreover, if $h'$ has non-coaxial monodromy, so has $h$.
%\end{corollary}

We are now ready to prove the main result of this section.

\begin{proof}[Proof of Theorem \ref{thm:existence}]
The existence of such a wished metric and its angle-deformability follow from
Lemma \ref{lemma:existence-one} and 
Proposition \ref{prop:splitting}.
Non-coaxiality follows as well if $g\geq 2$, or if $g=1$ and $\th_0\notin\ZZ$.

Suppose now that $g=1$ and $\th_0\in\ZZ$, but $\bm{\th}\notin\ZZ^n$.
Up to rearranging the indices, we can assume that $\th_{n-1},\th_n\notin\ZZ$.
Consider the metric $h$ obtained by choosing any $\eta<\frac{\pi}{2}$.
Inspecting the proof of Lemma 3.41 from \cite{mondello-panov:constraints}, we note that
the monodromy group of $h$ contains that of a $3$-punctured spherical
surface of genus $0$ with two conical points of angles $2\pi\th_{n-1}$ and
$2\pi\th_n$ at distance less than $\frac{\pi}{2}$, which is non-coaxial.
It follows that the monodromy of $h$ is non-coaxial by Lemma 2.11 of \cite{mondello-panov:constraints}.
\end{proof}

\section{Spherical systole and conical points of small angle}\label{sec:preliminaries}

Spherical systole $\sys(S,\bm x)$ is an important characteristic of a spherical surface that measures how far the surface is from a degenerate one. A closely related notion is that of injectivity radius sand  immersion radius of a conical point on a spherical surface. 

\begin{definition}[Injectivity and immersion radius at a conical point]\label{definj} 
Let $S$ be a spherical surface with conical points $\bm{x}$.
% and $x_i\in S$ be a conical point. 
The \emph{injectivity radius} at $x_i$ is the supremum $r_i$ of all $r>0$
such that $B_{x_i}(r)$ is isometric to the standard disk $\DD_{\th_i}(r)$.
The \emph{immersion radius} at $x_i$ is the supremum $\bar r_i$ of all $r\in (0,\pi)$
such that there is a locally isometric immersion from the standard disk $\DD_{\th_i}(r)$ to $S$
that maps $0$ to $x_i$. 
The {\it{maximal $1$-pointed ball at $x_i$}} is the locus $\W_{x_i}:=B_{x_i}(d_i)$,
where $d_i$ is the minimum distance between $x_i$ and any other conical point. 
\end{definition}

We emphasize that the ball $B_{x_i}(r)$ need not be
an embedded disk, and so $\W_{x_i}$ is not necessarily an embedded disk either.
Note also that the only conical point contained in the internal part $\Wint_{x_i}$ is $x_i$.

Moreover, it is clear that for each conical point $x_i$ we have $r_i\ge \sys(S,\bm x)$. \\

With the definition in hand we can summarise the content of this section.

In Subsection \ref{subsecessinj} we study injectivity and immersion radius. In Subsection \ref{subsecsystole} we introduce systole geodesics and show in Corollary \ref{birkhoff} that $\sys(S,\bm x)<\ell(\gamma)/2$ for any essential curve $\gamma$ in $\dot S$. 
In Subsection \ref{subsec4pi3} we focus on conical points with conical angle less than $4\pi/3$ and less than $2\pi/3$. The behaviour of spherical surface is more constrained close to such points and a number of phenomena are described in  Theorem \ref{thm:closest-points}. As a corollary we get Lemma \ref{anglesysbound} which states $\sys(S,\bm x)\le \pi\cdot \min\{\th_i\}$. This simple inequality is essential for the non-existence result of Theorem \ref{main:non-existence-small}.

%In Subsection \ref{subsec4pi3} we focus on conical points of small angle. Here the main result is Theorem \ref{thm:closest-points} which highlights certain phenomena particular for neighbourhoods of conical points of angles less than $4\pi/3$ and $2\pi/3$. As a corollary we get Lemma \ref{anglesysbound} which states $\sys(S,\bm x)\le \pi\cdot \min_i\th_i$.

%In this section we study metric properties of spherical surfaces
%with conical singularities on a relatively small scale. For each conical point $x_i$ we define the injectivity radius $r_i$ that measures the size of a maximal standard neighbourhood of $x_i$, and immersion radius $\bar r_i$. In particular, we study the spherical systole $\sys(S,\bm x)$
%and the corresponding systole geodesic.
%
%The second focus of this section are conical points of relatively small conical angle. Here the main result is Theorem \ref{thm:closest-points}. This theorem makes quantitative the following intuitively clear observation: on a spherical surface conical points $x_i$ and $x_j$ with small conical angles can not be the closest to each other.
%Another simple but important statement is Lemma \ref{anglesysbound} which states that the spherical systole is bounded by $\pi\min_i \th_i$.

The discussion of the spherical systole can be juxtaposed with the following simple lemma, that 
distinguishes spherical surfaces from hyperbolic and Euclidean ones.

\begin{lemma}[Diameter estimate for a spherical surface]\label{sphdiameter}
A spherical surface $S$ with $n$ conical points has diameter at most $\pi(n+1)$. Moreover, for every small $\e>0$
there exists a spherical surface with $n$ conical points
of diameter greater than $\pi(n+1-\e)$.
\end{lemma}
\begin{proof}
Let $p,q\in S$ be two points such that $d(p,q)=\mathrm{diam}(S)$.
Let $\gamma$ be a piecewise geodesic path of length $d(p,q)$ that connects $p$ and $q$. 
Since every length-minimizing geodesic arc in $\Sph$ has length at most $\pi$ and
$\gamma$ is length-minimizing, each smooth geodesic segment in $\gamma$ has length at most $\pi$ too. It follows that the number conical points hit by $\gamma$ is at least $\frac{d(p,q)}{\pi}-1$. Since $\gamma$ passes through each conical point on $S$ at most once, the first claim is proven.

\begin{center}
\begin{figurehere}
\psfrag{S}{$\textcolor{Blue}{\ol{S}'_\alpha}$}
\psfrag{S2}{$\textcolor{Blue}{\ol{\mathbb{S}}^2_\beta}$}
\psfrag{a1}{${\alpha_1}$}
\psfrag{a2}{${\alpha_2}$}
\psfrag{b1}{${\beta_1}$}
\psfrag{b2}{${\beta_2}$}
\psfrag{x1}{$\textcolor{BrickRed}{x_1}$}
\psfrag{x2}{$\textcolor{BrickRed}{x_2}$}
\psfrag{p}{$\textcolor{Brown}{p}$}
\psfrag{q}{$\textcolor{Brown}{q_-=q}$}
\psfrag{q'}{$\textcolor{Brown}{q_+}$}
\psfrag{Q}{$\textcolor{Brown}{Q_-}$}
\psfrag{Q'}{$\textcolor{Brown}{Q_+}$}
\psfrag{m}{$\textcolor{Brown}{m}$}
\psfrag{M}{$\textcolor{Brown}{M}$}
\includegraphics[width=0.9\textwidth]{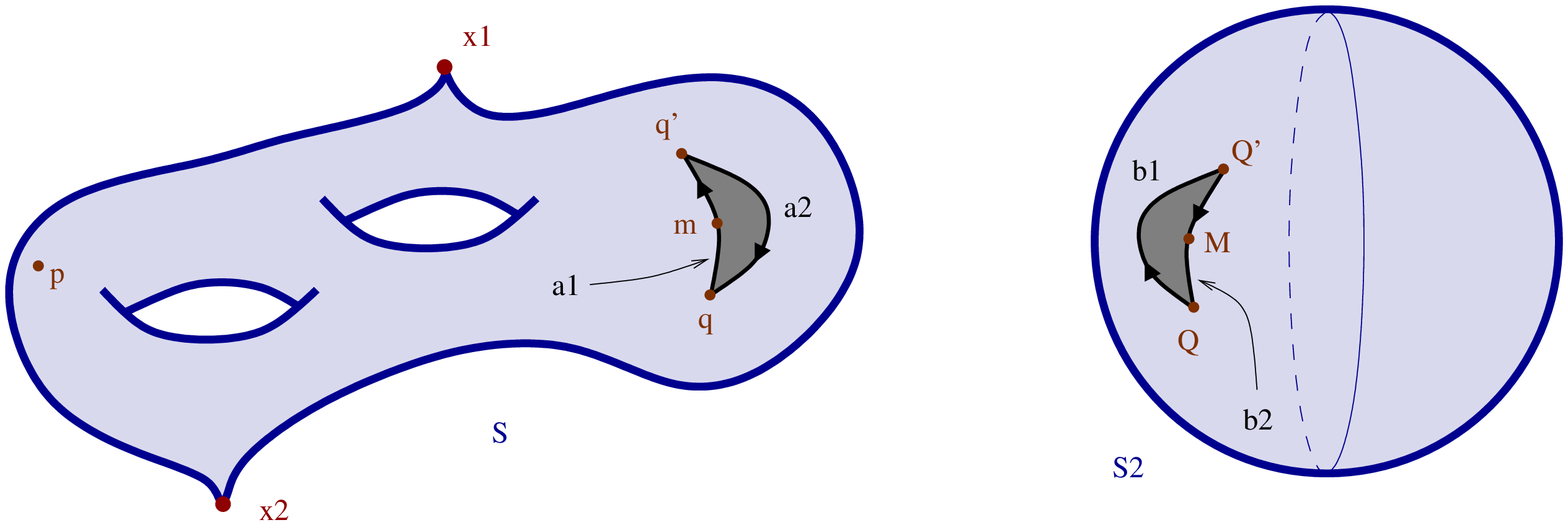}
\caption{{\small An example of the surfaces
$\ol{\mathbb{S}}^2_\beta$ and $\ol{S}'_\alpha$
(in this case, of genus $2$ with $2$ conical points).}}\label{fig:diameter}
\end{figurehere}
\end{center}

As for the second claim, we proceed by induction on $n\geq 0$.
For $n=0$, the round sphere $\Sph$ provides a surface of diameter $\pi$.
Take now $n\geq 1$, fix some $\e\in(0,\frac{1}{2})$
and suppose that there exists a surface $S'$ with $n-1$ conical points $x_1,\dots,x_{n-1}$,
genus $g$ and diameter greater than $\pi(n-\e/2)$.

Let $p,q\in S'$ be smooth points on $S'$ 
at distance at least $\pi(n-\e/2)$
and let $\alpha$ a smooth geodesic arc on $S'$ 
starting at $q=q_-$ and ending at some smooth point $q_+$, of length $\ell<\e\pi/4$.
Let $\beta$ be a geodesic arc on the round sphere $\Sph$
with endpoints $Q_-$ and $Q_+$ of length $\ell$.
Cut $S$ along $\alpha$, complete it and call $\alpha_1,\alpha_2$ the two
shores of the obtained surface $\ol{S}'_\alpha$; similarly,
cut $\Sph$ along $\beta$ and call $\beta_1,\beta_2$ the two obtained shores in $\ol{\mathbb{S}}^2_\beta$. 
Endow $\alpha_i,\beta_i$ with the induced orientations,
so that $\alpha_1$ (resp. $\beta_1$) runs from $q_-$ to $q_+$
(resp. from $Q_-$ to $Q_+$).
Call $m$ the midpoint of $\alpha_1$ and $M$ the midpoint of $\beta_2$, as in Figure \ref{fig:diameter}.

Consider now the surface $S$ obtained 
from $\ol{S}'_\alpha$ and $\ol{\mathbb{S}}^2_\beta$
by identifying $\alpha_2$ to $\beta_1$
in such a way that $q_-$ is glued to $Q_-$ and $q_+$
is glued to $Q_+$,
and by identifying
the arc $q_-m$ to the arc $MQ_+$ and the
arc $mq_+$ to the arc $Q_-M$.

Clearly, $S$ has genus $g+1$ and
$n$ conical points $x_1,\dots,x_n$,
where $x_1,\dots,x_{n-1}$ come from $S'$ and
the new conical point $x_n$ corresponds to $\{q_-,q_+,m,Q_-,Q_+,M\}$.
Moreover, it is easy to check that $S$
has diameter at least $\pi(n+1-\e)$.
\end{proof}

For every $n\geq 0$, the inductive construction
contained in the proof of the above lemma provides
surfaces of genus $g=n$ with $n$ conical points
of diameter arbitrarily close to $\pi(n+1)$.

\subsection{Essential curves and injectivity radius}\label{subsecessinj}

%Given a surface $S$ with a spherical metric with singularities, a point $p\in S$
%and $r>0$, we denote by $B(p,r)\subset S$ the closed ball of radius $r$ centred at $p$. 
%It is clear that for $r$ small enough such a ball is isometric to {\it{a}} standard disk,
%namely a closed disk isometric to some $D_{\theta}(r)$.
%
%
%
%\begin{definition}[Injectivity and immersion radius at a conical point]\label{definj} 
%Let $S$ be a spherical surface with conical points $\bm{x}$.% and $x_i\in S$ be a conical point. 
%The \emph{injectivity radius} at $x_i$ is the supremum $r_i=r(x_i)$ of all $r>0$
%such that $B(x_i,r)$ is isometric to the standard disk $D_{\th_i}(r)$.
%The \emph{immersion radius} at $x_i$ is the supremum $\bar r_i=\bar r(x_i)$ of all $r>0$
%such that there is a locally isometric immersion $\iota$ from the standard disk $D_{\th_i}(r)$ to $S$, so that $\iota(O)=x_i$. 
%\end{definition}

Here we discuss properties of the injectivity and immersion radius and relate the first to essential curves.

\begin{lemma}[Geodesic arcs and loops realizing the injectivity radius]\label{lemma:inj-rad}
Let $S$ be a surface with a spherical metric and conical points $\bm{x}$.
Then the injectivity radius $r_i$ at $x_i$ can be characterized by the following two properties:
\begin{itemize}
\item
for all $r\in \big(0,r_i\big)$ the ball $B_{x_i}(r)$ is isometric to the standard disk $\DD_{\th_i}(r)$;
\item
either there is a closed geodesic loop of length $2r_i$ based at $x_i$, 
or there is a different conical point $x_j$ at distance $r_i$ from $x_i$.
\end{itemize}
\end{lemma}

The above lemma is standard and follows from Definition \ref{definj}, so we omit the proof.

\begin{definition}[Essential curves, cylinders and loops]
A piecewise smooth simple closed curve $\gamma$ inside $\dot{S}$
is {\it{essential}} if each connected component of $\dot{S}\setminus\gamma$ has negative Euler characteristic.
An {\it essential cylinder in $\dot S$} is a cylindrical subsurface $C$ of $\dot S$ that retracts by deformation onto an essential simple
closed curve.
A {\it{loop based at $x_i$}} is a piecewise smooth simple closed curve $\gamma'$ inside $\dot{S}\cup\{x_i\}$ passing through $x_i$:
such loop $\gamma'$ is {\it{essential}} if each connected component of $\dot{S}\setminus\gamma'$ has negative Euler characteristic.
\end{definition}

\begin{lemma}[Essential curves bound the injectivity radius]\label{essloopsys} 
Let $S$ be a spherical surface with conical singularities $\bm{x}$.
Suppose that there exists a simple closed curve
$\gamma\subset S$ inside $\inte{B}_{x_i}(r)$,
which is either an essential curve in $\dot{S}$ or an essential loop based at $x_i$.
Then $r_i<r$, which means that either there is a conical point $x_j$ with $d(x_i,x_j)<r$, or there is a closed geodesic loop shorter than $2r$ based at $x_i$.
\end{lemma}
\begin{proof}
By definition of $r_i$ the open neighbourhood $\inte{B}_{x_i}(r_i)$
is isometric to the open standard disk $\inte{\DD}_{\th_i}\big(r_i\big)$.
Since $\gamma$ cannot be contained in such a neighbourhood, it follows that $r_i<r$. The last statement of the lemma follows from Definition \ref{definj}.
\end{proof}

%Alongside with the injectivity radius we  define the {\it immersion radius} $\bar r_i$ for each conical point.
%\begin{definition} Let $S$ be a spherical surface with conical points $\bm{x}$ and $x_i\in S$ be a conical point. 
%
%The \emph{immersion radius} at $x_i$ is the supremum $\bar r_i=\bar r_i$ of all $r>0$
%such that there is a locally isometric immersion $\iota$ from the standard disk $D_{\th_i}(r)$ to $S$, sending the centre of the disk to $x_i$. 
%\end{definition}

The next lemma evaluates the immersion radius of a conical point.

\begin{lemma}[Locally isometric immersion $\nu_i$]\label{immeradius} Let $S$ be a spherical surface and $x_i\in S$ be a conical point. 
%Let $d_i=d(x_i,{\bm x}\setminus x_i)$. 
Then the immersion radius at $x_i$ is given by $\bar r_i=\min(2r_i, d_i,\pi)$. Moreover, in case $\bar r_i<\pi$ there is a continuous map $\nu_i: \DD_{\th_i}(\bar r_i)\to S$ that is a local isometry on the interior of 
$\DD_{\th_i}(\bar r_i)$ and sends $0$ to $x_i$.
\end{lemma}
\begin{proof} It is easy to see that $\bar r_i\le \min(2r_i, d_i,\pi)$. To prove the converse, choose $r<\bar r_i$. One can check that a locally isometric immersion $\nu: \inte{\DD}_{\th_i}(r)\to S$ with $\iota(0)=x_i$  exists if and only if any locally immersed 
arc $\gamma\subset S$ with one endpoint in $x_i$ and the other endpoint in $\bm{x}$
%loop based at $x_i$
%or arc connecting $x_i$ to $x_j$ with $j\neq i$
%geodesic arc or loop $\gamma\subset S$ with endpoints in $\bm x$ and passing through $x_i$ 
has length at least $r$. Such a map can be defined by first sending isometrically a small neighbourhood of $0$ to a small neighbourhood of $x_i$ and then extending this map along each geodesic radius of $\inte{\DD}_{\th_i}(r)$. 

By our assumptions $r<d_i$, and so a length $r$ geodesic can not join $x_i$ with $x_j$ if $j\ne i$. Also, since $r<2r_i$ there is no locally immersed geodesic loop in $S$, based at $x_i$ with $\ell(\gamma)\le r$. Hence, a locally isometric immersion $\nu_i: \inte{\DD}_{\th_i}(r)\to S$ exists.
\end{proof}

\subsection{Spherical systole and systole geodesics}\label{subsecsystole}

%On the other hand, the spherical systole defined below measure how far the
%spherical metric is far from degenerating.
%
%\begin{definition}[Spherical systole]\label{sphersysdef}
%Let $S$ be a surface with spherical metric and conical singularities $\bm{x}=(x_1,\ldots,x_n)$. 
%The {\it systole $\displaystyle \sys(S,\bm{x})$ of $S$} is the supremum of all $r>0$
%such that all balls $B(x_i,r)$ are disjoint in $S$ and each $B(x_i,r)$ is isometric to the standard disk $D_{\th_i}(r)$.
%\end{definition}

From the definition of spherical systole it follows that on every spherical surface $(S,\bm x)$ there is geodesic of length  $2\sys(S,\bm{x})$ based at $\bm x$. We state this as a lemma but omit the proof. 

%We start with a simple lemma  telling that on every spherical surface $(S,\bm x)$ there is geodesic of length  $2\sys(S,\bm{x})$ based at $\bm x$. The lemma follows from the definition of spherical systole, so we omit the proof.

\begin{lemma}[Geodesics realizing the systole]\label{lemmarcloopsys}
Let $S$ be a surface with a spherical metric and conical points $\bm{x}$.
Then $\sys(S,\bm{x})$ is the minimum of all $r>0$ for which
at least one of the following two conditions is satisfied: 
\begin{itemize}
\item
there is a closed geodesic loop of length $2r$ based at some conical point;
\item
there is a geodesic arc of length $2r$ joining two distinct conical points.
\end{itemize}
\end{lemma}

%The next remark follows directly from Definition \ref{sphersysdef}.
\begin{definition}[Systole arcs and loops]
Let $S$ be a surface with a spherical metric and conical points $\bm{x}$, and let $\sigma_{\sys}$ be a geodesic arc or a loop of length  $2 \sys(S,\bm{x})$ with end points in $\bm{x}$. We call $\sigma_{\sys}$ a {\it systole geodesic}, or more specifically {\it systole arc} or {\it systole loop}.
\end{definition}

\begin{remark}[Midpoint of systole geodesics]\label{valueatmidpoint} 
Let $S$ be a surface with a spherical metric and conical points $\bm{x}$ and let $\sigma_{\sys}$ be a systole geodesic.  Let $s$ be the midpoint of $\sigma_{\sys}$. Then $d(s,\bm{x})=\sys(S,\bm{x})$.
\end{remark}

The next lemma gives us an upper bound on the value of $\sys(S,\bm{x})$.

\begin{lemma}\label{lemma:shorterthanpi/2} 
For any spherical surface
$S$ with conical singularities at $\bm{x}$ and $\chi(\dot{S})<0$, the systole satisfies $\sys(S,\bm{x})\le \frac{\pi}{2}$.
\end{lemma}
\begin{proof} Suppose by contradiction that $\sys(S,\bm{x})>\frac{\pi}{2}$ and choose $r\in (\frac{\pi}{2}, \sys(S,\bm{x}))$. By the definition of systole, points at distance at most $r$ from $\bm{x}$ form a disjoint union of standard disks $D_i$ (isometric to $\DD_{\th_i}(r)$) embedded in $S$. Note that the surface $S'=S\setminus \left(\bigcup_i D_i\right)$ has convex boundary and has no conical points. At the same time $S'\simeq \dot{S}$ and so
$\chi(S')=\chi(\dot{S})<0$. This clearly contradicts Gauss-Bonnet formula.
\end{proof}

%The geometry of spherical bigons also yield the following.
%
%\begin{lemma}
%Let $\gamma_0,\gamma_1$ be two homotopic
%smooth geodesic arcs in $S$ joining $x_i$ and $x_j$ with $i\neq j$.
%Then either of the three occurs:
%\begin{itemize}
%\item[(a)]
%$\ell(\gamma_0)=\ell(\gamma_1)=\pi$
%\item[(b)]
%$\ell(\gamma_0)=\ell(\gamma_1)<\pi$ and the homotopy $(\gamma_t)_{t\in[0,1]}$ between
%them sweeps the same angle $2\pi (2k)$ at $x_i$ and $x_j$, for some $k\in \NN_+$
%\item[(c)]
%$\ell(\gamma_0)=2\pi-\ell(\gamma_1)\neq \pi$ 
%and the homotopy $(\gamma_t)_{t\in[0,1]}$ between
%them sweeps the same angle $2\pi (2k+1)$ at $x_i$ and $x_j$, for some $k\in \NN$.
%\end{itemize}
%\end{lemma}

Finally, we explain how the length of an essential curve $\gamma\subset\dot{S}$
bounds the systole of $(S,\bm{x})$. The statement below is essentially
a consequence of Lemma \ref{essloopsys}.

\begin{corollary}[Essential curves bound the systole]\label{birkhoff} 
Let $S$ be a spherical surface with conical points $\bm{x}$ and let $\gamma\subset \dot S$ be a piecewise-smooth essential closed curve. Then either
\begin{itemize}
\item
for some $x_i\in \bm{x}$  there is a closed geodesic loop in $S$ based at $x_i$ of length less than $\ell(\gamma)$; or
\item
there is a geodesic segment of length less than $\ell(\gamma)/2$ with endpoints in $\bm{x}$.
\end{itemize}
In any case, $\sys(S,\bm{x})<\ell(\gamma)/2$.
\end{corollary}

\begin{proof} 
Note first, that there is a curve $\gamma'$ in $\dot S$ shorter than $\gamma$ and homotopic to it. Indeed, if $\gamma$ is not locally geodesic, we can straighten some small curvy bit of it. If $\gamma$ is a geodesic, then an equidistant of it will be shorter. 

Choose now $\varepsilon$ satisfying $\varepsilon<\min\left\{\sys(S,\bm{x}), d(\gamma', \bm{x}), \frac{1}{2}\Big(\ell(\gamma)-\ell(\gamma')\Big)\right\}$. 
Denote by $S_{\varepsilon}$ the complement in $S$ to the open $\varepsilon$-neighbourhood of $\bm x$,
which thus contains $\gamma'$, and let $\gamma_{\varepsilon}$ be a shortest curve in $S_{\varepsilon}$
homotopic to $\gamma'$. Such a curve $\gamma_\e$ is composed of points lying on the boundary of $S_{\varepsilon}$, and of geodesic pieces in $\inte{S}_{\varepsilon}$ with endpoints at $\partial S_{\varepsilon}$. Since $\gamma_{\varepsilon}$ is shortest in its homotopy class, it has to intersect $\partial S_{\varepsilon}$, i.e. there exists a conical point $x_i$ and
an $x_{i,\varepsilon}\in \gamma_{\varepsilon}\cap \pa S_\e$ with $d(x_i,x_{i,\varepsilon})=\varepsilon$. 

Since $\ell(\gamma_{\varepsilon})<\ell(\gamma')<\ell(\gamma)-2\varepsilon$, we see that $\gamma_{\varepsilon}$ is contained the open $\frac{1}{2}\ell(\gamma)$-neighbourhood of $x_i$. Now we can apply Lemma \ref{essloopsys}
to the essential curve $\gamma_\e\subset \dot{S}$ to either get a geodesic loop of length less than $\ell(\gamma)$ based at $x_i$ or get conical point $x_j$ with $d(x_i,x_j)<\ell(\gamma)/2$. This proves the main claim of the corollary. The inequality $\sys(S,\bm{x})<\ell(\gamma)/2$ follows now from Lemma \ref{lemmarcloopsys}. 
\end{proof}

\subsection{Neighbourhood of conical points with angles $\le 4\pi/3$ and $\le 2\pi/{3}$}\label{subsec4pi3}

In this section we study neighbourhoods of conical points with conical angle at most $4\pi/3$. We prove that the immersion radius of such points is equal to the distance to a closest conical point. We show as well that the closest conical point to a conical point with angle less than $2\pi/{3}$ must have conical angle larger than $2\pi/{3}$. This is summarised in Theorem \ref{thm:closest-points}.

%The goal of this section is to prove the following statement, from which we will deduce some useful corollaries.

\begin{theorem}[Couples of conical points closest to each other]\label{thm:closest-points}
Let $S$ be a spherical surface with $\chi(\dot S)<0$.
%Let $r_i$ be the injectivity radius at $x_i$, 
Suppose that $\th_i\le \frac{2}{3}$, and let $x_j$ be the conical point closest to $x_i$. Then the following hold.
\begin{itemize}
\item[(a)]
We have $d(x_i,x_j)<2r_i$. In particular, the immersion radius $\bar r_i$ of $x_i$ is equal to $d(x_i,x_j)$.
\item[(b)]
If $x_i$ is the conical point closest to $x_j$, then $\th_i+\th_j>\frac{2}{3}$. 
\item[(c)]
If $\th_i\le \frac{1}{3}$, then $\th_j>\frac{1}{3}$.
\end{itemize}
\end{theorem}

Before going into the proof of Theorem \ref{thm:closest-points} we give an important application.

\begin{lemma}[Systole bound in terms of the smallest angle]\label{anglesysbound} 
For any spherical surface $S$ with $\chi(\dot S)<0$ for any $i$ we have $\sys(S,\bm{x})\le \pi\th_i$.
\end{lemma}
\begin{proof} 
Assume without loss of generality that $\th_1=\min\{\th_k\}$ and let us prove $\sys(S)\le \pi\th_1$. 
%Let $r_1$ be the injectivity radius at $x_1$ and let $B(x_1,r_1)$ be the radius $r_1$ disk centred at $x_1$. 
Consider first the situation when for some $i$ we have $d(x_1,x_i)=r_1$.
Fix $\varepsilon>0$ and let $x_{j,\varepsilon}$ be the point on a geodesic segment $x_1x_j$ at distance $\varepsilon$
from $x_j$. Let $\gamma_{j,\varepsilon}$ be a loop based at $x_j$, composed of the segment $x_j x_{j,\varepsilon}$, the circle $\pa B_{x_1}(r_1-\varepsilon)$ and the segment $x_{j,\varepsilon}x_j$. Since $\chi(\dot S)<0$, the loop $\gamma_{j,\varepsilon}$ is essential and clearly
$$\ell(\gamma_{j,\varepsilon})=2\pi\th_1\sin(r_1-\varepsilon)+2\varepsilon\le 2(\pi\th_1+\varepsilon).$$
Hence, applying Lemma \ref{essloopsys} to $(x_j,\gamma_{j,\varepsilon})$ we get $\sys(S,\bm{x})\le \pi\th_1$.

Suppose now that there is no conical point at distance $r_1$ from $x_1$.
Then there must be a closed geodesic loop $\gamma_1$ in $S$ based at $x_1$ of length $2r_1$. 
In this case, by Theorem \ref{thm:closest-points}(a) we have $\th_i>\frac{2}{3}$. 
%At the same time we should have $2r_1\leq\pi$, since otherwise $\gamma$ can be slightly deformed to a shorter loop which contradicts the fact that $r_1$ is the injectivity radius. 
At the same time, the systole is at most $\frac{\pi}{2}$ by Lemma \ref{lemma:shorterthanpi/2}.
Hence, $\sys(S)\le \frac{\pi}{2}<\pi\th_1$ as required. 
\end{proof}

The remainder of the section is devoted to the proof of Theorem \ref{thm:closest-points}.
Instead of proving the statements (a), (b) and (c) of this theorem in one go, we will split it into
Lemma \ref{closeconepoint}, Lemma \ref{twopoints120} and Corollary \ref{closestcone} correspondingly.

Let us first make some basic observation of spherical trigonometry.

\begin{lemma}\label{rightangletr} 
Let $OPQ$ be a convex spherical triangle with angles $\wh{O},\wh{P},\wh{Q}$.
% such that $r=|O A|<\frac{\pi}{2}$,  $\angle O=\pi\th<\frac{\pi}{2}$, and $\angle B=\frac{\pi}{2}$. Then the following statements hold 
\begin{enumerate}
\item [(a)]
Suppose that $r=|O P|<\frac{\pi}{2}$,  $\wh{O}=\pi\th<\frac{\pi}{2}$, and $\wh{Q}=\frac{\pi}{2}$.\\
Then
$|O P|=\arctan\left(\frac{\tan(|O Q|)}{\cos(\pi \th)}\right)$ and $OP$ is the largest side of the triangle $OPQ$.
\item[(b)]
Suppose that $\wh{O}\ge \wh{P}$ and $\wh{O}+\wh{P}\le \frac{2}{3}\pi$.\\
Then $\wh{Q}>\wh{P}$ and so $|OP|>|OQ|$.
\item[(c)]
Suppose that  $\wh{O},\wh{P}\le \frac{1}{3}\pi$.\\
Then $\wh{Q}>\wh{P}$ and so $|OP|>|OQ|$.
%$O A$ is the largest side of the triangle $O A B$.
%\item If $\th\le \frac{1}{3}$, then $|OA|<2|OB|$.
\end{enumerate} 
\end{lemma}
\begin{proof}
Recall that in a convex spherical triangle the side opposite to a larger angle has larger length.

(a) The first claim is one of Napier's rules for right-angled spherical triangles. 
The second claim follows since $\wh{Q}>\max\{\wh{O},\, \wh{P}\}$.

(b) Two angle inequalities imply $\wh{P}\le \frac{1}{3}\pi$.
Since $\wh{O}+\wh{P}+\wh{Q}>\pi$, we deduce that $\wh{Q}>\frac{1}{3}\pi\ge \wh{P}$.

(c) Since $\wh{O}+\wh{P}+\wh{Q}>\pi$, we deduce that $\wh{Q}>\frac{1}{3}\pi\ge \wh{P}$.
\end{proof}

\begin{lemma}[Conical points close to a geodesic loop with a small angle]\label{closeconepoint} Let $S$ be a spherical surface and suppose that there is a geodesic loop $\gamma_i$ of length $2r_i$ based at $x_i$. Let $2\pi\th_i'$ and $2\pi\th_i''$ be the two angles into which $\gamma$ cuts the conical angle at $x_i$ and assume $\th_i'\le\frac{1}{3}$. Then there exists a conical point $x_k$ in $S$ with 
$$d(x_i,x_k)\le \arctan\left(\frac{\tan\big(r_i\big)}{\cos(\pi \th_i')}\right)=d_i'<2r_i.$$
In particular, if $d_i\ge 2r_i$, then $\th_i>\frac{2}{3}$.
\end{lemma}
\begin{proof} 
Note first that the second inequality $d_i'<2r_i$ is automatically satisfied since the function $t\mapsto \arctan(t)$ is increasing and concave for $t\ge 0$ and by our assumptions $\cos(\pi \th_i')\ge \frac{1}{2}$.

Assume by contradiction that for all $k\ne i$ we have $d(x_i,x_k)>d_i'$.
Since $d_i'<2r_i$, we have $d'_i<\bar{r}_i$ by Lemma \ref{immeradius}.

%Consider now the spherical kite 
%$OQPQ'$ with sides $|OQ|=|OQ'|=r_i$ and angles $\wh{O}=2\pi\th_i'$ and $\wh{Q}=\wh{Q}'=\frac{\pi}{2}$, which is unique up to isometries.
%The kite $OQPQ'$ is composed of two isometric right-angled triangles $OPQ$ and $OPQ'$, glued along $OP$. 
%From Lemma \ref{rightangletr}(a), it follows that $|OP|=d_i'>|OQ|=|OQ'|$. 
%Hence there is an  isometric embedding of $OQPQ'$ inside $D_{\th_i}(\bar{r}_i)$ that sends $O$ to the conical point $0$. By a slight abuse, assume that $OQPQ'$ is realised as a subset of $D_{\th_i}(\bar{r}_i)$. 

Consider now the spherical kite of vertices $0,w,z,w'$ embedded in $\DD_{\th_i}(\bar{r}_i)$,
with angle $2\pi\th_i'$ at $0$, right angles at $w,w'$ and sides $0w$ and $0w'$ of length $r_i$.
%with sides $|0w|=|0w'|=r_i$ and angles $\wh{0}=2\pi\th_i'$ and $\wh{w}=\wh{w}'=\frac{\pi}{2}$.
Such a kite is unique up to isometries and it is composed of 
two isometric right-angled triangles with vertices $0,z,w$ and $0,z,w'$, glued along the edge $0z$. 
From Lemma \ref{rightangletr}(a) it follows that $|0z|=d_i'>|0w|=|0w'|$,
which shows that indeed the kite can be isometrically embedded inside $\DD_{\th_i}(\bar{r}_i)$.
%Hence there is an  isometric embedding of $OQPQ'$ inside $D_{\th_i}(\bar{r}_i)$ that sends $O$ to the conical point $0$. By a slight abuse, assume that $OQPQ'$ is realised as a subset of $D_{\th_i}(\bar{r}_i)

\begin{center}
\begin{figurehere}
\psfrag{D}{$\DD_{\th_i}(\bar{r}_i)$}
\psfrag{S}{$\textcolor{Blue}{S}$}
\psfrag{O}{$\textcolor{Sepia}{0}$}
\psfrag{A}{$\textcolor{Cyan}{z}$}
\psfrag{B}{$\textcolor{BrickRed}{w}$}
\psfrag{B'}{$\textcolor{BrickRed}{w'}$}
\psfrag{i(A)}{$\textcolor{Cyan}{\nu_i(z)}$}
\psfrag{i(B)=i(B')}{$\textcolor{BrickRed}{\nu_i(w)=\nu_i(w')}$}
\psfrag{xi}{$\textcolor{Sepia}{x_i}$}
\psfrag{i}{$\nu_i$}
\psfrag{pith'}[][][0.9]{$\textcolor{Sepia}{\pi\th'_i}$}
\includegraphics[width=0.7\textwidth]{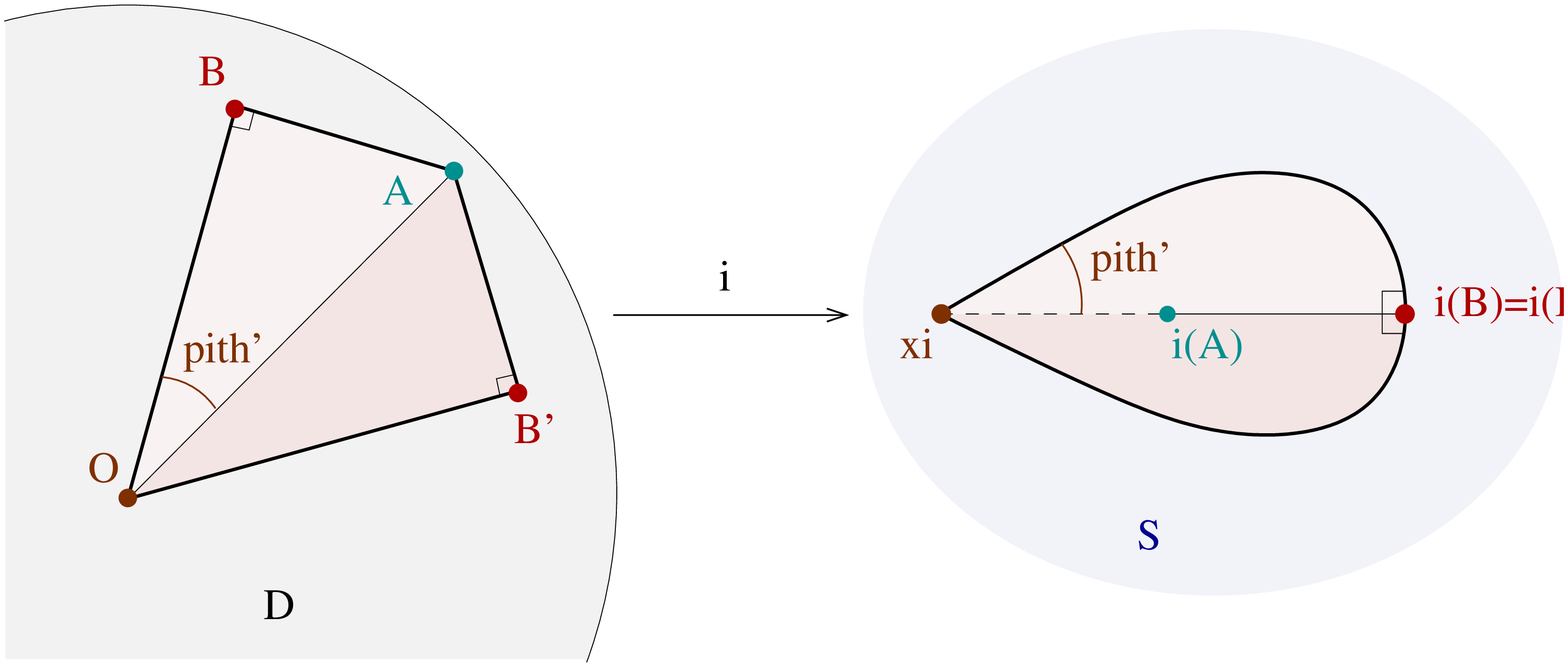}
\caption{{\small The map $\nu_i$ cannot restrict to a local isometry on the kite.}}\label{fig:point-close-a}
\end{figurehere}
\end{center}

Let $\nu_i:\DD_{\th_i}(\bar{r}_i)\to S$ be a map
that takes the origin $0$ to $x_i$ and which is a locally isometric immersion in the interior of the disk.
%a locally isometric immersion $\nu: D_{\th_i}(d_i'+\varepsilon)\to S$ of the standard disk $D_{\th_i}(d_i'+\e)$ that sends the conical point $0$ of $D_{\th_i}(d_i'+\e)$ to $x_i$. 
After precomposing $\nu_i$ with a rotation of $\DD_{\th_i}(\bar{r}_i)$, we may assume that the union of the segments $\nu_i(0w)$ and $\nu_i(0w')$ is the geodesic $\gamma_i$ and $\nu_i(w)=\nu_i(w')$ is the midpoint of $\gamma_i$.

Note finally that the images $\nu_i(zw)$ and $\nu_i(zw')$ should coincide in $S$. Hence, the map $\nu_i$ cannot be a local isometry at $z$. This contradicts our assumption that $\nu_i$ is a locally isometric immersion on $\inte{\DD}_{\th_i}(\bar{r}_i)$.

The last statement follows since we have proven that, in case $\th_i'\le\frac{1}{3}$ or $\th_i''\le\frac{1}{3}$, there exists a conical point $x_k$ with $d(x_i,x_k)<2r_i$. Hence, the inequality $d_i\ge 2r_i$ implies $\th_i',\th_i''> \frac{1}{3}$.
\end{proof}

%\begin{lemma}\label{triangle120} Let $OAB$ be a convex spherical triangle such that $\angle O\ge \angle A$, $\angle O+\angle A\le \frac{2}{3}\pi$. Then $\angle B>\angle A$ and so $|OA|>|OB|$.
%\end{lemma}
%\begin{proof} Two angle inequalities imply $\angle A\le \frac{1}{3}\pi$.
%Since $\angle O+\angle A+\angle B>\pi$ we deduce that $\angle B>\frac{1}{3}\pi\ge \angle A$. Hence the statement of lemma holds since in a convex spherical triangle the side opposite to a larger angle has larger length.
%\end{proof}

\begin{lemma}[Existence of conical points close to an arc between $x_i,x_j$ with small $\th_i+\th_j$]\label{twopoints120} 
Let $S$ be a spherical surface with conical singularities and suppose that for some $x_i, x_j$ we have $d(x_i,x_j)\leq d\Big(x_k, \{x_i,\, x_j\}\Big)$ for all $k\ne i,j$. Then $\th_i+\th_j>\frac{2}{3}$.
\end{lemma}
\begin{proof}
Assume without loss of generality that $\th_i\ge \th_j$. If  $\th_i>\frac{2}{3}$ we have nothing to prove, so we can assume $\th_i\le \frac{2}{3}$.
 
If there is no geodesic loop based at $x_i$ of length $2r_i$, then 
$x_i$ must sit at distance $r_i$ from another conical point. Since $d(x_i,x_j)\leq d\Big(x_k, \{x_i,\, x_j\}\Big)$ for all $k\neq i,j$, 
it follows that $d(x_i,x_j)=r_i$.
If there is a geodesic loop based at $x_i$ of length $2r_i$, then
$d(x_i,x_j)<2r_i$ by Lemma \ref{closeconepoint}.
In either case, $\bar{r}_i=d_{i}=d(x_i,x_j)<2r_i$.

Consider now a map $\nu_i:\DD_{\th_i}(\bar{r}_i)\rar S$ that sends $0$ to $x_i$
and which is a local isometric immersion in the interior of the disk.
  
%Note that its interior admits a locally isometric immersion $\iota$ inside $S$
%that sends the cone point $O$ of $D_{\th_i}(d_{ij})$ to $x_i$. Such immersion
%extends by continuity to the closure, thus giving a map $\iota:D_{\th_i}(d_{ij})\rar S$.
%Denote by $A$ the point of the boundary $\pa D_{\th_i}(d_{ij})$ such that
%the radius $OA$ is sent by $\iota$ to a geodesic segment of length $d_{ij}$ that joins $x_i$ with $x_j$.
%Call $A'$ the point on $\pa D_{\th_i}(d_{ij})$ opposite to $A$,
%so that the diameter $AA'$ splits the cone into two isometric sectors  of angle $\pi\th_i$. 

Denote by $z$ the point of the boundary $\pa \DD_{\th_i}(d_{i})$ such that
the radius $[0,z]$ is sent by $\nu_i$ to a geodesic segment of length $d_{i}$ that joins $x_i$ with $x_j$.
Let $-z$ the point on $\pa \DD_{\th_i}(d_{i})$ opposite to $z$,
so that the diameter $[-z,z]$ splits the cone into two isometric sectors  of angle $\pi\th_i$. 

\begin{center}
\begin{figurehere}
\psfrag{D}{$\DD_{\th_i}(d_i)$}
\psfrag{S}{$\textcolor{Blue}{S}$}
\psfrag{O}{$\textcolor{Sepia}{0}$}
\psfrag{A}{$\textcolor{Sepia}{z}$}
\psfrag{T}{$\textcolor{BrickRed}{T}$}
\psfrag{Tb}{$\textcolor{BrickRed}{\ol{T}}$}
\psfrag{B}{$\textcolor{Cyan}{w}$}
\psfrag{A'}{$\!{-z}$}
\psfrag{AB}[][][0.9]{${zw}$}
\psfrag{ABb}[][][0.9]{$\ol{zw}$}
\psfrag{i(AB)}[][][0.9]{$\nu_i(zw)$}
\psfrag{i(ABb)}[][][0.9]{$\nu_i(\ol{zw})$}
\psfrag{i(B)}{$\textcolor{Cyan}{\nu_i(w)}$}
\psfrag{xi}{$\textcolor{Sepia}{x_i}$}
\psfrag{xj}{$\textcolor{Sepia}{x_j}$}
\psfrag{i}{$\nu_i$}
\psfrag{pithj}[][][0.9]{$\textcolor{Sepia}{\pi\th_j}$}
\includegraphics[width=0.7\textwidth]{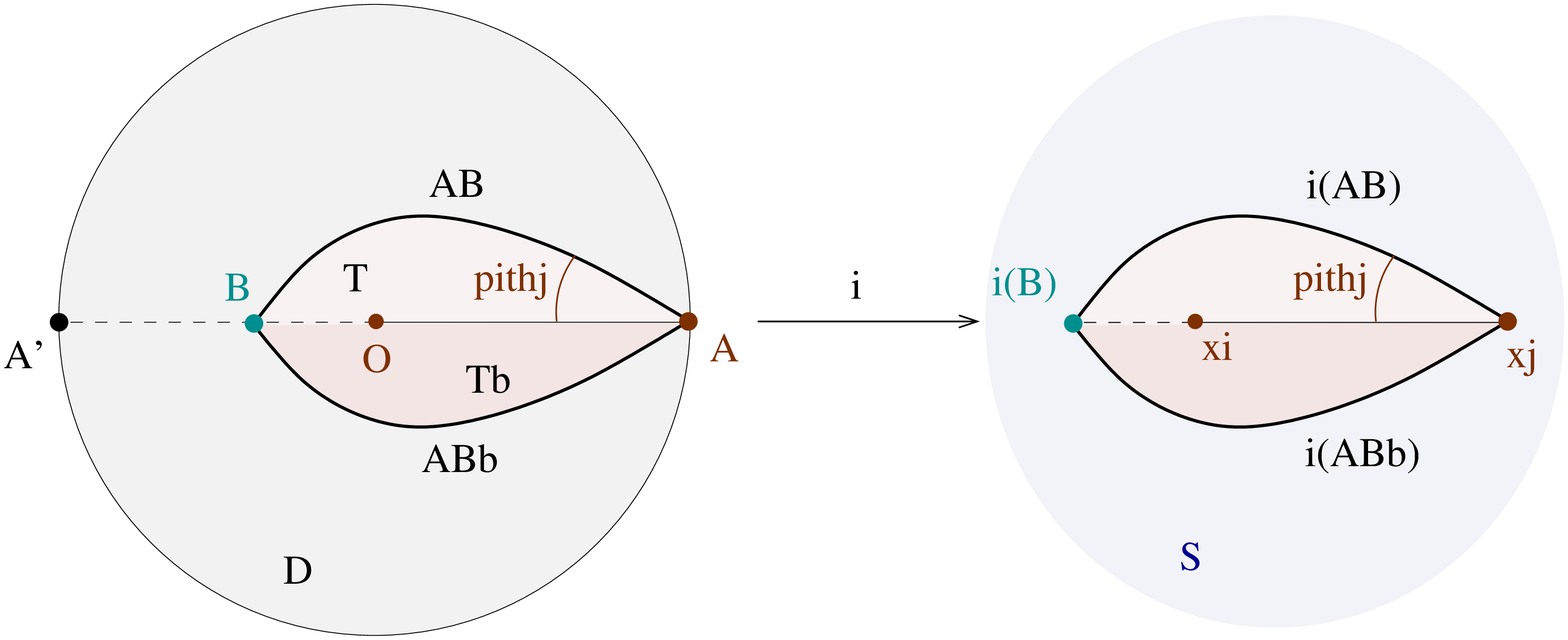}
\caption{{\small The map $\nu_i$ cannot restrict to a local isometry on $T\cup\ol{T}$.}}\label{fig:point-close-b}
\end{figurehere}
\end{center}

For every $w$ on the radius $[-z,0]$,
consider the two geodesic triangles
$T,\ol{T}\subset \DD_{\th_i}(d_{i})$
with vertices $0,z,w$ and call
$zw\subset T$ and $\ol{zw}\subset \ol{T}$ the two segments between $z$ and $w$.
Lemma \ref{rightangletr}(b) implies that we can choose $w$ so that both
triangles $T$ and $\ol{T}$  
form an angle $\pi\th_j$ at $z$. We deduce that the images of $\nu_i(zw)$ and $\nu_i(\ol{zw})$ in $S$ coincide, which contradicts to the fact that $\nu_i$ is a local isometry close to point $w$. 
\end{proof}

The following statement is a variation of the above lemma.

\begin{corollary}[Existence of a cone points of large angle close to a cone point of small angle]\label{closestcone}
Let $S$ be a spherical surface with conical singularities and suppose that $\th_i\le \frac{1}{3}$ and that $\chi(\dot S)<0$. Let $x_j$ be the conical point closest to $x_i$. Then $\th_j\ge\frac{1}{3}$.
\end{corollary}
\begin{proof} 
The proof of this statement repeats the proof of Lemma \ref{twopoints120} with the difference that instead of Lemma \ref{rightangletr}(b) one applies Lemma \ref{rightangletr}(c). As in that proof we consider the map $\nu_i:\DD_{\th_i}(\bar{r}_i)\rar S$ that is locally isometric on the interior of the disk. Using exactly the same notations and reasoning as the proof of  Lemma \ref{twopoints120} we construct a point $w$ in the interior of $\DD_{\th_i}(\bar{r}_i)$ where the map  $\nu_i$ is not a local isometry. This gives us a contradiction.
\end{proof}

We can now summarise the proof of the main result of this section.

\begin{proof}[Proof of Theorem \ref{thm:closest-points}]
Assertion (a) is equivalent to the last claim of Lemma \ref{closeconepoint}
and assertion (b) is exactly the content of Lemma \ref{twopoints120}.
Finally (c) is proven in Corollary \ref{closestcone}.
\end{proof}

%\begin{proof}[Proof of Proposition \ref{4pi3prop}] 
%Let $p$ be a saddle critical point of $V_S$ on the connected component $S_{\le t}$ with the lowest value of $V_S$. Then the point $p$ is the midpoint of a shortest geodesic arc $\alpha$ with endpoints in $\bm{x}$, which is thus completely contained in $S_{\le t}$. In particular, $d(x_k,\pa\alpha)\geq \ell(\alpha)$ for all conical points $x_k$ which are not endpoints of $\alpha$.
%
%Now there are two cases.
%\begin{itemize}
%\item[(a)]
%$\alpha$ is a loop or based at some conical point $x_i$.\\
%Since $d\Big(x_k,x_i\Big)\ge \ell(\alpha)=2r_i$,
%we conclude that $\th_i>\frac{2}{3}$ by Lemma \ref{closeconepoint}.
%\item[(b)]
%$\alpha$ joins two distinct conical points $x_i,x_j$.\\
%Since $d\Big(x_k,\{x_i,\,x_j\}\Big)\ge \ell(\alpha)=d(x_i,x_j)$,
%we conclude that $\th_i+\th_j>\frac{2}{3}$ by Lemma \ref{twopoints120}.
%\end{itemize}
%%Let us show that $\th_i>\frac{2}{3}$ in case (a), and that $\th_i+\th_j>\frac{2}{3}$ in case (b).
%%
%%In the first case we have $d\Big(x_k,\{x_i,\,x_j\}\Big)\ge d(x_i,x_j)$ for all $x_k\in S$, since otherwise there would be a geodesic shorter than $2V_S(p)$ in $S_{\le t}$ that joins $x_k$ with one of two points $x_i$ or $x_j$. The midpoint $p'$ of such a geodesic would be a saddle point with $V_S(p)<V_S(p')$. Hence, the inequality  $\th_i+\th_j>\frac{2}{3}$ follows from Lemma \ref{twopoints120}.
%%
%%By a similar reasoning it follows that in the second case we are in the setting of Lemma \ref{closeconepoint} and so $\th_i>\frac{2}{3}$. This proves the proposition.
%\end{proof}

%%%%%%%%%%%%%%%%%%%%%%%

\section{Geometry of the Voronoi function}\label{sec:voronoi}

In this section we begin our study of a central geometric object associated to a spherical surface: the Voronoi function. Here we recall its definition. 

\begin{definition}[Voronoi function and Voronoi graph] Let $S$ be a surface with a spherical metric and conical points $\bm{x}$. The {\it Voronoi function} $\Vor_S:S\rar\RR$ is defined as $\Vor_S(p):=d(p,\bm{x})$. 
The {\it Voronoi graph} $\Gamma(S)$ is locus of points $p\in\dot{S}$ at which
the distance $d(p,\bm{x})$ is realized by two or more arcs joining $p$ to $\bm{x}$.
\end{definition}

We will simply write $\Vor=\Vor_S$ and $\Gamma=\Gamma(S)$ when no ambiguity is possible.

In  Subsection \ref{vorongraphsec} we establish various elementary properties of $\Gamma$ and $\Vor$. In particular, we show that $\Gamma$ is a finite graph with at most $-3\chi(\dot S)$ geodesic edges. 

In Subsection \ref{morsetheorysec} we undertake Morse-theoretic study of the Voronoi function, this can be done even though $\Vor$ is non-smooth at $\Gamma\cup \bm{x}$. {\it Critical points} of $\Vor$ can be classified into local minima, local maxima and saddle points (see Theorem \ref{locmaxint}). An analogous study of the distance function to a finite subset of $\RR^2$ was conducted by Siersma in \cite{siersma}, however our case differs in several aspects.

In Subsection  \ref{saddlegeosec} we first derive a bound on the number of {\it critical values} of $\Vor$ in terms of $\chi(\dot S)$ (Proposition \ref{critnumber}). Then  we study saddle geodesics, namely the ``unstable submanifolds'' of saddle points of $\Vor$. We show in Proposition \ref{prop:Voronoi-cellular-decomposition} that saddle geodesics cut the surface into spherical disks and then prove a number of auxiliary results needed for Theorem \ref{main:many}.

\subsection{Voronoi graph}\label{vorongraphsec}
Here we derive some basic properties of the Voronoi graph $\Gamma$. In particular, we show that in Lemma \ref{voronoisgraph} and Corollary \ref{edgesvertices} that $\Gamma$ is a graph with at most $-3\chi(\dot S)$ geodesic edges. We prove as well the bound $\Vor<\pi$ and estimate from above the lengths of level sets of $\Vor$ (Corollary \ref{Corlengthoflevel}).
 
%Consider a surface $S$ with spherical metric and conical points $\bm{x}$
%of angles $2\pi\bm{\th}$.  

%\begin{definition}[Voronoi function]
%The {\it Voronoi function} $\Vor:S\rar\RR$ of the spherical surface $(S,\bm{x})$ is defined as $\Vor(p):=d(p,\bm{x})$. 
%\end{definition}

%The Voronoi function $V_{S}$ captures a lot of information about the spherical metric. In this section we study basic properties of $\Vor$ and of the corresponding {\it Voronoi graph}. 
\begin{lemma}[Upper bound for $\Vor$]\label{Vlessthanpi}
The Voronoi function satisfies the inequality $\Vor<\pi$.
\end{lemma}
\begin{proof}
Let $O,O'$ be antipodal points in $\Sph$.
By contradiction, suppose  that there is a point $p\in S$ such that $\Vor(p)\ge \pi$. 
Then there is a locally isometric map $\Sph\setminus \{O'\}\rar \dot{S}$
that takes the origin $O$ to $p$.
It is easy to see that it extends to a continuous map $\nu:\Sph\rar S$. 
%Denote by $O'$ the complementary point to $\inte{\DD}(\pi)$ in  $\Sph$. 
One can check that the map $\nu$ has to be a branched cover at $O'$. Hence $\nu: \Sph\to S$ is a branched cover map with at most one ramification. It follows that $\nu$ is an isometry, which contradicts $\chi(\dot{S})<0$.
\end{proof}

\begin{notation}
We use the symbol $\DD(r)$ to denote the standard disk $\DD_1(r)$.
For every $p\in \dot{S}$, let
$\nu_p:\DD(\Vor(p))\rar S$ be  a continuous map which takes the center $0$ to $p$
and which is a local isometry on the interior $\inte{\DD}(\Vor(p))$.
Such $\nu_p$ is clearly unique up to rotations of the disk.
\end{notation}

\begin{lemma}[Finitely many geodesics realize $\Vor$]\label{V_Smax} 
For any point $p\in \dot{S}$ there are finitely many smooth geodesic segments of length $\Vor(p)$ that join $p$ and $\bm{x}$.
\end{lemma}
\begin{proof}
Let $p$ be any point in $\dot{S}$.
The wished geodesic segments pull back via $\nu_p$ to radii of $\DD(\Vor(p))$
joining $0$ to a point of $\nu_p^{-1}(\bm{x})$.
It is easy to see that $\nu_p^{-1}(\bm{x})$ is a discrete and so finite subset of  $\partial\DD(\Vor(p))$, hence the proof is complete.
\end{proof}

\begin{definition}[Multiplicity of a point in the Voronoi graph]\label{longdefinition} 
The {\it{multiplicity}} $\mu_p$ of a point $p\in S$ is the number of
geodesic segments of length $\Vor(p)$ that join $p$ with $\bm{x}$.
\end{definition}

By definition, the Voronoi graph $\Gamma$ is the set of points
$p\in S$ of multiplicity greater than one.
The subset $\Gamma$ can be presented as the union of the locus
$\Gamma_0$ of points of multiplicity at least $3$
and of the locus $\Gamma_1$ of points of multiplicity exactly $2$.

\begin{notation}
Given $p$ be a point in $\Gamma$,
denote by $(z_1,z_2,\dots,z_{\mu_p})$ the
cyclically ordered subset of points in $\pa\DD(\Vor(p))$
that are mapped to $\bm{x}$ by $\nu_p$.
%
%The preimage $\nu_p^{-1}(\bm{x})$
%consists of cyclically ordered points 
%$\{y_j\,|\,j\in\ZZ/m_p\}$ in $\partial\DD(\Vor(p))$.
Denote by $R_j$ be the radius in $\DD(\Vor(p))$ that joins $0$ and 
the midpoint of the arc of $\pa\DD(\Vor(p))$ bounded by $z_j,z_{j+1}$
for $j\in\ZZ/m_p$, and by $R$ the union of all radii $R_j$.
\begin{center}
\begin{figurehere}
\psfrag{O}{$\textcolor{Green}{0}$}
\psfrag{y1}{$\textcolor{Sepia}{z_1}$}
\psfrag{y2}{$\textcolor{Sepia}{z_2}$}
\psfrag{y3}{$\textcolor{Sepia}{z_3}$}
\psfrag{y4}{$\textcolor{Sepia}{z_4}$}
\psfrag{r1}{$\textcolor{Green}{R_1}$}
\psfrag{r2}{$\textcolor{Green}{R_2}$}
\psfrag{r3}{$\textcolor{Green}{R_3}$}
\psfrag{r4}{$\textcolor{Green}{R_4}$}
\psfrag{D(V(p))}{$\textcolor{Blue}{\DD(\Vor(p))}$}
\psfrag{U}{$\textcolor{Blue}{U}$}
\includegraphics[width=0.35\textwidth]{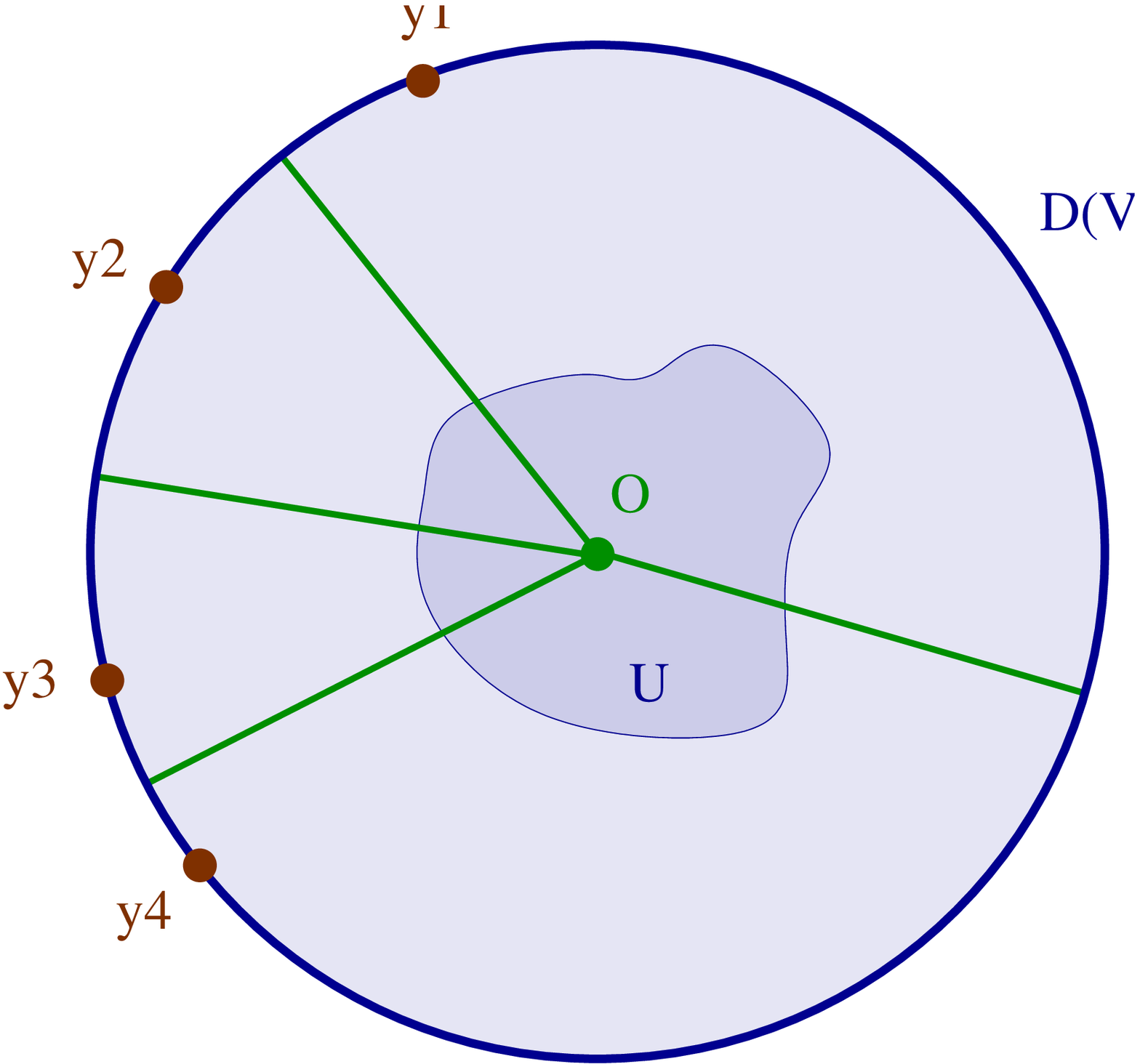}
\caption{{\small Local model of Voronoi graph near a vertex of multiplicity $4$.}}\label{fig:Voronoi-vertex}
\end{figurehere}
\end{center}
We denote by $\d_{z_j}$ the distance function $d(z_j,\cdot):\DD(\Vor(p))\rar\RR$
and by $\d_{\bm{z}}$ the minimum of all such $\d_{z_j}$.
\end{notation}

\begin{lemma}[$\Gamma$ is a finite graph with geodesic edges]\label{voronoisgraph}
The subset $\Gamma_0$ consists of finitely many points ({\it{vertices}}) and
$\Gamma_1$ is the disjoint union of finitely many locally
closed smooth geodesic segments ({\it{edges}}). Thus, $\Gamma$
is a $1$-dimensional CW complex embedded in $\dot{S}$
and the valence of each vertex coincides with its multiplicity.\\
Moreover, near a point $p\in \dot{S}$ the function $\Vor$ is locally the minimum
of $\mu_p$ smooth distance functions.
\end{lemma}
\begin{proof}
Let $p$ be a point in $\Gamma$.
%The preimage $\nu_p^{-1}(\bm{x})$
%consists of cyclically ordered points 
%$\{y_j\,|\,j\in\ZZ/m_p\}$ in $\partial\DD(\Vor(p))$.
%Let $R_j$ be the radius in $\DD(\Vor(p))$ that joins $O$ and the midpoint of $[y_j,y_{j+1}]$
%for $j\in\ZZ/m_p$ and let $R=\bigcup_j R_j$.
%
%
Since $\Vor(p)<\pi$, all the distance functions $\d_{z_j}$ are smooth.
It is easy to see that there is a small neighbourhood $U$ of $0\in\DD(\Vor(p))$ such that
$\Vor\circ\nu_p:U\rar\RR$ coincides with $\d_{\bm{z}}$.

As a consequence, $\Vor$ is the minimum of $\mu_p$ smooth functions near $p$
and $\nu_p(U\cap R)=\nu_p(U)\cap \Gamma$.
It follows that Figure \ref{fig:Voronoi-vertex}
describes a neighbourhood of a point of $\Gamma$ inside $S$.
\end{proof}

%\begin{notation} In the light of Lemma \ref{voronoisgraph},
%points of $\Gamma_0$ are called {\it{vertices}}
%and connected components of $\Gamma_1$ are called
%{\it{edges}} of the Voronoi graph. We will denote by
%$|\Gamma_0|$ the number of vertices of $\Gamma$
%and, by abuse of notation, by $|\Gamma_1|$ the number of edges.
%\end{notation}

\begin{definition}[Voronoi domains]
The {\it{open Voronoi domain}} $\DVint_{x_i}$ is the connected component of $S\setminus\Gamma$
that contains $x_i$ and the {\it{Voronoi domani}}
$\DV_{x_i}$ is the closure of $\DVint_{x_i}$ inside $S$.
We denote by 
$\DVbar_{x_i}$ the metric completion of $\DVint_{x_i}$.
%
%Furthermore, we denote by $\inte{D}_i$ the connected component of $S\setminus \Gamma$ that contains $x_i$, by $\bar{D}_i$ the metric completion of $\inte{D}_i$ and
%by $D_i$ the closure of $\inte{D}_i$ inside $S$.\\
%Let $e$ be an edge of $\Gamma$ e let $p\in\inte{e}$. There are exactly two oriented geodesic paths $\alpha,\alpha'$ of length $\Vor(p)$ from $p$ to $\bm{x}$.
%We denote by $e^\star$ the homotopy class of an arc with endpoints in $\bm{x}$ obtained by concatenating $\alpha^{-1}$ and $\alpha'$.
\end{definition}

Note that $\DVbar_{x_i}$ is a topological disk and that there is a continuous surjective map
$\DVbar_{x_i}\rar \DV_{x_i}$. Thus, $\DV_{x_i}$ is a topological disk if and only if such map is a homeomorphism,
which happens if and only if there is no point $p\in\Gamma$ that can
be joined to $x_i$ by more than one geodesic of length $\Vor(p)$. 
For example, in Figure \ref{fig:Voronoi-graph} both domains
$\DV_{x_1}$ and $\DV_{x_2}$ are topological cylinders.

\begin{center}
\begin{figurehere}
\psfrag{G(S)}{$\textcolor{Green}{\Gamma}$}
\psfrag{x1}{$\textcolor{Sepia}{x_1}$}
\psfrag{x2}{$\textcolor{Sepia}{x_2}$}
\psfrag{S}{$\textcolor{blue}{S}$}
\psfrag{D1}{$\textcolor{blue}{\DV_{x_1}}$}
\psfrag{D2}{$\textcolor{blue}{\DV_{x_2}}$}
\includegraphics[width=0.4\textwidth]{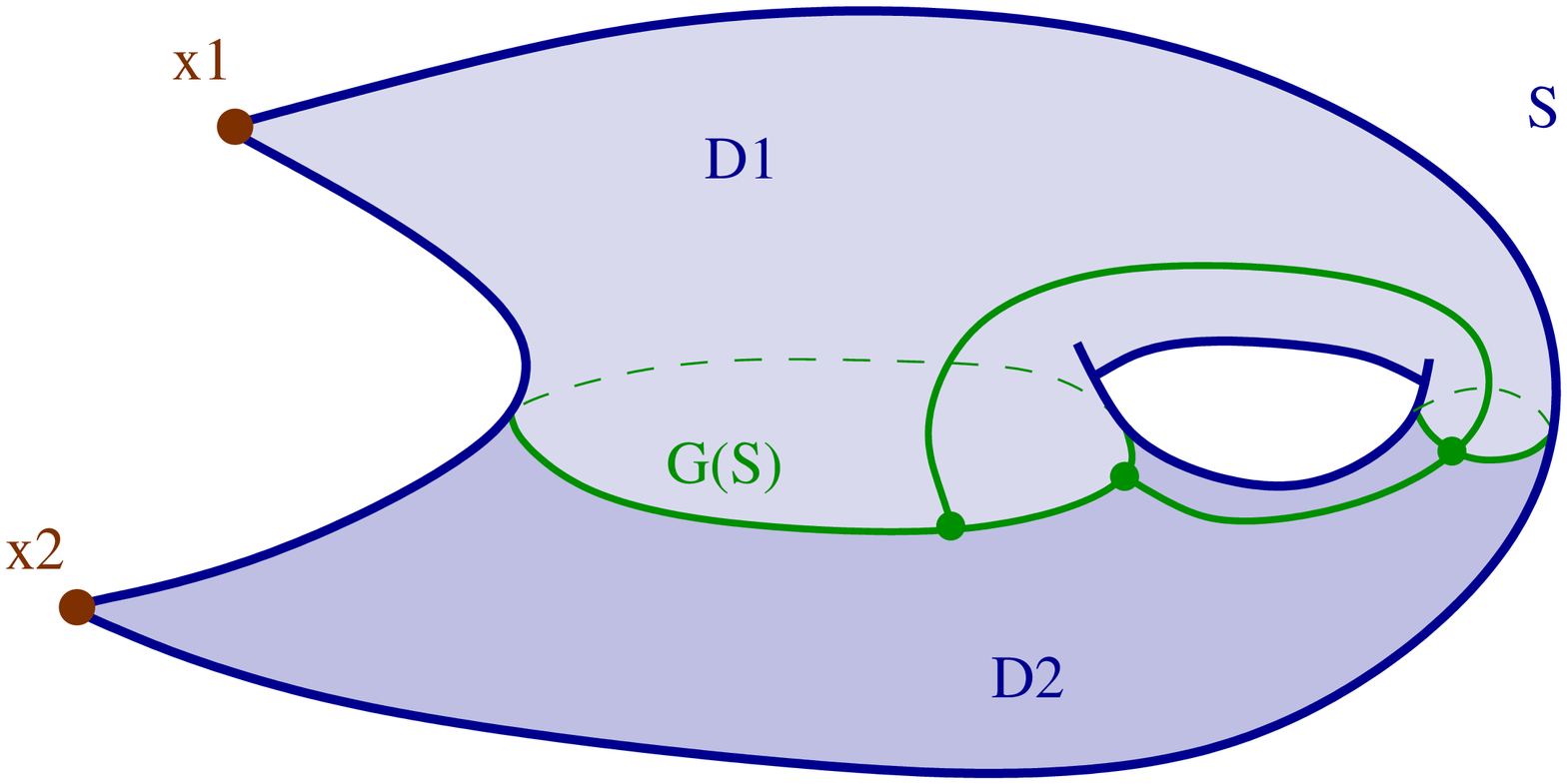}
\caption{{\small An example of Voronoi graph in the case $(g,n)=(1,2)$.}}\label{fig:Voronoi-graph}
\end{figurehere}
\end{center}

\begin{corollary}[Size of the Voronoi graph]\label{edgesvertices} 
The graph $\Gamma$ has at most $6g-6+3n$ edges and at most $4g-4+2n$ vertices. Moreover, its vertices have valence at least three.
\end{corollary}

We will denote by $|\Gamma_0|$ the number of vertices of $\Gamma$
and, by abuse of notation, by $|\Gamma_1|$ the number of edges.

\begin{proof}[Proof of Corollary \ref{edgesvertices}]
%Let $v$ be the number of vertices of $\Gamma$
%and $e$ be the number of edges. 
The last claim follows from Lemma \ref{voronoisgraph}. Hence,
we have $|\Gamma_0|\le \frac{2}{3}|\Gamma_1|$. 
Since $S=\Gamma_0\cup \Gamma_1\cup\left(\bigcup_i \DVint_{x_i}\right)$,
the Euler characteristic of $S$ satisfies
$|\Gamma_0|-|\Gamma_1|+n=\chi(S)=2-2g$. It easily follows that $|\Gamma_1|\le 6g-6+3n$ and $|\Gamma_0|\leq 4g-4+2n$.
\end{proof}

\begin{lemma}[Convexity of the disks $\DVbar_{x_i}$]\label{graphconvex}
Each completion $\DVbar_{x_i}$ is a convex polygon. In other words, 
$\DVbar_{x_i}$ is a disk with piecewise-geodesic boundary and
any two adjacent geodesic sides in $\DVbar_{x_i}$  form an angle strictly smaller than $\pi$. 
\end{lemma}
\begin{proof}
Fix $p\in\Gamma_0\cap \DV_{x_i}$.
%Assume that the multiplicity of $m_p$ of $p$ is $3$ to simplify notations (this does not affect the proof).
%Cyclically order the preimages $\nu_p^{-1}(\bm{x})$ counterclockwise and
Let $z_1,z_2,z_3\in\pa\DD(\Vor(p))$ be three points of $\nu_p^{-1}(\bm{x})$, which are consecutive in the natural cyclic order,
%Let $R_1$ (resp. $R_2$) be the ray of $\DD(\Vor(p))$
%that joins the origin $O$ to the midpoint of the arc $[y_1,y_2]$ on $\pa\DD(\Vor(p))$ (resp. of the arc $[y_2,y_3]$).
%
%Let $y_1,y_2,y_3\in\pa\DD(\Vor(p))$
%be the cyclically ordered preimages $\nu_p^{-1}(\bm{x})$ 
%and let $r_1$ (resp. $r_2$) be the ray of $\DD(\Vor(p))$
%that joins the origin $O$ to the midpoint of $[y_1,y_2]$ (resp. $[y_2,y_3]$).
and let $\hat{R}_1$ be the diameter of $\DD(\Vor(p))$ obtained by prolonging $R_1$.

\begin{center}
\begin{figurehere}
\psfrag{O}{$\textcolor{Green}{0}$}
\psfrag{y1}{$\textcolor{Sepia}{z_1}$}
\psfrag{y2}{$\textcolor{Sepia}{z_2}$}
\psfrag{y3}{$\textcolor{Sepia}{z_3}$}
\psfrag{y4}{$\textcolor{Sepia}{z_4}$}
\psfrag{D(VS)}{$\textcolor{blue}{\DD(\Vor(p))}$}
\psfrag{r1}{$\textcolor{Green}{R_1}$}
\psfrag{r2}{$\textcolor{Green}{R_2}$}
\psfrag{R1}{$\textcolor{Green}{\hat{R}_1}$}
\includegraphics[width=0.35\textwidth]{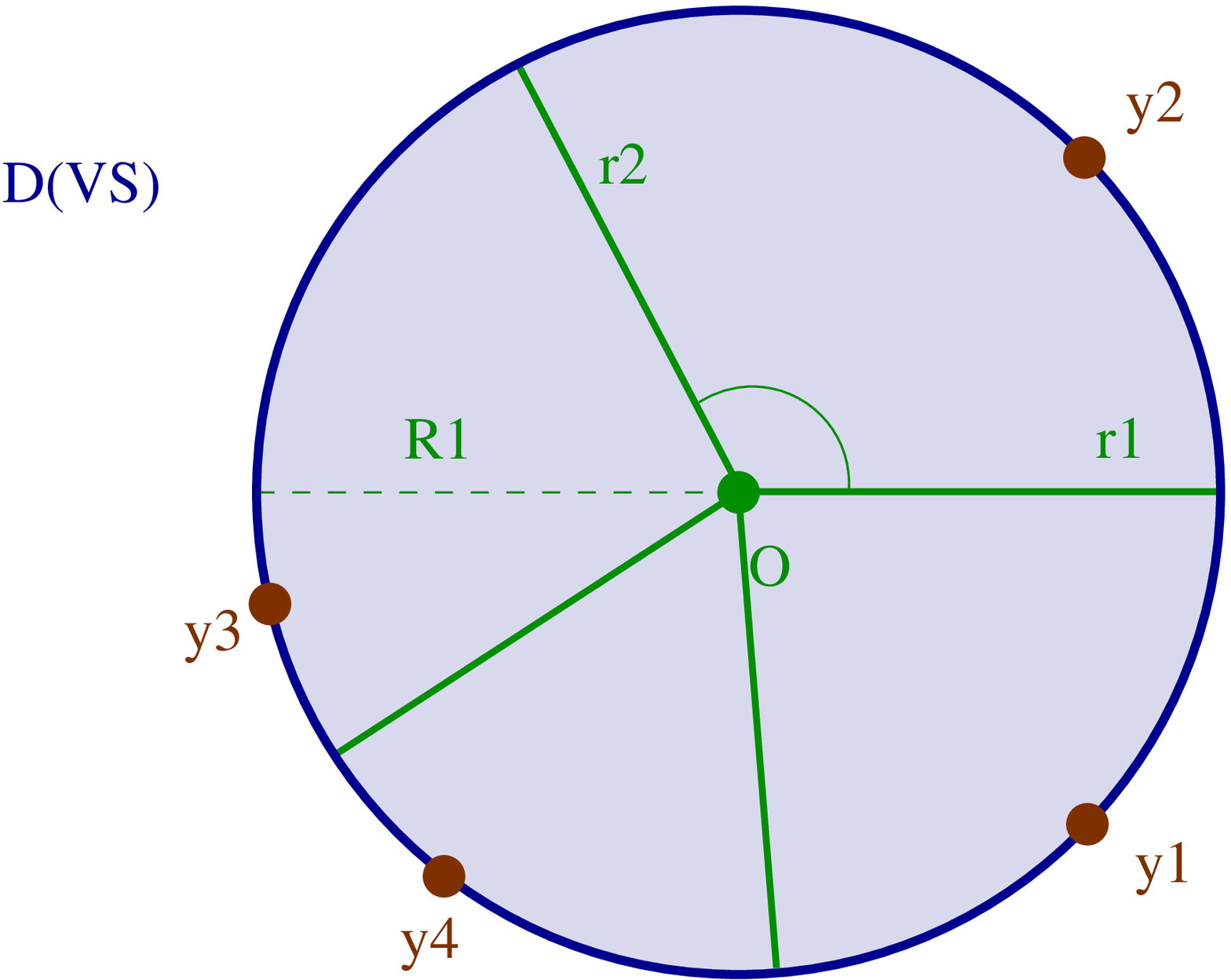}
\caption{{\small Local picture of an angle at a vertex of multiplicity $4$ of the Voronoi graph.}}\label{fig:Voronoi-angles}
\end{figurehere}
\end{center}

Since $d(w,z_1)>d(w,z_2)$ at all points $w\in R_2\setminus\{0\}$,
it follows that $R_2\setminus\{0\}$ must be contained in the component
of $\DD(\Vor(p))\setminus \hat{R}_1$ that contains $R_2$.
Hence, the component of $\DD(\Vor(p))\setminus(R_1\cup R_2)$ that contains
$z_2$ has an angle strictly smaller than $\pi$ at the origin
and the conclusion follows. Alternatively, the counter-clockwise angle at $0$ from $0z_1$ to $0z_3$ is twice the counter-clockwise angle from $R_1$ to $R_2$.
\end{proof}

The following corollary will be useful in the future.

\begin{corollary}[Upper bound for the length of level curves of $\Vor$]\label{Corlengthoflevel} 
Let $S$ be a spherical surface with conical points
of angles $2\pi\bm{\th}$. Then for every $r\in (0,\pi)$ the level set $\Vor^{-1}(r)$ is a finite union of arcs of spherical circles of radius $r$ and its length satisfies
$$
\ell(\Vor^{-1}(r))\leq
2\pi \sin(r) \|\bm{\th}\|_1\le 2\pi r\|\bm{\th}\|_1. %=2\pi r\sum_i\th_i.
$$
\end{corollary}
\begin{proof}
Since $S$ is covered by domains $\DV_{x_i}$, it is enough to prove that for each $i$ the intersection  $\Vor^{-1}(r)\cap \DV_{x_i}$ has total length at most $2\pi \sin(r)\th_i$. Clearly, this intersection is a locally isometric image of the curve in $\DVbar_{x_i}$ consisting of points at distance $r$ from $x_i$.  By Lemma \ref{graphconvex} the polygon $\DVbar_{x_i}$ is star-shaped at $0$
and so the latter curve is a union of
arcs at constant distance $r$ from $x_i$,
whose total length is at most $2\pi\sin(r)\th_i\leq 2\pi r\th_i$.
\end{proof}

\subsection{Critical points and critical values of Voronoi function}\label{morsetheorysec}

In this subsection we analyse the Voronoi function from a Morse-theoretic point of view. Even though $\Vor$ is non-smooth at $\bm x\cup \Gamma$, one can speak of regular and critical points of $\Vor$: regular points are points close to which the level sets of $\Vor$ define locally a continuous foliation. The main result of this subsection is Theorem \ref{locmaxint},
 which provides a classification of possible types of critical points of $\Vor$.

\begin{definition}[Regular and critical points of the Voronoi function]\label{def:critical}
A point $p\in S$ is called {\it regular} for $\Vor$ if
there exists a real-valued continuous function $f$ on some neighbourhood of $p$ 
such that the pair of functions $(\Vor,f)$ 
gives local continuous coordinates on $S$ at $p$. 
A point $p$ that is not regular for $\Vor$ is called {\it critical}:
such a critical point $p$ is called {\it saddle} if $\Vor$ the both subsets $\{\Vor>\Vor(p)\}$ and $\{\Vor<\Vor(p)\}$ contain $p$ in their closures.  A value $c\in \RR$ is called {\it critical} if the level set $\Vor^{-1}(c)$ contains some critical point. A value $c\in \RR$ is called {\it saddle} if the level set $\Vor^{-1}(c)$ contains a saddle point.
\end{definition}

\begin{theorem}[Classification of critical points of $\Vor$]\label{locmaxint}  
The locus of critical points of $\Vor$
is the union of $\bm{x}$ with a subset of $\Gamma$ that  
consists of some closed edges of $\Gamma$ and finitely many isolated points in $\Gamma$.
More precisely, all critical points of $\Vor$ can be classified in the following types.
\begin{itemize}
\item[(a)] 
\emph{Isolated minima} form the set $\bm{x}$ of conical points of $S$.
\item[(b] 
\emph{Isolated local maxima} are located in $\Gamma$. The value of $\Vor$ is larger than $\frac{\pi}{2}$ at isolated local maxima that occur on edges of $\Gamma$.
Such points are isolated local maxima for the restriction $\Vor|_{\Gamma}$.
\item[(c)] 
\emph{Saddle points} are contained in $\Gamma_1$ and the value of $\Vor$ at them is smaller than $\frac{\pi}{2}$. Any saddle point $p$ is a midpoint of a geodesic segment or a loop based at $\bm{x}$ of length $2\Vor(p)<\pi$. Such $p$ is an isolated local minimum
for $\Vor|_{\Gamma}$.
\item[(d)] 
\emph{Non-isolated local maxima} form a disjoint union of  closed edges of $\Gamma$ that lie in the level set $\Vor=\frac{\pi}{2}$. Such points are non-isolated local maxima for $\Vor|_{\Gamma}$.
\end{itemize}
Moreover, if $p\in\Gamma_0$ is a vertex which is not a local maximum, there is exactly one oriented edge $\vec{e}$ outgoing from $p$ such that $\Vor|_{\vec{e}}$ is increasing near $p$.
\end{theorem}

%Note that, if $\Vor$ is smooth at $p$ and $(d\Vor)_p\neq 0$, then $p$ is regular for $\Vor$. 
In the following sequence of lemmas we will analyse critical points of $\Vor$ according to their position in $S$ 
with respect to the Voronoi graph $\Gamma$. 

\begin{lemma}[Regularity of $\Vor$ on $\dot{S}\setminus \Gamma$]\label{regoutofGamma} 
Any point $p\in \dot{S}$ that does not belong to $\Gamma$ is regular.
\end{lemma}
\begin{proof}
Let $p\in \dot{S}\setminus \Gamma$. There exists $i$ so that $p\in \DVint_{x_i}$.
By Lemma \ref{Vlessthanpi}, points of the disk $\DVint_{x_i}$ are at distance less than $\pi$ from $x_i$. 
By Definition \ref{longdefinition} each point $p$ of 
 $\DVint_{x_i}$ is connected by a unique geodesic of length $\Vor(p)$ with $x_i$, while $d(p,x_j)>\Vor(p)$ for all $j\ne i$. Hence $\Vor$ is smooth at $p$ and has non-zero gradient at $p$, and so $p$ is regular. It follows that all critical points of $\Vor$ apart from $\bm{x}$ are contained in $\Gamma$.
\end{proof}

\begin{lemma}[Critical points for $\Vor$ on an edge of $\Gamma$]\label{lemma:two-critical}
Let $e$ be an edge of $\Gamma$ and $\inte{e}$ be its interior.
Then either of the two occurs:
\begin{itemize}
\item[(a)]
$\Vor$ is constantly equal to $\frac{\pi}{2}$ on $e$ and so all points of $e$ are critical;
\item[(b)]
$\Vor$ has isolated critical points on $e$ and attains at most two critical values on $\inte{e}$ of which at most one is a saddle critical value. A saddle critical value is always smaller than $\frac{\pi}{2}$. Moreover, each critical point  $p\in \inte{e}$ lies on a geodesic arc or a loop of length $2\Vor(p)$ based at $\bm{x}$.
\end{itemize}
\end{lemma}
\begin{proof} 
Assume for simplicity that the endpoints of $e$ are distinct and $e$ is adjacent to two distinct
Voronoi domains $\DV_{x_i}$ and $\DV_{x_j}$, the general case being very similar. Denote $y$ and $y'$ the endpoints of $e$. Then $S$ contains a spherical quadrilateral $\Lambda$, bounded by sides $x_i y$, $y x_j$, $x_j y'$, $y' x_i$, and symmetric with respect to its diagonal $e=yy'$.

Consider a developing map
$\iota:\Lambda\rar \Sph$ and
call $E$ the maximal circle that contains $\iota(e)$ and $X_i,X_j$ the two points
$X_i:=\iota(x_i)$ and $X_j:=\iota(x_j)$. 
Denote by $\d_{\bm{X}}:\Sph\rar\RR$ the distance functions from $\{X_i,X_j\}$.
% and $X_j$ respectively,
%and let $d_{\bm{X}}:=\min\{d_{X_i},\,d_{X_j}\}$.

It is easy to see that $\Vor\Big|_{{\Lambda}}$ coincides with
$\d_{\bm{X}}\circ\iota$ and so the two functions have the same critical values on ${\Lambda}$.

\begin{center}
\begin{figurehere}
\psfrag{Sph}{$\textcolor{Blue}{\Sph}$}
\psfrag{z}{$\vspace{0.2cm}\textcolor{Sepia}{X_i}$}
\psfrag{z'}{$\vspace{-0.2cm}\textcolor{Sepia}{X_j}$}
\psfrag{s}{$P$}
\psfrag{M}{$Q$}
\psfrag{i(e)}{$\textcolor{Green}{E}$}
\psfrag{A}{Case (a)}
\psfrag{B}{Case (b)}
\includegraphics[width=0.5\textwidth]{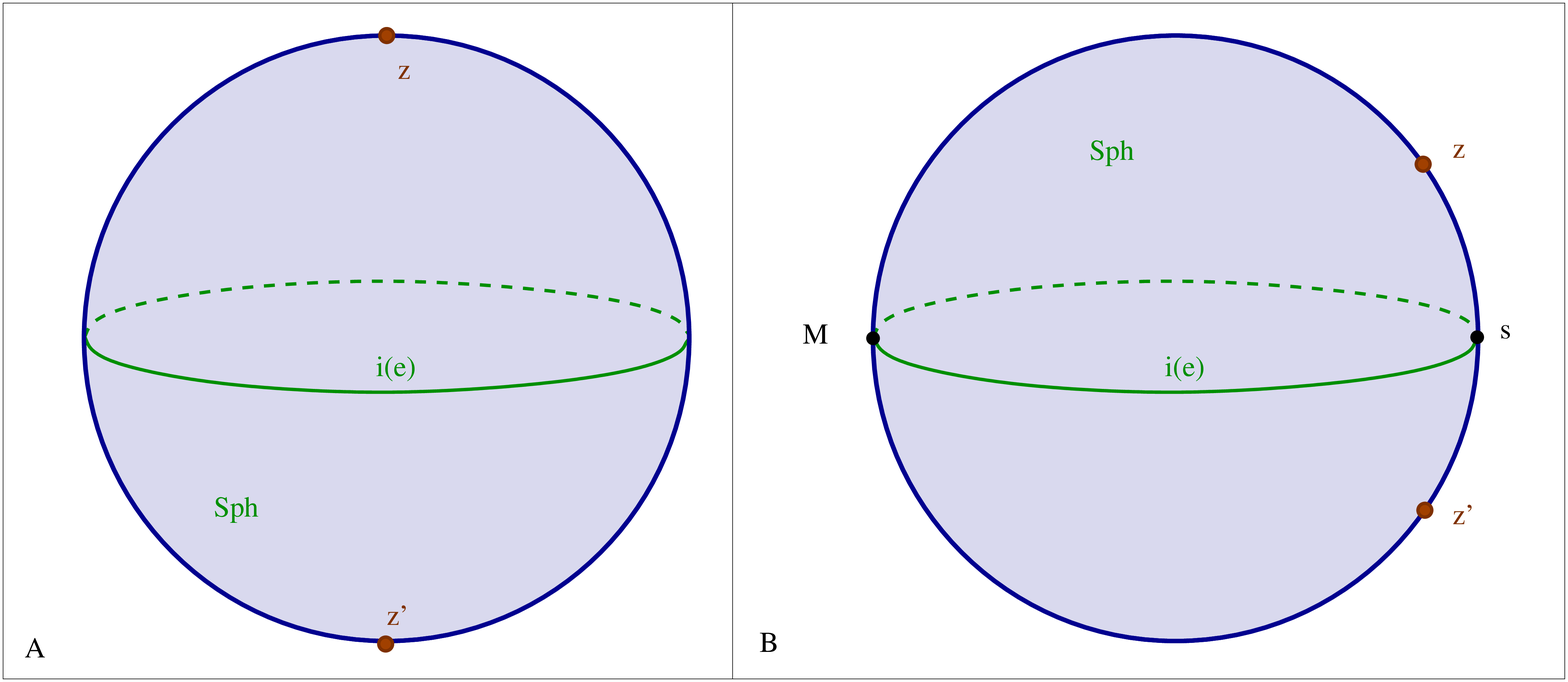}
\caption{{\small The two configurations for $X_i,X_j,E$ on $\Sph$.}}\label{fig:Voronoi-edges}
\end{figurehere}
\end{center}

The conclusion about critical values follows by noting
that the function $\d_{\bm{X}}:\Sph\rar\RR$ satisfies either of the following:
\begin{itemize}
\item[(a)]
$\d_{\bm{X}}$ is constantly equal to  $\frac{\pi}{2}$ along the maximal circle $E$;
\item[(b)]
$\d_{\bm{X}}$ is non-constant along $E$, with isolated critical points $P,Q$ 
that lie at the intersection of $E$ with the maximal circle passing through $X_i,X_j$.
The function $\d_{\bm{X}}$ attains two critical values $\d_{\bm{X}}(P)<\d_{\bm{X}}(Q)$ on $E$. 
Moreover, $\d_{\bm{X}}(P)+\d_{\bm{X}}(Q)=\pi$ and
the critical point $P$ is a saddle, whereas the
critical point $Q$ is a local maximum.
\end{itemize}
The geodesic arcs passing through the critical points on $\inte{e}$ are preimages in $\Lambda$ of two arcs of the maximal circle on $\Sph$ passing through $X_i$ and $X_j$.
\end{proof}

The following lemma will help us to analyse the behaviour of $\Vor$ at the vertices of $\Gamma$.

\begin{lemma}[Minimum function of finitely many smooth functions]\label{minsmoothfunc} Let $f_1,\ldots, f_\mu$ be smooth functions defined
on a neighbourhood of a point $P\in \Sph$ such that $f_1(P)=\ldots=f_\mu(P)$ and such that $(d f_i)_P\neq 0$ for all $i$. 
\begin{itemize}
\item[(a)]
Suppose that there is a vector $v\in T_P\Sph$ such that $(df_i)_P(v)>0$ for all $i$. Then the function $\min\{f_i\}$ is regular at $P$.
\item[(b)]
Suppose that for every non-zero vector $v\in T_P\Sph$ 
there exists $i$ such that $(d f_i)_P(v)<0$. Then the function $\min\{f_i\}$ has a local maximum at $P$.
%
%\item Suppose that functions $f_i$ are concave close to $p$. Then either $f$ is regular at $p$ or it has a local maximum at $p$.
%
\end{itemize}
\end{lemma}
This lemma is standard, so we only give a short proof.
\begin{proof}[Proof of Lemma \ref{minsmoothfunc}]
(a) Let $f$ be any smooth function in a neighbourhood of $P$ such that $df_P(v)=0$ and $df_P\neq 0$. Then it is not hard to check that the pair $\left(f, \min\{f_i\}\right)$ defines continuous local coordinates in a neighbourhood of $P$. Hence $\min\{f_i\}$ is regular at $P$ according to our definition.

(b) This can be proven by taking restriction of $\min\{f_i\}$ to any geodesic ray passing through $P$. 
%3. This is so since the function $\min_i f_i$ is concave at $p$.
\end{proof}

\begin{lemma}[Critical values for $\Vor$ at a vertex of $\Gamma$]\label{critvert} 
The Voronoi function $\Vor$ can have the following behaviour at a vertex $p\in\Gamma_0$ of $\Gamma$.
%Let $p\in \Gamma_0$ be a vertex of $\Gamma$, then $\Vor$ can have the following behaviour at $p$.
\begin{itemize}
\item[(a)] 
$\Vor$ is regular at $p$. 
\item[(b)]
$\Vor$  attains an isolated local maximum at $p$. 
\item[(c)]
$\Vor$ attains  a non-isolated local maximum at $p$. \\
This case occurs only if $\Vor(p)=\frac{\pi}{2}$ and $\Vor$ is identically equal to $\frac{\pi}{2}$ on an edge $e$ of $\Gamma$ adjacent to $p$. Moreover, such $e$ is not a loop and it is the unique edge incident at $p$ on which
$\Vor$ takes the constant value $\frac{\pi}{2}$.
%is unique and it does not from a loop in $\Gamma$ based at $p$.
\end{itemize}
\end{lemma}
\begin{proof}
We have seen that the map $\nu_p:\DD(\Vor(p))\rar S$
isometrically identifies a neighbourhood $U$ of the center $0\in\DD(\Vor(p))$ with a neighbourhood of $p\in S$.
It was explained in the proof of Lemma \ref{voronoisgraph} that,
up to restricting $U$, the function $\Vor$ on the neighbourhood $\nu_p(U)$ of $p$
can modelled on the function $\d_{\bm{z}}:\DD(\Vor(p))\rar\RR$, where  ${\bm z}$ is a collection of  points $(z_1,\ldots, z_{\mu_p})$ going counterclockwise along the boundary $\partial \DD (\Vor(p))$.

In order to analyse $\d_{\bm{z}}$ near $0$, we consider the following three cases (see Figure \ref{fig:Voronoi-critical}).

\begin{center}
\begin{figurehere}
\psfrag{O}{$\textcolor{Green}{0}$}
\psfrag{y1}{$\textcolor{Sepia}{z_1}$}
\psfrag{y2}{$\textcolor{Sepia}{z_2}$}
\psfrag{y3}{$\textcolor{Sepia}{z_3}$}
\psfrag{D(VS)}{$\textcolor{blue}{\DD(\Vor(p))}$}
\psfrag{A}{Case (a)}
\psfrag{B}{Case (b)}
\psfrag{C}{Case (c)}
\psfrag{b}{$\textcolor{Blue}{\zeta}$}
\includegraphics[width=0.85\textwidth]{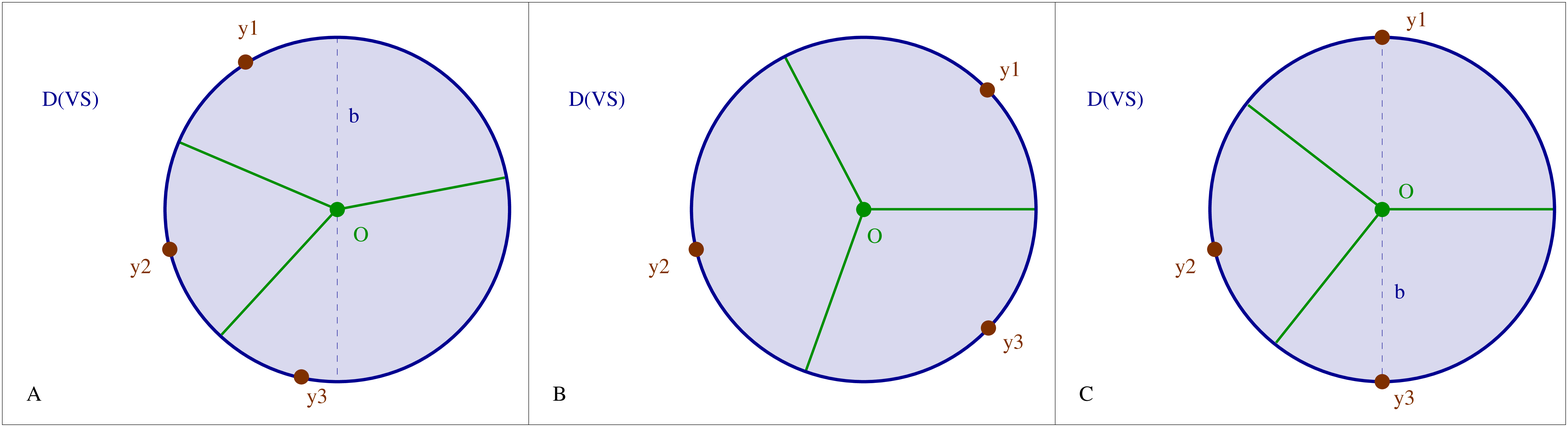}
\caption{{\small Types of critical point at a vertex of multiplicity $3$ of the Voronoi graph.}}\label{fig:Voronoi-critical}
\end{figurehere}
\end{center}

\begin{itemize}
\item[(a)]
There is a diameter $\zeta$ in  $\DD(\Vor(p))$ passing through $0$ such that $\bm z$ lies  in one connected component of $\DD(\Vor(p))\setminus \zeta$. Let us show in this case that $0$ is regular for $\d_{\bm{z}}$. 
%Indeed, function $d(z,{\bm y})$ is equal to $\min_id(z, y_i)$. 
At the point $0$ all the gradient vectors $\nabla \d_{z_j}$ are transversal $\zeta$ 
and point toward the same half of $\DD(\Vor(p))\setminus \zeta$ (namely, the half that does not contain $\bm{z}$). It follows that we are in Case (a) of Lemma \ref{minsmoothfunc} and $0$ is a regular point, which corresponds to Case (a) of this lemma.
\item[(b)]
Suppose now that for any diameter $\zeta$  passing trough $0$ there are two points $z_i$ and $z_j$ lying in its complement and separated by it. In such a case  $p$ is an isolated local maximum, since it is easy to see that we are in Case (b) of Lemma \ref{minsmoothfunc}.
\item[(c)]
The remaining situation to analyse is when (after a cyclic reordering) points $z_1$ and $z_{\mu_p}$ are opposite points on the circle $\partial \DD(\Vor(p))$, whereas $z_2,\ldots, z_{\mu_p-1}$ lie one half-circle with endpoints $z_1$, $z_{\mu_p}$. 
Call $\zeta$ the diameter of $\DD(\Vor(p))$ that joins $z_1$ and $z_{\mu_p}$ and
split again this situation into three subcases.
\begin{itemize}
\item[(c1)]
$\Vor(p)=d(0,\bm{z})<\frac{\pi}{2}$.\\
In this case $0$ is a regular point for $\d_{\bm{z}}$. 
In fact, for every $j$ the gradient vector $\nabla \d_{z_j}$ at $0$
is nonzero and it again points toward the half-disk of $\DD(\Vor(p))\setminus \zeta$ that does not contain $\bm{z}$.
So we are again in Case (a) of Lemma \ref{minsmoothfunc}.
%
%Instead of giving the full proof we just indicate the idea. Namely, by direct inspection one can check that for  some $\varepsilon>0$ and any value $t\in (d(O,y_i)-\varepsilon, d(O,y_i)+\varepsilon)$ the level set $d(z,{\bm y})=t$ is a connected piecewise smooth curve in $\mathbb D(\Vor(p))$ that cuts the disk into two smaller ones. Moreover such level sets vary continuously (in Hausdorff topology) with $t$.
\item[(c2)]
$\Vor(p)=d(0,\bm{z})>\frac{\pi}{2}$.\\
In this case the function $\min\{\d_{z_1},\, \d_{z_{\mu_p}}\}$ on $\DD(\Vor(p))$ attains its isolated global maximum at $0$. 
Clearly, the same holds for the function $\d_{\bm z}$. 
\item[(c3)]
$\Vor(p)=d(0,\bm{z})=\frac{\pi}{2}$.\\
We will show that we are in Case (c) of the current lemma. 

Let $w$ be endpoint of $R_{\mu_p}$ on $\partial \DD(\pi/2))$, which lies at distance $\pi/2$ from $z_1$ and $z_{\mu_p}$
on the arc of  $\partial \DD(\pi/2)$ going counter-clockwise from $z_{\mu_p}$ to $z_1$, and let $-w$ be the  point on $\partial \DD(\pi/2)$ opposite to $w$. The function $\min\{\d_{z_1},\, \d_{z_{\mu_p}}\}$ attains  its maximum $\pi/2$ 
on the diameter $[-w,w]$ of $\DD(\pi/2)$.
It is not hard to see that all points of the radius $[0,w]$ (resp. $[-w,0]$)
different from $0$
are at distance  larger (resp. smaller) than $\pi/2$ from $z_2,\dots,z_{\mu_p-1}$.
This proves that the level set $\{\d_{\bm{z}}=\pi/2\}$ coincides with the radius $[0,w]$. This finishes the analysis of this case and finishes the proof of the lemma. 
\end{itemize}
\end{itemize}
\end{proof}

\begin{proof}[Proof of Theorem \ref{locmaxint}] 
The first statement of this Corollary follows directly from Lemmas \ref{regoutofGamma} and \ref{lemma:two-critical}. The classification follows from  Lemmas \ref{lemma:two-critical} and \ref{critvert}. The last claim is a consequence of the classification of critical points at a vertex of $\Gamma$.
\end{proof}

\subsection{Saddle critical points and saddle geodesics}\label{saddlegeosec}
In this section basing on classification of  critical points of Voronoi function we give a bound on the number of its saddle critical values, see  Proposition  \ref{critnumber}. Additionally to this we start our study of {\it saddle geodesics} and prove in particular that they cut the surface in a union of disks, see Proposition \ref{prop:Voronoi-cellular-decomposition}.

\begin{definition}[Saddle geodesics] Let $S$ be a spherical surface and $\gamma$ be a geodesic arc or loop based at $\bm x$. We call $\gamma$ a {\it saddle geodesic} in case the midpoint $p$ of $\gamma$ is a saddle point for $\Vor$ and $\ell(\gamma)=2\Vor(p)$.
If $\gamma$ is an arc we call it a {\it saddle arc}; if it is a loop, we call it a {\it saddle loop}.
\end{definition}

By Theorem \ref{locmaxint} each saddle critical point belongs to a unique saddle geodesic.

%\begin{definition} Let $S$ be a spherical surface, and $\gamma$ be a geodesic arc or a geodesic loop based at $\bm x$. We call $\gamma$ a {\it saddle arc} or respectively a {\it saddle loop} in case the midpoint $p$ of $\gamma$ is a saddle point and $\ell(\gamma)=2\Vor(p)$. By Theorem \ref{locmaxint} each saddle critical point belongs to a unique saddle arc or loop.
%\end{definition}

\begin{proposition}[Number of critical values of $\Vor$]\label{critnumber} 
Let $S$ be a spherical surface with conical points $\bm{x}$ and assume that $\chi(\dot{S})<0$.
The Voronoi function $\Vor$ has the following properties.
\begin{itemize}
%
%All critical points of $\Vor$ apart from conical points lie on $\Gamma$.
\item[(a)] 
The systole of $S$ is the minimal non-zero critical value of $\Vor$ and it is a saddle value.
\item[(b)]
The number of non-zero critical values of $\Vor$ is at most $|\Gamma_0|+2|\Gamma_1|\leq -8\chi(\dot S)$. 
\item[(c)] 
The number of saddle critical values of $\Vor$ is at most $|\Gamma_1|\leq -3\chi(\dot S)$ and all saddle values lie in the interval 
$\left(0,\frac{\pi}{2}\right)$.
\end{itemize}
\end{proposition}
\begin{proof}
%Let $p\in \dot{S}\setminus \Gamma$. There exists $i$ so that $p\in \inte{D}_i$. By Lemma \ref{\Vormax}, points of the disk $\inte{D}_i$ are on distance less than $\pi$ from $x_i$. By Definition \ref{longdefinition} each point $p$ of  $\inte{D}_i$ is connected by a unique geodesic of length $\Vor(p)$ with $x_i$, while $d(p,x_j)>\Vor(p)$ for all $j\ne i$. Hence $\Vor(p)$ is smooth at $p$ and has non-zero gradient at $p$, and so $p$ is regular. It follows that all critical points of $\Vor$ apart from $\bm{x}$ are contained in $\Gamma$.
%
All values in the interval $(0,\sys(S,\bm{x}))$ are regular for $\Vor$ by definition of the systole. 
At the same time, it is easy to see that the midpoint $s$ of a geodesic $\sigma_{\sys}$ 
that realizes the systole is a critical point of $\Vor$.
According to Theorem \ref{locmaxint}, the point $p$ can only be a local maximum or a saddle.
The only case in which it could be an isolated local maximum is that $\Gamma$ consists just of 
the single point $p$, 
which happens only if $S$ is the round sphere with one conical point of angle $2\pi$.
Also, $p$ can be a non-isolated local maximum only if
$\Vor$ takes constant value $\pi/2$ on $\Gamma$,
namely only if $S$ has genus $0$ with $n=2$ conical point at distance $\pi$ from each other.
Thus, both cases above are ruled out by the hypothesis $\chi(\dot{S})<0$.
As a consequence, $p$ is a saddle point and (a) is proven.

As for (b), all critical points of $\Vor$ (apart from $\bm{x}$) belong to $\Gamma$ and the function $\Vor$ can attain at most two critical values in the interior of each edge of $\Gamma$
by Lemma \ref{lemma:two-critical}.
Thus, $\Vor$ has at most $|\Gamma_0|+2|\Gamma_1|$ critical values. By Corollary \ref{edgesvertices},
the number of vertices of $\Gamma$ is at most $-2\chi(\dot S)$ and the number of edges is at most $-3\chi(\dot S)$.
It follows that
$|\Gamma_0|+2|\Gamma_1|\leq -8\chi(\dot{S})$
and so (v) is proven.

To prove (c) recall that according to Lemma \ref{critvert} critical points at vertices of $\Gamma$ have to be local maxima. Hence, Lemma \ref{critvert} together with Lemma \ref{lemma:two-critical} imply that the number of saddle critical values  of $\Vor$ is bounded by the number of edges of $\Gamma$. Moreover, according to Lemma \ref{lemma:two-critical} each saddle value is less than $\frac{\pi}{2}$.
\end{proof}

\subsubsection{Delaunay-Morse decomposition of a spherical surface}\label{sec:delaunay}

As a consequence of the previous analysis, we can produce a cellular decomposition of $S$
by applying Morse theory to the Voronoi function $\Vor$.
Note first that the flow associated to the gradient
vector field $\nabla \Vor$ on $\dot{S}\setminus\Gamma$
determines a deformation retraction $R_{\Vor}:\dot{S}\rar \Gamma$.

\begin{proposition}[Delaunay-Morse decomposition of $S$]\label{prop:Voronoi-cellular-decomposition}
Let $\Gamma^s$ be the set of saddle points and $\Gamma^m$
the set of local maxima for $\Vor$.
\begin{itemize}
\item[(a)]
The $\Gamma\setminus\Gamma^s$ is a disjoint union of open trees $\tau_l$.
The intersection of the critical locus of $\Vor$ with $\tau_l$ is either an isolated local maximum
or an edge of non-isolated local maxima.
\item[(b)]
The surface $S$ has a cell decomposition with
$0$-cells given by $\bm{x}$,
open $1$-cells
$R_{\Vor}^{-1}(\Gamma^s)$ consisting of all open saddle arcs and loops,
and one open $2$-cell $R_{\Vor}^{-1}(\tau_l)$ for every open tree $\tau_l$.
%geodesic arcs through saddle points for $\Vor$ and with endpoints in $\bm{x}$, 
%and with $2$-skeleton $S^{(2)}=S$.
%Moreover, $R_{\Vor}^{-1}(\tau_l)$ is an open $2$-cell for every tree $\tau_l$.
\item[(c)]
The Euler characteristic of $\dot{S}$ satisfies $\chi(\dot{S})=|\Gamma^m|-|\Gamma^s|$,
where $|\Gamma^m|$ (resp. $|\Gamma^s|$) is the number of connected components
of $\Gamma^m$ (resp. the cardinality of $\Gamma^s$).
\end{itemize}
\end{proposition}
\begin{proof}
By Theorem \ref{locmaxint} the points in $\Gamma^s$ correspond to isolated local minima
for $\Vor|_{\Gamma}$ and the points in $\Gamma^m$
correspond to local maxima for $\Vor|_{\Gamma}$.

%First we make a preliminary observation.
%Let $G\subset\Gamma\setminus\Gamma^s$ be a subgraph in which all vertices
%are bivalent. Then the restriction $\Vor|_G$ has no local minimum in the interior of $G$.
%By contradiction, a local minimum $p\in G$ cannot belong to the interior of an edge
%because $G\cap \Gamma^s=\emptyset$. On the other hand, if $p$ is a vertex of $G$
%and $\vec{e}_1,\vec{e}_2$ are the two oriented edges of $G$ outgoing from $p$,
%the last claim of Corollary \ref{locmaxint} implies that $\Vor$ cannot be increasing
%both along $\vec{e}_1$ and $\vec{e}_2$, and so $p$ cannot be a local minimum.

In order to prove (a), consider
a connected component $\tau_l$ of $\Gamma\setminus\Gamma^s$.
The subgraph $\tau_l$ does not contain local minima for $\Vor|_{\Gamma}$.
Moreover, for each vertex $p$ of $\tau_l$ which is not a local maximum
there is a unique oriented edge $\vec{e}$ in $\tau_l$ outgoing from $p$
such that $\Vor$ increases along $\vec{e}$ near $p$.
It follows that the flow on $\tau_l$ induced
by the gradient vector field $\nabla \left(\Vor|_{\tau_l}\right)$ gives a deformation
retraction of $\tau_l$ onto $\tau_l\cap \Gamma^m$.
It follows that $\tau_l\cap \Gamma^m$ consists of one connected component and that
$\tau_l$ is an open tree.

Claim (b) is an easy consequence of (a) and the fact that $R_{\Vor}$ is a deformation retraction.
Also, (c) is obtained by computing $\chi(S)$ with respect to the cellular decomposition
defined in (b).
\end{proof}

\subsubsection{Additional properties of saddle critical points and geodesic}\label{sec:additional}
In this subsection we collect several results concerning saddle points and geodesics, needed for the proof of Theorem \ref{main:many}. The first lemma  gives a sufficient condition for a geodesic to be saddle.

\begin{lemma}[Geodesics far from other conical points are saddle]\label{saddlechar} Let $S$ be a spherical surface and $x_i, x_j$ be two conical points. Let $\gamma$ be a geodesic segment that joins $x_i$ with $x_j$. 
\begin{itemize}
\item[(a)]
If $\max(d(x_k,x_i),d(x_k,x_j))\ge \ell(\gamma)$
%$d(x_k,\{x_i,x_j\})\ge \ell(\gamma)$ 
for all $k\neq i,j$, then $\gamma$ is a saddle arc.
\item[(b)]
If $\gamma$ is a geodesic loop ($i=j$) and $d_i=d(x_i,\bm x\setminus x_i)\geq\ell(\gamma)$, then $\gamma$ is a saddle loop.
%
%Then $\gamma$ is a saddle arc if for any $k\ne i,j$ we have $\max(d(x_i,x_k),d(x_j,x_k))\ge \ell(\gamma)$. Similarly, a geodesic loop $\gamma$ based at $x_i$ is saddle if for any $k\ne i$ we have $d_i\ge \ell(\gamma)$.
\end{itemize}
\end{lemma}
\begin{proof} 
As for case (a), let $p$ be the midpoint of $\gamma$. To prove that this is a saddle point it is enough to show that for any $k\ne i,j$ we have $d(p,x_k)>\frac{\ell(\gamma)}{2}$. Assume the converse and let $\gamma'$ be a geodesic segment of length at most $\frac{\ell(\gamma)}{2}$ that joins $x_k$ with $p$. We can assume that $x_k$ does not belong to $\gamma$ and so $\gamma'$ and $\gamma$ meet at $p$ under non-zero angle. It is clear then that we can smooth the union of $\gamma'$ and $x_i p$ into a curve shorter than $\ell(\gamma)$ that joins $x_i$ and $x_k$. It follows $d(x_i,x_k)<\ell(\gamma)$. In the same way we prove that $d(x_j,x_k)<\ell(\gamma)$ and get a contradiction. The proof of case (b) is analogous.
\end{proof}

Next, we state a Morse-theoretic lemma that we will need to prove Theorem \ref{main:many}(b).

\begin{lemma}
%[Points in distinct components of a suplevel are joined by a geodesic through a saddle]
[Saddle geodesics from disconnected level sets]\label{saddlepath} 
Let $S$ be a spherical surface and $c$ be a regular value of Voronoi function $\Vor$. Suppose that the level set $\Vor^{-1}(c)$ has two connected components $\lev_{1,c}, \, \lev_{2,c}$ that bound a cylinder $S'\subset S$. Assume that for points $x\in S'$ close to $\lev_{1,c}\cup\lev_{2,c}$  we have $\Vor(x)\le c$.  
Then there is a saddle point $s\in S'$ with $\Vor(s)=c'<c$  with the following properties:
\begin{itemize}
\item There is a path $\alpha\subset S'$ that joins a point  $q_{1,c}\in \lev_{1,c}$ with a point $q_{2,c}\in \lev_{1,c}$, passes through $s$ and such that $\Vor(\alpha)\ge c'$.
\item The path $\alpha$ is transversal to the saddle geodesic $\sigma_s\subset S'$ passing through $s$. In particular, in the case $\sigma_s$ is a loop, $\alpha$ separates the boundaries of $S'$ inside $S'$.
\end{itemize}
\end{lemma}

\begin{proof} 
Let $c'\in(0,c)$ be the maximum value such that points $q_{1,c}$ and $q_{2,c}$ lie in the same connected component of $\Vor^{-1}[c',\pi)\cap S'$. Then points $q_{1,c}$ and $q_{2,c}$ can be connected by a path $\alpha$ in $\Vor^{-1}[c',\pi)\cap S'$.

\begin{center}
\begin{figurehere}
\psfrag{s}{$\textcolor{red}{s}$}
\psfrag{sigma}{$\textcolor{red}{\sigma_s}$}
\psfrag{q1c}{$\textcolor{Purple}{q_{1,c}}$}
\psfrag{q2c}{$\textcolor{Purple}{q_{2,c}}$}
\psfrag{l1c}{$\textcolor{Purple}{\lambda_{1,c}}$}
\psfrag{l2c}{$\textcolor{Purple}{\lambda_{2,c}}$}
\psfrag{lc'}{$\textcolor{Purple}{\lambda_{c'}}$}
\psfrag{S}{$\textcolor{blue}{S}$}
\psfrag{S'}{$\textcolor{blue}{S'}$}
\psfrag{alpha}{$\textcolor{BrickRed}{\alpha}$}
\includegraphics[width=0.4\textwidth]{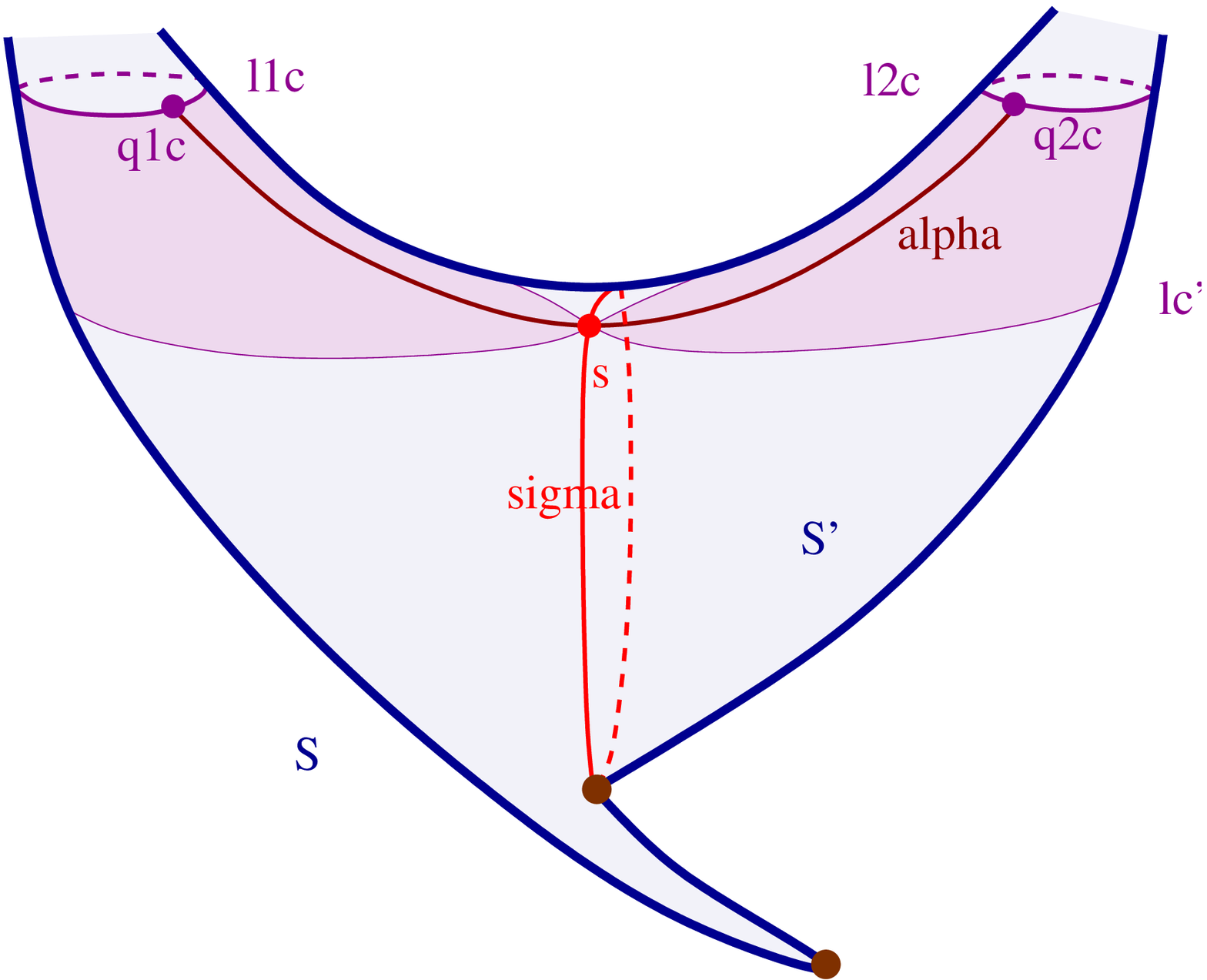}
\caption{{\small Existence of a saddle point $s$ on the level curve $\lambda_{c'}$.}}\label{fig:finding-saddle}
\end{figurehere}
\end{center}

It is not hard to see that removing local (isolated or non-isolated) maxima from $\Vor^{-1}[c',\pi)\cap S'$
does not produce more connected components. Thus,
if we remove all saddle points on the level $\Vor^{-1}(c')\cap S'$
from $\Vor^{-1}[c',\pi)\cap S'$, then $q_{1,c}$ and $q_{2,c}$ will be in two different connected components of the remaining surface. It follows, that $\alpha$ should pass through one of such saddle points $s$ and it will be transversal to a saddle arc of loop $\sigma_s$ through $s$. Clearly, if $\sigma_s$ is a loop, it has to separate $q_{1,c}$ from $q_{2,c}$ since it intersects $\alpha$ transversally at one point. 
\end{proof}

Finally, we analyse points of small conical angle and their neighbourhoods. We recall that $\W_{x_i}=B_{x_i}(d_i)$.

%\begin{definition} Let $S$ be a spherical surface, $x_i$ be a conical point and $d_i$ be the distance $d(x_i,{\bm x}\setminus x_i)$. Denote by $B_i$ the $d_i$-neighbourhood of $x_i$ in $S$ and by $\partial B_i$ the subset of points at distance $d_i$ from $x_i$. 
%\end{definition}

\begin{lemma}[$\DV_{x_i}$ and $\W_{x_i}$ at a point $x_i$ with small $\th_i$]\label{smallangleneigh} 
Let $S$ be a spherical surface with a conical point $x_i$ with $\th_i\le \frac{1}{3}$. 
Let $\sigma$ be a saddle geodesic passing through $x_i$ and $s\in \sigma$ be the corresponding saddle point. 
Let $x_j$ be a closest to $x_i$ conical point (i.e., $d(x_i,x_j)=d_i$).
Then 
\begin{itemize}
\item[(a)]
The Voronoi domain $\DV_{x_i}$ of $x_i$ belongs to the interior of $\W_{x_i}$. Consequently,  $\Vor(s)\in [\frac{d_i}{2}, d_i)$. 
\item[(b)]
%For any $x_k$ different from $x_i$ and $x_j$ we have $d(x_k,x_i)\geq
%d(x_k,x_j)+\max\{d_i-\pi \th_i,\ 0\}$.
For any $x_k$ different from $x_i$ and $x_j$ we have $d(x_k,x_i)>
d(x_k,x_j)$. 

Moreover, we have  $d(x_k,x_i)\ge d(x_k,x_j)+d_i-\pi \th_i$.

\item[(c)]
$\partial \W_{x_i}$ is contained inside $B_{x_j}(\pi\th_i)$.
\item[(d)]
If additionally $\th_i<\frac{1}{7}$, then $\partial \W_{x_i}$ is contained inside
$\inte{B}_{x_j}(d_i/2)$.
\end{itemize} 
\end{lemma}
\begin{proof}
(a) To prove that $\DV_{x_i}$ belongs to the interior of $\W_{x_i}$  it is sufficient to show that $\diam(\pa\W_{x_i})<d_i$.
Indeed, this would imply that for any $p\in \partial \W_{x_i}$ we have $d(p, x_j)<d_i=d(p,x_i)$, and so $\pa \W_{x_i}$ is disjoint from $\DV_{x_i}$. As a consequence, $\DV_{x_i}$ is contained in the connected
component of $S\setminus \pa\W_{x_i}$ that contains also $x_i$, which is in fact $\Wint_{x_i}$.
 
Since $\th_i\leq\frac{1}{3}$, we have $2r_i>d_i$ by Theorem \ref{thm:closest-points}(a),
and so $\bar{r}_i=d_i$. By Lemma \ref{immeradius} there is
a map $\nu_i:\DD_{\th_i}(d_i)\rar S$ that takes the origin to $x_i$ and which is a local isometry in
the interior, so that $\W_{x_i}=\nu_i(\DD_{\th_i}(d_i))$.
%, and view $\DD_{\th_i}(d_i)\subset \inte{\DD}_{\th_i}(\bar{r}_i)$.
%Since $r_i>\frac{d_i}{2}$, by Lemma \ref{immeradius} there is a locally isomeric map $\iota: D_{\th_i}(d_i)\to S$ sending the centre of this disk to $x_i$. 
Let us take points $z$ and $w$ in $\partial \DD_{\th_i}(d_i)$ such that $\nu_i(z),\nu_i(w)\in\pa \W_{x_i}$. Since $\th_i\le \frac{1}{3}$, it is easy to see that the distance between $z$ and $w$ 
inside $\DD_{\th_i}(d_i)$ is strictly less than $d_i$. Since $\nu_i$ is a local isometry on $\DD_{\th_i}(d_i)$, we deduce $d(\nu_i(z),\nu_i(w))<d_i$.

Let us now prove that $\Vor(s)\in [\frac{d_i}{2}, d_i)$. Since $s\in \DV_{x_i}\subset \Wint_{x_i}$, we already know that $\Vor(s)<d_i$. To see that $\Vor(s)\ge \frac{d_i}{2}$ consider two cases. If $\sigma$ is a saddle arc that connects $x_i$ with some conical point $x_k$, then we have $\Vor(s)=\frac{1}{2}d(x_i,x_k)\ge \frac{d_i}{2}$. On the other hand, if $\sigma$ is a saddle loop, we have $\ell(\sigma)\ge 2r_i>d_i$ and so $\Vor(s)=\frac{1}{2}\ell(\sigma)>\frac{d_i}{2}$.
 
%1) Let us first prove that $\Vor(p)\ge \frac{d_i}{2}$. Indeed, if $\gamma$ is a saddle arc that connects $x_i$ with some conical point $x_k$ we have $\Vor(p)=\frac{1}{2}d(x_i,x_k)\ge \frac{d_i}{2}$. On the other hand, if $\gamma$ is a saddle loop, using condition $\th_i<\frac{1}{3}$ we deduce from Theorem \ref{thm:closest-points} a) that $\ell(\gamma)\ge 2r_i>d_i$.
%
%Let us now prove that $p$ lies in the interior of $B_i$.
%Assume by contradiction that $d(x_i,p)\ge d_i$. Then there is a point $p'\in x_ip\subset \gamma$ that belongs to $\partial B_i$. By the definition of saddle arc the point $x_i$ is the closest conical point to $p'$. However we claim that $d(p',x_j)<d_i=d(p',x_i)$. To prove this, consider the standard disk $D_{\th_i}(d_i)$. Since $r_i>\frac{d_i}{2}$, by Lemma \ref{immeradius} there is a locally isomeric map $\iota: D_{\th_i}(d_i)\to S$ sending the centre of this disk to $x_i$. Let us take points $z_1$ and $z_2$ in $\partial(D_{\th_i}(d_i))$ such that $\iota(z_1)=x_j$ and $\iota(z_2)=p'$. Since $\th_i\le \frac{1}{3}$, it is easy to see that the distance between $z_1$ and $z_2$ inside $D_{\th_i}(d_i)$ is less than $d_i$. Since $\iota$ is a local isometry, we deduce $d(p',x_j)<d_i$, which is a contradiction.

\begin{center}
\begin{figurehere}
\psfrag{q}{$\textcolor{BrickRed}{q}$}
\psfrag{xi}{$\textcolor{Sepia}{x_i}$}
\psfrag{xj}{$\textcolor{Sepia}{x_j}$}
\psfrag{xk}{$\textcolor{Sepia}{x_k}$}
\psfrag{S}{$\textcolor{blue}{S}$}
\psfrag{Bmax}{$\textcolor{Green}{\W_{x_i}}$}
\psfrag{eta}{$\textcolor{BrickRed}{\ \alpha}$}
\includegraphics[width=0.35\textwidth]{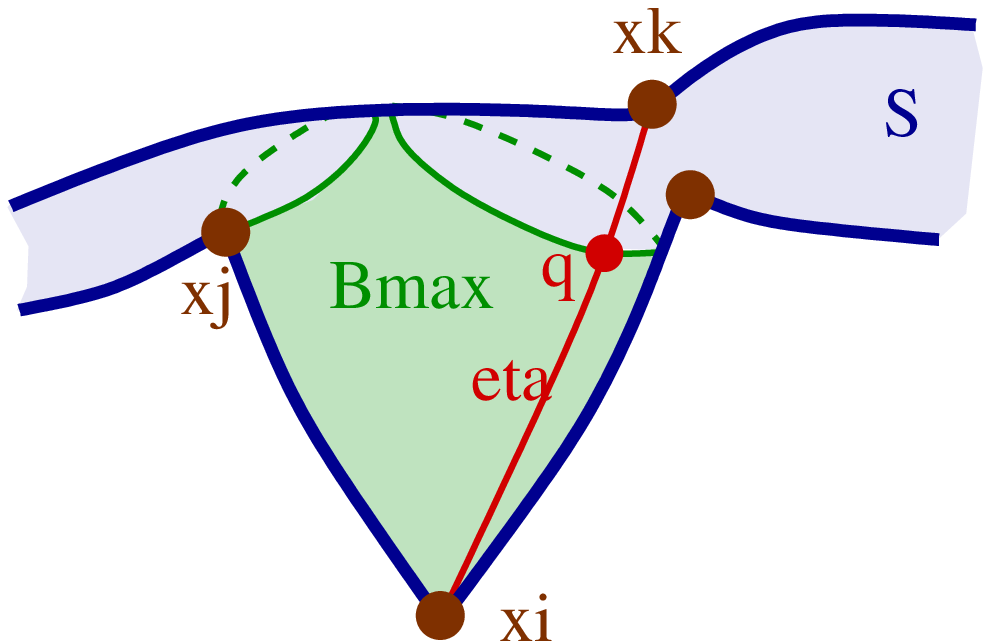}
\caption{{\small An example of maximal $1$-pointed ball $\W_{x_i}$.}}\label{fig:Bmax}
\end{figurehere}
\end{center}
%

%(b) Consider a path $\alpha$ of length $d(x_i,x_k)\ge d_i$ that joins $x_k$ with $x_i$, and let $q$ be the intersection of $\alpha$ with $\partial \W_{x_i}$. Clearly, $d(x_k,q)=d(x_i,x_k)-d_i$. At the same time, as it was explained in (a), we have $d(x_j,q)<d_i$. Since $\ell(\partial \DD_{\th_i}(d_i))\le 2\pi\th_i$, 
%we also have $\diam(\pa \W_{x_i})\leq \pi\th_i$ and so $d(x_j,q)\le \pi\th_i$.
%Using the triangle inequality we conclude $$d(x_k,x_j)\le d(x_j,q)+d(x_k,q)\leq \min\{d_i,\pi\th_i\}+(d(x_i,x_k)-d_i)$$
%which is equivalent to claim (b).

(b) Consider a path $\alpha$ of length $d(x_i,x_k)\ge d_i$ that joins $x_k$ with $x_i$, and let $q$ be the intersection of $\alpha$ with $\partial \W_{x_i}$. Clearly, $d(x_k,q)=d(x_i,x_k)-d_i$. At the same time, as it was explained in (a), we have $d(x_j,q)<d_i$. So, applying the triangle inequality, we get the first assertion:
$$d(x_k,x_j)\le d(x_j,q)+d(x_k,q)<d(x_i,x_k).$$

Let us now prove the second inequality. 
Since $\ell(\partial \DD_{\th_i}(d_i))\le 2\pi\th_i$, 
we also have $\diam(\pa \W_{x_i})\leq \pi\th_i$ and so $d(x_j,q)\le \pi\th_i$.
Using the triangle inequality we conclude $$d(x_k,x_j)\le d(x_j,q)+d(x_k,q)\leq d(x_i,x_k)-d_i+\pi\th_i.$$

%3) Using the notations of 2)  we note $$d(x_j,x_k)\le d(x_j,q)+d(q,x_k)=d(x_j,q)+d(x_i,x_k)-d_i+d(q,x_k).$$
%The statement now follows since $d(x_k,q)=d(x_i,x_k)-d_i$.
(c)  It is explained in (b) that $\diam(\partial\W_{x_i})\leq \pi\th_i$, and since the point $x_j$ belongs to $\partial \W_{x_i}$, the statement clearly holds.

(d) Again, it is not hard to see that the condition $\th_i<\frac{1}{7}$ implies that
the diameter of $\pa \DD_{\th_i}(d_i)$ is less than $\frac{d_i}{2}$. Hence,
$\diam(\pa\W_{x_i})<\frac{d_i}{2}$ and the conclusion follows.
\end{proof}

%%%%%%%%%%%%%%%%%%%%%%%%%%%%%%%%

\section{Voronoi cylinders and sublevel sets of Voronoi function}\label{sec:voronoi2}
 
In this section we turn our attention to two types of subsurfaces of spherical surfaces singled out by the Voronoi function. First, we study Voronoi cylinders (see  Definition \ref{defVoronoi} and Corollary \ref{cor:Vor}) and give a lower bound on the moduli of such cylinders, see Lemma \ref{cylinderslength}. This permits us to get a hold on conformal geometry of the surface. Next, we study various properties of the connected components of sublevel surfaces of $\Vor$, 
i.e. of subsets $\{\Vor\le c\}$.

\subsection{Voronoi cylinders and their modulus}

In this subsection we study simple subsurfaces of $S$ which are well foliated by level sets
of the Voronoi function $\Vor$.

\begin{definition}[Voronoi cylinders and caps]\label{defVoronoi} 
Let $(S,\bm x)$ be a spherical surface with conical points.
A cylindrical subsurface $C\subset \dot S$ without critical points of $\Vor$ and whose boundary components are  connected components of level sets of $\Vor$ is called a {\it Voronoi cylinder}. 
A disk in $\dot S$ whose boundary
is a connected component of a level set of $\Vor$
is called a {\it Voronoi cap} if the critical points of $\Vor$
contained in it consist
either of one isolated maximum, or of a segment in the level set $\Vor^{-1}\left(\frac{\pi}{2}\right)$. 
%A {\it small Voronoi cap} is a cap contained in $\Vor^{-1}(0,\frac{\pi}{2})$.
\end{definition}

In order to extract Voronoi cylinders from spherical surfaces we start with two standard lemmas.

\begin{lemma}[Local structure of $(S,\Vor)$ near a regular level set]\label{smallcylinder} 
Let $S$ be a spherical surface and $\lev_c$ be a connected component of a level set $\Vor^{-1}(c)$ such that all points of $\lev_c$ are regular. Then for some $\varepsilon>0$ the connected component $U_{\varepsilon}$  of $\Vor^{-1}([c-\varepsilon,c+\varepsilon])$ containing $\lev_c$ is a Voronoi cylinder. 
Moreover, the map
$\Vor: U_{\varepsilon}\to [c-\varepsilon,c+\varepsilon]$ is a continuous fibration 
with fibers homeomorphic to a circle.
%where each fibre is a connected component of a level set of $\Vor$ homeomorphic to a circle.
\end{lemma}
\begin{proof} This lemma is standard and follows from the fact that at small neighbourhood of a regular point of $\Vor$ the level sets of $\Vor$ form a continuous foliation (see also \cite{siersma}).
\end{proof}

\begin{lemma}[Local structure of $(S,\Vor)$ near a local maximum]\label{smallcap} 
Let $p$ be an isolated local maximum of $\Vor$ with $\Vor(p)=c$. Then there exists $\varepsilon$ such that the connected component $U_{\varepsilon}$ of $\Vor^{-1}([c-\varepsilon,c])$ containing $p$ is a Voronoi cap. Moreover the map $\Vor:(U_{\varepsilon}\setminus p)\to [c-\varepsilon,c)$ is a continuous fibration with fibers homeomorphic to a circle. The same statement holds if 
$p$ is replaced by an edge $e$ of $\Gamma$ such that $\Vor(e)=\frac{\pi}{2}$.
\end{lemma}
\begin{proof} 
This lemma follows from the analysis of local maxima given in Lemmas \ref{lemma:two-critical} and \ref{critvert} and from Lemma \ref{smallcylinder}.
\end{proof}

Combining these two lemmas we get the following corollary.

\begin{corollary}[Subsurfaces of type {$\Vor^{-1}([r',r''])$} without saddle points]\label{cor:Vor}
Let $S$ be a spherical surface with conical singularities and let $0<r'<r''<\pi$ be two regular values of $\Vor$. Suppose that the interval $[r',r'']$ does not contain saddle critical values of $\Vor$. Then
each connected component of $\Vor^{-1}([r',r''])$ is of the following type:
\begin{itemize}
\item 
a Voronoi cylinder bounded by a connected component of $\Vor^{-1}(r')$ 
and a connected component of $\Vor^{-1}(r'')$;
\item
a Voronoi cap whose boundary is a connected component of $\Vor^{-1}(r')$.
\end{itemize}
\end{corollary}

\begin{center}
\begin{figurehere}
\psfrag{S}{$\textcolor{Blue}{S}$}
\psfrag{xi}{$\textcolor{Sepia}{x_i}$}
\psfrag{xk}{$\textcolor{Sepia}{x_k}$}
\psfrag{xj}{$\textcolor{Sepia}{x_j}$}
\psfrag{G(S)}{$\textcolor{OliveGreen}{\Gamma}$}
\psfrag{C}{$\textcolor{Plum}{C}$}
\includegraphics[width=0.45\textwidth]{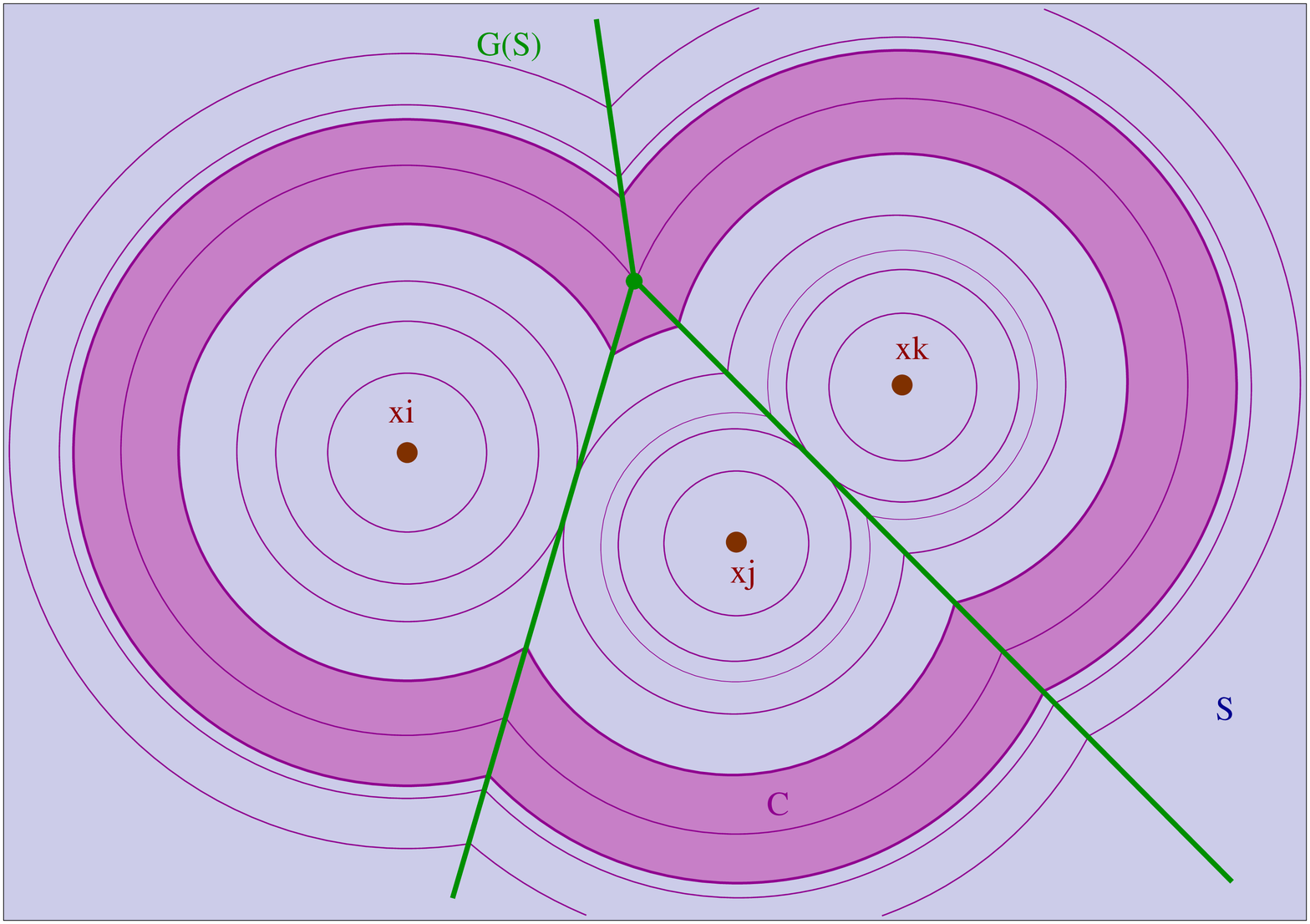}
\caption{{\small Example of level sets of $\Vor$
and of a component $C$ of $\Vor^{-1}([r',\,r''])$.}}\label{fig:Voronoi-foliation}
\end{figurehere}
\end{center}

\begin{proof} 
Consider first a connected component $C$ of $\Vor^{-1}([r',r''])$ without critical points.
It follows then from Lemma \ref{smallcylinder} that $C$ is a Voronoi cylinder and that
its boundary components should lie in the level sets $\Vor^{-1}(r')$ and $\Vor^{-1}(r'')$.

Consider now a connected component $A$ of $\Vor^{-1}([r',r''])$ that contains a critical point.
Then such a point should be a local maximum (by Theorem \ref{locmaxint} local minima are conical points of $S$). Hence the statement can be deduced from Lemma \ref{smallcap}.
%Let us  associate to $S$ the topological space $G_S$ (a graph) whose points are  connected components of level sets of $\Vor$. We call $V_{G_S}$ the function that $\Vor$ induces on $G_S$ and call $\pi: S\to G_S$ the projection.
%There is a natural topology of $G_S$ induced by the Hausdorff topology on subsets of $S$. It follows from Theorem \ref{locmaxint} that $G_S$ is a finite graph its vertices of valence $1$ correspond to local minima and maxima of $\Vor$ and its vertices of valence $3$ correspond to saddle points of $\Vor$. It is not hard to see 
%
%The proof of this result follows from  Theorem \ref{locmaxint}. Recall that according this theorem $\dot S$ does not contain local minima. Since $\Vor^{-1}([r_1,r_2])$ does not contain saddle points, it can only contain local maxima, which can be either isolated fixed points or segments in the level set $\Vor=\frac{\pi}{2}$.
%
%Let $U$ be a connected component of $\Vor^{-1}([r_1,r_2])$. Consider the image $\pi(U)\subset G_S$. Clearly, it is disjoint from all vertices of valence $3$ and from vertices of valence $1$ where $V_{G_S}=0$. It follows that the image $\pi(U)$ belongs to 
\end{proof}

\begin{lemma}[Modulus of a Voronoi cylinder]\label{cylinderslength} 
Let $C$ be a Voronoi cylinder with $\Vor(C)=[r',r'']$.
For every $t\in [r',r'']$ let $\lev_t$ be the component of the level set $\Vor^{-1}(t)$ contained in $C$. Then we have 
\[
M(C)> \int_{r'}^{r''}\frac{1}{\ell(\lev_t)}dt>\frac{1}{2\pi\|\bm{\th}\|_1}\log\left(\frac{r''}{r'}\right).
\]
\end{lemma}
\begin{proof} We will first establish the left inequality. Let $[t',t'']\subset [r',r'']$ and denote by $C_{t',t''}\subset C$ the cylinder bounded by the curves $\lev_{t'}$ and $\lev_{t''}$.   
Note, that since $\Vor$ a Lipschitz function with $|\nabla \Vor|=1$ on $\dot S\setminus \Gamma$
and $\ell(\lev_t)$ is a continuous function of $t$,
we have $\area(C_{t',t''})=\int_{t'}^{t''}\ell(\lev_t)dt$ by the co-area formula. Note at the same time that, since $\Vor$ is a Voronoi function, we have $d(\lev_{t'},\lev_{t''})=t''-t'$. 
Applying Lemma \ref{modulus-height} we get
\[
M(C_{t',t''})> \frac{(t''-t')^2}{\area(C_{t',t''})}=\frac{(t''-t')^2}{\int_{t'}^{t''}\ell(\lev_t)dt}.
\]
To get the inequality, it suffice now to cut cylinder $C$ into $k$ Voronoi cylinders of width $\frac{t''-t'}{k}$, use sub-additivity of modulus (Lemma \ref{subadditivity}) and send $k$ to infinity.

The right hand side inequality clearly holds since by Corollary \ref{Corlengthoflevel} we have  $\ell(\lev_t)<2\pi\|\bm{\th}\|_1 t$.
\end{proof}

By the very definition of extremal systole,
Corollary \ref{cor:Vor}
and
Lemma \ref{cylinderslength}
provide us a tool to detect non-essential Voronoi cylinders.

\begin{corollary}[Non-essentiality of Voronoi cylinders]\label{cor:non-essential}
Assume that
\[
\Ext\sys(\dot{S})\geq \frac{2\pi\|\bm{\th}\|_1}{\log(r''/r')},
\]
for some regular values $0<r'<r''<\pi$ of $\Vor$. Then the following holds.
\begin{itemize}
\item[(i)]
A Voronoi cylinder $C$ with $\Vor(C)=[r',r'']$ is non-essential.
\item[(ii)]
If additionally there are no saddle values in $[r',r'']$,
then every component of $\Vor^{-1}([r',r''])$
is either a disk without conical points or
a non-essential cylinder.
\end{itemize}
\end{corollary}

\subsection{Area and total angle of components of sublevels}
In this section we study sublevel surfaces $\{\Vor\le c\}$ and their connected components. First we estimate their area and then give a lover bound the total conical angle. Both results are needed for our proof of systole inequality.  

\subsubsection{Area of sublevel surfaces}
\begin{lemma}[Area and perimeter of sublevel sets of $\Vor$]\label{lengthoflevel}
For every $0<r<2\pi$ the following hold.
\begin{itemize}
\item[(a)]
Let $S'$ be a connected component of $\{\Vor\leq r\}$
and let $\{x_i\,|\,i\in I\}$ with $I\subseteq\{1,2,\dots,n\}$
be the collection of conical points in $S\setminus S'$.
Then
\[
\ell(\pa S')\leq
2\pi \sin(r) \|\bm{\th}_{I^c}\|_1\le 2\pi r\|\bm{\th}_{I^c}\|_1=2\pi r\sum_{i\in I^c}\th_i
\]
where $I^c=\{1,2,\dots,n\}\setminus I$.
\item[(b)]
The area of the sublevels of $\Vor$ is bounded above by
$\Area(\Vor^{-1}(0,r))\le \pi r^2\|\bm{\th}\|_1$.\\
If $S'$ is a connected component of $\{\Vor\leq r\}$
and $\{x_i\,|\,i\in I\}$ is the collection of conical points in $S\setminus S'$,
then
\[
\Area(S')\leq
\pi r^2\|\bm{\th}_{I^c}\|_1.
\]
\end{itemize}
Moreover,
\begin{itemize}
\item[(c)]
The maximum value of $\Vor$ is bounded below by
$\displaystyle\max(\Vor)\ge \sqrt{2\left(1+\chi(\dot S)\|\bm{\th}\|_1^{-1}\right)}=\sqrt{2\chi(S,\bm{\th})\|\bm{\th}\|_1^{-1}}$.
\end{itemize}
\end{lemma}
\begin{proof} 
%Recall that the surface $S$ is covered by domains $D_i$. So in order to prove (a) it is enough to prove that for each $i$ the intersection  $\Vor^{-1}(r)\cap D_i$ has total length at most $2\pi \sin(r)\th_i$. Clearly, this intersection is a locally isometric image of the curve in $\bar D_i$ consisting of points on distance $r$ from $x_i$.  By Lemma \ref{graphconvex} the polygon $\bar D_i$ is star-shaped and so the latter curve is a union of
%arcs at constant distance $r$ from $x_i$,
%whose total length is at most $2\pi\sin(r)\th_i\leq 2\pi r\th_i$.
The proof of (a) is identical to the proof of Corollary \ref{Corlengthoflevel}, where we estimate the length of the whole level set $\Vor^{-1}(r)$. To get the bound on $\ell(\pa S')$ one needs to note additionally that $\pa S'\cap \DVint_i=\emptyset$
for all $i\in I$.

The upper bound in (b) for the area is easily obtained by noting that
\[
\Area(\Vor^{-1}(0,r))=\int_0^r \ell(\Vor^{-1}(t))dt\le \pi\|\bm{\th}\|_1 r^2
\]
and similarly for the area of $S'$.

Finally, the upper bound for $r_{\max}=\max(\Vor)$ in (c) is a consequence of
\[
\Area(S)=\Area\left(\Vor^{-1}(0,r_{\max})\right)\leq \pi\|\bm{\th}\|_1 r^2_{\max}
\]
from (b) and of Gauss-Bonnet formula $2\pi\left(\|\bm{\th}\|_1+\chi(\dot{S})\right)=\Area(S)$.
\end{proof}

\subsubsection{Total conical angle of sublevel subsurfaces}

The following result is a simple application of Theorem \ref{thm:closest-points}.

\begin{proposition}[Lower bound for the total angle in a sublevel of $\Vor$]\label{4pi3prop} 
Let $S$ be a spherical surface with conical singularities $\bm{x}$
and let $c$ be a regular value of $\Vor$. Suppose that a connected component $S'$   
of $\Vor^{-1}[0,c]$ contains a saddle critical point. Then  the sum of conical angles in $S'$ is 
larger than $\frac{4\pi}{3}$.
\end{proposition}

\begin{example} \label{example:4-3}
The bound $\frac{4\pi}{3}$ in Proposition \ref{4pi3prop} can not be improved. 
Indeed, for any small $\varepsilon>0$
surfaces of genus zero
with three conical points of
angles $\left(\frac{4\pi}{3}+4\varepsilon,\,\frac{\pi}{3}, \,\frac{\pi}{3}\right)$ or 
$\left(\frac{2\pi}{3},\,\frac{2\pi}{3}+2\varepsilon, \,\frac{2\pi}{3}+2\varepsilon\right)$
necessarily contain a connected component $S'$ of a sublevel of $\Vor$
with a saddle critical point for which the sum of cone angles is $\frac{4\pi}{3}+4\varepsilon$.

\begin{center}
\begin{figurehere}
\psfrag{T}{$\textcolor{Blue}{T}$}
\psfrag{T<t}{$\textcolor{Purple}{T'}$}
\psfrag{A}{(a)}
\psfrag{B}{(b)}
\psfrag{p}{$\textcolor{BrickRed}{s}$}
\psfrag{x1}{$\textcolor{Sepia}{x_1}$}
\psfrag{x2}{$\textcolor{Sepia}{x_2}$}
\psfrag{x3}{$\textcolor{Sepia}{x_3}$}
\psfrag{pi/6}[][][0.9]{$\textcolor{Sepia}{\pi/6}$}
\psfrag{pi/3}[][][0.9]{$\textcolor{Sepia}{\pi/3}$}
\psfrag{2pi/3}[][][0.9]{$\textcolor{Sepia}{2\pi/3+2\e}$}
\psfrag{pi/3+}[][][0.9]{$\textcolor{Sepia}{\pi/3+\e}$}
\includegraphics[width=0.7\textwidth]{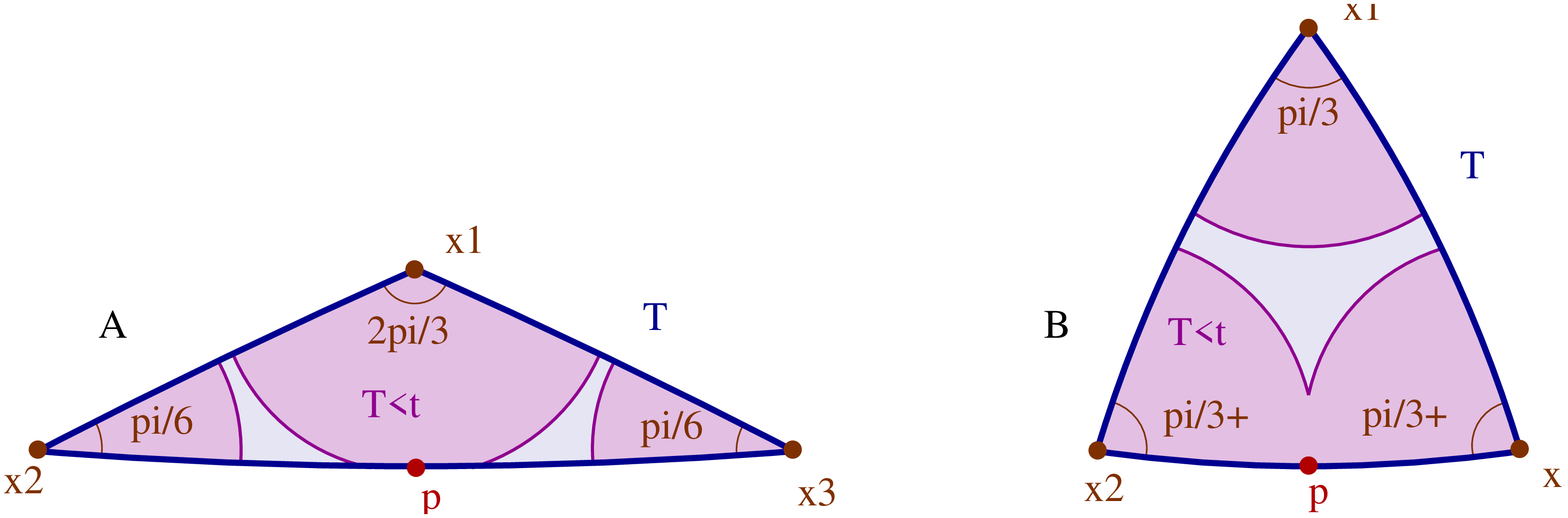}
\caption{{\small The surfaces of Example \ref{example:4-3} are obtained by doubling the triangles in the picture.}}\label{fig:examples4-3}
\end{figurehere}
\end{center}

In Figure \ref{fig:examples4-3} such surfaces are obtained by doubling the spherical
triangles $T$ in the pictures: the subsurface $S'$ is the double of $T'$. In $S'$ the point $s$ will correspond to a saddle point for $\Vor$ in both cases; however, in case (a) the saddle point $s$ will lie on a loop based at $x_1$,
whereas in case (b) it will lie on a geodesic arc joining $x_2$ and $x_3$. 
\end{example}

\begin{proof}[Proof of Proposition \ref{4pi3prop}] 
Let $s$ be a saddle critical point of $\Vor$  with the lowest value of $\Vor$ contained in the connected component $S'$. The point $s$ is the midpoint of a saddle geodesic $\sigma_s$ 
contained in $S'$. It is easy to see that, for any conical point $x_k$ which is not an endpoint of $\sigma_s$, we have
$d(x_k,\pa\sigma_s)\geq \ell(\sigma_s)$.

Now there are two cases.
\begin{itemize}
\item[(a)]
$\sigma_s$ is a saddle loop based at some conical point $x_i$.\\
Since $d(x_k,x_i)\ge \ell(\sigma_s)=2r_i$,
we conclude that $\th_i>\frac{2}{3}$ by Theorem \ref{thm:closest-points}(a).
\item[(b)]
$\sigma_s$ joins two distinct conical points $x_i,x_j$.\\
Since $d(x_k,\{x_i,\,x_j\})\ge \ell(\sigma_s)=d(x_i,x_j)$,
we conclude that $\th_i+\th_j>\frac{2}{3}$ by Theorem \ref{thm:closest-points}(b).
\end{itemize}
\end{proof}

\section{Consequences of the systole inequality}\label{sec:main}

%%%% Risultato principale

In this section we derive two corollaries from the systole inequality (Theorem \ref{main:systole}).
First, we give an obstruction for existence of spherical metrics with
a very small angle (Theorem \ref{main:non-existence-small}), next we prove the properness of the forgetful map (Theorem \ref{main:properness}).

%In this section we will see how to derive
%the non-existence of spherical metrics with
%a very small angle (Theorem \ref{main:non-existence-small})
%and the properness of the forgetful map (Theorem \ref{main:properness})
%from the systole inequality (Theorem \ref{main:systole}).

\subsection{Surfaces with one small conical angle}

%\begin{lemma}
%The systole of $S$ satisfies
%\[
%\sys(S,\bm{x})\leq \pi\th_1.
%\]
%\end{lemma}
%\begin{proof}
%Let $\gamma_r\subset S$ the locus of points at distance $r>0$ from $x_1$.
%For $r$ small, such $\gamma_r$ is a smooth curve of length $2\pi\th_1\sin(r)$ contained in $\dot{S}$.
%Let $\bar{r}$ be the smallest value of $r>0$ for which
%$\gamma_r$ is singular, or it passes through a conical pont.
%
%If $\gamma_{\bar{r}}$ passes through a conical point, then
%a portion $\gamma'$ of it will be a simple arc joining two conical points
%or a simple loop based at a conical point. In both cases, we will have
%$\ell(\gamma')\leq 2\pi\th_1\sin(\bar{r})$.
%
%If $\gamma_{\bar{r}}$ does not pass through a conical point,
%then it is singular. Hence, a portion $\gamma'$ of it forms an essential loop, which has thus length less than $2\pi\th_1\sin(\bar{r})$. 
%
%The conclusion follows by noting that $\sys(S,\bm{x})\leq \ell(\gamma')/2\leq \pi\th_1\sin(\bar{r})\leq\pi\th_1$.
%\end{proof}

The purpose of this section is to prove 
the following non-existence result for spherical metrics
in a fixed conformal class for which 
a conical point has a very small assigned angle.

\begin{mainthm}{\ref{main:non-existence-small}}[Non-existence of spherical metrics with one small angle]
Let $(S,J)$ be a Riemann surface with marked points $\bm{x}=(x_1,\dots,x_n)$
such that $\chi=\chi(\dot{S})<-1$.
Let $\bm{\hat{\th}}=(0,\th_2,\dots,\th_n)$ with
$\th_2,\dots,\th_n>0$ and suppose that
\begin{itemize}
\item[(i)]
$\chi(S,\bm{\hat{\th}})\geq 0$.
\item[(ii)]
$\Ab_{\bm{\hat{\th}}}(S,\bm{x})>0$.
\end{itemize}
Moreover, let
\[
\th_1^\star=
%\begin{cases}
%1 & \text{if $\Ab_{\bm{\hat{\th}}}(S,\bm{x})>1$}\\
\frac{1}{\pi}\left( \frac{\e}{\pi(1+4\|\bm{\hat{\th}}\|_1)}
\right)^{-3\chi+1} 
%& \text{otherwise}
%\end{cases}
\quad\text{with}\quad
\e=\min\left\{
\frac{1}{2}\Ab_{\bm{\hat{\th}}}(S,\bm{x}),\quad
\exp \left( \frac{-\pi(1+2\|\bm{\hat{\th}}\|_1)}{\Ext\sys(\dot{S},J)}\right)
\right\}.
\]
Then 
%there is a $\th_1^\star\in\left(0,10^{-6}\right)$
%that depends only on $\Ext\sys(\dot{S},J)$, $\|\bm{\hat{\th}}\|_1$, $\Ab_{\bm{\hat{\th}}}(S,\bm{x})$ and $\chi(\dot{S})$ such that
there exists no spherical metric on $S$
with angles $2\pi\bm{\th}=2\pi(\th_1,\th_2,\dots,\th_n)$ at $\bm{x}$
in the conformal class determined by $J$
for any $\th_1<\th_1^\star\in\left(0,10^{-6}\right)$. 
%In particular, one can choose
%\[
%\th_1^\star=
%%\begin{cases}
%%1 & \text{if $\Ab_{\bm{\hat{\th}}}(S,\bm{x})>1$}\\
%\frac{1}{\pi}\left( \frac{\e}{\pi(1+4\|\bm{\hat{\th}}\|_1)}
%\right)^{-3\chi+1} 
%%& \text{otherwise}
%%\end{cases}
%\quad\text{with}\quad
%\e=\min\left\{
%\frac{1}{2}\Ab_{\bm{\hat{\th}}}(S,\bm{x}),\quad
%\exp \left( \frac{-\pi(1+2\|\bm{\hat{\th}}\|_1)}{\Ext\sys(\dot{S},J)}\right)
%\right\}.
%\]
\end{mainthm}

Let us first comment on the assumptions of the theorem. If $\chi(S,\bm{\hat{\th}})<0$, the above statement 
would follow from Gauss-Bonnet if we set $\th_1^\star=-\chi(S,\bm{\hat{\th}})$.
Similarly, if $\dot{S}$ is a $1$-punctured torus,
it would support no spherical metrics for $\th_1\leq \th_1^\star=1$ by Gauss-Bonnet.
If $\dot{S}$ is a $3$-punctured sphere, then the statement is
again trivial: in fact, the assumption $\Ab_{\bm{\hat{\th}}}(S,\bm{x})>0$
quickly leads to the non-existence of a spherical metric for small $\th_1$
by monodromy considerations (as in \cite{eremenko:three} and
\cite{mondello-panov:constraints}). 

On the contrary, the condition $\Ab_{\bm{\hat{\th}}}(S,\bm{x})>0$
is essential, as shown in the following example.

\begin{example}
Let $\th_1\in(0,\pi/2)$ and consider a convex spherical triangle 
$T_{\th_1}$ with vertices $X_1,X_2,X_3$ and angles
$\pi\cdot\left(\th_1,\frac{1}{2},\frac{1}{2}\right)$ and let $X_4$ be the midpoint of the edge $X_2 X_3$.
Denote by $S_{\th_1}$ the spherical surface obtained by doubling $T_{\th_1}$
and mark by $x_i$ the point on $S_{\th_i}$ induced by $X_i$.
Then $(S_{\th_1})_{\th_1\in(0,\pi/2)}$ is a family of spherical metrics on
a surface of genus $0$ with $4$ conical points of angles $2\pi\cdot\left(\th_1,\frac{1}{2},\frac{1}{2},1\right)$.
It is easy to see by symmetry that all $S_{\th_1}$ are conformally equivalent to one another.
\end{example}

%
%\begin{remark}\label{rmk:th1star}
%In Theorem \ref{thm:small-angle} it is possible to take
%\[
%\th_1^\star=
%%\begin{cases}
%%1 & \text{if $\Ab_{\bm{\hat{\th}}}(S,\bm{x})>1$}\\
%\frac{1}{\pi}\left( \frac{\e}{\pi(1+4\|\bm{\hat{\th}}\|_1)}
%\right)^{-3\chi+1} 
%%& \text{otherwise}
%%\end{cases}
%\]
%where 
%\[
%\e=\min\left\{
%\frac{1}{2}\Ab_{\bm{\hat{\th}}}(S,\bm{x}),\quad
%\exp \left( \frac{-\pi(1+2\|\bm{\hat{\th}}\|_1)}{\Ext\sys(\dot{S},J)}\right)
%\right\}.
%\]
%\end{remark}

Let us now prove the above non-existence result.

\begin{proof}[Proof of Theorem \ref{main:non-existence-small}]
%If $\Ab_{\bm{\hat{\th}}}(S,\bm{x})>1$, then $\chi(S,\bm{\hat{\th}})<-1$
%by Lemma \ref{lemma:upper-bound-Ab}(ii). This implies that it is enough to take $\th_1^\star=1$,
%since the conclusion follows by Gauss-Bonnet.
%Note first that $\Ab_{\bm{\hat{\th}}}(S,\bm{x})\leq 1$ by (i) and Lemma \ref{lemma:upper-bound-Ab}(i).
%Choose $\e>0$ and $\th_1^\star$ as in the above Remark \ref{rmk:th1star}.
Since $-\chi(\dot{S})\geq 2$ and $\e\leq \frac{1}{2}$, we have $\th_1^\star\leq \frac{1}{\pi}\left(\frac{1}{2\pi}\right)^7<10^{-6}$.

In order to show that $\th_1^\star<\frac{1}{2}\Ab_{\bm{\hat{\th}}}(S,\bm{x})$,
note that $\th_1^\star\leq \frac{1}{\pi}\left(\frac{\Ab_{\bm{\hat{\th}}}(S,\bm{x})}{2\pi}\right)^7$
because $\e\leq \frac{1}{2}\Ab_{\bm{\hat{\th}}}(S,\bm{x})$.
%
%
%
%If $\chi(S,\bm{\hat{\th}})+\th_1^\star\leq 0$, then the statement is trivally true by Gauss-Bonnet. Thus, we assume now on that $\chi(S,\bm{\hat{\th}})+\th_1^\star>0$.
%
%If $\chi(S,\bm{\hat{\th}})\leq 0$, then $\Ab_{\bm{\hat{\th}}}(S,\bm{x})<\th_1^\star$.
%If $\chi(S,\bm{\hat{\th}})>0$, then $\Ab_{\bm{\hat{\th}}}(S,\bm{x})\leq 1$
%as in Lemma \ref{lemma:upper-bound-Ab}.
On the other hand, $\Ab_{\bm{\hat{\th}}}(S,\bm{x})\leq 1$ by (i) and Lemma \ref{lemma:upper-bound-Ab}(i)
and so 
$\frac{1}{\pi}\left(\frac{\Ab_{\bm{\hat{\th}}}(S,\bm{x})}{2\pi}\right)^7<
\frac{1}{2}\Ab_{\bm{\hat{\th}}}(S,\bm{x})$ gives the claim.

We will complete the proof arguing by contradiction: we let $\th_1\in(0,\th_1^\star)$ and we
suppose that a spherical metric on $S$ with conical points of angles $2\pi\bm{\th}$ exists.

%Since $\th_1<1/4$, we have $\pi(1+4\|\bm{\hat{\th}}\|_1)>4\pi\|\bm{\th}\|_1$
%and $\pi(1+2\|\bm{\hat{\th}}\|_1)>2\pi\|\bm{\th}\|_1$.

Certainly, $\Ab_{\bm{\th}}(S,\bm{x})\geq \Ab_{\bm{\hat{\th}}}(S,\bm{x})-\th_1
>\frac{1}{2}\Ab_{\bm{\hat{\th}}}(S,\bm{x})\geq \e$.
Moreover, Lemma \ref{anglesysbound} 
provides the systole bound $\sys(S,\bm{x})\leq \pi\th_1$
and so
\[
\sys(S,\bm{x})\leq \pi\th_1<\left( \frac{\e}{\pi(1+4\|\bm{\hat{\th}}\|_1)}
\right)^{-3\chi+1}
<\left( \frac{\e}{4\pi\|\bm{\th}\|_1}
\right)^{-3\chi+1}
\]
since $\pi(1+4\|\bm{\hat{\th}}\|_1)>4\pi\|\bm{\th}\|_1$.

By the systole inequality (Theorem \ref{main:systole}) it follows that $\Ext\sys(\dot{S},J)<\frac{2\pi\|\bm{\th}\|_1}{\log(1/\e)}$ and so 
\[
\Ext\sys(\dot{S},J)<\frac{2\pi\|\bm{\th}\|_1}{\log(1/\e)}
\leq \frac{2\pi\|\bm{\th}\|_1}{\pi(1+2\|\bm{\hat{\th}}\|_1)}\Ext\sys(\dot{S},J)<
\Ext\sys(\dot{S},J)
\]
because $\pi(1+2\|\bm{\hat{\th}}\|_1)>2\pi\|\bm{\th}\|_1$.
We have thus achieved the wished contradiction.
\end{proof}

%%%%%%%%%%%%%%%%%%%%%%%%%%%%%%%%

\subsection{Properness of the forgetful map}

In this section
we prove a qualitative counterpart to the systole inequality: the forgetful map
from the moduli space of spherical surfaces to the moduli space of Riemann surfaces
is proper if the non-bubbling parameter does not vanish.

Throughout the section, fix $g,n\geq 0$ in such a way that $2g-2+n>0$.
%For every $\bm{\th}\in\RR_+^n$ we will denote simply by $\Ab_{\bm{\th}}$
%the quantity $\Ab_{\bm{\th}}(S)$, where $S$ is a surface of genus $g$.

%\begin{definition}[Moduli space of spherical metrics]
%For every $\bm{\th}=(\th_1,\dots,\th_n)\in\RR_+^n$
%the set 
%\[
%\MSPH_{g,n}(\bm{\th})=
%\left\{
%(S,\bm{x},h)\ \Big|\
%\begin{array}{l}
%\text{$S$ surface of genus $g$ with spherical metric $h$ with}\\
%\text{conical singularities at $\bm{x}=(x_1,\dots,x_n)$ of angles $2\pi\bm{\th}$} 
%\end{array}
%\right\}\Big/\text{isometry}
%\]
%is called {\it{moduli space of spherical surfaces}} of genus $g$ with $n$ conical singularities
%of angles $2\pi\bm{\th}$. 
%\end{definition}

%Note that the above definition of $\MSPH_{g,n}(\bm{\th})$ is modelled on the more classical
%{\it{moduli space of Riemann surfaces of genus $g$ with $n$ marked points}}
%\[
%\Mcal_{g,n}:=
%\left\{
%(S,\bm{x},J)\ \Big|\
%\begin{array}{l}
%\text{$S$ surface of genus $g$ with complex structure $J$}\\
%\text{and marked points $\bm{x}=(x_1,\dots,x_n)$}
%\end{array}
%\right\}\Big/\text{isomorphism}
%\]
%Associating to a spherical metric its underlying complex structure
%defines the following {\it{forgetful map}}
%\[
%F_{g,n,\bm{\th}}:\MSPH_{g,n}(\bm{\th})\lra\Mcal_{g,n}.
%\]

\subsubsection{Compactness in $\Mcal_{g,n}$ and in $\MSPH_{g,n}(\bm{\th})$}\label{sec:moduli-spaces}

In this subsection we recall some basic results about the local structure of the moduli spaces $\Mcal_{g,n}$ and  $\MSPH_{g,n}(\bm{\th})$, see Definitions \ref{modulidef} and \ref{def:moduli-spherical}.
We also mention two standard criteria of compactness in such moduli spaces.\\

The following result is well-known (see for instance \cite{gac2}).

\gm{Statement of Proposition \ref{prop:Mgn} expanded. The presentation
of $\Mcal_{g,n}$ as an orbifold is included because it is used below.}
%%%%%%%%%%%%%%%%%%%%%%%%%%%%%%%%%%%%%%%
\begin{proposition}[Moduli space of Riemann surfaces]\label{prop:Mgn}
%The space $\widetilde{\Mcal}_{g,n}$ is a complex connected manifold of (complex) dimension $3g-3+n$ and
%$\MCG(S,\bm{x})$ acts properly and with finite stabilizers.
The moduli space
$\Mcal_{g,n}$ is a complex-analytic connected orbifold of (complex) dimension $3g-3+n$,
and in particular $\Mcal_{g,n}$ is the global quotient $\Mcal'_{g,n}/G'$
of a complex connected manifold $\Mcal'_{g,n}$ by a finite group $G'$
of biholomorphisms of $\Mcal'_{g,n}$.
Moreover, $\Mcal_{g,n}$ is compact if and only if $(g,n)=(0,3)$.
\end{proposition}

We also recall the standard criterion of compactness in $\Mcal_{g,n}$,
which can be easily derived from Teichm\"uller's work
(see \cite{hubbard:book1}, for instance).
%\cite{gardiner:schiffer} and 
%\cite{deligne-mumford}.

\begin{lemma}[Extremal systole function]\label{lemma:ext-systole-function}
The extremal systole function $\Ext\sys:\Mcal_{g,n}\rar\RR_{>0}$ 
that sends $[S,\bm{x},J]$ to $\Ext\sys(\dot{S},J)$
is continuous
and its superlevel sets $\{\Ext\sys\geq s\}$ are compact for all $s>0$.
\end{lemma}

The following result on the local structure of $\MSPH_{g,n}(\bm{\th})$
can be directly deduced from Luo's \cite{luo:monodromy}.

\begin{proposition}[Moduli space of spherical surfaces]\label{prop:luo}
Assume that $\Ab_{\bm{\th}}(g,n)>0$
 and no $\th_i$ is integral.
Then
$\MSPH_{g,n}(\bm{\th})$ is a real-analytic orbifold
of dimension $6g-6+2n$.
%\begin{itemize}
%\item[(i)]
%$\widetilde{\MSPH}_{g,n}(\bm{\th})$ is a real-analytic manifold 
%of dimension $6g-6+2n$
%and
%$\MCG(S,\bm{x})$ acts properly and with finite stabilizers on it.
%As a consequence,
%$\MSPH_{g,n}(\bm{\th})$ is a real-analytic orbifold.
%\item[(ii)]
%the holonomy map $\hol$ takes values in the orbifold locus of the representation space and
%it is a real-analytic local diffeomorphism;
%
%for every point $(S,\bm{x},h)$ of $\MSPH_{g,n}(\bm{\th})$
%there is a contractible open neighbourhood $U$ and a locally defined holonomy map
%$\hol:U\rar \Rep(S,\SO_3(\RR),\bm{\th})$ which is a diffeomorphism onto 
%a smooth open subset $\hol(U)$; 
%\item[(ii)]
%the forgetful map $F_{g,n,\bm{\th}}$ is smooth and real-analytic.
%\end{itemize}
\end{proposition}

\gm{The following sketch of an argument for Proposition \ref{prop:luo}
has been expanded 
(up to the following ``Remark'')
to deal with the orbifold structure. It is used below
to justify the terminology ``real-analytic map of orbifolds''.}
%%%%%%%%%%%%%%%%%%%%%%%%%%%%%%%%%%%%%%%%%
Let us briefly explain the idea behind the above proposition.
Consider the moduli space $\MP_{g,n}$ of surfaces 
of genus $g$ with $n$ marked points endowed
with a $\CC\PP^1$-structure whose Schwarzian derivative has poles of order at most $2$ at the marked points (with respect to any smooth $\CC\PP^1$-structure).
It is well-known that $\MP_{g,n}$ is a holomorphic affine bundle over $\Mcal_{g,n}$ of rank $3g-3+2n$. It follows that
$\MP'_{g,n}:=\MP_{g,n}\times_{\Mcal_{g,n}}\Mcal'_{g,n}\rar \Mcal'_{g,n}$
is a holomorphic affine bundle of the same rank, which is acted on
by a finite group of biholomorphisms isomorphic to $G'$
(where $G'$ is as in Proposition \ref{prop:Mgn}), so that
the (complex-analytic) orbifold $\MP_{g,n}\cong \MP'_{g,n}/G'$ is a global quotient.

As remarked by Luo, since 
the monodromy is non-coaxial
(which is ensured by $\Ab_{\bm{\th}}(g,n)>0$)
and the angles are non-integral, 
the set $\MSPH_{g,n}(\bm{\th})$
identifies with the locus of $\CC\PP^1$-structures inside
$\MP_{g,n}$ with monodromy in $\SO_3(\RR)$
and with quadratic residue $\frac{1}{2}(1-\th_i^2)$ of
the Schwarzian at the $i$-th conical point.
We similarly identify $\MSPH'_{g,n}(\bm{\th}):=\MSPH_{g,n}(\bm{\th})\times_{\Mcal_{g,n}}\Mcal'_{g,n}$ to the corresponding locus inside $\MP'_{g,n}$.

Now, Luo's main theorem in \cite{luo:monodromy}
ensures in particular that the map that sends
a $\CC\PP^1$-structure to its monodromy representation
gives a local biholomorphism of complex orbifolds between a neighbourhood of
a point of
$\MSPH_{g,n}(\bm{\th})$ inside $\MP_{g,n}$ and 
an open subset of the 
space $\Rep_{g,n}(\PSL_2(\CC))$ of
conjugacy classes of representations 
of the fundamental group of a surface of genus $g$ with $n$
ordered points removed inside $\PSL_2(\CC)$
(see also \cite[Corollary 1(a)]{luo:monodromy}).
%To sum up, 
%the locus $\MSPH_{g,n}(\bm{\th})$ can be described
%by imposing $n$ holomorphic conditions on the quadratic residues
%of the Schwarzian at the marked points
%and the real-analytic condition of having holonomy
%contained in a conjugate of $\SO_3(\RR)$ inside $\PSL_2(\CC)$
%
Note that the non-coaxiality and non-integrality
of the $\th_i$'s ensure that
the subspace $\Rep_{g,n}(\SO_3(\RR))_{\bm{\th}}\subset\Rep_{g,n}(\PSL_2(\CC))$
of conjugacy classes
of representations in $\SO_3(\RR)$
that assign a rotation of angle $2\pi\th_i$ to a loop about the $i$-th puncture
is a real-analytic suborbifold of dimension $6g-6+2n$.
We conclude that $\MSPH'_{g,n}(\bm{\th})$ is a smooth
real-analytic subvariety of $\MP'_{g,n}$
and so $\MSPH_{g,n}(\bm{\th})=\MSPH'_{g,n}(\bm{\th})/G'$
is a real-analytic orbifold.

\begin{remark}
The assumption in Proposition \ref{prop:luo}
on the non-integrality of the $\th_i$'s depends on the nature of Luo's proof and can be removed. 
%To remove such an assumption in Proposition \ref{prop:luo}(ii)
%we should construct a space finer than the representation variety. For the sake of conciseness we will not do it here.
%On the contrary, it is essential that the spherical surfaces have
On the other hand, non-coaxiality of the monodromy is important and it is ensured by the condition $\Ab_{\bm{\th}}(g,n)>0$. In order to treat the case of possibly coaxial spherical metrics, one should consider spherical metrics ``up to equivalence of $\CC\PP^1$-structures'' (which is coarser than ``up to isometry''): in this case, one could
endow the space of such equivalence classes of spherical metrics with the structure of real-analytic variety.
\end{remark}

%\begin{remark}[$\SU_2$-lifts of the holonomy map]\label{rmk:lifts-reps}
%%We incidentally recall that $[\rho]\mapsto \tr(\rho(\gamma))$
%%determines a smooth function $\tr_\gamma:\Rep(S,\SO_3(\RR),\bm{\th})\rar\RR$
%%for all closed curves $\gamma\subset\dot{S}$.
%Minor variations on the techniques in \cite{mondello-panov:constraints}
%allow to show that
%the holonomy map $\hol$ in Proposition \ref{prop:luo}(ii)
%admits a lifts to $\Hol:\widetilde{\MSPH}_{g,n}(\bm{\th})\rar \Rep(S,\SU_2,\bm{\th})$
%that are in correspondence with the elements of $H^1(S,\ZZ/2)$.
%In particular, if $g=0$ then there exists a canonical lift of the given $\hol$.
%\end{remark}

A consequence of the interpretation
of $\MSPH_{g,n}(\bm{\th})$ as a the moduli space of $(\SO_3(\RR),\Sph)$-structures
is that (locally defined) length and distance functions are continuous.
The following result on the continuity of the systole function is also rather standard. 
The compactness of the superlevel sets
depends on the fact that the curvature is constant
and the diameter is bounded (see Lemma \ref{sphdiameter}).

\begin{lemma}[Systole function]\label{lemma:systole-function}
Assume that $\Ab_{\bm{\th}}(g,n)>0$ and no $\th_i$ is integral.
The systole function $\sys:\MSPH_{g,n}(\bm{\th})\rar\RR_{>0}$ is continuous.
Moreover, its superlevel sets $\{\sys\geq s\}$ are compact for all $s>0$.
\end{lemma}
%\begin{proof}
%The function $\sys$ is locally the minimum of finitely many length functions and so it is continuous.
%Moreover, the diameter of a spherical surface in $\MSPH_{g,n}(\bm{\th})$ can be bounded
%from above in terms of $(g,n,\bm{\th})$. Thus, a degeneration occurs only if two conical points meet or if an essential loop is shrunk to a point: in both cases $\sys$ would go to zero.
%\end{proof}

\subsubsection{The forgetful map}

We recall from the introduction that associating to a spherical metric
its underlying complex structure determines a forgetful map
$F_{g,n,\bm{\th}}:\MSPH_{g,n}(\bm{\th})\rar\Mcal_{g,n}$,
which naturally lifts to a $G'$-equivariant
$F'_{g,n,\bm{\th}}:\MSPH'_{g,n}(\bm{\th})\rar\Mcal'_{g,n}$.
%
%%%%%%%%%
\gm{This first paragraph and Corollary \ref{cor:forgetful} are slightly
expanded to clarify the meaning of ``real-analytic map of real-analytic orbifolds''.}
%%%%%%%%%%%%%%%%%%%%%%%%

\begin{corollary}[Forgetful map]\label{cor:forgetful}
Suppose that $\Ab_{\bm{\th}}(g,n)>0$ and assume that no $\th_i$ is integral.
Then the forgetful map $F'_{g,n,\bm{\th}}$ is real-analytic. As a consequence,
$F_{g,n,\bm{\th}}$ is real-analytic too.
\end{corollary}
\begin{proof}
As discussed in Subsection \ref{sec:moduli-spaces},
$\Mcal'_{g,n}$ is a complex manifold and
$\MSPH'_{g,n}(\bm{\th})$ is a smooth real-analytic subvariety
of the complex manifold $\MP'_{g,n}$ by Proposition \ref{prop:luo}.
Thus, the $G'$-equivariant holomorphic affine bundle $\MP'_{g,n}\rar \Mcal'_{g,n}$
restricts to a $G'$-equivariant real-analytic map $\MSPH'_{g,n}(\bm{\th})\rar\Mcal'_{g,n}$,
which agrees with $F'_{g,n,\bm{\th}}$.
\end{proof}

The following consequence of the systole inequality 
allows us to extend the conclusions of Theorem \ref{thm:BT} to the case $\bm{\th}\in\RR^n_{>0}$.

\begin{mainthm}{\ref{main:properness}}[Properness of the forgetful map] %\label{prop:properness}
Let $g,n\geq 0$ with $2g-2+n>0$ and
let $\bm{\th}\in\RR_{>0}^n$ such that $\Ab_{\bm{\th}}(g,n)>0$.
%Assume moreover that no $\th_i$ is integral.
Then the forgetful map $F_{g,n,\bm{\th}}$
is proper.
\end{mainthm}
\begin{proof}
Consider a diverging sequence $(S_k,h_k)$
in $\MSPH_{g,n}(\bm{\th})$, namely a sequence that leaves all compact
subsets of $\MSPH_{g,n}(\bm{\th})$, and let $(S_k,J_k)=F_{g,n,\bm{\th}}(S_k,h_k)$.
By Lemma \ref{lemma:systole-function}, $\sys(S_k)\rar 0$.
Thus, for every $\e \in (0,\Ab_{\bm{\th}}(g,n))$. 
the systole of $S_k$ is smaller than $\left(\frac{\e}{4\pi\|\bm{\th}\|_1}\right)^{-3(2-2g-n)+1}$ for $k\geq k(\e)$.
Then Theorem \ref{main:systole} implies that $\Ext\sys(\dot{S}_k,J_k)\rar 0$,
and so the sequence $(S_k,J_k)$ in $\Mcal_{g,n}$ is divergent too by Lemma \ref{lemma:ext-systole-function}.
\end{proof}

In a similar fashion one can prove that 
\[
F_{g,n,\Acal^\circ}:\MSPH_{g,n}(\Acal^\circ)\lra \Mcal_{g,n}\times \Acal^\circ
\]
is proper, where $\Acal^\circ\subset\RR^n_{>0}$ is the open subset of all $\bm{\th}$
such that $\Ab_{\bm{\th}}(g,n)>0$ and
$\MSPH_{g,n}(\Acal^\circ)$ is the moduli space of spherical surfaces of genus $g$
with $n$ conical points of angles belonging to $2\pi\cdot \Acal^\circ$.

\section{Disconnectedness of the moduli space of spherical surfaces}\label{sec:disconnectedness}

The goal of this section is to provide examples of angle vectors $\bm{\th}$
for which the moduli space of spherical surfaces of genus $0$ with conical points of angles
$2\pi\bm{\th}$ is disconnected. According to our knowledge these are first examples of the type and they highlight the complexity of the moduli space of spherical surfaces.

%This series of examples can be generalized in various ways, but for the sake of exposition we present the case that requires the minimum amount of computations. 

%\begin{mainthm}{\ref{main:many}}
%Let $\MSPH_{0,n+3}(\bm{\th})$ be the moduli space of spherical surfaces of genus $0$ with $n+3$ conical points with $\bm{\th}=(n_1+\frac{1}{2},n_2+\frac{1}{2},n_3+\frac{1}{2},\varepsilon,\ldots,\varepsilon)$, where $n_i\ge n$, $\varepsilon<\frac{1}{2n+2}$. Then the following holds.
%
%\begin{enumerate}
%\item The moduli space $\MSPH_{0,n+3}(\bm{\th})$ has at least $3^n$ connected components. 
%
%\item If $\varepsilon<\exp(-2\pi \max_i n_i)$ then the image of the forgetful  map from $\MSPH_{0,n+3}(\bm{\th})$ to $M_{0,n+3}$ has at least $3^n$ connected components.
%\end{enumerate}
%
%\end{mainthm}

\begin{mainthm}{\ref{main:many}}[Moduli spaces of spherical surfaces with many components]
Let $m\geq 1$ and $\e\in\left(0,\frac{1}{2m+2}\right)$.
Fix $\e_1,\dots,\e_m\in (0,\e]$ and integers $m_1,m_2,m_3\geq m$ and
set
$
\bm{\th}=\left(\frac{1}{2}+m_1,\frac{1}{2}+m_2,\frac{1}{2}+m_3,\e_1,\dots,\e_m\right)$.
Then
\begin{itemize}
\item[(a)]
the moduli space $\MSPH_{0,3+m}(\bm{\th})$ has at least $3^m$ connected components;
\item[(b)]
If $\e<\frac{1}{16}\exp(-2\pi \max\{m_j\})$, then the image of the forgetful  map 
$F_{0,3+m,\bm{\th}}:\MSPH_{0,3+m}(\bm{\th})\rar\Mcal_{0,3+m}$ has at least $3^m$ connected components.
\end{itemize}
\end{mainthm}

\begin{remark}
An inspection of the proof presented in the following subsections 
shows that a statement analogous to that of Theorem \ref{main:many} holds
even if the values of $m_1,m_2,m_3$ are slightly perturbed. Thus, the fact
that the moduli space of spherical surfaces might have a large number of connected components
is not exceptional: such phenomenon indeed occurs for $\bm{\th}$ ranging over a subset
of $\RR^{3+m}_+$ that has non-empty internal part.
\end{remark}

The proof of part (a) of Theorem \ref{main:many} is less involved, the reader interested in this part of the theorem can read this section up to Proposition \ref{gap} and the go directly to Subsection \ref{disconnectedproofsub}.

For convenience, from now on the conical points will be denoted by $x_1,x_2,x_3,y_1,\ldots,y_m$.
%
%and their Voronoi disks by $\DV_{x_1},\dots,\DV_{y_m}$. We also denote by $\W_{y_i}$
%the maximal $1$-pointed ball at $y_i$.
%

\begin{center}
\begin{figurehere}
\psfrag{S}{$\textcolor{blue}{S}$}
\psfrag{C3}{$\textcolor{Purple}{C_3}$}
\psfrag{x1}{$\textcolor{Sepia}{x_1}$}
\psfrag{x2}{$\textcolor{Sepia}{x_2}$}
\psfrag{x3}{$\textcolor{Sepia}{x_3}$}
\psfrag{y1}{$\textcolor{Sepia}{y_1}$}
\psfrag{y2}{$\textcolor{Sepia}{y_2}$}
\psfrag{y3}{$\textcolor{Sepia}{y_3}$}
\psfrag{y4}{$\textcolor{Sepia}{y_4}$}
\psfrag{a13}{$\textcolor{Sepia}{\alpha_{13}}$}
\psfrag{lc}{$\textcolor{Purple}{\lambda_c}$}
\psfrag{qhat}{$\textcolor{Purple}{\hat{q}}$}
\psfrag{qc}{$\textcolor{Purple}{q_c}$}
\includegraphics[width=0.45\textwidth]{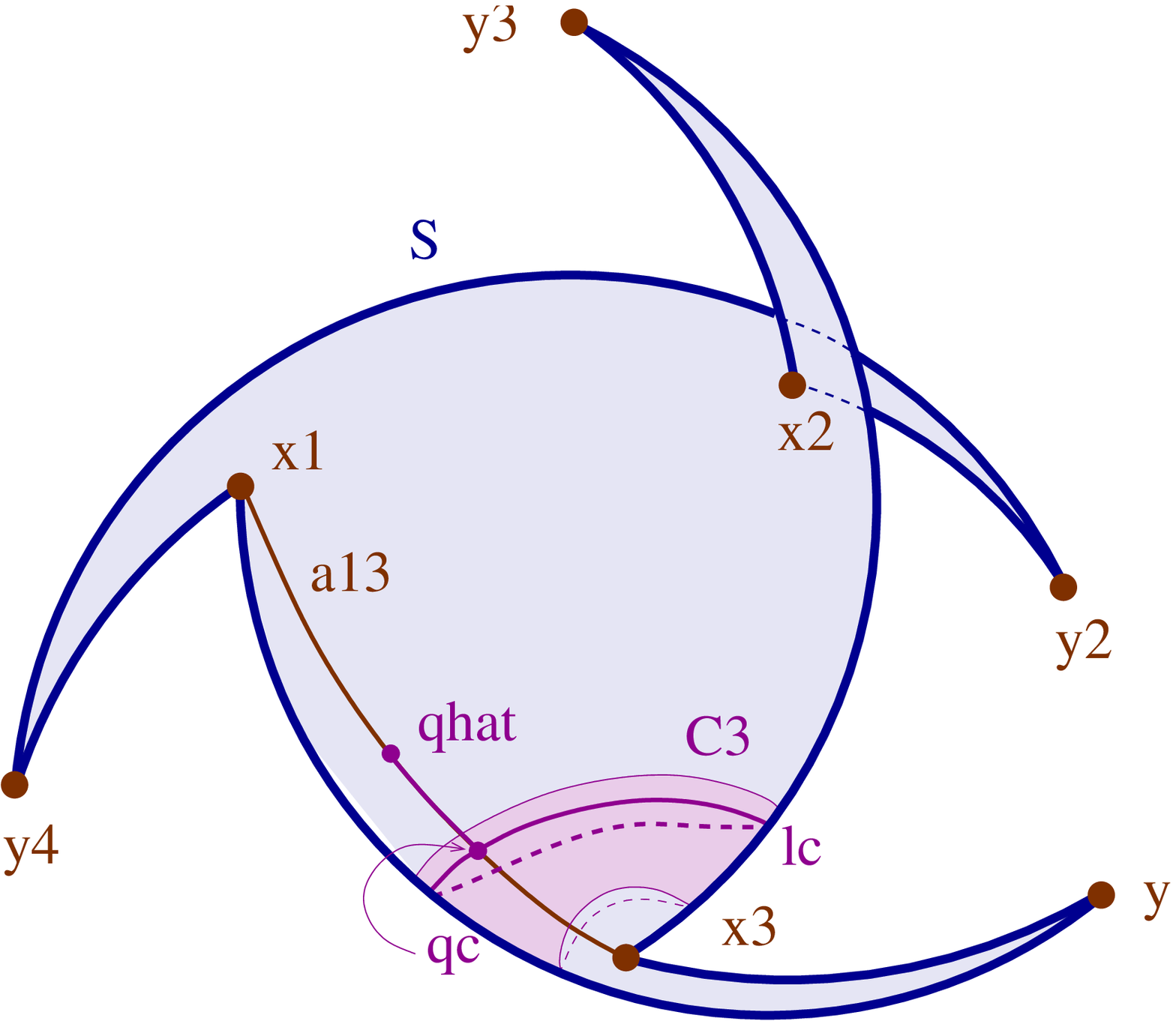}
\caption{{\small An example of a spherical surface $S$ in $\MSPH_{0,7}\left(\frac{3}{2},\frac{5}{2},\frac{3}{2},\e,\e,\e,\e\right)$ with the Voronoi cylinder $C_3$, the arc $\alpha_{13}$ and its midpoint $\hat{q}$ and the level curve $\lambda_c$ that intersects $\alpha_{13}$ at $q_c$.}}\label{fig:disconnected}
\end{figurehere}
\end{center}

\subsection{Monodromy of spherical surfaces of genus $0$}\label{sec:holonomy0}

In this subsection %we prove the gap properties stated in Proposition \ref{gap},
we establish some relations between the monodromy of a spherical surfaces of genus $0$,
the length of certain geodesics and the distances between certain pairs of conical points.

We begin by recalling \cite{mondello-panov:constraints} that to each spherical surface of genus $0$ with
$n$ conical singularities of angles $2\pi\bm{\th}$
one can associate a {\it standard set} of $n$ elements $\Q_1,\ldots,\Q_n\in \SO_3(\RR)$ such that $\Q_1\cdot\ldots\cdot \Q_n=I$ and  each $\Q_i$ is a rotation of $\Sph$ by angle $2\pi\th_i$. 

To construct such a set, choose a basepoint $s_0\in \dot S$ and a {\it standard set of loops} 
$\beta_1,\ldots,\beta_n$ based at $s_0$, namely a collection of simple peripheral loops such that
$\beta_1\ast\ldots \ast \beta_n=\id$ in $\pi_1(\dot S,s_0)$
and each $\beta_i$ simply winds about the $i$-th conical point counterclockwise.
Moreover, fix a universal cover $\tilde{\dot{S}}\rar\dot{S}$ and a lift $\tilde{s}_0\in\tilde{\dot{S}}$ of $s_0$.
%and let $\tilde{\beta}_i$ be the lift of $\beta_i$ that starts at $\tilde{s}_0$.

Associated to the given spherical metric on $S$ we can choose a 
locally isometric developing map $\iota:\tilde{\dot{S}}\rar\Sph$,
which is equivariant with respect to a monodromy representation $\pi_1(\dot{S},s_0)\rar\SO_3(\RR)$.
The element $\Q_i$ is then the image of $\beta_i$ through such monodromy representation.
%
%holonomy of $\iota$ along $\tilde{\gamma}_i$, namely
%$\Q_i:=\rho(\gamma_i)$.

We stress that the $n$-tuple $(\Q_1,\dots,\Q_n)$, and so equivalently the monodromy representation, is unique only up to
conjugation; namely, a different choice of the universal cover, of the basepoint, or of the developing map would
be associated to the $n$-tuple $(\P\Q_1 \P^{-1},\dots,\P\Q_n \P^{-1})$ for some $\P\in\SO_3(\RR)$.

\begin{remark}\label{holboundary} 
A similar construction can be performed on a spherical disk $D$ with $n$ conical points in its interior.
We can pick a basepoint $s_0\in\pa D$ and 
a standard set of loops $\beta_1,\ldots,\beta_n$ in $\dot{D}$ based at $s_0$.
As above, the choice of a developing map $\iota:\tilde{\dot{D}}\rar\Sph$ will produce
an $n$-tuple $(\Q_1,\ldots, \Q_n)$. Since the loop $\beta_{\partial}=\beta_1\ast\ldots \ast\beta_n$ is homotopic to the boundary $\pa D$ in $\dot{D}$, 
the product $\Q_1\cdot \ldots \cdot \Q_n$ is equal to the monodromy $\Q_{\partial D}$ along the boundary of $D$. 
\end{remark}

%
%If $\tilde{b}$ is a lift of $b$ to the universal cover $\tilde{\dot{S}}$ of $\dot{S}$
%and $\tilde{\gamma}_i$ is the lift of $\gamma_i$ that starts at $\tilde{b}$,
%
%Then $Q_i$ is the monodromy of the developing map from the universal cover of $\dot S$ to $\Sph$ along $\gamma_i$.

\begin{definition}
The {\it{rotation number}} $\rot:\SO_3(\RR)\rar\left[0,\frac{1}{2}\right]$
is the function defined by requiring that $\Q\in\SO_3(\RR)$ is a rotation
of angle $2\pi\cdot\rot(\Q)$.
\end{definition}

%For any element $Q\in \SO_3(\RR)$ we denote by $\Rot(Q)\in \left[0,\frac{1}{2}\right]$ the unique number, such that $Q$ is rotation of $\Sph$ of angle $2\pi\Rot$. 
We omit the proof of the following lemma, which summarizes 
a few properties of the rotation number of a composition.

\begin{lemma}[Basic properties of the rotation number]\label{orthogprod} 
The  function $\rot$ is invariant under conjugation and it satisfies the following properties.
\begin{itemize}
\item[(a)] 
Let $\Q_1$ and $\Q_2$ be two elements of $ \SO_3(\RR)$ with $\rot(\Q_1)=\rot(\Q_2)=\frac{1}{2}$ and let $\pi\phi\in \left[0,\frac{\pi}{2}\right]$ be the angle between the rotation axes of $\Q_1$ and $\Q_2$. Then $\rot(\Q_1\cdot \Q_2)=\phi$.
\item[(b)]
For any $\Q_1,\Q_2\in \SO_3(\RR)$   we have  $\rot(\Q_1\cdot \Q_2)\ge |\rot(\Q_1)-\rot(\Q_2)|$.
%\item For $Q_1,Q_2\in SO(3)$ with $\Rot(Q_1)=\frac{1}{2}$  we have  $\Rot(Q_1\cdot Q_2)\ge \frac{1}{2}-\Rot(Q_2)$.
\item[(c)]
For any  $\Q_1,\ldots, \Q_n\in \SO_3(\RR)$ we have $\rot(\Q_1\cdot \Q_2\cdot \ldots \cdot \Q_n)\le \sum_i \rot(\Q_i)$.
\end{itemize}
\end{lemma}
In the following lemma we show that the distance between two conical points in the spherical surface
$S$ whose angles are odd multiples of $\pi$ is controlled by the monodromy.

\begin{lemma}[Monodromy and distance between conical points]\label{geodesichol} 
Consider a spherical surface $S$ of genus $0$ with $n$ conical points of angles $2\pi\bm{\th}$
and call $x_1,x_2$ the first two conical points.
Suppose that $\th_1,\th_2\in\frac{1}{2}+\ZZ_{\geq 0}$.
Then there exists a standard set of loops $\{\beta_i\}$ such that monodromies $\Q_1,\Q_2$ corresponding to $\beta_1,\beta_2$ satisfy $\rot(\Q_1\cdot \Q_2)\le \frac{1}{\pi}d(x_1,x_2)$. If moreover $x_1$ and $x_2$ are joined by a geodesic of length $d(x_1,x_2)$, then $\Q_1\cdot \Q_2$  is a rotation of angle $2d(x_1,x_2)$.
\end{lemma}
\begin{proof} 
Let $\epsilon>0$ and let $\gamma\subset \dot{S}\cup \{x_1, x_2\}$ be an arc of length at most $d(x_1,x_2)+\epsilon$ that joins $x_1$ with $x_2$. First, we choose the standard set of loops $\{\beta_i\}$ on $\dot{S}$  based at $s_0\in \gamma$
in such a way that
the loop $\beta_1$ follows $\gamma$ from $s_0$ to a point very close to $x_1$, then encircles $x_1$ and then goes back, and the loop $\beta_2$ is chosen in an analogous way.
% and all other loops are chosen in any way, so that $\gamma_1,\ldots,\gamma_n$ is a standard set of loops. 
One sees from this construction that monodromies $\Q_1$ and $\Q_2$ are rotations of $\Sph$ around points at distance at most $d(x_1,x_2)+\epsilon$ and so $\rot(\Q_1\cdot \Q_2)\le \frac{1}{\pi}\left(d(x_1,x_2)+\epsilon\right)$ by Lemma \ref{orthogprod}(a). Letting $\epsilon\rar 0$, we get the first claim. 

Finally, if $x_1$ and $x_2$ are joined by a geodesic of length $d(x_1,x_2)$ 
completely contained in $\dot{S}\cup\{x_1,x_2\}$,
we repeat the same proof by choosing $\gamma$ to be this geodesic. Then $\Q_1$ and $\Q_2$ are rotations around two points on $\Sph$ that can be joined by a geodesic of length $2d(x_1,x_2)$ and so we are done by Lemma \ref{orthogprod}(a).
\end{proof}

The next lemma gives a constraint for the monodromy of a spherical surface along a piecewise-geodesic boundary curve.

\begin{lemma}[Length-monodromy constraint for a surface bounded by a geodesic loop]\label{geoloophol} 
Let $\Sigma$ be a spherical surface with boundary and with conical points in its interior.
Suppose that a boundary component $\beta$ of $\Sigma$ is a geodesic loop based at $s_0$
of length $\ell(\beta)<\pi$
and let $2\pi\phi$ be the angle formed by $\beta$ at $s_0$.
Then
\begin{equation}\label{lengtholin}
  \rot(\Q_{\beta})\ge \frac{1}{2\pi}\ell(\beta)\ge  \left|\rot(\Q_{\beta})-d_{\RR}\left(\phi-\frac{1}{2},
  \,\ZZ\right)\right|
 \end{equation}
where $\Q_\beta$ denotes the monodromy of $\Sigma$ along $\beta$.
\end{lemma}
\begin{proof}
We can clearly restrict our attention to a cylindrical neighbourhood $C$ of $\beta$ inside $\Sigma$.
Let $\tilde{C}$ be the universal cover of $C$, and $\tilde{\beta}$ be a geodesic segment in $\partial \tilde{C}$ whose interior projects isomorphically to $\beta\setminus \{s_0\}$. Let $\iota: \tilde{C}\to \Sph$ be a developing map and denote by $Y$ and $Y'$ the endpoints of $\iota(\tilde{\beta})$ inside $\Sph$. Clearly $\Q_{\beta}(Y)=Y'$ and so $\ell(\beta)=|YY'|\le 2\pi\cdot\rot(\Q_{\beta})$.

In order to prove the right inequality, let us introduce two auxiliary elements of $\SO_3(\RR)$:
let $\Q_{YY'}$ be the rotation that preserves the  geodesic $YY'$ and sends $Y$ to $Y'$,
and let $\Q_{Y',\phi}$ be the clockwise rotation of $\Sph$ about $Y'$ by angle $2\pi(\phi+\frac{1}{2})$. Clearly $\Q_{\beta}=\Q_{Y',\phi}\cdot \Q_{YY'}$, and so
$$\Q_{YY'}=\Q^{-1}_{Y',\phi}\cdot \Q_{\beta}.$$
Since $\rot(\Q_{YY'})=\frac{1}{2\pi}\ell(\beta)$ the desired inequality follows now from Lemma \ref{orthogprod}(b).
\end{proof}

%%%%%%%%%%%%%%%%%%%%%%%%%%%%555

\subsection{Gap properties}

In this subsection we prove two types of gap properties for spherical surfaces under consideration in Theorem \ref{main:many}: 
the first one concerns distances between conical points, the second one lengths of loops based at $x$-points.

The following proposition is the last preparatory result needed for our proof of Theorem \ref{main:many}(a).
The proposition states that the points $x_1,x_2,x_3$ 
are always sufficiently far from one another and
that each $y$-point is closest to a unique conical point, which is an $x$-point.

\begin{proposition}[Gap properties for the distances between conical points]\label{gap} 
Suppose we are in the setting of Theorem \ref{main:many}. Then the following hold.
\begin{itemize}
\item[(a)]
For $j,k\in \{1,2,3\}$ we have $d(x_j,x_k)\ge \pi(\frac{1}{2}-m\varepsilon)$.
\item[(b)]
Each point $y_i$ has a unique closest conical point, which has to be $x_1$, $x_2$, or $x_3$.
 If $x_j$ is the closest conical point to $y_i$ then for  $k\ne j$ we have
\begin{equation}\label{gapinequality}
d(y_i,x_j)\le d(y_i, x_k)-\pi(1/2-(m+1)\varepsilon)<d(y_i, x_k).
\end{equation}
%\item[(c)]
%Suppose $x_j$ is the closest to $y_i$ conical point. Then for $k\ne j$ we have $d(x_k,y_i)>d(x_j,y_i)$.
\item[(c)]
If $x_j$ is the closest conical point to $y_i$ then $d(y_i,x_k)>d(x_j, x_k)$ for  $k\ne j$.
\item[(d)]
If $d(x_2,x_3)\ge d(x_1,x_3)$
and $\alpha_{13}$ is an arc between $x_1$ and $x_3$ of length $\ell(\alpha_{13})=d(x_1,x_3)$, then
$\alpha_{13}$ is a saddle arc. The same statement holds for any permutation of indices $\{1,\,2,\,3\}$.
%  then $x_1$ and $x_3$ can be joined by a saddle arc $\alpha_{13}$ with $\ell(\alpha_{13})=d(x_1,x_3)$.
\end{itemize}
\end{proposition}

\begin{proof}
As for (a), we can assume that $j=1$ and $k=2$. Choose a standard set of loops as in Lemma \ref{geodesichol} and denote the monodromy of the loop encircling $y_i$ by $R_i$
and the monodromy of the loop encircling $x_j$ by $\Q_j$. We have 
\begin{equation}\label{holn3}
\Q_1\cdot \Q_2\cdot \Q_3\cdot \R_1\cdot \ldots\cdot  \R_m=I.
\end{equation}
 Since $\rot(\R_i)\leq \varepsilon$, we have $\rot(\R_1\cdot \ldots\cdot \R_m)\le m\varepsilon$
by Lemma \ref{orthogprod}(c). 
Hence, $\rot(\Q_3\cdot \R_1\cdot \ldots\cdot \R_m)\ge \frac{1}{2}-m\varepsilon$
by Lemma  \ref{orthogprod}(b). 
Now, Equation (\ref{holn3}) implies that $\rot(\Q_1\cdot \Q_2)=\rot(\Q_3\cdot \R_1\cdot \ldots\cdot \R_m)$ and so we conclude by Lemma \ref{geodesichol} that $d(x_1,x_2)\ge \pi(\frac{1}{2}-m\varepsilon)$.

Let us now turn to claim (b). 
Since $\varepsilon<\frac{1}{4}$, it immediately follows from Corollary \ref{closestcone} that every marked point closest to $y_i$ belongs to the set $\{x_1,\,x_2,\,x_3\}$. Suppose that $x_j$ is a marked point closest to $y_i$, so that $d_i=d(y_i,x_j)$. Given $k\ne j$, we want prove Inequality (\ref{gapinequality}) 
%\[
%d(y_i,x_j)\leq d(y_i, x_k)-\pi \left(1/2-(m+1)\varepsilon\right)<d(y_i, x_k),
%\]
which implies, in particular, that $x_j$ is the unique marked point closest to $y_i$. The second inequality of (\ref{gapinequality}) holds since $\varepsilon<\frac{1}{2m+2}$. 
To derive the first inequality, we first apply
Lemma \ref{smallangleneigh}(b) to the points $y_i,\,x_j,\,x_k$ and then use part (a)
\[
d(x_k,y_i)\ge d(x_k,x_j)+d_i-\pi\th_i\geq \pi\left(1/2-m\e\right)+d(y_i,x_j)-\pi\e.
\]

%
%The first inequality follows from part (a) and Lemma \ref{smallangleneigh}(2) applied to the points $y_i,\,x_j,\,x_k$.
%
%Consider the disk $D_{\varepsilon}(d_i)$. By Lemma \ref{closeconepoint} we have $2r(y_i)>d_i$ and so by Lemma \ref{immeradius} there is a locally isometric map $\iota: D_{\varepsilon}(d_i)\to S$ that sends the center of the disk to $y_k$. Consider  the shortest (piecewise geodesic) path $\eta_{jk}$ that joins $x_j$ with $y_k$. Let $z\in \eta_{jk}$ be the unique point such that $d(z,y_k)=d_k$. Since both $x_i$ and $z$ belong to $\iota(\partial(D_{\varepsilon}(d_k)))$, we see that $d(z,x_i)\le \pi\varepsilon$. At the same time by 1) we have $d(x_i,x_j)\ge \pi(\frac{1}{2}-n\varepsilon)$, and so $d(z,x_j)\ge \pi(\frac{1}{2}-(n+1)\varepsilon)$ which proves the claim.

%3) This statement immediately follows by applying Lemma \ref{smallangleneigh}(2) to points $y_i,\,x_j,\,x_k$.

Claim (c) follows immediately from Lemma \ref{smallangleneigh} (b), since $\e<\frac{1}{3}$.

In order to prove (d),
note first that $\alpha_{13}$ cannot pass through
any $y_i$, because $\alpha_{13}$ is length-minimizing and $\e<\frac{1}{2}$.
On the other hand, $\alpha_{13}$ clearly cannot pass through $x_2$ because
$d(x_2,x_3)\geq d(x_1,x_3)$. Hence, $\alpha_{13}$ is a geodesic segment. In order to prove that $\alpha_{13}$ is saddle, we will apply Lemma \ref{saddlechar} (a). We need to show that for any conical point $p\in S$ different from $x_1$ and $x_3$ we have  $\max(d(p,x_1),d(p,x_3))\ge\ell(\alpha_{13})$. We will split this consideration into three cases.
\begin{itemize}
\item  $p=x_2$. Then 
by our assumptions $d(x_2,x_3)\ge d(x_1,x_3)=\ell(\alpha_{13})$.
\item $p=y_i$ and the closest to $y_i$ conical point is $x_1$ or $x_3$. In this case $\max(d(y_i,x_1),d(y_i,x_3))> \ell(\alpha_{13})$ by assertion (c).
\item $p=y_i$ and the closest to $y_i$ conical point is $x_2$. In this case by assertion (c) $d(y_i,x_3)>d(x_2,x_3)$ and again by our assumptions $d(x_2,x_3)\ge d(x_1,x_3)=\ell(\alpha_{13})$.
\end{itemize}

\end{proof}
\begin{definition}[Tied $y$-points] Suppose that we are in the setting of Theorem \ref{main:many}. Then by Proposition \ref{gap}(b) for each point $y_i$ there is a unique closest point $x_i$. We will say that $y_i$ is {\it tied to} $x_j$.
\end{definition}

The next result is an application of Lemma \ref{geoloophol} and it states that a simple geodesic loop based at $x_j$ can be either rather long
or quite short depending on whether it separates or not the other two $x$-points.

% to the situation of Theorem  \ref{main:many}.

\begin{proposition}[Dichotomy for geodesic loops based at $x_j$]\label{goeodesicloop} 
Suppose that we are in the setting of Theorem \ref{main:many} and that $\varepsilon<\frac{1}{8m}$. Let
$\{j,k,l\}=\{1,2,3\}$ and let $\gamma$ be a simple geodesic loop based at $x_j$ with $\ell(\gamma)<\pi$.
Then exactly one of the following occurs:
\begin{itemize}
\item[(a)]
$\ell(\gamma)\le 2\pi m\varepsilon$,
the points $x_k$ and $x_l$ belong to the same component of $S\setminus\gamma$
and the other component contains at least one $y$-point;
%$\gamma$ bounds a disk that contains at least one $y$-point in its interior and does not contain the points $x_{k}$ and $x_{l}$;
\item[(b)]
$\ell(\gamma)\ge  \frac{\pi}{2}-2\pi m\varepsilon$ and 
the points $x_k$ and $x_l$ belong to distinct components of $S\setminus\gamma$.
%$\gamma$ separates on $\Sph$  points $x_{k}$ and $x_{l}$.
\end{itemize} 
\end{proposition}

\begin{proof}
Assume, without loss of generality,
that $(k,l,j)=(1,2,3)$, so that $\gamma$ is based at $x_3$.
Since $\ell(\gamma)<2\pi$, both connected components $S\setminus \gamma$ have 
at least one conical point in their interior: denote by $D_2$ the component that contains $x_2$
and by $D_1$ the other component.

(a) Consider first the situation when $D_2$ contains both $x_1$ and $x_2$. In this case,
by Remark \ref{holboundary} 
the monodromy $\Q_{\pa D_1}$ along $\partial D_1$ is a product of at most $m$ rotations of $\Sph$, of angle at most $2\pi\varepsilon$ each. 
Hence, we can apply Lemma \ref{geoloophol} to $\Sigma=D_1$ and obtain $\ell(\gamma)\le 2\pi m\varepsilon$.

(b) Suppose now that $x_1\in D_1$ and $x_2\in D_2$.
Let $2\pi\th_{13}$ and $2\pi\th_{23}$ be the angles that $\partial D_1$ and  $\partial D_2$ form at $x_3$. Since $\left(\th_{13}-\frac{1}{2}\right)+\left(\th_{23}-\frac{1}{2}\right)=\th_3-1=2m_3-\frac{1}{2}$, 
we must have
either $d_\RR\left(\th_{13}-\frac{1}{2},\,\ZZ\right)\le \frac{1}{4}$ or $d_\RR\left(\th_{23}-\frac{1}{2},\,\ZZ\right)\le \frac{1}{4}$. 
Up to exchanging the roles of $x_1$ and $x_2$, we can assume the former.

Applying  first Inequality (\ref{lengtholin}) and then Lemma \ref{orthogprod}(c,b) we obtain
$$\frac{1}{2\pi}\ell(\partial D_1)+\frac{1}{4}\ge \frac{1}{2\pi}\ell(\partial D_1)+
d_\RR\left(\th_{13}-\frac{1}{2},\,\ZZ\right)\ge \rot(\Q_{\pa D_1})\ge \frac{1}{2}-m\varepsilon.$$
This proves the result.
\end{proof}

%%%%%%%%%%%%%%%%%%%%%%%%%

\subsection{Construction of the separating Voronoi cylinders $C_j$}\label{sec:separating-cylinders}

Stated informally, the goal of this subsection is to explain that spherical surfaces under consideration in Theorem \ref{main:many}(b) {\it are geometrically similar} to the surface depicted on Figure \ref{fig:disconnected}. More precisely, we prove that on such surfaces for each $x_j$ there is a Voronoi cylinder $C_j$ that
separates $x_j$ and $y$-points tied to $x_j$ from all the other conical points. Cylinders $C_1,\, C_2,\, C_3$ are constructed in Lemma \ref{lemma:Cj} and some of their basic properties are stated in Proposition \ref{cyl-Cj}. In Lemma \ref{modulus-Cj} we show that under assumptions of Theorem \ref{main:many}(b) cylinders $C_j$ have modulus at least $\frac{1}{2}$.

%In this subsection we show that each spherical surface under consideration in Theorem \ref{main:many}(b)
%contains Voronoi cylinders $C_1,C_2,C_3$ of modulus at least $\frac{1}{2}$ and that each $C_j$
%separates $x_j$ and the $y$-points close to $x_j$ from all the other conical points.

\begin{proposition}[$C_j$ does not separate $x_j$ from the $y$-points tied to $x_j$]\label{cyl-Cj}
Suppose that we are in the setting of Theorem \ref{main:many} and that $\e<\frac{1}{8(m+1)}$.
Let $\Ical_\e$ be the interval $\left[\pi (m+1)\e,\,\pi\left(\frac{1}{4}-(m+1)\e\right)\right]$. Then there exist connected components $C_1,\, C_2,\, C_3$ of $\Vor^{-1}(\Ical_\e)$ such that
\begin{itemize}
\item[(a)] each $C_j$ is  a Voronoi cylinder completely contained in the interior of the Voronoi domain $\DV_{x_j}$;
\item[(b)]  if $y_i$ is tied to $x_j$, then $y_i$ and $x_j$ belong
to the same connected component of $S\setminus C_j$.

\end{itemize}

\end{proposition}

Let us remind here that the Voronoi domain $\DV_{x_j}$ is not necessarily a disk, and it can contain it its interior some open edges of the Voronoi graph $\Gamma$. 

We subdivide most of the argument needed to prove the above proposition
into the following three lemmas, in which we assume $j=3$ 
and $d(x_1,x_3)\leq d(x_2,x_3)$ to simplify the notation.

\begin{lemma}[Construction of $C_3$ the interior of  $\DV_{x_3}$]\label{lemma:Cj}
Suppose that $d(x_1,x_3)\leq d(x_2,x_3)$ and let $\alpha_{13}$ be the length-minimizing arc
between $x_1$ and $x_3$. 
For all $c\in \Ical_\e$, denote by $\lev_c$ the connected component of $\Vor^{-1}(c)$
that meets $\alpha_{13}$.
Then the union $C_j:=\bigcup_{c\in \Ical_\e}\lev_c$ is completely contained in the interior of $\DV_{x_3}$.
\end{lemma}
\begin{proof}
By Proposition \ref{gap}(c) the path $\alpha_{13}$ is a saddle arc. The situation is illustrated in Figure \ref{fig:disconnected}.

Let $\hat{q}$ be the  midpoint of $\alpha_{13}$. Then, since $\alpha_{13}$ is a saddle arc, the segment $\hat{q} x_3$ belongs to the Voronoi domain $\DV_{x_3}$. In particular $\Vor(q)=d(x_3,q)$ for any $q\in \hat{q}x_3$. 

Let us show first that for every $c\in \Ical_\e$ the curve $\lev_c$ is disjoint form any Voronoi domain $\DV_{y_i}$. Indeed, let $q_c$ be the point of intersection of $\lev_c$ with $\alpha_{13}$ and consider the connected set $\hat{\lev}_c=\hat{q} q_c\cup \lev_c$. Clearly $\Vor(\hat{\lev}_c)\ge c>\pi\varepsilon$. Applying Lemma \ref{smallangleneigh}(c) to $y_i$, we see that  $\hat{\lev}_c$ does not intersect $\partial \W_{y_i}$. At the same time $\hat{q}$ does not belong to $\W_{y_i}$. Hence $\hat{\lev}_c$ is disjoint from $\W_{y_i}$, and so it is disjoint from $\DV_{y_i}$ by Lemma \ref{smallangleneigh}(a). 

To finish the proof is suffices to show that $\lev_c$ is disjoint from $\DV_{x_1}$ and $\DV_{x_2}$. 
By contradiction, suppose that $\lev_c$ meets $\DV_{x_1}$.
Then there exists a point on $\lev_c$ at distance $c$ from $x_3$ and $x_1$,
and so $d(x_3,x_1)\le 2c$, which contradicts Proposition \ref{gap}(a). Similarly, $\lev_c$
does not meet $\DV_{x_2}$.
%We thus shown that the intersection of $C_3$ and $\pa\DV_{x_3}$ 
%is a closed subset consisting of points
%where the map $\DVbar_{x_3}\rar\DV_{x_3}$ is not injective and which are not vertices of $\Gamma$.
%We have to show that such intersection is empty.
%By contradiction, let $s\in C_3\cap \pa\DV_{x_3}$ be a point at which $\Vor(s)=c$ is minimal.
%
%
%It is easy to see that $s$ must be a saddle point for $\Vor$ and so
%there is a geodesic loop $\sigma_s$ based at $x_3$ through $s$
%of length $2c\leq \frac{\pi}{2}-2(m+1)\e$.
%By Proposition \ref{goeodesicloop}, such loop must have $\ell(\sigma_s)\leq 2\pi m\e$
%and so $c\leq \pi m\e$ gives the wished contradiction.
\end{proof}

\begin{lemma}[$C_3$ is a Voronoi cylinder]\label{Cj-Voronoi}
The locus $C_3$ is a Voronoi cylinder.
\end{lemma}
\begin{proof}
%It is enough to show that each curve $\lev_c$ contains only regular points. 
%Suppose by contradiction that $s\in \lev_c$ is a critical point. 
%Since $\Vor(s)>0$, 
%the point $s$ is not a local minimum for $\Vor$. Also, since $c<\frac{\pi}{2}$, any local maximum 
%for $\Vor$ in the level set $\Vor^{-1}(c)$ is isolated.
%On the other hand, points of $C_3\cap \alpha_{13}$ are not local maxima of $\Vor$ by Proposition 
%\ref{gap}(c). It follows that $s$ is not a local maximum of $\Vor$ either. 

It is enough to show that each curve $\lev_c$ for $c\in \Ical_\e$ contains only regular points. Since $\Vor(\lev_c)=c>0$, points in $\lev_c$ are not local minima. Let us show that $\lev_c$ doesn't contain local maxima. Indeed, $\alpha_{13}$ is a saddle geodesic by Proposition \ref{gap}(c) and so all points in $C_3\cap \alpha_{13}$, including the point $\lev_c\cap \alpha_{13}$ are regular. At the same time,  $\lev_c$ is connected by construction, while any local maximum for $\Vor$ in the level set $\Vor^{-1}(c)$ is isolated.

Suppose now that $s\in \lev_c$ is a saddle point for $\Vor$ and let $\sigma_s$ be a saddle geodesic passing through $s$. 
By Lemma \ref{lemma:Cj}, the curve $\lev_c$
is contained in the interior of $\DV_{x_3}$, and so $\sigma_s$ is a saddle loop based at $x_3$ of length $2c$. This contradicts Proposition \ref{goeodesicloop} and proves that $C_3$ is a Voronoi cylinder.
\end{proof}

%\begin{lemma}[$C_3$ is disjoint from $\W_{y_i}$ for $x_3$ closest to $y_i$]\label{C3-disjoint-Byi}
%Let $x_3$ be the conical point closest to $y_i$.
%Then $C_3\cap \W_{y_i}=\emptyset$.
%\end{lemma}

\begin{lemma} \label{C3-disjoint-Byi}
Suppose that $y_i$ is tied to $x_3$, then $C_3$ is disjoint from $\W_{y_i}$.
\end{lemma}
\begin{proof}
By Lemma \ref{smallangleneigh}(c), the boundary $\pa \W_{y_i}$
is contained in the interior of $B_{x_3}(\pi\e)$.
So, since $\Vor(C_3)>\pi\varepsilon$, we have $C_3\cap \pa \W_{y_i}=\emptyset$.
Suppose by contradiction that $C_3\cap \W_{y_i}\neq\emptyset$, i.e., $C_3\subset \W_{y_i}$.
Pick $q\in \alpha_{13}\cap C_3$ and note the the portion $qx_1$ of $\alpha_{13}$
must cross $\pa \W_{y_i}$ at some point $q'$, since $d(y_i,x_1)>d(y_i,x_3)$.
Since $\alpha_{13}$ is length-minimizing, we have
$d(x_3,q')\geq d(x_3,q)$
but this contradicts that fact that $d(x_3,q)\geq \pi(m+1)\e$
and $d(x_3,q')\leq \pi\e$, and so we can conclude.
\end{proof}

We can now prove that the separation properties of the Voronoi cylinder $C_j$
distinguishes the points $y_i$ tied to $x_j$
from the other $y$-points.

\begin{proof}[Proof of Proposition \ref{cyl-Cj}]
We can assume $j=3$: the other cases $j=1,2$ are analogous.
Also, up to switching the roles of $x_1,x_2$, we can assume that $d(x_1,x_3)\leq d(x_2,x_3)$.

(a) By Lemma \ref{lemma:Cj} and Lemma \ref{Cj-Voronoi},
the locus $C_3$ is a Voronoi cylinder contained in the interior of $\DV_{x_3}$.

(b) Suppose now that $y_i$ is tied to $x_3$
and let $x_3 y_i$ be the length-minimizing arc between them, which is thus
contained inside $\W_{y_i}$.
By Lemma \ref{C3-disjoint-Byi}, the cylinder $C_3$ is disjoint from $\W_{y_i}$
and so $x_3 y_i$ does not meet $C_3$.
It follows that $x_3$ and $y_i$ belong to the same connected component
of $S\setminus C_3$.
\end{proof}

Now we can complete the analysis of how $C_j$ separates $x_j$ from the other conical points.

\begin{proposition}[$C_j$ separates $x_j$ from the other $x$-points]\label{Cj-splits-xj}
The Voronoi cylinder $C_j$ separates $x_j$ from $\{x_k,x_l\}$
for all $\{j,k,l\}=\{1,2,3\}$, namely
$x_j$ and $\{x_k,x_l\}$ are 
not contained in the same connected component of $S\setminus C_j$.
\end{proposition}
\begin{proof}
As above, we can assume that $j=3$ and that $d(x_1,x_3)\leq d(x_2,x_3)$. Then, as in Proposition \ref{gap}(d) we have a saddle arc $\alpha_{13}$ that joins $x_1$ with $x_3$, and we can consider the cylinder $C_3$ constructed as in Lemma \ref{lemma:Cj}.
We denote by $\partial_3 C_3$ the boundary curve of $C_3$ 
that is at distance $\pi (m+1)\varepsilon$ from $x_3$.
By construction,
$\partial_3 C_3$  separates $x_1$ from $x_3$. It remains to show that $\pa_3 C_3$ separates $x_2$ from $x_3$ as well. 

Suppose, in fact, that $\partial_3 C_3$ does not separate $x_2$ from $x_3$. 
We will achieve a contradiction with Proposition \ref{goeodesicloop}  
by constructing a simple geodesic loop based at $x_3$ of length at most 
$2\pi (m+1)\varepsilon<\frac{\pi}{2}-2\pi m\e$ that separates $x_1$ from $x_2$.

We begin by showing in three steps that a piecewise geodesic curve $\alpha_{23}$
of length $d(x_2,x_3)$ that joins $x_2$ and $x_3$ is disjoint from $C_3$ and is, in fact, a saddle arc. 

{\it i) $\alpha_{23}$ is disjoint from $C_3$}. By our assumption $\partial_3 C_3$ does not separate $x_2$ from $x_3$, and so to prove that $\alpha_{23}$ is disjoint from $C_3$ it is enough to show that $\alpha_{23}$ is disjoint from 
$\partial_3 C_3$. Assume the converse. Clearly, the distance function $d(x_3,.)$ parametrizes $\alpha_{23}$. So, since  $\partial_3 C_3$ is equidistant from $x_3$, $\alpha_{23}$ can intersect $\partial_3 C_3$ at most once and the intersection must be transverse. Hence, $\partial_3 C_3$  separates $x_2$ from $x_3$, which is a contradiction.

{\it ii) $\alpha_{23}$ is a geodesic}. It follows from {\it i)}, that $\alpha_{23}$ does not pass through $x_1$.
On the other hand, $\alpha_{23}$ cannot pass through any $y_i$
because it is length-minimizing and $\e<\frac{1}{2}$. It follows that $\alpha_{23}$ is a geodesic arc. 

{\it iii) $\alpha_{23}$ is a saddle arc}. By Proposition \ref{gap}(d), it suffices to prove 
that  $d(x_1,x_2)\ge d(x_1,x_3)$. Since $\partial_3 C_3$ is equidistant from $x_3$, we have $d(x_1,x_3)=d(x_1,\partial_3 C_3)+d(\partial_3 C_3,x_3)$. Since $\partial_3 C_3$ lies in the interior of $\DV_{x_3}$, we have $d(\partial_3 C_3,x_3)<d(\partial_3 C_3,x_2)$. Finally, $d(x_1,x_2)\ge d(x_1,\partial_3 C_3)+d(\partial_3 C_3,x_2)$.
 %Proposition \ref{gap}(c), we conclude that $\alpha_{23}$ is a saddle arc, in particular for any $q\in \alpha_{23}$ we have $\Vor(q)=d(q,\{x_2,x_3\})$. 

%Note next that, by construction, any path that joins $x_1$ with a point of $\alpha_{23}$ necessarily intersects $C_3$, which lies in the interior of $\DV_{x_3}$. It follows that
%each point of the geodesic $\alpha_{23}$ is further from $x_1$ than from $x_3$, in particular, $d(x_1,x_2)>d(x_1,x_3)$.
%By Proposition \ref{gap}(c),
%we conclude that $\alpha_{23}$ is a saddle arc. 

\begin{center}
\begin{figurehere}
\psfrag{S}{$\textcolor{blue}{S}$}
\psfrag{C3}{$\textcolor{Purple}{\! C_3}$}
\psfrag{d3C3}{$\textcolor{Purple}{\!\pa_3 C_3}$}
\psfrag{x1}{$\textcolor{Sepia}{x_1}$}
\psfrag{x2}{$\textcolor{Sepia}{x_2}$}
\psfrag{x3}{$\textcolor{Sepia}{x_3}$}
\psfrag{yi}{$\textcolor{Sepia}{y_i}$}
\psfrag{a13}{$\textcolor{Sepia}{\alpha_{13}}$}
\psfrag{a23}{$\textcolor{Sepia}{\alpha_{23}}$}
\psfrag{a}{$\textcolor{BrickRed}{\alpha}$}
\psfrag{l1c}{$\textcolor{Purple}{\lambda_{1,c}}$}
\psfrag{l2c}{$\textcolor{Purple}{\!\!\lambda_{2,c}}$}
\psfrag{q1c}{$\textcolor{Purple}{q_{1,c}}$}
\psfrag{q2c}{$\textcolor{Purple}{q_{2,c}}$}
\psfrag{Wi}{$\textcolor{OliveGreen}{\W_{y_i}}$}
\psfrag{si}{$\textcolor{red}{\!\sigma_s}$}
\psfrag{s}{$\textcolor{red}{\! s}$}
\includegraphics[width=0.45\textwidth]{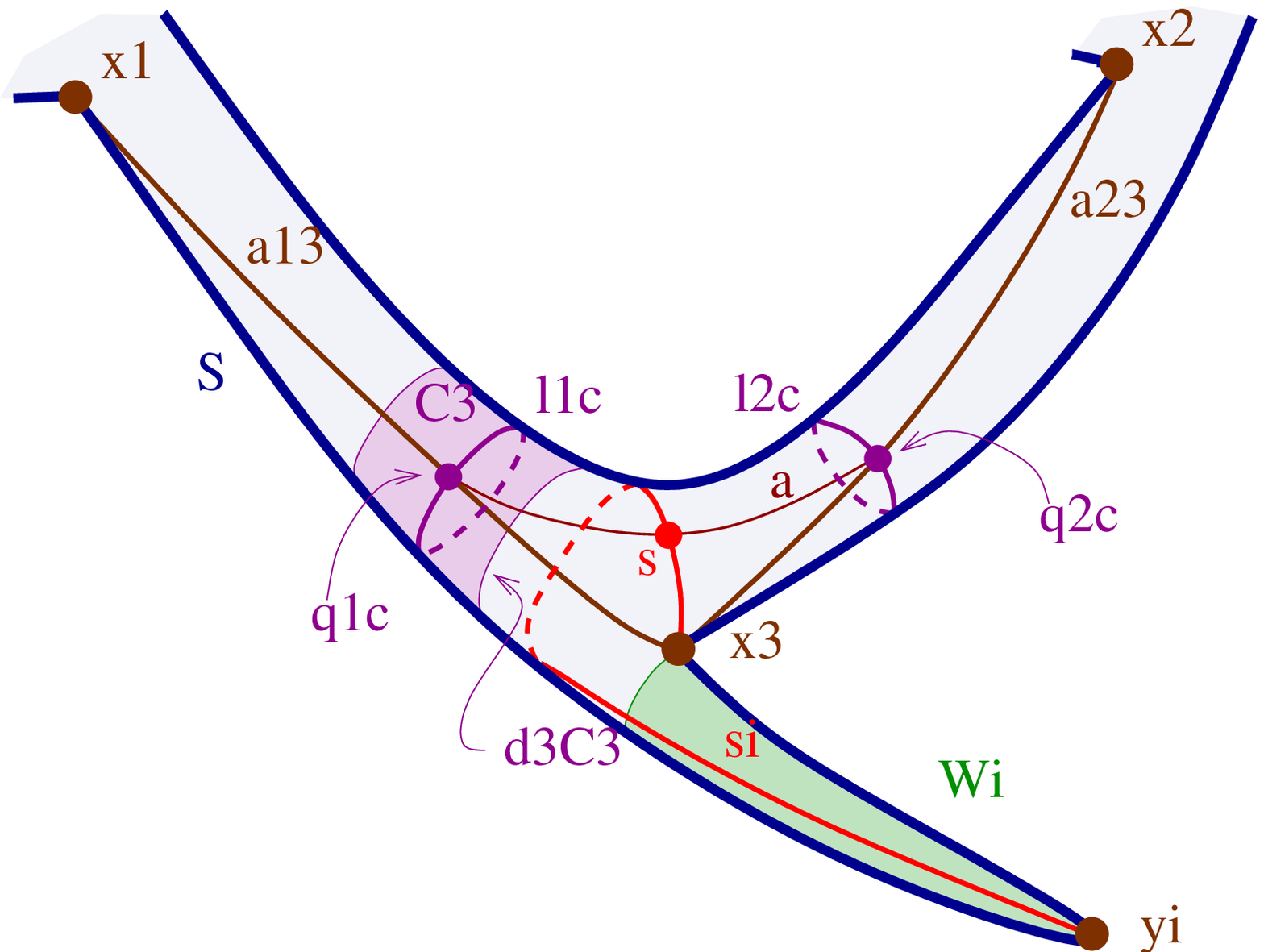}
\caption{{\small If $C_3$ does not separate $x_2$ from $x_3$, the saddle geodesic $\sigma_s$ cannot end at some $y_i$.}}\label{fig:finding-loop}
\end{figurehere}
\end{center}

Now,
fix a regular $c\in \Ical_\e$ and let $q_{1,c}\in \alpha_{13}$ and $q_{2,c}\in \alpha_{23}$ 
be points at distance $c$ from $x_3$. Since $\alpha_{13}$ and $\alpha_{23}$ are saddle geodesics,  we have $\Vor(q_{1,c})=\Vor(q_{2,c})=c$.
Since $q_{1,c}\in C_3$ and $q_{2,c}\notin C_3$,
the points $q_{1,c},q_{2,c}$ belong to different connected components of the level set $\Vor^{-1}(c)$. We 
denote such components by $\lev_{1,c}$ and $\lev_{2,c}$ correspondingly, and denote by $S'$ the cylindrical subsurface of $S$ bounded by $\lev_{1,c}$ and $\lev_{2,c}$. 
By applying Lemma \ref{saddlepath}, we obtain
a path $\alpha$ inside $S'$ with the following properties
\begin{itemize}
\item
$\alpha$ passes through a saddle point $s$ with $\Vor(s)=c'<c$
\item
$\alpha$ is contained in $\Vor^{-1}([c',\pi))\cap S'$
\item
$\alpha$ intersects transversally the saddle geodesic $\sigma_s\subset S'$ passing through $s$.
\end{itemize}
We claim that
$\sigma_s$ is a saddle loop based at $x_3$.
As a consequence, by Lemma \ref{saddlepath} the loop $\sigma_s$ separates $q_{1,c}$ from $q_{2,c}$ on $S$, and so 
it also separates $x_1$ and $x_2$. 
This gives us a desired contradiction with Proposition \ref{goeodesicloop}.

In order to finally prove the above claim, note first that
$\sigma_s$ cannot end at $x_1$ or at $x_2$, since it is contained in $S'$.
%because points on $\sigma_s$ are at distance at most $c'<d(x_3,\{x_1,x_2\})$ from $x_3$.
%
%is contained in a connected component of $\Vor^{-1}([0,c'))$ 
%that hosts $x_3$.
%and it has one endpoint in $C_3$, 
%it cannot end at $x_1$ or at $x_2$.
Suppose by contradiction that $\sigma_s$ has one endpoint at $y_i$ and let $d_{y_i}=d(y_i,x_3)$.
Lemma \ref{smallangleneigh}(a) and  Lemma \ref{smallangleneigh}(d) state correspondingly that 
$$d_{y_i}/2<\Vor(s)=c'<d_{y_i},\;\;\; \Vor(\pa\W_{y_i})<d_{y_i}/2.$$
But since $\Vor(\alpha)\ge c'>d_{y_i}/2$, the second inequality implies that $\alpha\cap \pa\W_{y_i}=\emptyset$. Hence $\alpha$ cannot enter $\W_{y_i}$, in particular $s\notin \W_{y_i}$, i.e., $d(s,y_i)>d_{y_i}$. However $\Vor(s)=d(s,y_i)$ and so we get a contradiction with Lemma \ref{smallangleneigh}(a).
%This contradicts the fact that the portion $sy_i$ of $\sigma_s$ has length $c'<d_{y_i}$.
Hence, $\sigma_s$ must be a saddle loop based at $x_3$ and the proof is complete.
\end{proof}

\begin{corollary} If $y_i$ is not tied to $x_j$ then $y_i$ and $x_j$ belong to distinct connected components of $S\setminus C_j$.
\end{corollary}
\begin{proof} By definition $C_1,\,C_2,\,C_3$ are pairwise disjoint. And so the conclusion follows from Propositions \ref{Cj-splits-xj} and \ref{cyl-Cj}(b).
\end{proof}

As a last piece of information on the cylinders $C_j$,
here we give an estimate
of their modulus that will be needed in the proof of Theorem \ref{main:many}(b).

\begin{lemma}\label{modulus-Cj}
Under the hypotheses of Theorem \ref{main:many}(b),
each Voronoi cylinder $C_j$ constructed above has modulus $M(C_j)>\frac{1}{2}$.
\end{lemma}
\begin{proof}
Assume $j=3$, the other cases $j=1,2$ being analogous.

Consider the Voronoi cylinder $C_3$
constructed in Proposition \ref{cyl-Cj}.
The boundary curves of $C_3$ belong to the level sets $\Vor^{-1}(\pi (m+1)\e)$ and
$\Vor^{-1}(\pi/4-\pi (m+1)\e)$
and
the total angle of the conical points contained in the component of $S\setminus C_3$ that contains 
$x_3$ is at most $2\pi\left(m_1+\frac{1}{2}+m\e\right)<2\pi(m_1+1)$.
By Lemma \ref{cylinderslength}, it follows that
\[
M(C_1)> \frac{1}{2\pi(m_1+1)}\log\left(\frac{1}{4(m+1)\varepsilon}-1\right).
\]
As a consequence, $M(C_1)>\frac{1}{2}$ if 
\[
\varepsilon< \frac{1}{4(m+1)\left(1+\exp \pi(\max\{m_i\}+1)\right)}
\]
is satisfied. 
This is the case, because the right hand side is not smaller than $\frac{1}{16}\exp(-2\pi \max\{m_i\})$.
\end{proof}

\subsection{Disconnectedness of the moduli spaces}\label{disconnectedproofsub}

We can finally assemble all the information on the cylinders $C_j$ obtained above
and prove the main result of this section.

\begin{proof}[Proof of Theorem \ref{main:many}]
%Let $S$ be a spherical surface in $\MSPH_{0,m+3}(\bm{\th})$ with conical points $x_1,x_2,x_3,y_1,\dots,y_m$,
%and let $J$ be its underlying complex structure.
%
(a) Consider a spherical surface in $\MSPH_{0,m+3}(\bm{\th})$ with conical points $x_1,x_2,x_3,y_1,\dots,y_m$.
By Proposition \ref{gap}(b), each point $y_j$ on $S$ has a unique 
closest conical point among $x_1,\,x_2,\,x_3$. Hence, to each spherical surface
in $\MSPH_{0,3+m}(\bm{\th})$ we can associate a function 
$\kappa:\{1,\ldots, m\}\to \{1,2,3\}$
in such a way that $x_{\kappa(i)}\in \{x_1,x_2,x_3\}$ is the point closest to $y_i$. 
Since the
$3^m$ functions $d(y_i,x_j)$ are continuous on the moduli space of metrics, by Proposition \ref{gap}(b) the application 
$\Kcal:\MSPH_{0,3+m}(\bm{\th})\rar \{1,2,3\}^m$
that sends a spherical surface to its associated vector
$(\kappa(1),\ldots, \kappa(m))$ is locally constant. 
It follows that $\MSPH^{\kappa}_{0,3+m}(\bm{\th}):=\Kcal^{-1}(\kappa)$ is a union of connected components
of $\MSPH_{0,3+m}(\bm{\th})$.
Hence, to prove claim (a) we need to show that each $\MSPH^{\kappa}_{0,3+m}(\bm{\th})$ is non-empty.

Consider a partition of $\{1,2,\dots,m\}$ into three disjoint subsets $I_1,I_2,I_3$. 
In order to construct the metric such that $\kappa^{-1}(j)=I_j$, we start first with a spherical surface of genus $0$ with three conical points $(x_1',x_2',x_3')$ such that $\th'_j=m_j+\frac{1}{2}+\sum_{i\in I_j}(\e_i-1)$. 
Such a spherical surface exists by \cite{eremenko:three} and, since the monodromy is not coaxial, we can apply Proposition \ref{prop:splitting} to split each point $x'_j$ into a point $x_j$ of angle $2\pi\left(m_j+\frac{1}{2}\right)$ and a collection of nearby points $y_i$ with $i\in I_j$ of conical angles $2\pi\e_i$. 
We can arrange in such a way that
the function $f$ associated to such spherical surface satisfies $\kappa(I_j)=j$ for $j=1,\,2,\,3$ simply by taking 
$\eta>0$ small enough in Proposition \ref{prop:splitting}.

(b) For every $i=1,\dots,m$ consider the map $\Pi_i:\Mcal_{0,3+m}\rar \Mcal_{0,4}$
that sends $(S,J,x_1,x_2,x_3,y_1,\dots,y_m)$ to $(S,J,x_1,x_2,x_3,y_i)$
and let $\Pi=(\Pi_1,\dots,\Pi_m):\Mcal_{0,3+m}\rar \left(\Mcal_{0,4}\right)^m$. Let us denote subregions $\Mcal_{0,4}^{(1,4)}$, $\Mcal_{0,4}^{(2,4)}$ and $\Mcal_{0,4}^{(3,4)}$ of $\Mcal_{0,4}$ as in Lemma \ref{lemma:peripheral}.
For every $\kappa\in\{1,2,3\}^m$, we denote by $\Mcal^\kappa_{0,3+m}$ the subset of $\Mcal_{0,3+m}$ defined as
\[
\Mcal^\kappa_{0,3+m}:=
\Pi^{-1}\left(
\Mcal_{0,4}^{(\kappa(1),4)}\times\Mcal_{0,4}^{(\kappa(2),4)}\times\dots\times\Mcal_{0,4}^{(\kappa(m),4)}
%\prod_{i=1}^m\Mcal_{0,4}^{(\kappa(i),4)}
\right)
\]
and we note that such $\Mcal^\kappa_{0,3+m}$ are pairwise disjoint by Lemma \ref{lemma:peripheral}.
In order to conclude, it is enough to show that
$F_{0,3+m,\bm{\th}}$ maps $\MSPH^\kappa_{0,3+m}(\bm{\th})$ to $\Mcal^\kappa_{0,3+m}$.

We recall that on each such surface $S$ there exists a cylinder $C_{\kappa(i)}$
(Proposition \ref{cyl-Cj})
of modulus greater than $\frac{1}{2}$ (Lemma \ref{modulus-Cj}) such that
the pairs of conical points $\{x_{\kappa(i)},\,y_i\}$ and $\{x_1,x_2,x_3\}\setminus\{x_{\kappa(i)}\}$
are contained in different components of $S\setminus C_{\kappa(i)}$ (Proposition \ref{Cj-splits-xj}).
This in particular implies that a simple closed curve $\gamma_j$ homotopic to $C_j$
has $\Ext_{\gamma_j}(\dot{S},J)<2$ and so $\Ext_{\gamma_j}(S\setminus\{x_1,x_2,x_3,y_i\},J)<2$.
It follows from Lemma \ref{lemma:peripheral} that $(\Pi_i\circ F_{0,3+m,\bm{\th}})(\MSPH^\kappa_{0,3+m}(\bm{\th}))$ is contained inside
$\Mcal_{0,4}^{(\kappa(i),4)}$ and so
\[
F_{0,3+m,\bm{\th}}(\MSPH^\kappa_{0,3+m}(\bm{\th}))\subseteq \Mcal^\kappa_{0,3+m}
\]
as desired.
%By part (i) of Theorem \ref{main:many},
%the whole $\MSPH_{0,3+m}(\bm{\th})$ is the disjoint union of all
%$\MSPH^f_{0,3+m}(\bm{\th})$, each of which is non-empty.
%Thus, the image of $F_{0,3+m,\bm{\th}}$
%is the disjoint union of 
%the non-empty subsets $\IM(F_{0,3+m,\bm{\th}})\cap \Mcal^f_{0,3+m}$.
%Therefore, it has at least $3^m$ connected compoents.
\end{proof}

%\subsubsection{Extremal estimate}
%
%\begin{lemma} Consider a flat metric on $\mathbb CP^1$ with conical points $x_1,\ldots,x_4$ of cone angles $\pi$. Then we have $\min_{i\ne j}\frac{(2d(x_i))^2}{\area(\mathbb CP^1)}\le 2$.
%\end{lemma}
%\begin{proof}
%Consider the double cover $E$ of $\mathbb CP^1$ with ramification at points $x_1,\ldots,x_4$. Then $\min_{i,j}d(x_i,x_j)$ is half the length of the shortest geodesic $\gamma_{\sys}$  on $E$, while $\area(E)=2\area(\mathbb CP^1)$. Since $\ell(\gamma_{\sys})^2/\area(E)\le 1$, the statement follows. 
%\end{proof}
%
%\begin{corollary} Consider $\mathbb CP^1$
%\end{corollary}

%%%%%%%%%%%%%%%%%%%%%%%%

\section{Disks with one conical point}\label{sec:disks-one}

%\input non-existence-disks-cylinders.tex

%\subsection{Disks with one conical point}
The purpose of this section is to estimate the area of a 
topological disk 
endowed with a spherical metric with at most one conical point and a short boundary.
Such computation will be needed to calculate the area of almost bubbling surfaces in Section \ref{sec:systole}.
For this reason, we will use the symbol $\BU$ to denote such 
topological disks.

In order to measure how short the boundary of $\BU$ is, we introduce the following quantity.

\begin{definition}[$\lambda$-invariant]
Let $\BU^0$ be a spherical disk without conical points and $\BU^1$ be a spherical with one conical point $x$. We define their
{\it{$\lambda$-invariant}} as
\[
\lambda_0(\BU^0):=\left(\frac{\ell(\partial \BU^0)}{2\pi}\right)^2
\qquad\text{and}\qquad
\lambda_1(\BU^1):=\frac{\ell(\pa \BU^1)}{d(x,\pa \BU^1)}.
\]
\end{definition}

%\begin{definition}[$\lambda$-invariant]
%Let $\BU$ be a spherical monodromy with at most one conical point.
%If $\BU$ has a conical point $x$, we let
%\[
%\lambda(\BU):=\frac{\ell(\pa \BU)}{d(x,\pa \BU)}.
%\]
%If $\BU$ has no conical point, we define
%$\displaystyle
%\lambda(\BU):=\frac{\ell(\pa \BU)}{\max_{p\in \BU}d(p,\pa \BU)}$.
%\end{definition}

The following result is an essential ingredient of our proof of the systole inequality.

\begin{theorem}[Disks with one conical point and short boundary]\label{diskstructure} 
Let $\BU^1$ be a spherical disk with one conical point $x$ of angle $2\pi\th$
and assume that $\lambda_1(\BU^1)<\frac{1}{2}$.
Then there exists $b^1\in\ZZ_{\geq 0}$ such that
\[
\frac{1}{2\pi}\Big|\Area(\BU^1)-4\pi(\th+b^1)\Big|<\frac{\ell(\pa \BU^1)}{d(x,\pa \BU^1)}=\lambda_1(\BU^1).
\]
\end{theorem}

%The case of a disk without a conical point can be deduced from the above result
%just marking by $x$ a point of $\BU$ which is farthest from $\pa \BU$. Thus the following is immediate.
%
%\begin{corollary}[Disks without conical points and with short boundary]\label{cor:diskstructure}
%If $\BU$ is a spherical disk without conical points such that $\lambda(\BU)<\frac{1}{2}$,
%then
%\[
%|4\pi b-\Area(\BU)|<2\pi\lambda(\BU).
%\]
%for a suitable integer $b\geq 1$.
%\end{corollary}

%\begin{remark}\label{rmk:shorterthanpi/2} 
%Since the distance from $x$ to $\partial \BU$ is less than $\pi$, by the assumptions of the theorem $\ell(\partial \BU)<\frac{\pi}{2}$.
%\end{remark}

\begin{remark}
The hypothesis $\lambda_1(\BU^1)<\frac{1}{2}$
is only used at the very end of proof of Theorem \ref{diskstructure}.
The analogous estimate for disks without conical points
is proven in Corollary \ref{shorthindisk}.
\end{remark}

The proof of Theorem \ref{diskstructure} proceeds as follows.
First we reduce the calculation of the area of the disk $\BU^1$ 
to that of a disk $D_\alpha$ without conical points,
which is obtained
by cutting $\BU^1$ along a geodesic
arc $\alpha$ that joins $x$ to $\pa \BU^1$.
Then we compute the area of $D_\alpha$
using the degree function of 
a developing map $\iota:D_\alpha\rar\Sph$
and the {\it algebraic area} of
its boundary $\iota(\pa D_\alpha)$ relative to a base point $Z\in\Sph$. Such algebraic area is
an invariant of oriented loops in $\Sph$ which has good additive properties and is tightly related to the degree function of $\iota$.
The final estimate for the algebraic area of $\iota(\pa D_\alpha)$
relies on an area estimate for isosceles spherical triangles
embedded in $\Sph$ and on
the isoperimetric inequality for domains in $\Sph$.

\subsection{Degree functions and algebraic area on $\Sph$}\label{sec:algebraic-area}

In this subsection we introduce the degree function 
and the algebraic area associated to an oriented loop in $\Sph$
and we prove some elementary relation between such quantities,
the classical degree and the standard area.

\begin{definition}[Degree function of an oriented loop in $\Sph$ relative to a base point]
Let $\xi$ be a piecewise smooth oriented closed curve on $\Sph$. Let $Z$ be a point in the complement $\Sph\setminus \xi$. The {\it degree function} $\deg_{Z}(\xi):\Sph\setminus \xi\rar\ZZ$ is defined as
\[
\deg_{Z}(\xi)(Y):=[YZ]\cdot \xi
\]
for every point $Y\in\Sph\setminus\xi$,
where $[YZ]$ is (the relative homology class of) any smooth path that runs from $Y$ to $Z$ and $[YZ]\cdot \xi$ is the intersection number (with sign) of two curves. 
\end{definition}

\begin{definition}[Algebraic area of an oriented loop in $\Sph$]
The {\it algebraic area} of $\xi$ with respect to $Z$ is defined as
\[
\Alg_Z(\xi)=\int_{\Sph}\deg_{Z}(\xi)\cdot\,\omega
\]
where $\omega$ is the standard area form on $\Sph$.
\end{definition}

Note that the algebraic area $\Alg_Z(\xi)$ is continuous for piecewise-smooth deformations of $\xi$ in $\Sph\setminus Z$. The following standard lemma justifies the name of degree function.

\begin{lemma}[Degree of a map from a disk to $\Sph$ and degree of its boundary]\label{degreefunction} 
Let $D$ be an oriented disk and $ \varphi: D\to \Sph$ be a piecewise smooth map and let 
$\xi=\pa\varphi:\pa D\rar\Sph$ be the restriction of $\varphi$ to the boundary $\pa D$, which is endowed with
the induced orientation.
%with induced counterclockwise orientation. 
Choose a regular point $Z\in \Sph\setminus\xi$ for $\varphi$. Then
the degree function $\deg(\varphi):\Sph\setminus\xi\rar\ZZ$ of the map $\varphi$
coincides with $\deg(\varphi)(Z)+\deg_{Z}(\xi)$.% coincides with 
%  of the map $\varphi$ over  $\Sph\setminus \alpha$.
\end{lemma}

We omit the proof since it is clear. The next two lemmas are straightforward as well.

\begin{lemma}[Additivity of the degree function]\label{additivity} Let $P, Q$ be two points on $\Sph$ and 
$\xi_1$, $\xi_2$, $\xi_3$ piecewise smooth paths that run from $P$ to $Q$. 
Choose  $Z\in \Sph\setminus (\xi_1\cup \xi_2\cup\xi_3)$ and let $\xi_{ij}$ be the oriented loop based at $P$ obtained by travelling first along $\xi_j$ and then along $\xi_i^{-1}$.
%composed of $\eta_i$ and $\eta_j$ with the opposite orientation. 
Then for any $Y\in \Sph\setminus (\xi_1\cup \xi_2\cup\xi_3)$ we have
$$\deg_{Z}(\xi_{13})(Y)=\deg_{Z}(\xi_{12})(Y)+\deg_{Z}(\xi_{23})(Y).$$
\end{lemma}

\begin{lemma}[Area of a spherical disk and algebraic area of its boundary]\label{areadegree} Let $D$ be a disk with a spherical metric and piecewise smooth boundary. Let $\iota: D\to \Sph$ be an orientation preserving developing map
and let $Z\in \Sph\setminus \iota(\partial D)$.
% and let $\deg(\iota)(Z)$ be the degree of $\iota$ over $Z$.  
Then the following equality holds
$$\Area(D)=4\pi \deg(\iota)(Z)+\Alg_Z(\iota(\partial D)).$$
\end{lemma}
\begin{proof} 
Since $\iota$ is locally isometric we clearly have $\Area(D)=\int_{D}\iota^*(\omega)$. Now the statement follows from Lemma \ref{degreefunction}, since $\int_{D}\iota^*(\omega)=\int_{\Sph}\deg(\iota)\omega$.
\end{proof}

\subsection{Algebraic area of short curves}

In this subsection we estimate the algebraic area of curves on $\Sph$ of length less than $2\pi$. Such estimate takes the form of an isoperimetric inequality.
%
%(Lemma \ref{smallarea}) and of isosceles  triangles (Lemma \ref{triangleineq}).

\begin{lemma}[Isoperimetric inequality in $\Sph$ for the algebraic area]\label{smallarea}  
Let $\xi\subset \Sph$ be an oriented piecewise smooth curve with  $\ell(\xi)<2\pi$. Let $Z$ be a point on $\Sph$ separated from $\xi$ by a geodesic circle in $\Sph$. Then $|\Alg_Z(\xi)|<\ell(\xi)^2/2\pi$. 
\end{lemma}

The above lemma is probably well-known but for the sake of completeness we will provide a proof that relies on the following
classical isoperimetric inequality (see \cite[Lemma 6.1]{hayman:meromorphic}).
%due to Paul L\'evy.

\begin{lemma}[Isoperimetric inequality for disk domains in $\Sph$]\label{isolemma} 
Let $\Omega\subset\Sph$ be a disk with piecewise smooth boundary.
%Let $\gamma$ be a (piecewise smooth)  simple closed curve %of length $L$ 
%in $\Sph$ that bounds a disk $D$.% of area $A$. 
Then $\ell(\pa\Omega)^2\ge \area(\Omega)\left(4\pi-\area(\Omega)\right)$. 
In particular, if  $\area(\Omega)<2\pi$, then $\ell(\pa\Omega)^2>2\pi \cdot\area(\Omega)$.
\end{lemma}

Before proving Lemma \ref{smallarea}, we recall that 
the concatenation at times $(t_1,t_2)$
of two loops $\xi_1,\xi_2:S^1\rar\Sph$
that satisfy $\xi_1(t_1)=\xi_2(t_2)$
is the loop $(\xi_1)\,_{t_1}\!\!\ast_{t_2}\!(\xi_2):S^1\rar\Sph$ defined as
\[
(\xi_1)\,_{t_1}\!\!\ast_{t_2}\!(\xi_2)(t)=\begin{cases}
\xi_1(t_1+2t) & \text{if $t\in [0,\frac{1}{2}]$}\\
\xi_2(t_2+2t-1) & \text{if $t\in[\frac{1}{2},1]$}
\end{cases}
\]
where we have identified $S^1$ with $\RR/\ZZ$. 
It is clearly possible to concatenate more than two loops (at different times)
iterating the above procedure.

\begin{lemma}[Loops generically immersed in $\Sph$ are concatenations of simple loops]\label{simple-decomposition}
Let $\xi:S^1\rar \Sph$ be a piecewise smooth curve with a finite number of self-intersections.
Then, up to a reparametrization, $\xi$ 
is a concatenation of finitely many simple loops $\xi_1,\dots,\xi_k$
such that $\xi_i$ and $\xi_j$ intersect at finitely many points for every $i\neq j$.
\end{lemma}
\begin{proof}
By assumption, 
the set of non-injectivity points
$N=\{t\in S^1\,|\,\exists t'\neq t\ \text{such that $\xi(t')=\xi(t)$}\}$
is finite. We proceed by induction on $|N|$.

For $|N|=0$, the loop $\xi$ is already simple.
Assume then $N\neq\emptyset$
and consider $u=\min\{d_{S^1}(t',t)\,|\,\xi(t')=\xi(t)\ \text{with $t'\neq t$}\}>0$.
Let $(t_1,t_1+u)$ be a couple such that $\xi(t_1)=\xi(t_1+u)$.
%and suppose that $u$ is attained at the couple $(t_1,t_1+u)$.
Then $\xi$ is the reparametrized concatenation $(\xi_1)\ast(\xi'_1)$,
where $\xi_1$ is a reparametrization of $\xi|_{[t_1,t_1+u]}$ and
$\xi'_1$ is a reparametrization of the remaining portion of $\xi$.
By the minimality of $u$, the loop $\xi_1$ is simple.
Moreover, the set $N'$ of non-injectivity points of $\xi'_1$ 
satisfies $|N'|<|N|$ and so the curve $\xi'_1$ is a reparametrized concatenation of finitely many
simple loops $\xi_2,\dots,\xi_k$ by induction.
It follows that $\xi$ is a reparametrized concatenation of $\xi_1,\dots,\xi_k$.
Finally, for $i\neq j$ the simple loops $\xi_i$ and $\xi_j$ can only intersect
at $\xi(N)$, which is a finite subset of $\Sph$.
\end{proof}

\begin{proof}[Proof of Lemma \ref{smallarea}]
Note first that, since $\ell(\xi)<2\pi$, a geodesic circle $E$ 
on $\Sph$ that does not intersect $\xi$ indeed exists, and so we can choose $Z$
in the component of $\Sph\setminus E$ that does not contain $\xi$.

After perturbing $\xi$, if necessary, we may assume that $\xi$ has finite number of self-intersections in $\Sph$. Denote by $\Omega_0$ the connected component of $\Sph\setminus\xi$ that contains $Z$.

Let us now give a different presentation of the function $\deg_{Z}(\xi)$.
Thanks to Lemma \ref{simple-decomposition}, the curve $\xi$ can be 
%realized as a concatenation of a finite number
%of simple oriented loops $\gamma_1,\dots,\gamma_k$ in such a way that
%$\gamma_i$ has length $L_i$ and encloses the domain $\Omega_i$,
%of lengths $L_i$
decomposed into a finite number of simple loops $\xi_1,\ldots, \xi_k$. Each loop $\xi_i$ encloses a disk $\Omega_i$, whose interior is disjoint from $\Omega_0$. 
We underline that two disks $\Omega_i$ and $\Omega_j$ may well overlap if $i,j\neq 0$.

\begin{center}
\begin{figurehere}
\psfrag{S2}{$\Sph$}
\psfrag{O1}{${\Omega_1}$}
\psfrag{O2}{${\Omega_2}$}
\psfrag{O3}{${\Omega_3}$}
\psfrag{O0}{${\Omega_0}$}
\psfrag{y1}{$\textcolor{BrickRed}{Y_1}$}
\psfrag{y2}{$\textcolor{BrickRed}{Y_2}$}
\psfrag{g1}{$\textcolor{BrickRed}{\xi_1}$}
\psfrag{g2}{$\textcolor{BrickRed}{\xi_2}$}
\psfrag{g3}{$\textcolor{BrickRed}{\xi_3}$}
\psfrag{z0}{${Z}$}
\includegraphics[width=0.3\textwidth]{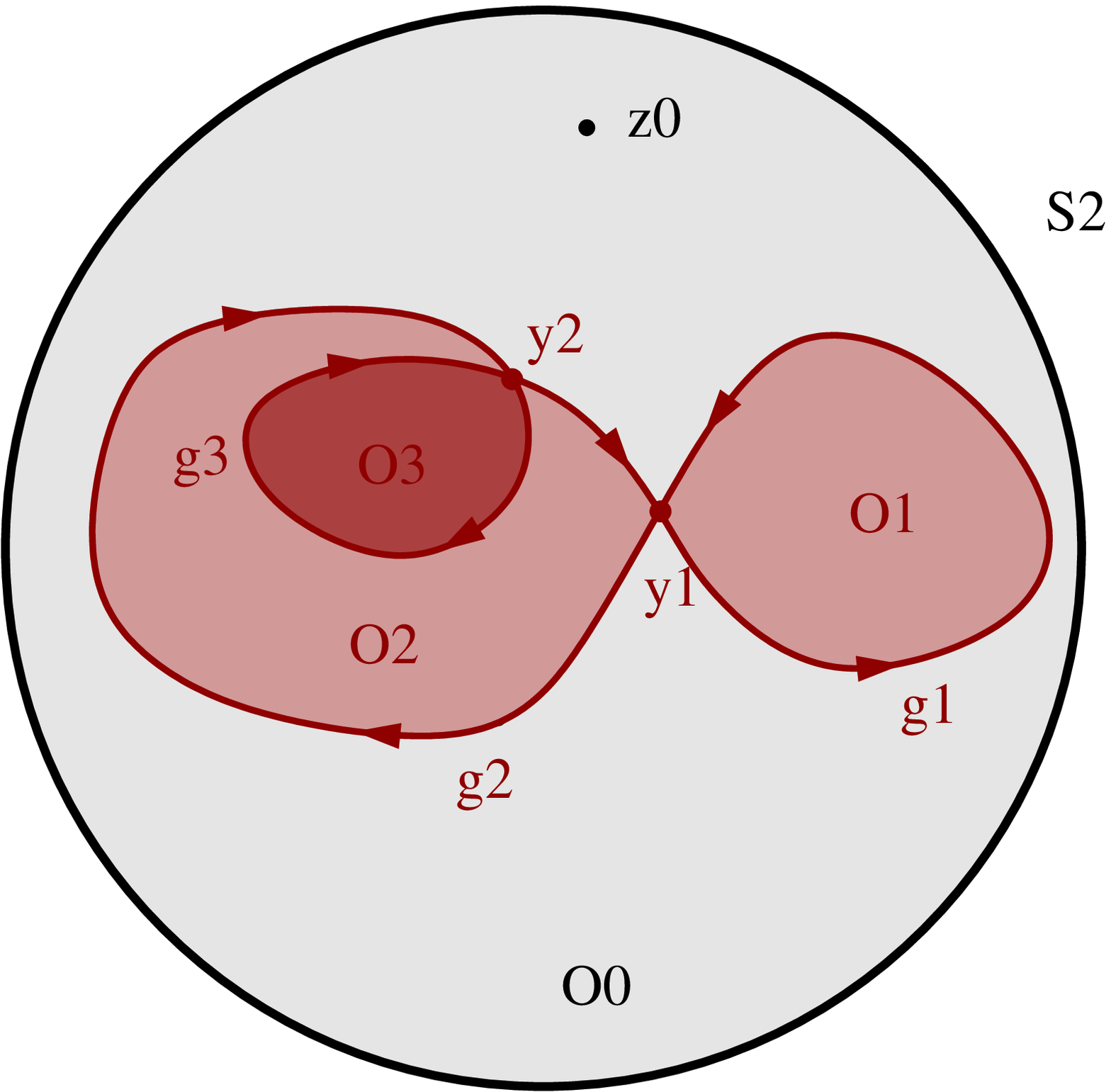}
\caption{{\small An example of decomposition of $\xi$: the loops $\xi_1,\xi_2$ are based at $Y_1$, $\xi_3$ is based at $Y_2$ and $\Omega_3$ is contained in $\Omega_2$.}}\label{fig:algarea1}
\end{figurehere}
\end{center}

It follows from Lemma \ref{additivity} that 
$$\deg_{Z}(\xi)=\sum_i \deg_{Z}(\xi_i).$$
Since the curves $\xi_i$ are separated from $Z$ by the geodesic circle $E$, we deduce
$$\Alg_Z(\xi)=\int_{\Sph}\deg_{Z}(\xi)\omega=\int_{\Sph}\sum_i \deg_{Z}(\xi_i)\omega=
\sum_i\pm \Area(\Omega_i).$$
Using Lemma \ref{isolemma} we conclude that
\[
|\Alg_Z(\xi)|\le\sum_i \Area(\Omega_i)< \sum_i \ell(\xi_i)^2/2\pi\le \ell(\xi)^2/2\pi
\]
because  $\sum_i \ell(\xi_i)=\ell(\xi)$.
\end{proof}

As a consequence of Lemma \ref{smallarea}, we obtain
the following isoperimetric inequality for the area of an abstract
spherical disk $\BU^0$ with no conical points, which will be
needed in the proof of the systole inequality.

%\begin{corollary}[Isoperimetric inequality for spherical disks with short boundary]\label{shorthindisk}
%Let $D$ be a disk with spherical metric such that $\partial D$ is piecewise smooth, $\ell(\partial D)<2\pi$ and for any point $p\in D$ we have $d(p,\partial D)\le\frac{\pi}{2}$. Then $\area(D)<\ell(\partial D)^2/2\pi$.
%\end{corollary}
%\begin{proof}
%Consider a developing map $\iota: D\to \Sph$. Since $\ell(\partial D)<2\pi$, there is a geodesic circle $E\subset \Sph$ disjoint from $\iota(\partial D)$. Let $O$ be the center of the 
%component of $\Sph\setminus E$
%which does not contain $\iota(\partial D)$. Since $d(O,\iota(\partial D))>\frac{\pi}{2}$, it follows from assumptions that $O$ does not belong to $\iota(D)$. Hence $\Alg_O(\iota(\partial D))=\area(D)$. Now the statement follows from Lemma \ref{smallarea}.
%\end{proof}

\begin{corollary}[Isoperimetric inequality for spherical disks with short boundary]\label{shorthindisk}
Let $\BU^0$ be a disk with spherical metric 
and no conical points
such that $\partial \BU^0$ is piecewise smooth and $\ell(\partial \BU^0)<2\pi$. Then 
\[
\frac{1}{2\pi}\Big|\Area(\BU^0)-4\pi b^0\Big|<\left(\frac{\ell(\partial \BU^0)}{2\pi}\right)^2=\lambda_0(\BU^0)
\]
for some $b^0\in \mathbb Z_{\ge 0}$.
\end{corollary}
\begin{proof}
Consider a developing map $\iota: \BU^0\to \Sph$. Since $\ell(\partial \BU^0)<2\pi$, there is a geodesic circle $E\subset \Sph$ disjoint from $\iota(\partial \BU^0)$. Let $O$ be the center of the 
component of $\Sph\setminus E$
which does not contain $\iota(\partial \BU^0)$. By Lemma \ref{areadegree} we have
$\Area(\BU^0)=4\pi \deg(\iota)(O)+\Alg_O(\iota(\partial \BU^0))$.
%Hence $\Alg_O(\iota(\partial D))=\area(D)$. 
Now the statement follows from Lemma \ref{smallarea},
because $\iota$ is orientation-preserving and so $b^0:=\deg(\iota)(O)\geq 0$.
\end{proof}

\subsection{Isosceles spherical triangles}\label{sec:isosceles}

In this section we estimate the area of certain isosceles triangles
embedded in $\Sph$ with one short side.

The following lemma will be needed in the proof of Theorem \ref{diskstructure}.

\begin{lemma}[Area estimate for isosceles triangles]\label{triangleineq} 
Let $T\subset\Sph$ be a solid triangle with geodesic sides
and endow its boundary $\pa T$ with the induced orientation
and then its vertices $P,O,Q$
with the cyclic orientation $P\prec O\prec Q$.
Suppose that $|OP|=|OQ|=r$ and that $|PQ|<r\cdot \lambda_1$
for some $\lambda_1\in (0,1)$.
Then 
\[
|\Alg_Z(\pa T)-4\pi\theta|=|\Area(T)-4\pi\theta|<\pi\lambda_1
\]
for any point $Z\in \Sph\setminus T$, where $2\pi\theta=\wh{O}$ is the internal angle of $T$ at $O$.
%Let $POQ$ be a solid triangle in $\Sph$
%with internal
%angle $2\pi \theta<2\pi$ at $O$, and such that $|OP|=|OQ|=r$ and $|PQ|<r\cdot \lambda$. Suppose that points $P$, $O$, $Q$ go in the anti-clockwise order on the boundary of $POQ$. Let $Z$ be a point in $\mathbb S^2\setminus POQ$, then 
%$$|\Alg(POQ,Z)-4\pi\theta|=|\Area(POQ)-4\pi\theta|<\pi\lambda.$$
\end{lemma}

\begin{proof} 
The  equality holds since $\Alg_Z(\pa T)=\Area(T)$ thanks to the choice of the cyclic orientation of the vertices of $T$ and because $Z\notin T$.  So we only need to prove the inequality on the right.

Suppose first $r\in\left[\frac{\pi}{2},\pi\right)$. Let $O'$ be the point on $\Sph$ opposite to $O$ and let $T'\subset\Sph\setminus\inte{T}$ be the solid triangle with
edges consisting of the shortest geodesics $O'P$ and $O'Q$ and
of the edge $PQ$ of $T$. Since $T\cup T'$ is a bigon with angles $2\pi\theta$,
\[
\Area(T)+\Area(T')=4\pi\theta.
\]
Moreover, $|O'P|=|O'Q|\le \frac{\pi}{2}$ and so the triangle $T'$ can be embedded inside an isosceles triangle with sides of lengths $\frac{\pi}{2},\,\,\frac{\pi}{2}\, |PQ|$. Hence the area of $T'$ is bounded by the length of the side $PQ$.
As a consequence, $|\Area(T)-4\pi\theta|\le |PQ|<r\cdot\lambda_1\le \pi \lambda_1$ because $|PQ|< r\cdot \lambda_1$ by assumption.

%Since $|O'P|=|O'Q|\le \frac{\pi}{2}$, the area of triangle $T'$ is bounded by the length of the side $PQ$. By assumptions $|PQ|< r\cdot \lambda$. If $\theta\le \frac{1}{2}$ we have 
%$$\Area(T')+\Area(T)=4\pi\theta,$$ 
%if $\theta\ge \frac{1}{2}$ we have 
%$$\Area(T')-\Area(T)=4\pi\theta.$$ 
%Either way we conclude $|\Area(T)-4\pi\theta|\le |PQ|<r\cdot\lambda\le \pi \lambda$.

Suppose now $r\in\left(0,\frac{\pi}{2}\right)$. Assume first $\theta\le \frac{1}{2}$.
%
%  and let $M$ be the midpoint of $PQ$, so that $|PQ|< 
%2\lambda r$ with $\lambda=\frac{\lambda}{2}$.
Applying Lemma \ref{rightriang} to triangle $T$,
we get $\theta<\frac{\lambda_1}{4}$. Since $\Area(T)<4\pi\theta$, we conclude that
\[
|\Area(T)-4\pi\theta|<4\pi\theta< \pi\lambda_1.
\]
Finally, consider the case $\theta\geq\frac{1}{2}$ and let
$T'=\Sph\setminus \inte{T}$ be the complementary of $T$ inside $\Sph$, so that $\Area(T)+\Area(T')=4\pi$.
Since $T'$ has internal angle $2\pi(1-\theta)$ at $O$,
we have just seen that $(1-\theta)<\frac{\lambda_1}{4}$
and so $4\pi(1-\theta)<\pi\lambda_1$.
From $\Area(T')<4\pi(1-\theta)$ we deduce that
$\Area(T)\in\left(4\pi\theta,4\pi\right)$
and so we conclude that $|\Area(T)-4\pi\theta|<\pi\lambda_1$.
\end{proof}

The above proof relies on the following computation
in the case $r\in\left(0,\frac{\pi}{2}\right)$ and $\theta\leq\frac{1}{2}$.

\begin{lemma}\label{rightriang}
Fix $\theta\in\left(0,\frac{1}{2}\right)$,
$r\in\left(0,\frac{\pi}{2}\right)$, and $\lambda_1\in (0,1)$.
%Chose positive $\theta<\frac{1}{2}$, $r<\frac{\pi}{2}$, and $\lambda<1$. 
Let $POQ$ be a convex spherical triangle with angle $2\pi\theta$ at $O$, and with $|OP|=|OQ|=r$ and $|QP|=\lambda_1 r$.
Then $\theta<\lambda_1/4$.
\end{lemma}
\begin{proof} 
Let $R$ be the midpoint of $QP$ and consider the triangle $POR$,
which has angle $\pi\theta$ at $O$ and $\frac{\pi}{2}$ at $R$
and such that $|RP|=r\lambda_1 /2$.

Applying the sine rule  to the sides  $PO$ and $OR$ of triangle $POR$, we get
$$\sin(\pi \theta)=\frac{\sin(r\lambda_1 /2)}{\sin (r)}=\varphi(r).$$
Note that for $r=\frac{\pi}{2}$ we get $\theta=\lambda_1/4$. Hence, to prove the lemma, it is enough to show that $\varphi$ is strictly increasing on  $(0,\frac{\pi}{2})$, namely that for $r\in (0,\frac{\pi}{2})$ we have $\varphi'(r)>0$. The latter is equivalent to proving 
$$(\lambda_1/2) \cos(r\lambda_1 /2)\sin(r)-\cos(r)\sin(r\lambda_1 /2)>0.$$
This inequality holds since $t\tan(r)>\tan(t\eta)$ for all $t\in(0,1)$ and $r\in (0, \pi/2)$. 
\end{proof}

\subsection{Area estimate of a disk with one conical point and short boundary} 

We can finally assemble all the ingredients to prove our wished
estimate for the area a disk with one conical point and short boundary.

\begin{proof}[Proof of Theorem \ref{diskstructure}] 
Let us write   $\th=\floor{\th}+\theta$, where $\theta=\{\th\}\in [0,1)$. 
Let $\alpha\subset \BU^1$ be a geodesic of length $r=d(x,\partial \BU^1)$ that joins $x$ with a point $y$ on $\partial \BU^1$. 
Denote by $D_{\alpha}$ the disk obtained from $\BU^1$ by cutting it along $\alpha$ and completing. 
Denote by $x_O$ the point in $D_{\alpha}$ corresponding to $x$ and by $y_P$ and $y_Q$ two points corresponding to $y$. The points $x_O$, $y_P$, $y_Q$ cut $\partial D_{\alpha}$ into two geodesic segments $y_Px_O$, $x_Oy_Q$ and a 
non-geodesic path $y_Qy_P$.
We choose the orientation on $D_{\alpha}$ so that points $y_P\prec x_O\prec y_Q$ 
according to the induced orientation on $\pa D_{\alpha}$.
%in the anti-clockwise order on $\partial D_{\xi}$.

Let now $\iota: D_{\alpha}\to \Sph$ be an orientation preserving developing map.
Call $O=\iota(x_O)$, $P=\iota(y_P)$ and $Q=\iota(y_Q)$
and denote by $PO$ the geodesic arc $\iota(y_Px_O)$ and by $OQ$ the geodesic arc $\iota(x_Oy_Q)$. 
To define $PQ$ note that $d(P,Q)\le \ell(\partial \BU^1)=\lambda_1(\BU^1)r<r<\pi$ and so we can choose $PQ$ to be the unique geodesic segment between $P$ and $Q$
of length less than $\pi$. 
Finally, let $T$ be a solid triangle in $\Sph$ 
with vertices $P,O,Q$ and sides $PO$, $OQ$, and $PQ$, 
and with angle $2\pi\theta$ at $O$ (such a $T$ is unique, unless $\theta=\frac{1}{2}$). 

Pick a point $Z$ in $\Sph\setminus T$ at distance $\lambda_1(\BU^1)$ from $O$. 
Since the conical angle at $x$ is equal to $2\pi\left(\floor{\th}+\theta\right)$, 
it is easy to see now that $\deg_Z(\iota)=\floor{\th}+b^1$ for some $b^1\geq 0$. 
We deduce from Lemma \ref{areadegree}, that 
$$\Area(\BU^1)=\Area(D_{\alpha})=\Alg_Z(\iota(\partial D_{\alpha}))+4\pi(\floor{\th}+b^1).$$

Hence, we need to show $|\Alg_Z(\iota(\partial D_{\alpha}))-4\pi\theta|\le 2\pi\lambda_1(\BU^1)$.

\begin{center}
\begin{figurehere}
\psfrag{S2}{$\Sph$}
\psfrag{O}{${O}$}
\psfrag{P}{${P}$}
\psfrag{Q}{${Q}$}
\psfrag{ipD}{${\iota(\pa D_\alpha)}$}
\psfrag{z0}{${Z}$}
\psfrag{th}{$\! 2\pi\theta$}
\includegraphics[width=0.85\textwidth]{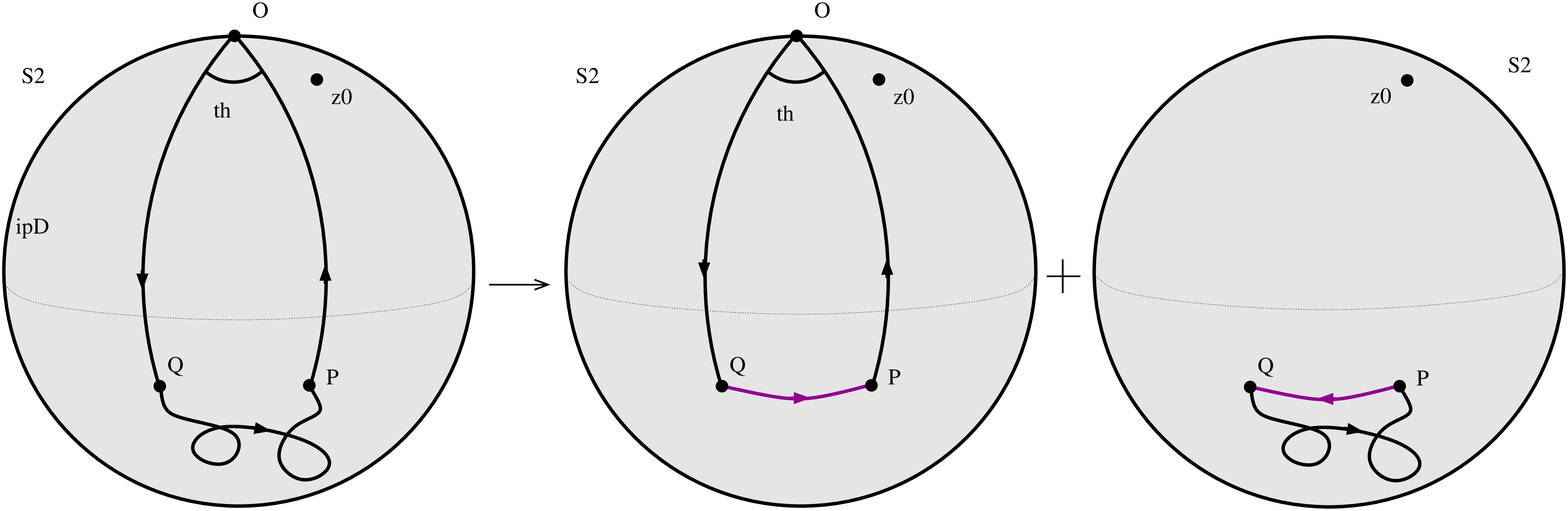}
\caption{{\small Splitting $\iota(\pa D_\alpha)$ into $\pa T$ and a curve $\xi$.}}\label{fig:algarea2}
\end{figurehere}
\end{center}

In order to prove this inequality let us present $\iota(\partial D_{\alpha})$ as a sum of two curves. The first is the boundary of triangle $T$ and the second is a curve $\xi$, composed of segment $PQ$ and the curvy part of $\iota(\partial D_{\alpha})$ going from $Q$ to $P$. Clearly $\ell(\xi)\le 2 \ell(\pa \BU^1)$ and so, by our hypotheses,
$\ell(\xi)\le 2r\lambda_1(\BU^1)$. Applying the additivity property of Lemma \ref{additivity} and the 
inequalities of Lemma \ref{smallarea} and Lemma \ref{triangleineq} we obtain
\begin{align*}
|\Alg_Z(\iota(\partial D_{\alpha}))-4\pi\theta| &=|\Alg_Z(T)-4\pi\theta+\Alg_Z(\xi)|\le \\
& \leq |\Alg_Z(T)-4\pi\theta|+|\Alg_Z(\xi)|\leq
\pi \lambda_1(\BU^1)+\left(2r\lambda_1(\BU^1)\right)^2/2\pi<2\pi\lambda_1(\BU^1)
\end{align*}
because $\lambda_1(\BU^1)<\frac{1}{2}$ implies $2r\lambda_1(\BU^1)<\pi$.
\end{proof}

\section{Almost bubbling surfaces and systole inequality}\label{sec:systole}

In this section we will prove Theorem \ref{main:systole} which
relates the extremal systole, 
that measures how much the surface
is conformally far from degenerating, and the spherical systole,
that measures how much the surface is metrically far from
degenerating.
As mentioned in the introduction, if the extremal systole is small, then
the spherical systole is small. In the following theorem we show
that the converse also holds, whenever $\Ab_{\bm{\th}}>0$.

\begin{mainthm}{\ref{main:systole}}[Systole inequality] %\label{main:systole}
Let $S$ be a surface with spherical metric and conical singularities 
at $\bm{x}$ of angles $2\pi\bm{\th}$. Assume that $\chi(\dot{S})<0$
and $\dot{S}$ is not a $3$-punctured sphere.\\
Suppose that there exists $\varepsilon\in\left(0,\frac{1}{2}\right)$
such that
\begin{itemize}
\item[(i)]
$\Ext\sys(\dot{S})\geq \frac{2\pi\|\bm{\th}\|_1}{\log(1/\e)}$.
%
%The surface $\dot S$ does not have essential cylinders  of modulus exceeding $(2\pi|\theta|)^{-1}\log(\varepsilon^{-1})$.
\item[(ii)]
$\Ab_{\bm{\th}}(S)\geq\varepsilon$.
\end{itemize}
Then the following inequality holds:
\begin{equation}\label{sysineq}\tag{$\clubsuit$}
\sys(S,\bm{x})\ge \left(\frac{\varepsilon}{4\pi\|\bm{\th}\|_1}\right)^{-3\chi(\dot S)+1}.
\end{equation}
\end{mainthm}

In Appendix \ref{app:example} we will show through a sequence of examples
that the estimate for the extremal systole determined by Theorem \ref{main:systole}
is reasonably sharp.

\subsection{Non-bubbling parameter}\label{sec:non-bubbling}

We recall from Section \ref{sec:bubbling} that the value
$\frac{1}{2\pi}\Area(S)\in\RR_{\geq 0}$ of an $\e$-bubbling spherical surface $S$
with conical points of angles $2\pi\bm{\th}$
sits at distance at most $\e$ from the subset
\[
\Acrit_{\bm{\th}}(S)=\left\{2b+2\|\bm{\th}_I\|\ \big|\ I\subseteq\{1,\dots,n\},\ b\in\ZZ_{\geq 0} \right\}.
\]
Thus, an obstruction $\e$-bubbling is exactly given
by $\Ab_{\bm{\th}}(S)\geq \e$.
Here we collect a few basic properties of such quantity.

\begin{lemma}[Elementary properties of the non-bubbling parameter]\label{lemma:upper-bound-Ab}
The subsets $\Crit_{\bm{\th}}(S)$ and $\Acrit_{\bm{\th}}(S)$ satisfy
$\Acrit_{\bm{\th}}(S)=\Crit_{\bm{\th}}(S)+\|\bm{\th}\|_1$ and so
$\Ab_{\bm{\th}}(S)=d_\RR(\chi(S,\bm{\th}),\, \Acrit_{\bm{\th}}(S))$.
Moreover, the following holds.
\begin{itemize}
\item[(a)]
If $\chi(S,\bm{\th})\geq 0$, then $\Ab_{\bm{\th}}(S)\leq 1$. In particular,
this holds whenever there
exists a spherical metric on $S$ with conical singularities at $\bm{x}$
of angles $2\pi\bm{\th}$.
\item[(b)]
If $\chi(S,\bm{\th})\leq 0$, then $\Ab_{\bm{\th}}(S)=-\chi(S,\bm{\th})$
since $0\in\Acrit_{\bm{\th}}(S)$.
\end{itemize}
\end{lemma}
\begin{proof}
The first claim is obvious.
As for (a), note that $\Acrit_{\bm{\th}}\supseteq 2\ZZ_{\geq 0}$ and so
\[
\Ab_{\bm{\th}}(S)=d_\RR(\chi(S,\bm{\th}),\,\Acrit_{\bm{\th}}(S))
\leq d_{\RR}(\chi(S,\bm{\th}),\,2\ZZ_{\geq 0})\leq 1
\]
Property (b) is immediate, since $\Acrit_{\bm{\th}}$ contains only non-negative values
and $0\in\Acrit_{\bm{\th}}$.
\end{proof}

\begin{remark}[Upper bound for $\Vor$ in terms of $\Ab_{\bm{\th}}(S)$]\label{maxofV_S}
If $S$ is endowed with a spherical metric with singularities of angles $2\pi\bm{\th}$, then
\[
\frac{1}{2\pi}\Area(S)=\chi(S,\bm{\th})=d_\RR\left(\chi(S,\bm{\th}),\,0\right)
\geq \Ab_{\bm{\th}}(S).
\]
As a consequence,
\[
\max(\Vor)\ge \sqrt{2\Ab_{\bm{\th}}(S)\|\bm{\th}\|_1^{-1}}
\]
by applying Lemma \ref{lengthoflevel} and the previous inequality.
\end{remark}

As mentioned in the introduction, another important property of the non-bubbling parameter is that
its positivity prevents spherical metrics from having coaxial monodromy.

\begin{lemma}[Coaxial metrics have vanishing non-bubbling parameter]\label{lemma:coaxial-zero-NB}
Let $(S,\bm{x},h)$ be a spherical surface with conical points of angles
$2\pi\cdot\bm{\th}$. If $(S,\bm{x})$ has coaxial monodromy,
then $\Ab_{\bm{\th}}(S,\bm{x})=0$.
\end{lemma}
\begin{proof}
Fix a universal cover $\tilde{\dot{S}}$ of $\dot{S}$,
a developing map $\iota:\tilde{\dot{S}}\rar\Sph$ to the standard sphere $(\Sph,h_{\Sph})$
and a compatible monodromy representation $\rho:\pi_1(\dot{S})\rar\SO_3(\RR)$.
Since $\rho$ is coaxial by hypothesis, there exist antipodal points $O,O'$ in $\Sph$
that are fixed by the image of $\rho$. Let $G\subset\PSL_2(\CC)$
be the real 1-parameter subgroup of M\"obius transformations of $\Sph\cong\CC\PP^1$
that fix $O,O'$ and send every meridian between the two poles $O,O'$ to itself.
Clearly, every element $g\in G$ commutes with the image of $\rho$
and so the metric $\iota^*g^*h_{\Sph}$ descends to a spherical metric $h_g$ on $(S,\bm{x})$,
which is conformally equivalent to $h$.

Let $\sigma\subset S$ be a geodesic arc that realizes the systole
and let $\tilde{\sigma}\subset\tilde{\dot{S}}$ be a lift of its.
Since $\sigma$ has
length at most $\frac{\pi}{2}$ by Lemma \ref{lemma:shorterthanpi/2}, 
the closure of $\iota(\tilde{\sigma})$ inside $\Sph$
cannot contain both $O$ and $O'$.

Hence, for every $s>0$ there exists $g\in G$ such that 
the $h_g$-length of $\sigma$ is at most $s$ and so
$(S,\bm{x},h_g)$ has systole at most $s$.
Since $s>0$ is arbitrary, it follows from Theorem \ref{main:systole} that
$\Ab_{\bm{\th}}(S,\bm{x})=0$.
\end{proof}

We underline that the above Lemma \ref{lemma:coaxial-zero-NB} can be proven
in a different way, and in particular without invoking Theorem \ref{main:systole}.
Indeed, one could show that, by using suitable elements $g\in G$, one
can produce metrics $h_g$ whose mass concentrates near certain points
$x_i$ with $i\in I$ and near other $b\geq 0$ smooth points,
so that the value of the area of $h_g$ (and so of $h$) can be estimated to be
arbitrarily close to $2\pi\cdot (2b+2\|\bm{\th}_I\|_1)$.

\subsection{Almost bubbling surfaces}\label{sec:almost-bubbling}

In this section we give a precise definition of $\e$-bubbling spherical surface mentioned in the introduction
and we show that such surfaces have small non-bubbling parameter.
The definition relies on the $\lambda$-invariant of a spherical disk with at most one conical point, introduced
in Section \ref{sec:disks-one}.

\gm{Definition \ref{def:bubbling} minimally touched.}
%%%%%%%%%%%%%%%%%%%%%%%%%%%%%%%%%
\begin{definition}[$\e$-bubbling surfaces]\label{def:bubbling}
Let $\e\in(0,\frac{1}{2})$. An {\it{$\e$-bubbling decomposition}} of a spherical surface $S$ with conical points $\bm{x}$ is the datum of a subset
$I\subseteq\{1,2,\dots,n\}$ and a partition of $S$ as a union of 
\begin{itemize}
\item
finitely many disks $\BU^0_j$ without conical points,
\item
a disk $\BU^1_i$ with one conical point $x_i$ for each $i\in I$, and
\item
a subsurface $S^c$ (called {\it{core}}),
\end{itemize}
%a {\it{core}} $S^c$, of disks 
%$\BU^0_j$ without conical points
%and disk $\BU^1_i$ with one conical point 
such that
\[
\frac{1}{2\pi}\Area(S^c)+\sum_{j}\lambda_0(\BU^0_j)+\sum_{i\in I}
\lambda_1(\BU^1_i)<\e.
\]
We say that $(S,\bm{x})$ is {\it{$\e$-bubbling}} if it admits an $\e$-bubbling decomposition
and that it is {\it{almost bubbling}} if it is $\e$-bubbling for some $\e$. %$0<\e<\frac{1}{2}$.
\end{definition}

The following statement is a simple application of Theorem \ref{diskstructure} and Corollary \ref{shorthindisk}.

\begin{theorem}[Almost bubbling implies small non-bubbling parameter]\label{thm:almost-bubbling}
Let $S$ be a spherical surface with conical points $\bm{x}$ of angles $2\pi\bm{\th}$.
If $S$ is $\e$-bubbling, then $\Ab_{\bm{\th}}(S,\bm{x})<\e$.
\end{theorem}
\begin{proof}
Consider an $\e$-bubbling decomposition of $S$.
% with core $S^c$, unmarked disks $\BU^0_j$
%and disks $\BU^1_i$ marked by $x_i$ for $i\in I=\{i\in \{1,\dots,n\}\ |\ x_i\notin S^c\}$.
An example for $(g,n)=(1,7)$ with $I=\{2,4\}$ and one $\BU^0_1$
%$k=1$ 
%and $r=0$ 
is illustrated in Figure \ref{fig:nearly-complete-bubbling}.

Applying Corollary \ref{shorthindisk} to  each $\BU^0_j$
and Theorem \ref{diskstructure} to each $\BU^1_i$,
we get 
\begin{equation}\label{eq:bubbles}
\frac{1}{2\pi}\left|\Area(\BU^0_j)-4\pi b^0_j\right|<\lambda_0(\BU^0_j),
\quad
\frac{1}{2\pi}\left|\Area(\BU^1_i)-4\pi(b^1_i+\th_i)\right|<\lambda_1(\BU^1_i)
\end{equation}
for suitable integers $b^0_j,\,b^1_i\geq 0$.
%(In the situation illustrated in Figure \ref{fig:nearly-complete-bubbling}
%we have $b_1=2$ and $b'_2=1$, $b'_4=0$.) 
As a consequence, setting $b=\sum_j b^0_j+\sum_{i\in I} b^1_i$,
we obtain
\begin{align*}
\Ab_{\bm{\th}}(S,\bm{x}) &\leq 
\left |\frac{1}{2\pi}\Area(S)-(2\|\bm{\th}_I\|_1+2b)\right|
\leq \\
&\leq
\frac{1}{2\pi}\Area(S^c)+
\sum_j\left|\frac{1}{2\pi}\Area(\BU^0_j)-2b^0_j\right|+
\sum_{i\in I}\left|\frac{1}{2\pi}\Area(\BU^1_i)-(2\th_i+2b^1_i)\right|<
\\
&<\frac{1}{2\pi}\Area(S^c)+\sum_j\lambda_0(\BU^0_j)+\sum_{i\in I}\lambda_1(\BU^1_i)<\e
\end{align*}
where the first inequality follows from the definition of non-bubbling parameter,
the second inequality is a simple rearrangement of the summands,
the third inequality is a consequence of Inequalities \eqref{eq:bubbles} and
the last inequality is the definition of $\e$-bubbling.
\end{proof}

\subsection{Voronoi core subsurfaces}

In this section we introduce Voronoi core subsurfaces of a given spherical surface
$S$ with conical points $\bm{x}$ such that $\chi(\dot{S})<0$.
%Such subsurfaces substantially capture all the topology of $\dot{S}$, and
%they are the union of a component of a sublevel of $\Vor$
%and of some small Voronoi caps contained in $\Vor^{-1}(0,\frac{\pi}{2})$. 
In Theorem \ref{Abound} we will use Voronoi cores as
core subsurfaces of some $\e$-bubbling decomposition
in the sense of Definition \ref{def:bubbling}: for this reason, we will again use
the symbol $S^c$ to denote a Voronoi core.

\begin{definition}[Voronoi core subsurfaces]\label{def:Voronoi-core}
Let $S$ be a spherical surface with $\chi(\dot{S})<0$ and let $r\in \left(0,\frac{\pi}{2}\right)$ be a regular value for $\Vor$.
A connected component $S^c$ of $\Vor^{-1}([0,r])$
is a {\it{Voronoi $r$-core}} of $S$
if
the complement $S\setminus S^c$ is the disjoint union
of disks $\BU^0_j$ without conical points
and disks $\BU^1_i$ with at most one conical point each.
We will say that $S^c$ is a {\it{Voronoi core}} if it is the Voronoi $r$-core for some $r$.
\end{definition}

%It follows from the above definition that
%$r>\sys(S,\bm{x})$ and that $\chi(\dot{S}^c)<0$.
%Moreover, 
%by Theorem \ref{locmaxint}(d) (and Definition \ref{defVoronoi})
%each cap $\CAP_l$ contains exactly one critical point for $\Vor$,
%which is an isolated local maximum.

%\begin{center}
%\begin{figurehere}
%\psfrag{S}{$\textcolor{Blue}{S}$}
%\psfrag{Scc}{$\textcolor{Blue}{S^c}$}
%\psfrag{xi}{$\textcolor{Sepia}{x_i}$}
%\psfrag{xj}{$\textcolor{Sepia}{x_j}$}
%\psfrag{xk}{$\textcolor{Sepia}{x_k}$}
%\psfrag{xh}{$\textcolor{Sepia}{x_h}$}
%\psfrag{G(S)}{$\textcolor{OliveGreen}{\Gamma(S)}$}
%\psfrag{C}{$\textcolor{Plum}{C}$}
%\psfrag{C'}{$\textcolor{Plum}{C'}$}
%\psfrag{D}{$\textcolor{Plum}{\CAP}$}
%\includegraphics[width=0.4\textwidth]{core.eps}
%\caption{{\small Essential $C$ and non-essential $C'$ Voronoi cylinders and a cap $\CAP$.}}\label{fig:core}
%\end{figurehere}
%\end{center}

The existence of a Voronoi $r$-core can be verified
by studying the nature of the curves in $\Vor^{-1}(r)$ as follows.

\begin{lemma}\label{lemma:core-systole}
There are no Voronoi $r$-cores in $S$ with $r\leq \sys(S,\bm{x})$.
Moreover, a Voronoi core
contains every systole geodesic of $S$.
\end{lemma}
\begin{proof}
The first claim is immediate, since
for $r<\sys(S,\bm{x})$ all connected components of $\Vor^{-1}([0,r])$ are standard disks with one conical point. 

Concerning the second claim, suppose by contradiction that
$\sigma_{\sys}$ is a systole geodesic 
contained in a connected component $S'$ of $S\setminus S^c$.
Then $S'$ has one conical point $x_i$ and $\sigma_{\sys}$ is a loop based at $x_i$.
Moreover, $S'\setminus \sigma_{\sys}$ should have a component which is a disk $D$ whose geodesic boundary $\sigma_{\sys}$ has at most one non-smooth point, namely $x_i$.
Consider the developing map $\iota: D\to \Sph$. Then $\iota(\sigma_{\sys})$ should be a geodesic loop in $\Sph$ based at $\iota(x_i)$. At the same time, the length of $\sigma_{\sys}$ is at most $\pi$ by Lemma \ref{lemma:shorterthanpi/2}. 
This is a contradiction, since closed geodesics on $\Sph$ have length $2\pi$.
\end{proof}

\begin{lemma}[Characterization of Voronoi cores]\label{lemma:subcore}
Let $r<\frac{\pi}{2}$ be a regular value for $\Vor$.
Then there exists a component $S^c$
of
$\Vor^{-1}([0,r])$ which is the Voronoi $r$-core
if and only if the following conditions hold
\begin{itemize}
\item[(i)]
$r>\sys(S,\bm{x})$
\item[(ii)]
every component of $\Vor^{-1}(r)$ is a non-essential simple closed curve.
\end{itemize}
Moreover, the $r$-Voronoi core is unique, whenever it exists.
\end{lemma}

\begin{proof}
Given a Voronoi $r$-core $S^c$,
property (ii) is satisfied by definition; moreover,
the uniqueness of $S^c$ and (i) follow from
Lemma \ref{lemma:core-systole}.

Conversely, suppose that (i) and (ii) are satisfied.
Since all connected components of $\Vor^{-1}(r)$ are non-essential in $\dot{S}$, each one bounds a unique open disk in $S$ with at most one conical point. Any two such disks are either disjoint or one is completely contained inside the other. For this reason, the subsurface $S^c$ of $S$ consisting of points that do not lie inside
any of these open disks is connected.
By construction,
each connected component of $S\setminus S^c$
is a disk with at most one conical point.
As in the proof of Lemma \ref{lemma:core-systole},
none of such disks can contain a systole geodesic.
It follows that $S^c$ contains all systole geodesics
and, in particular, at least one conical point.
We conclude that $S^c$ is a connected component of
$\Vor^{-1}([0,r])$ and so it is a Voronoi $r$-core.
%It follows that $\chi(\dot{S}^c)<0$ and that 
%each boundary component of $S^c$ is a non-essential curve. 
\end{proof}

%\begin{proof}
%Since $\Vor^{-1}(r)$ is a disjoint union of circles,
%the Euler characteristic of $\{\Vor\neq r\}$
%agrees with $\chi(\dot{S})<0$. Thus, there exists a connected component $S^{cc}$ of $\{\Vor\neq r\}$ with $\chi(\dot{S}^{cc})<0$.
%
%Since $r>\sys(S,\bm{x})$, a geodesic arc or loop that realizes the systole 
%is completely contained either in $S^{cc}$ or in $S\setminus S^{cc}$.
%On the other hand,
%each boundary component of $S^{cc}$ is a non-essential curve and so
%each connected component of $S\setminus S^{cc}$
%is a disk with at most one conical point. 
%
%We claim that systole geodesics are in fact contained in $S^{cc}$.
%In particular, $S^{cc}$ contains some conical point and so
%$S^{cc}$ is a connected component of $\Vor^{-1}([0,r])$.
%
%
%By contradiction, suppose that $\sigma_{\sys}$ is a systole geodesic 
%contained in a connected component $S'$ of $S\setminus S^{cc}$.
%Then $S'$ has one conical point $x_i$ and $\sigma_{\sys}$ is a loop based at $x_i$.
%Moreover, $S'\setminus \sigma_{\sys}$ should have a component which is a disk whose geodesic boundary $\sigma_{\sys}$ has at most one non-smooth point, namely $x_i$. Denote such a disk by $D$ and consider the developing map $\iota: D\to \Sph$. Then $\iota(\sigma_{\sys})$ should be a geodesic loop in $\Sph$ based at $\iota(x_i)$. At the same time, the length of $\sigma_{\sys}$ is at most $\pi$ by Lemma \ref{lemma:shorterthanpi/2}. 
%This is a contradiction, since closed geodesics on $\Sph$ have length $2\pi$.
%\end{proof}

Finally, we give an upper bound for the total area of the $\BU^0_j$'s
and a lower bound for the total angle of the conical points sitting in a Voronoi core.

\begin{lemma}[Estimate for the total area of the $\BU^0_j$'s]\label{balls0area} 
Assume that the length of each $\pa\BU^0_j$ is smaller than $2\pi$.
Then
\[
\frac{1}{2\pi}\sum_{j} |\Area(\BU^0_j)-4\pi b^0_j|< \left(\frac{\ell\left(\Vor^{-1}(r)\right)}{2\pi}\right)^2.
\]
\end{lemma}

\begin{proof} 
By Corollary \ref{shorthindisk}, the area of $\BU^0_j$ satisfies
$\frac{1}{2\pi}| \Area(\BU^0_j)-4\pi b^0_j|<\left(\ell(\partial \BU^0_j)/2\pi\right)^2$.
Since the boundaries $\partial \BU^0_j$ are disjoint and
they belong to $\Vor^{-1}(r)$, we have
$\sum_j \ell(\pa \BU^0_j)\leq \ell\Big(\Vor^{-1}(r)\Big)$ and
the inequality follows.
\end{proof}

\begin{lemma}[Bound from below on the angles in a Voronoi core]\label{lemma:core-angles}
The sum of the angles of the conical points sitting inside a Voronoi core $S^c$
is larger than $\frac{4\pi}{3}$.
\end{lemma}
\begin{proof}
Since the subsurface $S^c$ is a connected component of $\Vor^{-1}([0,r])$ for some $r\in(\sys(S,\bm{x}),\frac{\pi}{2})$,
and $\sys(S,\bm{x})$ is a saddle value for $\Vor$, it follows that $S^c$ contains a saddle critical point for $\Vor$.
We then conclude by  Proposition \ref{4pi3prop}.
\end{proof}

%%%%%%%%%%%%%%%%%%%%%%%%%%%%%%%%%%%%%%%%

\subsection{Detecting almost bubbling surfaces through their Voronoi function}

In Section \ref{sec:almost-bubbling} it was shown that almost bubbling spherical surfaces
have small non-bubbling parameter.
%
%It follows from the discussion in Section \ref{sec:bubbling}
%that the non-vanishing of $\Ab_{\bm{\th}}$ is an obstruction to complete bubbling degenerations.
Here we prove that a spherical surface with small systole
whose Voronoi function satisfies certain geometric properties is in fact almost bubbling.

\begin{theorem}[Detecting almost bubbling surfaces via $\Vor$]\label{Abound} 
Let $S$ be a surface with $n$ conical points $\bm{x}$ of angles $2\pi\bm{\th}$
and $\chi(\dot{S})<0$.
%, and let $\wh{I}\subsetneq\{1,2,\dots,n\}$.
Let $\delta\in\left(0,\frac{\pi}{2}\right)$ and $\varepsilon\in 
\left(0,\frac{1}{2}\right)$
be such that
\[
\sys(S,\bm{x})<\frac{\delta\varepsilon}{4\pi\|\bm{\th}\|_1}< \delta<\max(\Vor)
\]  
%with $\e<2\|\bm{\th}_{\wh{I}^c}\|_1$
and suppose 
that the following conditions hold:
\begin{itemize}
\item[(i)]
the function $\Vor$ does not have saddle critical values in the interval 
$\left[\frac{\delta\varepsilon}{4\pi\|\bm{\th}\|_1},\, \delta \right]$; 
%\item[(ii)] 
%there exists one connected component $S^c$ with $\chi(\dot{S}^c)<0$ of the sublevel set 
%$\left\{\Vor\leq\frac{\delta\varepsilon}{4\pi\|\bm{\th}_{\wh{I}^c}\|_1}\right\}$ in $S$
%with the property that $x_i\notin S^c$ for all $i\in \wh{I}$;
%%contains points $x_1,\ldots,x_k$ and does not contain $x_{k+1},\ldots,x_n$. 
%%
\item[(ii)]
no component of $\Vor^{-1}\left[\frac{\delta\varepsilon}{4\pi\|\bm{\th}\|_1},\, \delta \right]$ is an essential cylinder.
%the surface $S\setminus S^c$ is a disjoint union of topological disks
%and every such disk contains at most one conical point.
%
% such that each disk contains at most one conical point $x_j$ ($k<j\le n$). 
\end{itemize}
Then $S$ is $\left(\frac{3\e}{5}\right)$-bubbling.
% and
%the surface $S\setminus S^c$ is a disjoint union of topological disks
%and every such disk contains at most one conical point.
%Then for some $l\in \mathbb Z_+$ the following inequality holds: 
%$$|\sum_{i=1}^k\theta_i-\sum_{i=k+1}^n\theta_i+\chi(\dot S)-2l|\le 2\pi|\theta|\cdot \varepsilon$$
\end{theorem}

\begin{center}
\begin{figurehere}
\psfrag{S'}{$\textcolor{Purple}{S^c}$}
\psfrag{S}{$\textcolor{blue}{S}$}
\psfrag{S4}{$\textcolor{blue}{\!\!\BU^1_4}$}
\psfrag{S2}{$\textcolor{blue}{\BU^1_2}$}
\psfrag{B1}{$\textcolor{blue}{\BU^0_1}$}
\psfrag{x2}{$\textcolor{Sepia}{x_2}$}
\psfrag{x4}{$\textcolor{Sepia}{x_4}$}
\includegraphics[width=0.4\textwidth]{almost-bubbling.eps}
\caption{{\small An example of almost bubbling spherical surface.}}\label{fig:nearly-complete-bubbling}
\end{figurehere}
\end{center}

%The hypotheses of the above theorem cover the case in which
%not all the points $x_h$ with $h\in\wh{I}^c$ are contained in $S^c$.

%\begin{remark}
%In the case $\wh{I}=\emptyset$ 
%we would have $\|\bm{\th}_{\wh{I}^c}\|_1=\|\bm{\th}\|_1$,
%and so the condition
%$\e<2\|\bm{\th}_{\wh{I}^c}\|_1$ becomes redundant,
%because $\|\bm{\th}\|_1\geq -\chi(\dot{S})\geq 1$.
%\end{remark}

%\begin{remark}\label{Morseremark}Since all critical values of $\Vor$ in the interval $\left[\frac{\delta\varepsilon}{4\pi\|\bm{\th}_{\wh{I}^c}\|_1},\, \delta \right]$ correspond to local maxima, a Morse theory argument implies that there exist two types of connected components of the set $\Vor^{-1}\left[\frac{\delta\varepsilon}{4\pi\|\bm{\th}_{\wh{I}^c}\|_1},\, \delta \right]$. One type are cylinders that join a connected component of  the level set $\Vor^{-1}(\frac{\delta\varepsilon}{4\pi\|\bm{\th}_{\wh{I}^c}\|_1})$ with a connected component of the level set $\Vor^{-1}(\delta)$. Another type are disks with one local maximum of $\Vor$ within, boundaries of such disks are connected components of $\Vor^{-1}(\frac{\delta\varepsilon}{4\pi\|\bm{\th}_{\wh{I}^c}\|_1})$. We will denote such disks $D_1,\ldots, D_r$.
%\end{remark}

\begin{proof}[Proof of Theorem \ref{Abound}]
%By Lemma \ref{lemma:subcore},
%a Voronoi core $S^c$ of $S$ can be
%obtained as 
%the unique component of 
%$\Vor^{-1}\left[0,\frac{\delta\e}{4\pi\|\bm{\th}\|_1}\right]$ 
%with $\chi(\dot{S}^{c})<0$.
By (i), components of $\Vor^{-1}\left[\frac{\delta\e}{4\pi\|\bm{\th}\|_1},\,\delta\right]$
are disks without conical points or Voronoi cylinders. Hence, (ii) implies that
every connected component of $\Vor^{-1}\left(\frac{\delta\e}{4\pi\|\bm{\th}\|_1}\right)$
is non-essential.

By Lemma \ref{lemma:subcore},
a Voronoi core $S^c$ of $S$ can be
obtained as a component of 
$\Vor^{-1}\left[0,\frac{\delta\e}{4\pi\|\bm{\th}\|_1}\right]$.
The complement
$S\setminus S^c$ is a disjoint union of disks with at most one conical point each. 

For every $i\in I:=\{i\in \{1,\dots,n\}\ |\ x_i\notin S^c\}$,
denote by 
$\BU^1_i$ the connected component of $S\setminus S^c$
that contains $x_i$, and let
%$B_1,\ldots,B_k$ 
$\{\BU^0_j\}$ be the components of $S\setminus S^c$
that do not contain conical points. 

In order to show that $(S^c,\{\BU^0_j\},\{\BU^1_i\})$ is an $\left(\frac{3\e}{5}\right)$-bubbling partition of $S$, we need to show that
$$\frac{1}{2\pi}\Area(S^c)+\sum_j\lambda_0(\BU^0_j)+\sum_i\lambda_1(\BU^1_i)<\frac{3\e}{5}.$$

The area of $S^{c}$ is estimated using Lemma \ref{lengthoflevel}(b) as follows
\begin{equation}\label{areas'}
\Area(S^{c})\le 
\pi\left(\frac{\delta\e}{4\pi\|\bm{\th}\|_1}\right)^2\|\bm{\th}_{I^c}\|_1
\le
\frac{(\delta\varepsilon)^2}{16\pi}
<
\frac{\pi^2}{4}\, \frac{1}{16\pi}
\,\frac{\e}{2}=
%\frac{\pi^2}{16 \pi}\frac{1}{2}\frac{3}{4 \pi}\e=
\frac{\pi}{128}\e
\end{equation}
because $\|\bm{\th}_{I^c}\|_1/\|\bm{\th}\|_1\leq 1$ and, by Gauss-Bonnet, $\|\bm{\th}\|_1>-\chi(\dot{S})\geq 1$.

Now, by Lemma \ref{lengthoflevel}(a)
\begin{equation}\label{allcurves}
\sum_{j} \ell(\pa \BU^0_j)+\sum_{i\in I}\ell(\pa \BU^1_i)
=\ell(\pa S^{c})
\leq 2\pi \frac{\delta\varepsilon}{4\pi\|\bm{\th}\|_1} \|\bm{\th}_{I^c}\|_1
\leq \frac{\e}{2}\delta< \frac{\pi}{8}.
\end{equation}

Inequality (\ref{allcurves}) allows us to estimate the total area of the bubbles $\BU^0_j$.
It clearly implies that $\sum_{j}  \ell(\pa \BU^0_j)\le \frac{\delta}{2}\varepsilon$.
Now, applying Lemma \ref{balls0area} we get 
\begin{equation}\label{allitledisks}
\sum_{j}|\Area(\BU^0_j)-4\pi b_j'|<\frac{1}{2\pi}\left(\frac{\delta\e}{2}\right)^2
<\frac{1}{2\pi}\, \frac{\pi^2}{4}\,\frac{1}{4}\,\frac{\e}{2}=
\frac{\pi}{64}\e.
\end{equation}
%so that $\Area(S^c)=\Area(S^{cc})+\sum_l\Area(\CAP_l)\leq \left(\frac{1}{128}+\frac{1}{64}\right)\pi\e=\frac{3\pi}{128}\e$.

Again using Inequality (\ref{allcurves}), we are able to estimate the sum of the $\lambda_1$-invariants of the disks $\BU^1_i$. In fact, all Voronoi cylinders in $\Vor^{-1}\left[\frac{\delta\e}{4\pi\|\bm{\th}\|_1},\delta\right]$ have height 
larger than $\left(1-\frac{\e}{4\pi\|\bm{\th}\|_1}\right)\delta>\frac{24}{25}\delta$. It follows that
for any $i\in I$ we have  $d(x_i,\partial \BU^1_i)>\frac{24}{25}\delta$. 
Thus, $\sum_i \lambda_1(\BU^1_i)<\frac{\e\delta/2}{24\delta/25}=\frac{25}{48}\e$.

Putting together the previous estimates, we conclude that
\[
\frac{1}{2\pi}\Area(S^c)+\sum_j\lambda_0(\BU^0_j)+\sum_i \lambda_1(\BU^1_i)<\left(\frac{3}{256}+\frac{25}{48}\right)\e<\frac{3}{5}\e,
\]
which proves the theorem.
\end{proof}

\subsection{Proof of the systole inequality}

The following elementary lemma motivates why the power $-3\chi(\dot{S})$ appears
in the statement of Theorem \ref{main:systole}.

\begin{lemma}[Pigeonhole principle]\label{lemma:pigeon}
Let $N\geq 2$ be an integer and $[r,t]$ be an interval contained in $(0,1)$ such that $r<t^N$.
For every collection of $N-2$ points $c_1,\dots,c_{N-2}$ in $[r,t]$, there exists
$\delta\in (r,t)$ such that the interval $[t\delta,\delta]$ 
is inside $(r,t)$ and it
does not contain any $c_i$.
\end{lemma}
\begin{proof}
Since $r<t^N$, there exists a small $\eta>0$ such that $r<(t-\eta)^{N-1}(t+\eta)$.
Consider the $N-1$ disjoint intervals $\big((t-\eta)^{k+1}(t+\eta),\,(t-\eta)^k(t+\eta)\big]$ 
contained in $[r,t]$ for $k=0,\dots,N-2$. There must be one such interval
that does not contain any $c_i$: suppose it is $\big((t-\eta)^{k_0+1}(t+\eta),\,(t-\eta)^{k_0}(t+\eta)\big]$.
It is enough to choose $\delta=(t-\eta)^{k_0}(t+\eta)$.
\end{proof}

We have now all the ingredients to prove the main result of this section.

\begin{proof}[Proof of Theorem \ref{main:systole}]
We will assume that Condition (i) is satisfied and Inequality (\ref{sysineq}) is violated,
and we will deduce that Condition (ii) cannot hold.

According to Proposition \ref{critnumber}, the number of non-zero saddle values of the function $\Vor$ is at most $-3\chi(\dot S)$. By Remark \ref{maxofV_S} and Condition (ii) we have $\max(\Vor)\ge \sqrt{2\pi\varepsilon\|\bm{\th}\|_1^{-1}}$. 

At the same time, the systole is shorter than $
\left(\frac{\varepsilon}{4\pi\|\bm{\th}\|_1}\right)^{-3\chi(\dot S)}$.
Since $\frac{\varepsilon}{4\pi\|\bm{\th}\|_1}< \sqrt{2\pi\varepsilon\|\bm{\th}\|_1^{-1}}$,
by Lemma \ref{lemma:pigeon} applied to $r=\sys(S,\bm{x})$, $t=\frac{\e}{4\pi\|\bm{\th}\|_1}$ and $N=-3\chi(\dot{S})$,
there exists a $\delta$ satisfying
\[
%\left(\frac{\varepsilon}{4\pi\|\bm{\th}\|_1}\right)^{-8\chi(\dot S)}
\sys(S,\bm{x})
< \delta < \frac{\varepsilon}{4\pi\|\bm{\th}\|_1}
<\max(\Vor)
\]
such that
there are no saddle values of $\Vor$ in the interval $\left[\frac{\varepsilon}{4\pi\|\bm{\th}\|_1}\delta,\, \delta\right]$. 
Hence, by Corollary \ref{cor:non-essential}
every connected component of 
$\Vor^{-1}\left[\frac{\varepsilon}{4\pi\|\bm{\th}\|_1}\delta,\, \delta\right]$
is either a disk without conical points or
a non-essential Voronoi cylinder.
% with extremal length $\Ext(C)\leq \frac{2\pi\|\bm{\th}\|_1}{\log(1/\varepsilon)}$. Hence, such cylinder is non-essential by Condition (i).

%By Corollary \ref{cor:core},
%there exists a component $S^c$ of the sublevel set $\left\{\Vor\leq \frac{\varepsilon}{4\pi\|\bm{\th}\|_1}\delta\right\}$ with $\chi(\dot{S}^c)<0$; moreover, each component of $S\setminus S^c$
%is a disk with at most one conical point.
%
%Let now $S^c$ be a connected component of the surface $S\setminus \mathcal{C}$, 
%that contains a geodesic arc 
%%$\varsigma$ 
%that realizes the systole.
%In particular, $S^c$ is a connected component of the sublevel $\left\{\Vor\leq \frac{\varepsilon}{4\pi\|\bm{\th}\|_1}\delta\right\}$.
%
%Consider all cylinders in $\mathcal{C}$
%that share a boundary component with $S^c$. Since neither of these cylinders is essential, we can apply Lemma \ref{SgammaS} and conclude that each component of $\mathcal{C}$ cuts from $S$ a disk with at most one conical point.
%It follows that the complement $S\setminus S^c$ is
% a disjoint union of disks with at most one conical point, and so $\chi(\dot{S}^c)\leq\chi(\dot{S})<0$.
%%In particular, $\chi(\dot{S}^c)<0$.
%
%It follows that $\left\{\Vor\geq \frac{\varepsilon}{4\pi\|\bm{\th}\|_1}\delta\right\}$ is a disjoint union of
%disk with at most one conical point each, and so
%$S^c=\left\{\Vor<\frac{\varepsilon}{4\pi\|\bm{\th}\|_1}\delta\right\}$ is connected and $\chi(\dot{S}^c)<0$.

Finally, Theorem \ref{Abound} ensures that 
$S$ is $\left(\frac{3\e}{5}\right)$-bubbling and so
$\Ab_{\bm{\th}}(S)<\varepsilon$ by Theorem \ref{thm:almost-bubbling}.
This contradicts Condition (ii).
\end{proof}

%\subsubsection{Short systoles versus long essential cylinders}

%The goal of this subsection is to explain that surfaces with a short systole and with $\theta_i$ sufficiently far from walls contain a long essential cylinder whose boundary is composed of two connected components of level sets of $\Vor$. 

The argument used in the proof of Theorem \ref{main:systole} can be adapted to 
prove the following result, that we formulate separately as a corollary.

%The goal of this subsection is to explain that surfaces with a short systole and with $\theta_i$ sufficiently far from walls contain a long essential cylinder whose boundary is composed of two connected components of level sets of $\Vor$. 

\begin{corollary}[Small systole and large $\Ab_{\bm{\th}}$ give long essential Voronoi cylinders]\label{longcylfromshortsys} 
Let $S$ be a surface with conical points $\bm{x}$ of angles $2\pi\bm{\th}$. 
Suppose that for some $\varepsilon\in\left(0,\frac{1}{2}\right)$  the following hold:
\begin{itemize}
\item[(i)]
$\displaystyle \sys(S,\bm{x})< \left(\frac{\varepsilon}{4\pi\|\bm{\th}\|_1}\right)^{-3\chi(\dot S)+1}$
\item[(ii)]
$\displaystyle \Ab_{\bm{\th}}(S)\geq  \varepsilon$.
\end{itemize}
Then there exists $\delta\in\left(\sys(S,\bm{x}), \frac{\varepsilon}{4\pi\|\bm{\th}\|_1}\right]$  such that 
\begin{itemize}
\item[(a)]
the function $\Vor$ has no saddle values in the interval $\left[\frac{\delta\varepsilon}{4\pi\|\bm{\th}\|_1},\,\delta \right]$
\item[(b)]
a connected component of $\Vor^{-1}\left(\left[\frac{\delta\varepsilon}{4\pi\|\bm{\th}\|_1},\,\delta \right]\right)$ is an essential cylinder.
\end{itemize}
\end{corollary}

\begin{proof}
Statement  (a) is proven as in the proof of Theorem \ref{main:systole}. We use  the upper bound for $\sys(S,\bm{x})$ given by  Condition (i)  and the fact that $\Vor$ has at most $-3\chi(\dot S)$ saddle critical values of which one is equal to $\sys(S,\bm{x})$.

%As in the proof of Theorem \ref{main:systole},  the upper bound for $\sys(S)$ given by  Condition (1)  implies the existence of $\delta\in\left(\sys(S), \frac{\varepsilon}{4\pi\|\bm{\th}\|_1}\right]$ such that  all values in the interval
%$\left[\frac{\delta\varepsilon}{4\pi\|\bm{\th}\|_1},\, \delta \right]$ are regular for $\Vor$. Hence, (a) holds.

About claim (b),
note that $\delta<\max (\Vor)$ and so the set $\Vor^{-1}\left(\left[\frac{\delta\varepsilon}{4\pi\|\bm{\th}\|_1},\, \delta \right]\right)$ is non-empty; in fact, it is a union of cylinders 
and caps by (a).
Hence, Condition (ii), Theorem \ref{Abound} 
and Theorem \ref{thm:almost-bubbling} imply (b).
%
%
%By contradiction, suppose each of these cylinders is not essential.
%In such case Corollary \ref{cor:core} tells us that conditions (ii) and (iii) of Theorem \ref{Abound} are satisfied. Hence we can apply this theorem with $\wh{I}=\emptyset$ and get
%the bound $\Ab_{\bm{\th}}(S)\le \varepsilon$, thus violating Condition (2). 
%
%
%
%The proof of this corollary repeats word in word the proof of Theorem \ref{main:systole}. The only difference is the following. From the very beginning we have the upper bound 1) on ${\rm Sys}(S)$. As a result we find $\delta<\max(\Vor)$ such that $\Vor$ has no critical values in the interval $[\frac{\delta\varepsilon}{4\pi|\theta|},\delta ]$. If we assume now that all connected components of $\Vor^{-1}([\frac{\delta\varepsilon}{4\pi|\theta|},\delta ])$ are not essential cylinders, applying Lemma \ref{SgammaS} and Theorem \ref{Abound} we get a contradiction with the bound $Ab_{\theta}(S)> 7\varepsilon$.
\end{proof}

\appendix

%\section{Angle inequalities for clustering}
%%
%\input non-existence-clustering.tex
%

\section{On the extremal length}\label{sec:appendix}

%\subsection{Extremal length}

In this appendix we will recall the definition of extremal length of an essential simple closed
curve in a Riemann surface and of extremal systole, and we will list some basic facts about them.

\begin{definition}[Extremal length]\label{defofextermal}
Let $\Sigma$ be a Riemann surface and let $\gamma\subset\Sigma$ be an essential simple closed curve. The {\it{extremal length}}
of $\gamma$ in $\Sigma$ is
\[
\Ext_\gamma(\Sigma):=\sup_{\rho} \frac{\inf_{\check{\gamma}}\ell_{\rho}(\check{\gamma})^2}{\area(\rho)}
\]
where the inf is taken over all $\check{\gamma}$ freely homotopic to $\gamma$
and the sup is taken over all conformal metrics $\rho$ on $\Sigma$ of finite area.
If $C$ is a cylinder and $\gamma$ is its waist, then 
%$\Ext(A):=\Ext_A(\gamma)$
%where $\gamma$ is the waist of $A$ and 
the {\it{modulus}} of
$C$ is $M(C)=1/\Ext_\gamma(C)$.
\end{definition}

It is a fact that all cylinders with finite positive extremal length are isomorphic
to a standard annulus as in the below example.

\begin{example}[Modulus of a standard plane annulus]\label{example:groeztsch}
For every $0<r'<r''$, the modulus of the annulus 
$C=\{z\in\CC\,|\,r'<|z|<r''\}$ is
$M(C)=\frac{1}{2\pi}\log(r''/r')$
and it is attained at Euclidean metrics homothetic to $\frac{|dz|^2}{|z|^2}$.
%This is essentially due to Gr\"otzsch \cite{groetzsch}.
\end{example}

Since cylinders are biholomorphic to standard annuli of
Example \ref{example:groeztsch},
their extremal length is attained at the standard flat metric.
Moreover, the following well-known 
variational characterization holds.

\begin{lemma}[Modulus and height of a cylinder]\label{modulus-height}
Let $C$ be a cylinder with metric $\rho'$ of area $\Area(C)$ and such that the distance between the two boundary components is $H$. Then $M(C)\geq  H^2/\Area(C)$.
Moreover, equality holds if and only if $(C,\rho')$ is a flat straight cylinder.
\end{lemma}
%\begin{proof} 
%Let $\alpha$ be a simple curve joining the two boundary components of $C$. Then we have the following formula 
%\[
%M(C)=\sup_{\rho} \frac{\inf_{\check{\alpha}}\ell_{\rho}(\check{\alpha})^2}{\area(\rho)}
%\]
%where the inf is taken over all $\check{\alpha}$ freely homotopic to $\alpha$
%and the sup is taken over all conformal metrics $\rho$ on $C$. 
%Hence to prove the lemma we just need to take $\rho=\rho'$
%and to recall that $M(C)$ is attained only at metrics (unique up to rescaling) that make $C$ into a flat straight cylinder (see Example \ref{example:groeztsch}).
%\end{proof}

The following standard subadditivity property of modulus directly descends from its definition
and is used to estimate the modulus of a Voronoi cylinder.

\begin{lemma}[Subadditivity of modulus]\label{subadditivity} 
Suppose that a cylinder  $C$ is cut into two cylinders $C'$ and $C''$ by a
homotopically nontrivial simple loop. Then $M(C)\geq M(C')+M(C'')$. 
\end{lemma}

In order to understand what metric
on a general punctured surface realizes the extremal length
of a given simple closed curve $\gamma$,
let us first recall
the relation between 
$\Ext_\gamma(\dot{S})$
and modulus of subcylinders $C$ homotopic to $\gamma$.

\begin{remark}[Modulus and extremal length]\label{modulevialength}
It is well-known that the extremal length can be characterized as
\[
\Ext_\gamma(\Sigma)=\inf_{C\sim\gamma} \frac{1}{M(C)}
\] 
where the infimum is taken over all annuli $C\subset \Sigma$ homotopic to $\gamma$.
It follows that, if $\Sigma\subset \Sigma'$ is a conformal embedding, then
$\Ext_\gamma(\Sigma)\geq \Ext_\gamma(\Sigma')$.
\end{remark}

%We will apply the following well-known result in particular in the case of a cylinder
%endowed with a spherical metric.

Using the above remark,
it can be shown that
on a general punctured surface $\dot{S}$
the sup in the definition of extremal length
is achieved at 
the flat metric $|\q_\gamma|$ with conical singularities
associated to the
Strebel differentials $\q_\gamma$ on $\dot{S}$ introduced in the below proposition
(see, for instance, Strebel's book \cite{strebel:quadratic}).

\begin{proposition}[Strebel differentials]\label{prop:strebel}
Let $(S,J)$ be a Riemann surface with marked points $\bm{x}$ and let $\gamma\subset\dot{S}$ be
an essential simple closed curve. Then there exists a non-zero holomorphic quadratic differential $\q_\gamma$ on $\dot{S}$  
(unique up to positive multiples) such that
\begin{itemize}
\item
$\q_\gamma$ has at worst simple poles at $\bm{x}$
\item
every horizontal trajectory of $\q_\gamma$ is either smooth, closed and freely homotopic to $\gamma$
or it is an arc with endpoints in $\bm{x}\cup \sing(\q_\gamma)$,
where $\sing(\q_\gamma)\subset\dot{S}$ is the zero locus of $\q_\gamma$.
\end{itemize}
Moreover, the union $C_\gamma$ of all smooth closed horizontal trajectories of $\gamma$ 
is the complement of finitely many arcs.
\end{proposition}

In view of Remark \ref{modulevialength},
Strebel's study of the extremal properties of the cylinder $C_\gamma$
leads to the following characterization of the extremal length.

\begin{corollary}[Extremal length and embedded cylinders]\label{cor:ext-from-cyl}
Let $\gamma$ be a simple closed essential curve in $\dot{S}$
and let $C\subset \dot{S}$ be a cylinder that retracts by deformation onto $\gamma$. Then 
$\Ext_{\gamma}(\dot{S},J)\leq 1/M(C)$. Equality is attained if and only if 
both the following conditions hold:
\begin{itemize}
\item
the metric $\rho$ is $|\q_\gamma|$,
where $\q_\gamma$ is a Strebel differential associated to $\gamma$;
\item
the cylinder $C$ is $C_\gamma$.
\end{itemize}
\end{corollary}

Now we introduce a quantity associated to a punctured Riemann surface $\Sigma$
%$(\dot{S},J)$,
which is invariant under biholomorphisms. 
Such quantity
measures how close 
%$(\dot{S},J)$ 
$\Sigma$
is to be conformally degenerate
(see Lemma \ref{lemma:ext-systole-function}).

\begin{definition}[Extremal systole]
Let $\Sigma$ be a connected punctured Riemann surface.
The {\it{extremal systole}} $\Ext\sys(\Sigma)$ is the minimum of $\Ext_{\gamma}(\Sigma)$ as
$\gamma$ ranges over all essential simple closed curves on $\Sigma$.
%We will shortly denote it by $\Ext\sys(\dot{S})$ when the complex structure $J$ is understood.
\end{definition}

Finally, we show how
the extremal length of a simple closed curve $\gamma$
inside a punctured spherical surface $\dot{S}$
provides a non-trivial upper bound for the length of shorter curves homotopic to $\gamma$ inside $\dot{S}$.

\begin{proposition}[Extremal length bounds length from above]\label{systolfromcylinder} 
Let $S$ be a surface with spherical metric and conical singularities at $\bm{x}$ of angles $2\pi\bm{\th}$ and assume $\chi(\dot{S})<0$.
For every essential, simple closed curve $\gamma$ on $\dot{S}$
there exists a simple closed curve $\check{\gamma}\subset\dot{S}$ freely homotopic to $\gamma$
of length 
$$\ell(\check{\gamma})<\sqrt{\Area(S)\cdot\Ext_{\gamma}(\dot{S})}<\sqrt{2\pi\Ext_\gamma(\dot{S})\|\bm{\th}\|_1}.$$
\end{proposition}

\begin{proof} 
By definition \ref{defofextermal} of extremal length,
$\Area(S)\cdot\Ext_{\gamma}(\dot{S})\geq \inf_{\gamma'}\ell(\gamma')^2$, where
$\gamma'$ ranges over all simple closed curves in $\dot{S}$ freely homotopic to $\gamma$.
By Corollary \ref{cor:ext-from-cyl}, the sup in the
definition of $\Ext_{\gamma}(\dot{S})$ is not attained at a spherical metric,
and so there exists $\e>0$ such that
$\Area(S)\cdot\Ext_{\gamma}(\dot{S})>2\e+\inf_{\gamma'}\ell(\gamma')^2$.
Furthermore, there exists $\check{\gamma}\simeq \gamma$ such that $\ell(\check{\gamma})^2\leq \e+\inf_{\gamma'}\ell(\gamma')^2$,
and so $\Area(S)\cdot\Ext_{\gamma}(\dot{S})>\e+\ell(\check{\gamma})^2$.
In other words, $\ell(\check{\gamma})<\sqrt{\Area(S)\cdot\Ext_{\gamma}(\dot{S})}$.
The conclusion then follows, since $\Area(S)=2\pi\left(\chi(\dot{S})+\|\bm{\th}\|_1\right)< 2\pi\|\bm{\th}\|_1$ by Gauss-Bonnet.
%
%
%Note first, that for any $\varepsilon>0$ there is a curve $\gamma_{\varepsilon}$  on $S\setminus \bm{x}$ that is homotopic to $\gamma$ and has length at most $\sqrt{2\pi \Ext\sys(\dot{S})\|\bm{\th}\|_1}+\varepsilon$. Indeed, by Gauss-Bonnet $\area(S)\le 2\pi\|\bm{\th}\|_1$, so this statement follows directly by applying Definition \ref{defofextermal} and Remark \ref{modulevialength} to the spherical metric on $\dot{S}$.
%
%Now, any sequence of curve in homotopy class of $\gamma$ that minimizes the length of curves in this class converges to a broken geodesic on $S$ with singularities at some of the conical points. 
%This geodesic is either a loop or it is a concatenation of geodesic segments. This proves the proposition. 
\end{proof}

As a consequence,
combining Proposition \ref{systolfromcylinder} and Lemma \ref{birkhoff}, we obtain the following bound
from above for the systole in terms of the extremal systole.

\begin{corollary}[Extremal systole bounds systole from above]\label{cor:systole-upper-bound}
In a spherical surface $S$ with conical singularities at $\bm{x}$ of angles $2\pi\bm{\th}$, 
the systole satisfies $\sys(S,\bm{x})\le \sqrt{(\pi/2) \Ext\sys(\dot{S})\|\bm{\th}\|_1}$. 
\end{corollary}

\subsection{Peripheral regions in $\Mcal_{0,4}$}

%Throughout this subsection, we will denote by $S$ a surface of genus $0$
%with $4$ distinct marked points $\bm{x}=(x_1,x_2,x_3,x_4)$, and let $\dot{S}=S\setminus\bm{x}$.
%In this subsection we study the region of $\Mcal_{0,4}$ on which the extremal systole
%is smaller than $2$ and we show that it consists of three disjoint, open, connected subsets.
%

We recall that the moduli space $\Mcal_{0,4}$ 
of Riemann surfaces of genus $0$ with $4$ distinct marked points
is isomorphic to $\CC\PP^1\setminus\{0,1,\infty\}$.

For every $1\leq i\leq 3$ denote by $\Mcal_{0,4}^{(i,4)}$ 
the subset of $\Mcal_{0,4}$
consisting of isomorphism classes of Riemann surfaces $(S,J,\bm{x})$ such that
there exists a simple closed curve $\gamma\subset\dot{S}$ 
with $\Ext_{\gamma}(\dot{S},J)<2$
such that a connected component of $S\setminus \gamma$ contains $\{x_i,x_4\}$.

In this subsection we want to prove the following result.

\begin{lemma}[Peripheral regions of $\Mcal_{0,4}$]\label{lemma:peripheral}
The subsets
$\Mcal_{0,4}^{(1,4)}$, $\Mcal_{0,4}^{(2,4)}$, $\Mcal_{0,4}^{(3,4)}$
of $\Mcal_{0,4}$ are non-empty, open, connected and disjoint.
\end{lemma}

%We will also denote by $\pi:\hat{S}\rar S$ the double cover branched at $\bm{x}$
%and by $\hat{x}_i\in\hat{S}$ the ramification point over $x_i$. 

%Given a complex structure $J$ on the surface $S$ of genus $0$.
%
%, we denote by $\q_J$ a non-zero
%holomorphic quadratic differential on $\dot{S}$ with simple poles at $\bm{x}=(x_1,x_2,x_3,x_4)$.

We begin by exhibiting an explicit construction of Riemann surfaces of genus $0$
with marked points $x_1,x_2,x_3,x_4$ endowed with a Strebel differential $\q_{12}$ associated
to a simple closed curve $\gamma$ that separates $x_1,x_2$ from $x_3,x_4$.

\begin{center}
\begin{figurehere}
%\psfrag{R}{$\textcolor{Blue}{\Rcal_r}$}
\psfrag{S}{$\textcolor{Blue}{\Ccal_r}$}
%\psfrag{r}{$\textcolor{PineGreen}{r}$}
%\psfrag{1/r}{$\textcolor{Blue}{\frac{1}{r}}$}
\psfrag{x1}{$\textcolor{Sepia}{x_1}$}
\psfrag{x2}{$\textcolor{Sepia}{x_2}$}
\psfrag{x3}{$\textcolor{Sepia}{x_3}$}
\psfrag{x4}{$\textcolor{Sepia}{x_4}$}
\psfrag{a12}{$\textcolor{PineGreen}{\alpha_{12}}$}
\psfrag{a34}{$\textcolor{Purple}{\alpha_{34}}$}
\includegraphics[width=0.25\textwidth]{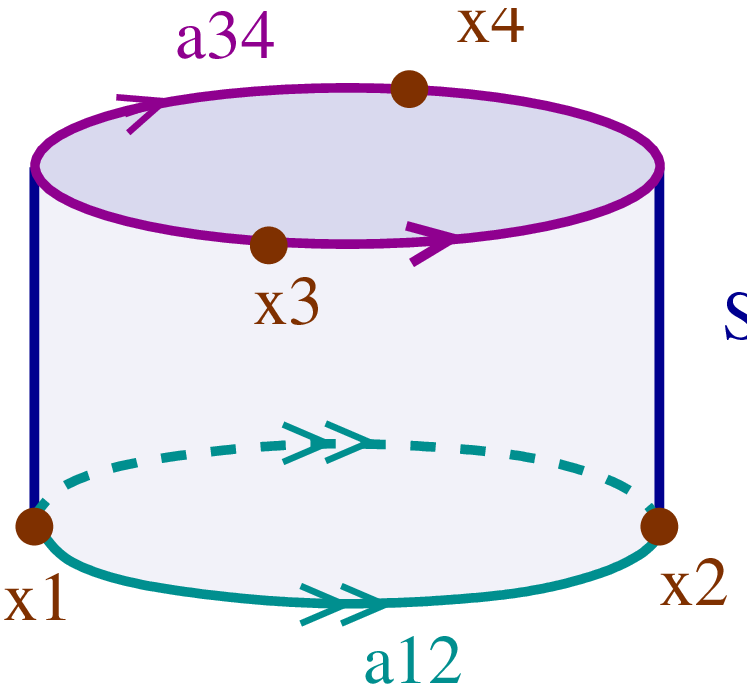}
\caption{{\small The surface $S_{r,\phi}$ is obtained from a cylinder $\Ccal_r$ via identification on $\pa\Ccal_r$.}}\label{fig:M04-bis}
\end{figurehere}
\end{center}

\begin{example}\label{example:strebel0,4}
%Consider a rectangle $\Rcal_r=\{z\in\CC\,|\, \RE(z)\in[0,r],\ \IM(z)\in[0,\frac{1}{r}]\}$ in the complex plane for some $r>0$
%and, by little abuse, denote by $dz^2$ the induced quadratic differential on the double $D\Rcal_r$ of $\Rcal_r$.
%Then mark the vertices $(0,\,r,\,\frac{i}{r},\,r+\frac{i}{r})$ of $D\Rcal_r$ by $(x_1,x_2,x_3,x_4)$.
%Finally, we obtain a surface $S_{r,\phi}$ by performing a twist along $\bar{\gamma}=\{\IM(z)=\frac{1}{2r}\}$ of amplitude $\phi\in \RR/\ZZ$, namely
%by cutting $D\Rcal_r$ along $\bar{\gamma}$, shearing the upper half-surface by $2r\phi$ along the fault and then gluing it back
%to the lower half-surface. 
Let $r>0$ and consider a flat cylinder $\Ccal_r=(\RR/2r\ZZ)\times [0,\frac{1}{r}]$ of waist $2r$ and height $\frac{1}{r}$,
which can be obtained from the strip $\{z\in\CC\,|\,\IM(z)\in[0,\frac{1}{r}]\}$ by identifying $z\sim z+2r$.
Given $\phi\in\RR/\ZZ$, let $S_{r,\phi}$ be the flat surface obtained from $\Ccal_r$
by identifying $(u,0)\sim(-u,0)$ and $(u-2r\phi,\frac{1}{r})\sim (2r\phi-u,\frac{1}{r})$, and then marking the conical points
$[0,0],[r,0],[2r\phi,\frac{1}{r}],2r(\phi+\frac{1}{2}),\frac{1}{r}]$ by $x_1,x_2,x_3,x_4$.
The quadratic differential $dz^2$ on the strip descends to a quadratic differential $\q_{12}$ on $S_{r,\phi}$
and $|\q_{12}|$ agrees with the induced flat metric on $S_{r,\phi}$.
Note that the only two horizontal non-periodic trajectories of $\q_{12}$
have length $r$: they are the arc $\alpha_{12}$ between $x_1,x_2$ 
induced by the boundary curve $\{\IM(z)=0\}$ of $\Ccal_r$
and the arc $\alpha_{34}$ between $x_3,x_4$ induced by the boundary curve $\{\IM(z)=\frac{1}{r}\}$ of $\Ccal_r$.
\end{example}

%
%It can be easily seen that the surface $(S,q_\gamma)$ can
%be obtained from one such $(DR,dz^2)$ by performing a twist along
%the horizontal trajectory $\gamma'=\{\IM(z)=b/2\}$, and by marking the four corners of $DR$ by $\bm{x}$.
%
%Since the complex vector space of such quadratic differentials is one-dimensional, 
%for any homotopically non-trivial simple
%closed curve $\gamma\subset\dot{S}$ the Strebel differential $\q_\gamma$
%agrees with $\q_J$ up to multiplication by a nonzero complex number.
%%by a complex number of norm $1$.

In the next lemma we show that all surfaces in $\Mcal_{0,4}$ endowed with Strebel differentials
are of the type seen above.

\begin{lemma}[Strebel differentials that separate $x_1,x_2$ from $x_3,x_4$]\label{lemma:strebel0,4}
Let $(S,\bm{x},J)\in\Mcal_{0,4}$.
Let $\gamma$ be a simple closed curve in $\dot{S}$ that separates $x_1,x_2$ from $x_3,x_4$
and let $\q_\gamma$ be the associated Strebel differential of area $2$.
Then $(S,\q_\gamma)$ is isomorphic to a $(S_{r,\phi},\q_{12})$ constructed in Example \ref{example:strebel0,4}
for some $r>0$ and $\phi\in\RR/\ZZ$. As a consequence, $\Ext_{\gamma}(S_{r,\phi})=2r^2$.
\end{lemma}
\begin{proof}
Since $\q_\gamma$ has simple poles at $\bm{x}$, there is a unique horizontal trajectory outgoing from each $x_i$.
It follows that the complement of the cylinder $C_\gamma$ is the union of two segments joining $x_1$ to $x_2$ and $x_3$ to $x_4$.
It is now easy to see that $\q_\gamma$ must be of the type produced in Example \ref{example:strebel0,4}.
The last claim follows by noting that a geodesic loop homotopic to ${\gamma}$ has length $2r$.
\end{proof}

In the lemma below we show that there can be 
at most one simple closed curve with extremal length smaller than $2$
on a Riemann surface in $\Mcal_{0,4}$. This is a special case of a more general lower bound
for the product the extremal lengths of two simple closed curves in terms of their geometric intersection product.
We include a short proof for completeness.

%The pull-back $\pi^*q_J$ to $\hat{S}$ is the square
%of a holomorphic $(1,0)$-form $\varphi_J$ on $\hat{S}$.
%We note that $\hat{S}$ has genus $1$ and $\varphi_J$ has no zeroes and no poles: in fact,
%$|\varphi_J|^2$ is a flat metric on $\hat{S}$, which is again unique up to rescaling.

\begin{lemma}[Small extremal systole is realized at one curve]\label{lemma:M04-extsys}
For every Riemann surface $(S,J)$ of genus $0$ with $4$ marked points
there exists at most one (essential) simple closed curve $\gamma\subset \dot{S}$ such that
$\Ext_\gamma(\dot{S},J)<2$.
\end{lemma}
\begin{proof}
%Consider the complex plane $\CC$ with the quadratic differential $dz^2$.
%Call $DR$ the double of 
%the flat rectangle $R=\{z\in\CC\,|\, \RE(z)\in[0,a],\ \IM(z)\in[0,b]\}$ for some $a,b>0$
%and, by little abuse, denote by $dz^2$ the induced quadratic differential on $DR$.
%It can be easily seen that the surface $(S,q_\gamma)$ can
%be obtained from one such $(DR,dz^2)$ by performing a twist along
%the horizontal trajectory $\gamma'=\{\IM(z)=b/2\}$, and by marking the four corners of $DR$ by $\bm{x}$.
Up to relabelling the marked points, we can assume that $\gamma$ separates $x_1,x_2$ from $x_3,x_4$.
Up to rescaling, we can also assume that $(S,|\q_\gamma|)$ has area $2$.
By Lemma \ref{lemma:strebel0,4}, the couple $(S,\q_\gamma)$ is isomorphic to
a couple $(S_{r,\phi},\q_{12})$ constructed in Example \ref{example:strebel0,4}.

%Since the area of $|\q_\gamma|$ is $2$ and the length of $\bar{\gamma}$ is $2r$,
%it follows that $\Ext_\gamma(\dot{S})=2r^2$.

If $\beta$ is any other essential simple closed curve in $\dot{S}$ not homotopic to $\gamma$, then
any geodesic representative $\bar{\beta}$ of $\beta$ must cross both the arcs
$\alpha_{12}$ and $\alpha_{34}$.
Hence, $\bar{\beta}$ must have
length at least $\frac{2}{r}$ and so $\Ext_\beta(\dot{S},J)\geq \frac{1}{2}\left(\frac{2}{r}\right)^2=\frac{2}{r^2}$.
It follows that $\Ext_\gamma(\dot{S},J)\cdot\Ext_\beta(\dot{S},J)\geq 4$.
As a consequence, if $\Ext_\gamma(\dot{S},J)<2$, then $\Ext_\beta(\dot{S},J)>2$.
\end{proof}

We can now prove the main result of this subsection.

%Let $1\leq i<j\leq 4$. Denote by $\Mcal_{0,4}^{(i,j)}$ 
%the subset of $\Mcal_{0,4}$
%consisting of isomorphism classes of Riemann surfaces $(S,J,\bm{x})$ such that
%there exists a simple closed curve $\gamma\subset\dot{S}$ 
%with $\Ext_{\gamma}(\dot{S},J)<2$
%such that a connected component of $S\setminus \gamma$ contains $x_i,x_j$.
%
%
%\begin{lemma}[Peripheral regions of $\Mcal_{0,4}$]\label{lemma:peripheral}
%The subsets
%$\Mcal_{0,4}^{(1,4)}$, $\Mcal_{0,4}^{(2,4)}$, $\Mcal_{0,4}^{(3,4)}$
%of $\Mcal_{0,4}$ are non-empty, open, connected and disjoint.
%\end{lemma}
\begin{proof}[Proof of Lemma \ref{lemma:peripheral}]
In view of Lemma \ref{lemma:M04-extsys}, the above regions are disjoint.
Since their union is $\Ext\sys^{-1}(0,2)$ and $\Ext\sys:\Mcal_{0,4}\rar\RR_+$
is continuous by Lemma \ref{lemma:ext-systole-function}, it follows that
each region is open.

Let us prove that $\Mcal_{0,4}^{(3,4)}$ is non-empty and connected.
The cases $(1,4)$ and $(2,4)$ will be analogous.

Consider the map $\Psi:\CC^*\rar \Mcal_{0,4}$ that sends
$z=re^{2\pi i\phi}$ to the Riemann surface $(S_{r,\phi},\bm{x})$ constructed in
Example \ref{example:strebel0,4}. It is not difficult to see that such map is continuous.

Since the curve $\gamma$ inside $S_{r,\phi}$ that separates $x_1,x_2$ from $x_3,x_4$
satisfies $\Ext_{\gamma}(S_{r,\phi})=2r^2$, it follows that
$\Mcal_{0,4}^{(3,4)}$ is the image of the punctured unit disk $\Delta^*=\{z=re^{2\pi i\phi}\in\CC^*\,|\,r<1\}$
via $\Psi$.
Hence, $\Mcal_{0,4}^{(3,4)}$ is non-empty and connected.
%
%
%
%Consider the Deligne-Mumford compactification $\ol{\Mcal}_{0,4}$
%(which is isomorphic to $\CC\PP^1$)
%and let $\Sigma^{(1,4)}\in\pa\Mcal_{0,4}$ be the singular
%surface with a node that separates the marked points
%$x_1,x_4$ from $x_2,x_3$.
%
%The estimate of the extremal lengths of the standard domains
%in \cite[Chap.3B]{ahlfors:book} ensure that $\Mcal_{0,4}^{(1,4)}\cup\{\Sigma^{(1,4)}\}$
%contains a small contractible open neighbourhood $U^{(1,4)}$ of the point $\Sigma^{(1,4)}$,
%and so $\Mcal_{0,4}^{(1,4)}$ is non-empty.
%
%Let now $(S,\bm{x},J)$ be a point in $\Mcal_{0,4}^{(1,4)}$ and let $\gamma\subset\dot{S}$
%be a simple closed curve that separates $x_1,x_4$ from $x_2,x_3$.
%The Teichm\"uller flow acting through the matrix 
%$\left(\begin{array}{cc}e^{-t/2} & 0\\ 0 & e^{t/2}\end{array}\right)$
%on the half-translation surface $(S,\bm{x},J,q_\gamma)$ produces
%a smooth path $(S,\bm{x},J_t,q_{\gamma,t})$ in $\Mcal_{0,4}$ with $J_0=J$ and $q_{\gamma,0}=q_\gamma$.
%Since $\Ext_\gamma(\dot{S},J_t)=e^{-t}\cdot\Ext_\gamma(\dot{S},J)$, such path
%is contained inside $\Mcal_{0,4}^{(1,4)}$; moreover,
%the surface $(S,\bm{x},J_t)$ belongs to $U^{(1,4)}$ for $t>0$ large enough.
%It follows that
%$\Mcal_{0,4}^{(1,4)}$ is connected.
\end{proof}

\subsection{Comparing $\sys$ and $\Ext\sys$ in a sequence of explicit examples}\label{app:example}

The content of Theorem \ref{main:systole} can be rephrased as an upper bound
for $\Ext\sys$ as in Corollary \ref{cor:sys-ext-main} below.
The aim of this section is to show that, in the case of spherical surfaces with small systole,
such an upper bound for $\Ext\sys$ 
is reasonably optimal, and more precisely it is optimal up to a factor $3$.

\begin{corollary}[Bound for $\Ext\sys$ in terms of $\sys$]\label{cor:sys-ext-main}
Let $S$ be a surface with spherical metric and conical singularities at $\bm{x}$ of angles $2\pi{\bm{\th}}$.
Assume that $\chi=\chi(\dot{S})<0$ and that $\dot{S}$ is not a $3$-punctured sphere. Then
\[
\Ext\sys(\dot{S})\cdot
\frac{\log(1/\sys(S,\bm{x}))}{2\pi\|\bm{\th}\|_1(1-3\chi)}
\left(
1-\frac{(1-3\chi)\log(4\pi\|\bm{\th}\|_1)}{\log(1/\sys(S,\bm{x}))}
\right)\leq 1
\]
provided $\sys(S,\bm{x})\leq \left(\frac{1}{4\pi\|\bm{\th}\|_1}\min\left\{\frac{1}{2},\,\Ab_{\bm{\th}}(S,\bm{x})\right\}\right)^{1-3\chi}$.
In particular, if $\chi(\dot{S})$ and $\bm{\th}$ are fixed, then
\[
\limsup
\Ext\sys(\dot{S})\cdot
\frac{\log\left(1/\sys(S,\bm{x})\right)}{2\pi\|\bm{\th}\|_1(-\chi)}
\leq 3-\frac{1}{\chi}
\]
as $\sys(S,\bm{x})\rar 0$.
\end{corollary}
\begin{proof}
Let $\e=4\pi\|\bm{\th}\|_1\sys(S,\bm{x})^{\frac{1}{1-3\chi}}$, so that the last inequality in the statement of Theorem 
\ref{main:systole} becomes an equality.
The assumption on $\sys(S,\bm{x})$ implies that
$\e<\frac{1}{2}$ and $\Ab_{\bm{\th}}(S,\bm{x})\geq \e$.
By Theorem \ref{main:systole}, 
\[
\Ext\sys(\dot{S})\geq 
\frac{2\pi\|\bm{\th}\|_1}{\log(1/\e)}=
\frac{2\pi\|\bm{\th}\|_1(1-3\chi)}{\log(1/\sys(S,\bm{x}))-(1-3\chi)\log(4\pi\|\bm{\th}\|_1)}
\]
and the conclusion follows.
\end{proof}

%The aim of this section is to show that, in the case of spherical surfaces with small systole,
%such an upper bound for $\Ext\sys$ is reasonably optimal, and more precisely it is optimal up %to a factor $3$. 

In particular, we will
produce an explicit sequence of surfaces of genus $0$
with an increasing number of marked points and we will estimate their systole and extremal systole,
thus proving the following statement.

\begin{proposition}[Bound for $\Ext\sys$ in terms of $\sys$ in some examples]\label{prop:bound-example}
There exist spherical surfaces $S$ of genus $0$ with 
conical singularities at $\bm{x}$
of angles $2\pi\bm{\th}$ and $\chi=\chi(\dot{S})<-1$
and
$\Ab_{\bm{\th}}(S,\bm{x})=\frac{1}{2}$
such that the ratio
\[
\Ext\sys(\dot{S})\cdot
\frac{\log\left(1/\sys(S,\bm{x})\right)}{2\pi\|\bm{\th}\|_1
(-\chi)}
\]
can be made as close to $1$ as desired.
\end{proposition}

Incidentally, we recall that the extremal systole can be always bounded from below in terms of the systole as in
Corollary \ref{cor:systole-upper-bound}. \\

The construction of such sequence of spherical surfaces proceeds as follows.

\subsubsection{Construction of the examples}\label{subsec:construction}
Fix positive integers $m$ and $N\geq 1$ 
and a real parameter $\e\in(0,\frac{1}{2})$.
Let $\bm{\th}=(2N+\frac{1}{2},\frac{1}{2},\frac{1}{2},1^{m+1})$
and denote by $\epsilon$ the quantity $\frac{\e}{4\pi\|\bm{\th}\|_1}$.

Consider the surface $S$ obtained by doubling a spherical triangle with vertices
$x_1,x_2,x_3$ and angles $2\pi\cdot(2N+\frac{1}{2},\frac{1}{2},\frac{1}{2})$. Such an $S$ comes with a natural orientation-reversing involution.

Along the geodesic arc $\alpha_{12}$ that goes from $x_1$ and $x_2$, mark
by $y_i$ the point that sits at distance $\epsilon^i$ from $x_1$ for $i=0,\dots,m$.
Thus, $S$ is a spherical surface of genus $0$ with
$4+m$ conical points 
$\bm{x}=(x_1,x_2,x_3,y_0,y_1,\dots,y_m)$
of angles $2\pi\bm{\th}$.
%
%and the angle at each $y_i$ is $2\pi$, so that
%$\bm{\th}=(2N+\frac{1}{2},\frac{1}{2},\frac{1}{2},1^{m+1})$.
Denote by $\dot{S}$ the punctured surface $S\setminus \bm{x}$.
For simplicity, we will use the shorter notation
$\Ab(S)$ and $\sys(S)$
with the obvious meaning.

For all $i=1,\dots,m$ define $\gamma_i$ to be the simple closed curve in $\dot{S}$ consisting of points
at distance $\epsilon^{i-\frac{1}{2}}$ from $x_1$,
which in particular separates $x_1,y_m,\dots,y_i$ from all the other punctures.

\begin{center}
\begin{figurehere}
%\psfrag{R}{$\textcolor{Blue}{\Rcal_r}$}
\psfrag{Sk}{$\textcolor{Blue}{S}$}
\psfrag{Sk+}{$\textcolor{Blue}{S^+_m}$}
\psfrag{Sk-}{$\textcolor{Blue}{S^-_m}$}
\psfrag{Ck}{$\textcolor{PineGreen}{C_m}$}
\psfrag{gk}{$\textcolor{PineGreen}{\gamma_m}$}
%\psfrag{r}{$\textcolor{PineGreen}{r}$}
%\psfrag{1/r}{$\textcolor{Blue}{\frac{1}{r}}$}
\psfrag{x1}{$\textcolor{Sepia}{x_1}$}
\psfrag{x2}{$\textcolor{Sepia}{x_2}$}
\psfrag{x3}{$\textcolor{Sepia}{x_3}$}
\psfrag{y0}{$\textcolor{Sepia}{y_0}$}
\psfrag{yk-2}{$\textcolor{Sepia}{y_{m-2}}$}
\psfrag{yk-1}{$\!\!\textcolor{Sepia}{y_{m-1}}$}
\psfrag{yk}{$\textcolor{Sepia}{y_m}$}
\psfrag{a23}{$\textcolor{blue}{\alpha_{23}}$}
\psfrag{a12}{$\textcolor{blue}{\alpha_{12}}$}
\includegraphics[width=0.7\textwidth]{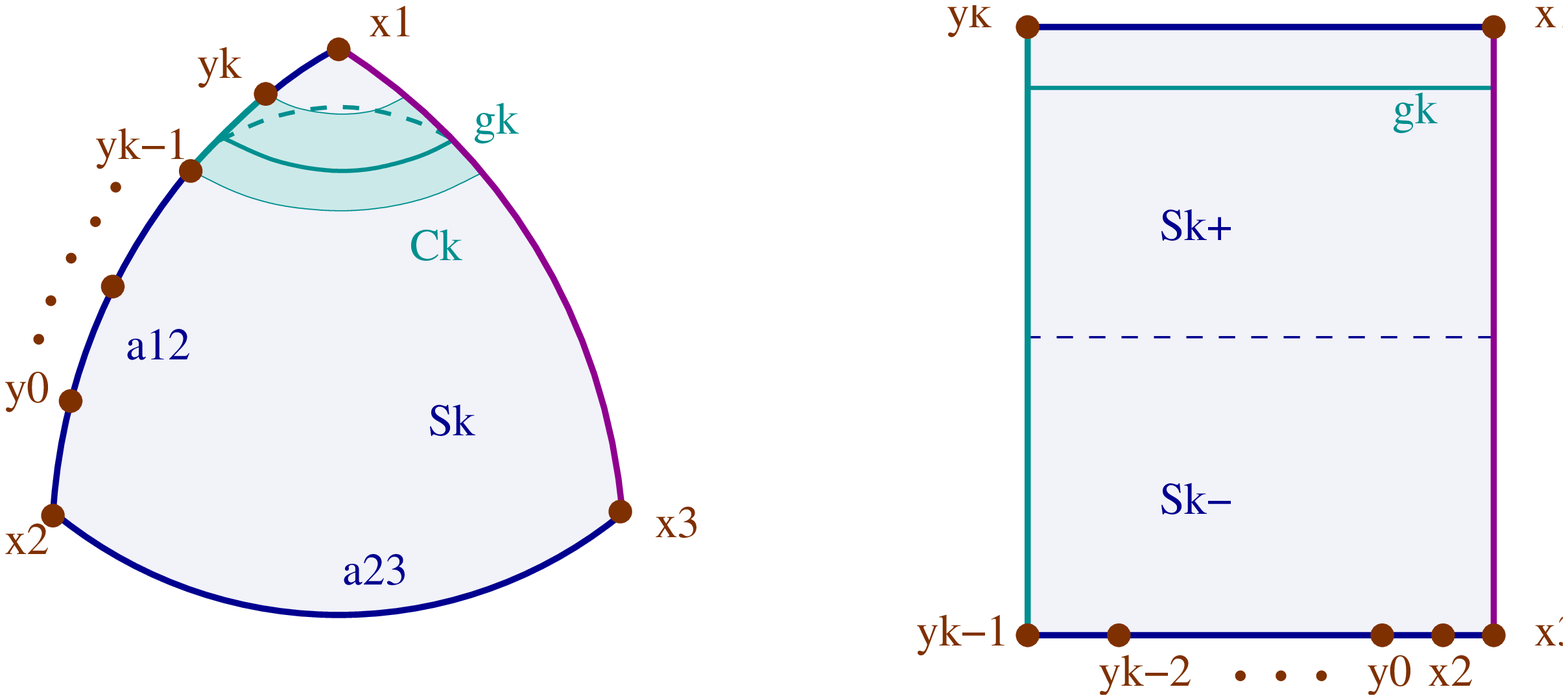}
\caption{{\small $S$ on the left is biholomorphic to the flat surface obtained by doubling the figure on the right.}}\label{fig:example1}
\end{figurehere}
\end{center}

Certainly, the systole is $\sys(S)=\epsilon^m$ and $\Ab(S)=\frac{1}{2}$.
Also, $\|\bm{\th}\|_1=2N+m+2+\frac{1}{2}$.

Since the ratio $\frac{-\chi-2}{-\chi}$ can be made close to $1$ by choosing $m$ large,
Proposition \ref{prop:bound-example} will then be a consequence of the following statement.

\begin{lemma}[Comparison between $\Ext\sys$ and $\sys$ for $S$]\label{lemma:example}
%Setting $\epsilon=\frac{\e}{4\pi\|\bm{\th}\|_1}$,
For the surfaces $S$ constructed in Subsection \ref{subsec:construction},
we have $\Ab(S)=\frac{1}{2}>\e$
and
\[
\sys(S)= %\epsilon_m=
\left(\frac{\e}{4\pi\|\bm{\th}\|_1}\right)^{-\chi-2}.
\]
%and
%\[
%\frac{2\pi(1-a)\|\bm{\th}\|_1}{16(1-a)|\bm{\th}\|_1+\log(4\pi\|\bm{\th}\|_1)+(1+o(\varepsilon))\log(1/\varepsilon)}
%\leq \Ext\sys(\dot{S})\leq
%\frac{2\pi(1-a)\|\bm{\th}\|_1}{\log(4\pi\|\bm{\th}\|_1)+(1+o(\varepsilon))\log(1/\varepsilon)}
%\]
%with $a=\frac{k+2}{k+2N+(5/2)}$.\\
Moreover, choosing $N,\e$ such that
$N\gg m$ and $\e<\exp\left[-8(4N+1)^2\right]$,
%$\frac{\log(4\pi\|\bm{\th_m}\|_1/\e_m)}{2N}\ll 1$, 
the ratio
\[
\Ext\sys(\dot{S})\frac{\log(4\pi\|\bm{\th}\|_1/\e)}{2\pi\|\bm{\th}\|_1}
=\Ext\sys(\dot{S})\frac{\log(1/\sys(S))}{2\pi\|\bm{\th}\|_1(-\chi-2)}
\leq 1
\]
can be made as close to $1$ as desired.
\end{lemma}

The systole of our surface $S$ was easily computed above.
To understand the extremal systole,
we consider the surfaces $\dot{S}_k$ obtained from $\dot{S}$
by filling the punctures $y_{k+1},\dots,y_m$ for all $k=1,\dots,m$.
We begin by estimating the extremal length of $\gamma_k$ inside $\dot{S}_k$
and we will show that in fact $\gamma_k$ realizes the extremal systole inside
$\dot{S}_k$ proceeding by induction on $k$.

\subsubsection{Planar model}

In order to reduce the problem to some well-known estimates of conformal moduli of plane annular domains, we will use
the following description of the complement in $S$ of
the geodesic arc $\alpha_{23}$ between $x_2$ and $x_3$.

\begin{lemma}[Planar model for $S\setminus\alpha_{23}$]\label{lemma:planar}
There is a biholomorphism between
$S\setminus\alpha_{23}$ and the unit disk $S':=\Delta$ such that
\begin{itemize}
\item[(a)]
the points $x_1$, $x_2$, $x_3$ in $S$ correspond to 
$x'_1=0$, $x'_2=1$, $x'_3=-1$ in $S'$
\item[(b)]
the point 
$y'_j=\tan(\epsilon^j/2)^{1/\th_1}$ in $(0,1)\subset S'$
corresponds to $y_j$ for all $j$
\item[(c)]
the orientation-reversing involution of $S\setminus\alpha_{23}$ corresponds to the
conjugation, the two shores of the cut in $S\setminus\alpha_{23}$ correspond to the two
arcs on $\pa S'$ between $1$ and $-1$,
and the arc $\alpha_{12}\subset S$ corresponds to $[0,1]\subset S'$.
\end{itemize}
\end{lemma}

\begin{center}
\begin{figurehere}
\psfrag{S'}{$\textcolor{Blue}{\dot{S}'}$}
\psfrag{Om}{$\textcolor{Blue}{\Omega}$}
\psfrag{x1}{$\textcolor{Sepia}{x'_1}$}
\psfrag{x2}{$\textcolor{Sepia}{x'_2}$}
\psfrag{x3}{$\textcolor{Sepia}{x'_3}$}
\psfrag{ym-1}{$\textcolor{Sepia}{y'_{m-1}}$}
\psfrag{ym}{$\textcolor{Sepia}{y'_m}$}
\psfrag{a23}{$\textcolor{blue}{\alpha_{23}}$}
\includegraphics[width=0.7\textwidth]{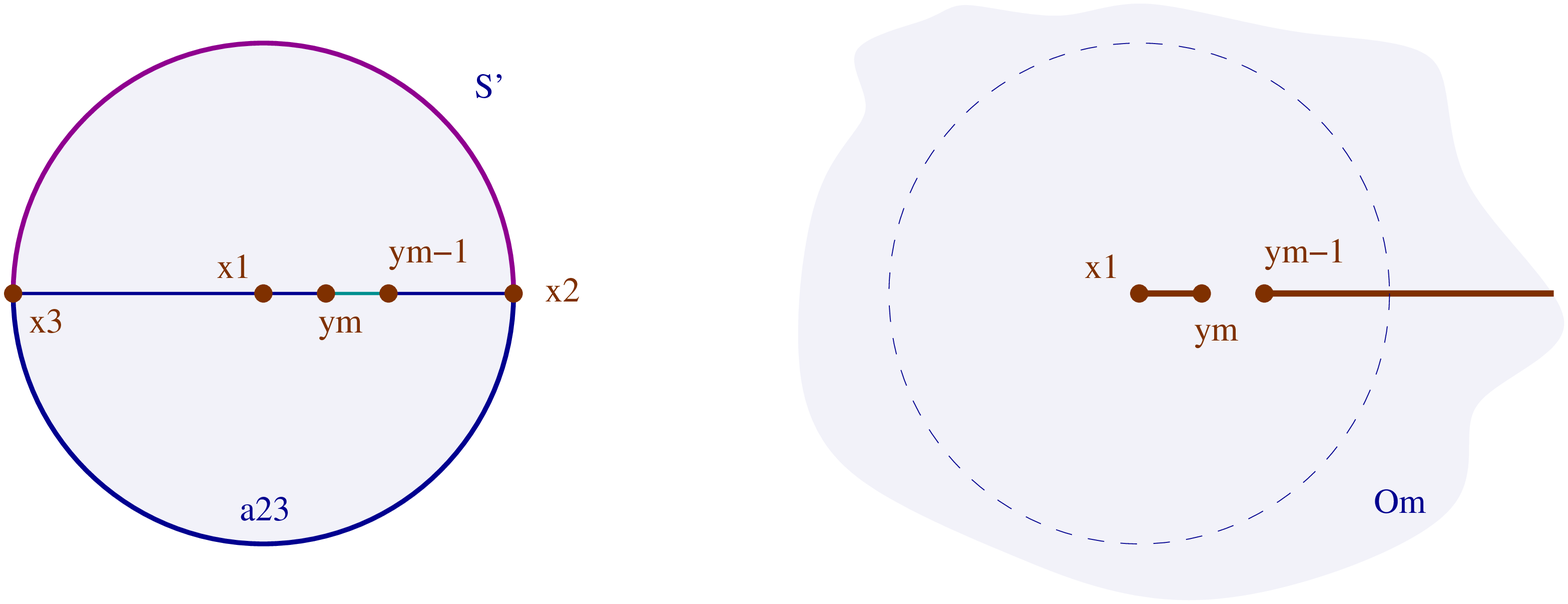}
\caption{{\small The planar model $\dot{S}'$ and the planar domain $\Omega$.}}\label{fig:example2}
\end{figurehere}
\end{center}

\begin{proof}[Proof of Lemma \ref{lemma:planar}]
Let $\CC$ with the natural coordinate $w$
be endowed with the standard 
spherical metric $\left(\frac{2 |dw|}{1+|w|^2}\right)^2$,
for which $\{|w|=1\}$ is a maximal circle.
A (multivalued) developing map from $\iota:S\setminus\alpha_{23}\rar
\CC$ that sends $x_1$ to the origin
has image equal to the open unit disk and it has order $\th_1$
at $x_1$. Thus, there exists a biholomorphism
$\psi:S\setminus\alpha_{23}\rar S'$ that sends $x_1$ to $0$
and such that $\iota=\iota'\circ\psi$, where
$\iota'(z)=z^{\th_1}$. Moreover $\psi$ can be uniquely chosen
so that $x_2$ corresponds to $x'_2=1$.

Since $[0,1)\subset S'$ running from $x'_1=0$ to $x'_2=1$
is sent to a geodesic by $\iota'$,
such segment corresponds to $\alpha_{12}\subset S$.
Thus the orientation-reversing involution of $S\setminus\alpha_{23}$ that fixes $\alpha_{12}$ must correspond to the conjugation in $S'$ and so $x_3\in S$ corresponds to $x'_3=-1$.

Finally, the point $w_j=\tan(\epsilon^j/2)$ in $\CC$ is at distance $\epsilon^j$ from $0$.
Hence,
the point $y'_j=\tan(\epsilon^j/2)^{1/\th_1}$ in $(0,1)\subset S'$ is at distance $\epsilon^j$ from $x'_1=0$,
and so $y'_j\in S'$ corresponds to $y_j\in S$.
\end{proof}

We denote by $\dot{S}'_k$ the punctured domain
$S'\setminus \{x'_1,x'_2,x'_3,y'_0,\dots,y'_k\}$
and by $\dot{S}'=\dot{S}'_m$.

\subsubsection{The Strebel differential $\q_k$ on $\dot{S}_k$}

In order to estimate the extremal lengths
in $\dot{S}_k$,
consider the Strebel differential $\q_k$ on $\dot{S}_k$
associated to $\gamma_k$, such that the total area of $\q_k$ is
exactly $\Ext_{\gamma_k}(\dot{S}_k)$.
%
%to consider a holomorphic quadratic differential $q_{k}$ on $\dot{S}_{k}$
%whose horizontal trajectories are either saddle connections or closed loops homotopic to $\gamma_{k}$, and such that the total area
%of $|q_{k}|$ is exactly $\Ext_{\gamma_{k}}(\dot{S}_{k})$.

\begin{lemma}\label{lemma:q_m}
The couple $(S,|\q_k|)$ 
is isomorphic to a doubled rectangle of vertices $x_1,y_k,y_{k-1},x_3$
with base $x_1 y_k$ of length $\frac{1}{2}\Ext_{\gamma_k}(\dot{S}_k)$ and 
height $x_1 x_3$ of length $1$. The horizontal trajectories of $\q_k$ are parallel to
the base. The points $y_{k-1},\dots,y_0,x_2$ lie in this order on the horizontal trajectory
running from $y_{k-1}$ to $x_3$.
\end{lemma}
\begin{proof}
The curve $\gamma_k$ is essential inside $\hat{S}_k=S_k\setminus\{x_1,y_k,y_{k-1},x_3\}$.
Consider the quadratic differential $\wh{\q}_k$ associated to $\gamma_k$ inside $\hat{S}_k$ such that $|\wh{\q}_k|$ has total area $\frac{1}{2}\Ext_{\gamma_k}(\dot{S}_k)$
(see Figure \ref{fig:example1} on the right, for the case $k=m$).
The analysis contained in Example \ref{example:strebel0,4} shows that 
$\hat{S}_k$ is isomorphic to $S_{r,\phi}$ with $\phi=0$ and $r^2=\frac{1}{2}\Ext_{\gamma_k}(\dot{S}_k)$, and so $(\hat{S}_k,|\wh{\q}_k|)$ is isometric to a doubled rectangle with corners
$x_1,y_k,y_{k-1},x_3$. Moreover, the orientation-reversing involution of $\dot{S}_k$ is an isometry and so conformal; hence, it fixes the metric $|\wh{\q}_k|$ and so it is the natural involution of the doubled rectangle.
In particular, $\alpha_{12}$ is fixed by the involution of the doubled rectangle and so it is the union
of the horizontal segments $x_1 y_k$ and $y_{k-1}x_3$
and the vertical segment $y_k y_{k-1}$. It is easy now to check that $\q_k=\wh{\q}_k$ and the conclusion easily follows.
\end{proof}

%
%Since $\dot{S}_m$ and the homotopy class of $\gamma_m$ are invariant under the
%above-mentioned orientation-reversing involution, it follows that $|\q_m|$ must be too
%and that $\alpha_{12}$ and $\alpha_{23}$ must be union of horizontal and vertical trajectories of $\q_m$.
%
%
%It is easy to see that $\q_{m}$ has simple poles at $x_1,y_m,y_{m-1},x_3$ and $(S,|\q_{m}|)$ is a flat surface with singularities, obtained by doubling a rectangle $ABCD$ with base $|CD|=|AB|$ of length $\frac{1}{2}\Ext_{\gamma_{k}}(\dot{S}_{k})$ and height $|AD|=|BC|$ of length $1$, in such a way that $(v_1,y_k,y_{k-1},v_3)$ correspond to $(A,B,C,D)$, and the ordered points $y_{k-1},y_{k-2},\dots,y_0,v_2,v_3$ all lie on the segment $CD$.
%Moreover, the horizontal trajectories of $q_{k}$ are parallel to $AB$ and $CD$ and the arc $\beta$ is a saddle connection that horizontally runs between $v_2$ and $v_3$.

We call $\q'_k$ the corresponding quadratic differential on $\dot{S}'_k$ 
and $\gamma'_k$ the curve in $\dot{S}'_k$ that corresponds to $\gamma_k$.

\subsubsection{Estimate for the extremal length of $\gamma_k$ in $\dot{S}_k$}

The extremal length of $\gamma_k\subset\dot{S}_k$
is estimated from above using the spherical metric and superadditivity of the modulus.
The bound from below uses the plane model of $\dot{S}\setminus\alpha_{23}$
and a classical estimate stating that the modulus of the annulus obtained from
$\CC$ by removing the two segments $[-1,0]$ and $[t-1,+\infty)$
is at most $\frac{1}{2\pi}\log(16t)$ for all real $t>1$ (see \cite[Chap.3B]{ahlfors:book}).

\begin{lemma}[Extremal length of $\gamma_k$]\label{lemma:ext-gamma}
The extremal length of $\gamma_k$ inside $\dot{S}_k$ satisfies
\[
%\frac{2\pi\th_1}{\log(1/\epsilon)}
%-O\left(\frac{\th_{m,1}^2}{\log^2(1/\epsilon)}\right)
%<
\frac{2\pi\th_1}{\th_1\log(16)+\log(1/\epsilon)+O(\epsilon)}
<\Ext_{\gamma_k}(\dot{S}_k)< \frac{2\pi\th_1}{\log(1/\epsilon)}
\]
and so in particular $\displaystyle \Ext_{\gamma_k}(\dot{S}_k)-\frac{2\pi\th_1}{\log(1/\epsilon)}=O\left(\frac{\th_1^2}{\log^2(1/\epsilon)}\right)$.
\end{lemma}

\begin{proof}
In order to bound $\Ext_{\gamma_k}(\dot{S}_k)$ from below,
consider the locus $C_k\subset\dot{S}_k$ of points at distance
between $\epsilon^k$ and $\epsilon^{k-1}$ from $x_1$.
By Lemma \ref{cylinderslength} we immediately have
$M(C_k)>\frac{1}{2\pi\th_1}\log(1/\epsilon)$ and so
$\Ext_{\gamma_k}(\dot{S}_k)<\frac{2\pi\th_1}{\log(1/\epsilon)}$.

%
%In $S'$ such cylinder corresponds to $C'_{k}=\{y'_k<z<y'_{k-1}\}$.
%Here we note that $\tan(\epsilon^k/2)=\epsilon(1-O(\epsilon))\tan(\epsilon^{k-1}/2)$
%for $0<\epsilon<1$ and so
%\[
%\log\left(\frac{y'_{k-1}}{y'_k}\right)=\frac{1}{\th_1}\log\left(\frac{\tan(\epsilon^{k-1}/2)}{\tan(\epsilon^k/2)}\right)
%=\frac{\log(1/\epsilon)+O(\epsilon)}{\th_1}.
%\]
%Thus, the modulus of $C'_{k}$ satisfies
%\[
%M(C'_{k})=\frac{1}{2\pi}\log\left(\frac{y'_{k-1}}{y'_k}\right)> \frac{1}{2\pi\th_1}\log(1/\epsilon).
%\]
%It follows that
%\[
%\Ext_{\gamma_{k}}(\dot{S}_{k})<\Ext_{\gamma_{k}}(C_{k})=\Ext_{\gamma'_{k}}(C'_{k})<\frac{2\pi\th_1}{\log(1/\epsilon)}
%\]
In order to bound $\Ext_{\gamma_k}(\dot{S}_k)$ from above, consider the quadratic differential $\q_k$
described in Lemma \ref{lemma:q_m}.
Since $\alpha_{23}$ is a horizontal segment, 
$\Ext_{\gamma_k}(\dot{S}_k)=\Ext_{\gamma_k}(\dot{S}_k\setminus\alpha_{23})=
\Ext_{\gamma'_k}(\dot{S}'_k)$ and it is achieved at the 
metric $|\q'_k|$.
Thus, $\Ext_{\gamma'_k}(\dot{S}'_k)=\Ext_{\gamma'_k}
\Big(\inte{S}'_k\setminus \left([0,y'_k]\cup[y'_{k-1},1]\right)\Big)$.
%, where
%$\Omega=\CC\setminus \left([0,z_k]\cup[z_{k-1},+\infty)\right)$.
Since $\Omega_k=\CC\setminus \Big([0,y'_k]\cup[y'_{k-1},+\infty)\Big)$ is biholomorphic to 
$\CC\setminus \left([-1,0]\cup[\frac{y'_{k-1}}{y'_k}-1,+\infty)\right)$,
it follows 
%from \cite[Chap.3B]{ahlfors:book} 
that
\begin{align*}
\Ext_{\gamma_k}(\dot{S}_k)^{-1} &=
M(\inte{S}'_k\cap\Omega_k)<
M(\Omega_k)\leq \frac{1}{2\pi}\log\left(\frac{16 y'_{k-1}}{y'_k}\right)
<\frac{1}{2\pi}\left(\log(16)+\frac{\log(1/\epsilon)+O(\epsilon)}{\th_1}\right)
\end{align*}
and so $\dis\Ext_{\gamma_k}(\dot{S}_k)\geq \frac{2\pi\th_1}{\th_1\log(16)+\log(1/\epsilon)+O(\epsilon)}$.
The last assertion is a straightforward calculation.
\end{proof}

\subsubsection{Estimate for the extremal systole of $\dot{S}_k$}

Now we inductively show that the extremal systole 
of $\dot{S}_k$ is realized at the closed curve $\gamma_k$.
%The conclusion will follow since $\dot{S}_m=\dot{S}$.

\begin{lemma}[Extremal systole of $\dot{S}_k$]\label{lemma:ext-sys-gamma}
Suppose that
$\e<\exp\left[-8(4N+1)^2\right]$.
%$\epsilon<\exp\left[-\left(\frac{2\pi\th_{m,1}}{\pi-2}\right)^2\right]$.
Then the extremal systole of the punctured surface $\dot{S}_k$
is achieved at $\gamma_k$ only.
% and so it satisfies
%\[
%\frac{2\pi\th_1}{\th_1\log(16)+\log(1/\epsilon)+O(\epsilon)}
%\leq\Ext\sys(\dot{S}_N)\leq \frac{2\pi\th_1}{\log(1/\epsilon)}.
%\]
\end{lemma}
\begin{proof}
Preliminarly observe that the total area of $S$ for the spherical metric is $2\pi\th_1$
and that our assumption on $\e$
implies that $\log\left(\frac{1}{\epsilon}\right)>8(4N+1)^2>
\left(\frac{2\pi\th_1}{\pi-2}\right)^2$.

Now, 
$\Ext_{\gamma_k}(\dot{S}_k)<\frac{2\pi\th_1}{\log(1/\epsilon)}$
by Lemma \ref{lemma:ext-gamma}.
We claim that, if $\gamma\subset\dot{S}_k$
is a simple closed curve not homotopic to $\gamma_k$, then
$\Ext_\gamma(\dot{S}_k)\geq \frac{2\pi\th_1}{\log(1/\epsilon)}$.
It will follow that $\Ext\sys(\dot{S}_k)$ is attained at $\gamma_k$ only.

Denote by $\dot{S}^+_k$ the region of $\dot{S}_k$ consisting of points at $|\q_k|$-distance less than $\frac{1}{2}$ from the segment $x_1 y_k$, by $\dot{S}^-_k$ the region of points at $|\q_k|$-distance less than $\frac{1}{2}$ from the segment $y_{k-1}x_3$ (see Figure \ref{fig:example1} on the right in the case $k=m$).

Let us prove the above claim by induction on $k\geq 1$.

Consider the case $k=1$.
The surface $\dot{S}_1$ has $5$ punctures
$x_1,x_2,x_3,y_0,y_1$
and the only closed curve in $\dot{S}^+_1$ is $\gamma_1$. 
A closed curve $\gamma\subset\dot{S}_1$ that cannot be deformed inside $\dot{S}^+_1$ or $\dot{S}^-_1$
must cross both segments $y_1 x_1$ and $y_0 x_3$, and so it must have
$|\q_1|$-length at least $2$.
Hence, 
\[
\Ext_\gamma(\dot{S}_1)\geq
\frac{4}{\Ext_{\gamma_1}(\dot{S}_1)}> 
\frac{2}{\pi\th_1}\log(1/\epsilon)
\geq \frac{2\pi\th_1}{\log(1/\epsilon)}
%>\Ext_{\gamma_1}(\dot{S}_1)
\]
where the first inequality on the left
follows from the very definition of extremal length using the metric $|\q_1|$,
the second inequality
relies on Lemma \ref{lemma:ext-gamma}
and the third inequality
is a rephrasing of $\epsilon<\exp(-\th_1\pi)$,
which follows from our assumption on $\e$.

Finally, a closed curve $\gamma$ contained
in $\dot{S}^-_1$ must be homotopic either
to $\gamma_{23}$ that separates $x_2,x_3$, or to $\gamma_{0,2}$
that separates $y_0,x_2$
or to $\gamma_{0,3}$ that separates $y_0,x_3$
from the other points.

Any curve homotopic to $\gamma_{2,3}$ 
(resp. $\gamma_{0,2}$, $\gamma_{0,3}$)
has spherical length at least $2$
(resp. at least $\pi-2$, at least $\pi$).
Hence, $\Ext_\gamma(\dot{S}_1)\geq \frac{(\pi-2)^2}{2\pi\th_1}$,
which is larger than $
\frac{2\pi\th_1}{\log(1/\epsilon)}$
by our observation at the very beginning of the proof.

%
%It can be shown that
%every closed curve in $\dot{S}_k$ that separates $x_2,x_3$ from the other marked points has length greater than $2$, and so
%extremal length greater than $\frac{4}{2\pi\th_1}$ by its very definition.
%Similarly, every closed curve in $\dot{S}_k$ that separates $x_2,y_0$ from the other marked points
%has length greater than $\pi-2$, and so its
%extremal length is greater than $\frac{(\pi-2)^2}{2\pi\th_1}$; every closed curve
%in $\dot{S}_k$ that separates
%$y_0,x_3$ from the other marked points
%has length greater than $\pi$ and so extremal length $\frac{\pi^2}{2\pi\th_1}$.
%Thus, the extremal length of these curves is at least $\frac{(\pi-2)^2}{2\pi\th_1}$,
%which is larger than $\frac{2\pi\th_1}{\log(1/\epsilon)}$.
%Such curves do not realize then the extremal systole in $\dot{S}_k$
%because 

Consider now the case $k>1$.
Again, every simple closed curve
in $\dot{S}^+_k$ is homotopic to $\gamma_k$.
Analogously to the case $k=1$,
a simple closed curve $\gamma\subset\dot{S}_k$
that is not homotopic to a curve in $\dot{S}^+_k$
or in $\dot{S}^-_k$
must have $|\q_k|$-length at least $2$
and so $\Ext_\gamma(\dot{S}_k)>\frac{2\pi\th_1}{\log(1/\epsilon)}$.
Finally, if $\gamma\subset\dot{S}^-_k$ is a simple closed curve, then
$\Ext_{\gamma}(\dot{S}_k)\geq \Ext_{\gamma}(\dot{S}_{k-1})\geq \frac{2\pi\th_1}{\log(1/\epsilon)}$ by induction.  
\end{proof}

We can now prove the main result of this section.

\begin{proof}[Proof of Lemma \ref{lemma:example}]
First note that $\Ext\sys(\dot{S})=\Ext_{\gamma_m}(\dot{S})$
by Lemma \ref{lemma:ext-sys-gamma}, because $\dot{S}=\dot{S}_m$.
By taking $N/m$ large enough, the ratio $\frac{\|\bm{\th}\|_1}{\th_1}>1$
can be made as close to $1$ as desired.
Moreover,
our assumption on $\e$ implies that
$\frac{\th_1}{\log(1/\epsilon)}<\frac{1}{2(4N+1)}$
and so this ratio can be made as small as desired by taking $N$ large enough.
Lemma \ref{lemma:ext-gamma} then ensures that $\Ext_{\gamma_m}(\dot{S})$ is
as close to $\frac{2\pi\th_1}{\log(1/\epsilon)}$,
and so to $\frac{2\pi \|\bm{\th}\|_1}{\log(1/\epsilon)}$,
as we wish.
%and that $\log(1/\e)<\log(1/\epsilon)$ and again their ratio is close to $1$ because $\log(1/\epsilon)\ll N_m$.
\end{proof}

%In comparing the above estimate with the systole inequality we notice that they have the same essential
%structure, except for the factor $-3\chi$ appearing in the systole inequality which is here replaced by
%the (weaker) $-\chi-2$.

%
%
%\begin{corollary}
%For $N\gg k$ and $\frac{\log(1/\epsilon)}{2N}\ll 1$ we have
%\[
%\exp\left[-\frac{2\pi\|\bm{\th}\|_1}{\Ext\sys(\dot{S}_k)}\right]^{-\chi-2}
%\Big/\sys(S,\bm{x})\leq 1
%\]
%and their ratio can be made as close to $1$ as desired.
%\end{corollary}
%\begin{proof}
%Another way of phrasing the inequality
%$\Ext\sys(\dot{S}_k)\Big\leq \frac{2\pi\|\bm{\th}\|_1}{\log(1/\e)}$
%in Proposition \ref{prop:example}
%is to say that
%\[
%\Ext\sys(\dot{S}_k)\leq \frac{2\pi\|\bm{\th}\|_1(-\chi-2)}{\log(1/\sys(S,\bm{x}))}.
%\]
%The claim then immediately follows.
%\end{corollary}
%
%

%\input{non-existence-almost-degenerate.tex}

%\input{non-existence-cylinder.tex}

%%%%%%%%%%%%%I%%%%%%

\setcounter{secnumdepth}{-1}

%\input section-glossary.tex

%\appendix

%\input draft-appendix.tex

\bibliographystyle{amsplain}
\bibliography{non-existence-PART1-biblio}

\end{document}